\documentclass[11pt, psamsfonts, reqno]{amsart}

\usepackage[T1]{fontenc}
\usepackage[latin1]{inputenc}
\usepackage[frenchb, english]{babel}
\usepackage{amsthm}
\usepackage{amscd}
\usepackage{tensor}
\usepackage{amssymb,amsmath,bbm}
\usepackage{yhmath}
\usepackage{tikz-cd}
\usepackage[all]{xy}
\usepackage{mathrsfs}
\usepackage[left=0.0cm,right=0.0cm,top=1.7cm,bottom=1.7cm]{geometry}
\usepackage{hyperref}
\usepackage{xcolor}
\usepackage{relsize}
\usepackage{comment}

\numberwithin{equation}{section}

\theoremstyle{plain}
\newtheorem{theorem}{Theorem}[subsection]

\newtheorem{lemma}[theorem]{Lemma}
\newtheorem{proposition}[theorem]{Proposition}
\newtheorem{prop}[theorem]{Proposition}
\newtheorem{corollary}[theorem]{Corollary}
\newtheorem{cor}[theorem]{Corollary}
\newtheorem*{theorem*}{Theorem}

\theoremstyle{definition}
\newtheorem{definition}[theorem]{Definition}
\newtheorem{example}[theorem]{Example}

\newtheorem*{claim}{Claim}
\newtheorem{construction}[theorem]{Construction}

\newtheorem{warning}[theorem]{Warning}

\theoremstyle{remark}
\newtheorem{remark}[theorem]{Remark}

\usepackage{tikz}
\usetikzlibrary{calc}

\def\Fil{\mathrm{Fil}}
\def\Hom{\mathrm{Hom}}

\def\Sym{\mathrm{Sym}}
\def\Hom{\mathrm{Hom}}
\def\iHom{R\underline{\mathrm{Hom}}}

\DeclareMathOperator{\colim}{colim}

\def\epsilon{\varepsilon}
\def\det{\mathrm{det}}
\def\Sh{{\operatorname{Sh}}}

\def\GL{\mathbf{GL}}

\def\Cond{\mathrm{Cond}}

\def\iHom{\underline{\mathrm{Hom}}}

\def\sol{\mathsmaller{\square}}

\def\Kos{\mathrm{Kos}}

\DeclareMathOperator{\Mod}{Mod}

\DeclareMathOperator{\Spa}{Spa}
\DeclareMathOperator{\Spec}{Spec}

\DeclareMathOperator{\gr}{gr}

\DeclareMathOperator{\Lie}{Lie}

\DeclareMathOperator{\Cont}{{Cont}}

\DeclareMathOperator{\Extdis}{Extdis}

\DeclareMathOperator{\id}{id}

\DeclareMathOperator{\NS}{NS}

\DeclareMathOperator{\Map}{Map}

\DeclareMathOperator{\End}{End}
\DeclareMathOperator{\Alg}{Alg}
\DeclareMathOperator{\CAlg}{CAlg}

\DeclareMathOperator{\AnRing}{AnRing}
\DeclareMathOperator{\AnCRing}{AnCRing}
\DeclareMathOperator{\Ani}{Ani}
\DeclareMathOperator{\SpecAn}{AnSpec}
\DeclareMathOperator{\AniRing}{AniRing}
\DeclareMathOperator{\Aff}{Aff}

\DeclareMathOperator{\pr}{pr}

\DeclareMathOperator{\Cat}{Cat}
\DeclareMathOperator{\CatLoc}{CatLoc}

\DeclareMathOperator{\CondRing}{CondRing}
\DeclareMathOperator{\AniAlg}{AniAlg}
\DeclareMathOperator{\AdicSp}{AdicSp}
\DeclareMathOperator{\AnSpec}{AnSpec}
\DeclareMathOperator{\AffRing}{AffRing}
\DeclareMathOperator{\Corr}{Corr}
\DeclareMathOperator{\LZ}{LZ}
\DeclareMathOperator{\PSh}{PSh}

\DeclareMathOperator{\Fl}{\n{F}\ell}
\DeclareMathOperator{\Nil}{Nil}
\DeclareMathOperator{\St}{St}
\DeclareMathOperator{\BUN}{BUN}

\DeclareMathOperator{\cone}{cone}

\DeclareMathOperator{\cofib}{cofib}

\DeclareMathOperator{\Prof}{Prof}
\DeclareMathOperator{\Pic}{Pic}
\DeclareMathOperator{\Pro}{Pro}
\DeclareMathOperator{\FinSet}{FinSet}

\newcommand{\n}[1]{\mathcal{#1}}
\newcommand{\bb}[1]{\mathbb{#1}}
\newcommand{\bbf}[1]{\mathbf{#1}}
\newcommand{\f}[1]{\mathfrak{#1}}
\newcommand{\s}[1]{\mathscr{#1}}

\DeclareMathOperator{\op}{\begin{scriptsize}
op
\end{scriptsize}}

\DeclareMathOperator{\red}{\begin{scriptsize}
red
\end{scriptsize}}
\DeclareMathOperator{\zar}{\begin{scriptsize}
zar
\end{scriptsize}}
\DeclareMathOperator{\cons}{\begin{scriptsize}
cons
\end{scriptsize}}
\DeclareMathOperator{\gen}{\begin{scriptsize}
gen
\end{scriptsize}}

\DeclareMathOperator{\an}{\begin{scriptsize}
an
\end{scriptsize}}
\DeclareMathOperator{\alg}{\begin{scriptsize}
alg
\end{scriptsize}}

\DeclareMathOperator{\ex}{\begin{scriptsize}
ex
\end{scriptsize}}

\title{The analytic de Rham stack in rigid geometry}

\author{Juan Esteban Rodr\'iguez Camargo}

\makeatletter
\def\@tocline#1#2#3#4#5#6#7{\relax
  \ifnum #1>\c@tocdepth 
  \else
    \par \addpenalty\@secpenalty\addvspace{#2}%
    \begingroup \hyphenpenalty\@M
    \@ifempty{#4}{%
      \@tempdima\csname r@tocindent\number#1\endcsname\relax
    }{%
      \@tempdima#4\relax
    }%
    \parindent\z@ \leftskip#3\relax \advance\leftskip\@tempdima\relax
    \rightskip\@pnumwidth plus4em \parfillskip-\@pnumwidth
    #5\leavevmode\hskip-\@tempdima
      \ifcase #1
       \or\or \hskip 1em \or \hskip 2em \else \hskip 3em \fi%
      #6\nobreak\relax
    \dotfill\hbox to\@pnumwidth{\@tocpagenum{#7}}\par
    \nobreak
    \endgroup
  \fi}
\makeatother
\usepackage{hyperref}

\begin{document}

\begin{abstract}
Applying the new theory of analytic stacks of Clausen and Scholze we introduce a general notion of  derived Tate adic spaces.  We use this formalism to define the  analytic de Rham stack in rigid geometry, extending  the theory of  $\wideparen{\n{D}}$-modules of Ardakov and Wadsley to the theory of analytic $D$-modules.  We prove some foundational results such as the existence of a six functor formalism and Poincar\'e duality for analytic $D$-modules, generalizing previous work of Bode. Finally, we relate the theory of analytic $D$-modules to previous work of the author with Rodrigues Jacinto on solid locally analytic representations of $p$-adic Lie groups. 
\end{abstract}

\maketitle
\tableofcontents

\section{Introduction}
\label{s:Intro}

\subsection*{Motivation}

The main objective of this paper is to geometrically  construct a six functor formalism for a suitable category of $D$-modules over rigid spaces. To further develop this idea let us  first recall some aspects of the classical theory of $D$-modules.

Let $K$ be a field of characteristic $0$ and $X$ a smooth scheme over $K$. Classically, the category of $D$-modules over $X$ is constructed by first defining a ring of algebraic differential operators $D_{X}$ over $X$, and then taking the category of $D_{X}$-modules whose underlying $\s{O}_X$-module is quasi-coherent. If $X$ admits an \'etale map to an affine space $X\to \bb{A}^d_{K}$, then $D_X$ can be explicitly constructed as the Weyl algebra 
\[
D_X= \s{O}_X[\partial_{T_1},\ldots, \partial_{T_d}],
\]
where $\bb{A}^d_K=\Spec K[T_1,\ldots, T_d]$, and the variables $\partial_{T_i}$ correspond to the partial derivations along the coordinate $T_i$.

In \cite{SimpsonDeRham,SimpsonTelemandeRham}, Simpson has proposed a different perspective on the theory of $D$-modules by the employ of stacks. Let $X$ be a smooth variety over $K$, Simpson attaches a space $X_{dR}^{alg}$ from commutative rings over $K$ to sets  whose theory of quasi-coherent sheaves is naturally isomorphic to the theory of $D$-modules over $X$. More concretely, let $\mathrm{Ring}_K$ be the category of commutative rings of finite type over $K$, then Simpson's de Rham stack is defined as the stack in the \'etale topology given by 
\[
X_{dR}^{alg}(R)= X(R^{\red}),
\] 
 where $R^{\red}$ is the reduction of the ring $R$. An advantage from this definition is that one can easily construct categories of $D$-modules for any variety over $K$ without the smooth assumption, namely the previous formula for the de Rham stack extends to any stack over $K$. The study of $D$-modules via the de Rham stack, and its application to geometric Langlands, can be found in the work of Gaitsgory and Rozenblyum \cite{GRdeRhamStack}. 
 
 Specializing to $p$-adic geometry, let $K$ be a non-archimedean extension of $\bb{Q}_p$ and let $X$ be a smooth rigid space over $K$. In the works \cite{ArdakovWadsleyII,ArdakovWadsleyDI} Ardakov and Wadsley have developed the theory of coadmissible $\wideparen{\n{D}}$-modules over rigid spaces. The departure point to define the category of coadmissible $\wideparen{\n{D}}$-modules is again a sheaf $\wideparen{\n{D}}_X$ of "infinite order $p$-adic differential operators over $X$". To describe this sheaf, let us suppose for simplicity that $X=\Spa A$ is affinoid and that we have an \'etale map $f:X\to \bb{D}^d_{K}$ to a  polydisc $\bb{D}^d_K=\Spa K\langle T_1,\ldots, T_d \rangle$. Then, as for schemes, we first consider the algebra of algebraic differential operators of $X$:
 \[
 D_X= A[\partial_{T_1},\ldots, \partial_{T_d}].
 \]
 The action of $\partial_{T_i}$ on $A$ by derivations is continuous, so we can find $N>0$ such that for all $n\geq N$ the left sub-$A^{\circ}$-module $A^{\circ}[p^n \partial_{T_1},\ldots, p^n \partial_{T_d}]\subset D_X$ is stable under multiplication. Thus, by taking $p$-adic completions $ \wideparen{\n{D}}_X^{(n)}:=A\langle p^n \partial_{T_1},\ldots, p^n \partial_{T_d} \rangle$, and taking limits along $n\to \infty$, one constructs the algebra of infinite order differential operators
 \[
 \wideparen{\n{D}}_X=\varprojlim_{n}\wideparen{\n{D}}_X^{(n)}.
 \]
 The algebra $ \wideparen{\n{D}}_X$ is known as a Fr\'echet-Stein algebra and its construction is motivated from the algebra of analytic distributions of $p$-adic Lie groups of Schneider-Teitelbaum \cite{SchTeitDist}.  In particular, there is a well defined category of coadmissible $\wideparen{\n{D}}_X$-modules given by the limit along pullbacks of the categories of finite type $\wideparen{\n{D}}_X^{(n)}$-modules. Categories of coadmissible  $\wideparen{\n{D}}_X$-modules have been extended using bornological vector spaces, and a six functor formalism for $\wideparen{\n{D}}_X$-modules has been constructed by Bode in \cite{bode2021operations}.

The theory of the analytic de Rham stack developed in this work is then a conciliation between the  geometric theory of $D$-modules of Simpson via the de Rham stack, and the theory of  $\wideparen{\n{D}}_X$-modules of Ardakov and Wadsley. To justify the tools used to construct the analytic de Rham stack let us start with an example. Let $K$ be a field of characteristic $0$, and  let $\bb{G}_a=\Spec K[T]$ be an affine space over $K$ seen as an additive group. It follows from the definition of the algebraic de Rham stack that 
\[
\bb{G}_{a,dR}^{alg}=\bb{G}_a/\widehat{\bb{G}}_{a}
\]
where $\widehat{\bb{G}}_a$ is the formal completion at $0\in \bb{G}_a$, acting by translations.  It turns out that the Cartier dual of  the stack $
*/\widehat{\bb{G}}_a$ is just $\bb{G}_a$, then, it is expected (and indeed the case) that modules over $\bb{G}_{a,dR}^{alg}$ are given by sheaves on $\bb{G}_a$ together with an operator $\partial_T$ (Cartier dual of $*/\widehat{\bb{G}_a}$), that is equivariant with respect to the module structure on $\bb{G}_a$ via the additive action of $\widehat{\bb{G}}_a$ on $\bb{G}_a$ (i.e. that $\partial_T$ acts by derivations).

Let us now take $K/\bb{Q}_p$ a non-achimedean extension, and let  $\bb{G}_a=\bb{D}^1_{K}=\Spa K\langle T\rangle$ be the open affinoid disc seen as an additive group. We would like to define an analytic de Rham stack $\bb{G}_{a,dR}$ whose theory of quasi-coherent sheaves is related to the theory of $\wideparen{\n{D}}$-modules. By construction, the sheaf $\wideparen{\n{D}}$ of infinite order $p$-adic differential operators has cotangent variables (i.e. the derivations $\partial_T$) that look like global sections of an analytic affine space $\bb{A}^{1,\an}_{K}$. Moreover, the category of coadmissible $\wideparen{\n{D}}$-modules is a non-commutative  analogue of the category of coherent sheaves on a relative analytic affine space (eg. $\bb{A}^{1,\an}_{K}$). Therefore, we would like to define the analytic de Rham stack in such a way that 
\[
\bb{G}_{a,dR}=\bb{G}_a/\bb{G}_a^{\dagger},
\]
with  $\bb{G}_a^{\dagger}\subset \bb{G}_a$ a subgroup acting by translations, and such that $*/\bb{G}_a^{\dagger}$ is the Cartier dual of $\bb{A}^{1,\an}_{K}$ (in a suitable sense). It is  known  that the continuous dual of $\s{O}(\bb{A}^{1,\an}_K)$ is given, as a Hopf algebra,   by the ring of germs of functions  at $0\in \bb{G}_a$
\[
K\{T\}^{\dagger}=\varinjlim_{n} K\langle \frac{T}{p^n} \rangle.
\]
Therefore, a reasonable candidate for $\bb{G}_a^{\dagger}$ would be given by 
\[
\bb{G}_a^{\dagger}= \Spa K\{T\}^{\dagger}. 
\]

Here is where several foundational problems appear. First, we are obligated to work with topological rings, and in order to have a good theory of analytic $D$-modules as quasi-coherent sheaves of a stack, we also need to work with topological modules. This problem is solved thanks to the theory of analytic geometry and condensed mathematics of Clausen and Scholze \cite{ClausenScholzeCondensed2019,ClauseScholzeAnalyticGeometry,CondensedComplex}. Second, if we ever expect to built up a six functor formalism of analytic $D$-modules  from the theory of complete modules of analytic rings, we need to construct categories of analytic stacks, and have a strong descent theory. Again, analytic geometry comes to the rescue, this time making use of the abstract theory  of six functor formalisms and $\s{D}$-stacks as in \cite{MannSix,MannSix2} and \cite{SixFunctorsScholze}. 

Once we have analytic stacks and complete modules at our disposal, we can make sense to objects such as $\bb{G}_a/\bb{G}_a^{\dagger}$, as well as to its category of complete  modules. However, in order to make a good definition of the analytic de Rham stack, we would need to mimic Simpson's construction of the algebraic de Rham stack. The key idea is that, while the space $\widehat{\bb{G}}_a$ represents the nilradical of a discrete ring $R$ (i.e. those elements $a$ such that $a^n=0$ for some $n$), the space $\bb{G}_a^{\dagger}$  represents a "$\dagger$-nilradical" for a suitable category of analytic rings $R$. In other words, $\bb{G}_a^{\dagger}$ represents elements $a\in R$ "of spectral norm zero", i.e.  such that $|a|\leq |p^n|$ for all $n\geq 0$. Thus, our first step is to restrict the theory of analytic rings to a theory of "bounded affinoid rings" $R$ for which we can construct a $\dagger$-nilradical $\Nil^{\dagger}(R)$.  Another motivation for the introduction of bounded affinoid rings arises from Tate algebras: we should only expect to construct an analytic de Rham stack for rings that look like affinoid Tate algebras, namely, for rings admitting a pseudo-uniformizer in a suitable sense. 

After bounded affinoid rings are introduced, we study some fundamental geometric properties of them in Sections \ref{SectionAnalyticAdicSpaces} and  \ref{SectionAnalyticAdicStacks}:  an analytic topology analogue to the analytic topology of adic spaces; a theory of derived Tate  adic spaces obtained by gluing bounded affinoid rings via rational localizations; the theory of the cotangent complex of analytic rings; different notions of morphisms of finite  presentation appearing in non-archimedean analytic geometry; a deformation theoretic description of smooth maps of morphisms of (solid) finite presentation of derived Tate adic spaces; Serre duality; a new deformation condition involving $\dagger$-nilradicals. This study on derived rigid geometry settles the basis for the theory of the analytic the Rham stack. 

Finally, once all the prerequisites in derived rigid geometry are done, we can start the study of the analytic de Rham stack. Let $X$ be a rigid space over $\bb{Q}_p$, and let $\Aff_{\bb{Q}_p}^{b}$ be the ($\infty$-)category  of bounded affinoid rings over $\bb{Q}_p$, then the analytic de Rham stack $X_{dR}$ will be defined as a suitable sheafification of the prestack mapping $\n{A}\in \Aff_{\bb{Q}_p}^b$ to 
\[
X_{dR}(\n{A})=X(\n{A}^{\dagger-\red}),
\] 
where $\n{A}^{\dagger-\red}:=\n{A}/\Nil^{\dagger}(\n{A})$ is the "$\dagger$-reduction of $\n{A}$". Our workhorse to prove properties on the analytic de Rham stack, such as the existence of six functors, Poincar\'e duality, and the construction of the Hodge filtration, will be a new theory of Cartier duality for analytic vector bundles over derived Tate adic spaces in Section \ref{SectionCartierDualityVectorBundles}. The main theory of the analytic de Rham stack, in particular the construction of six functors for analytic $D$-modules,  is the content of Section \ref{SectionAnalyticDeRham}.  We finish the paper with a generalization of the analytic de Rham stack of smooth morphisms in the equivariant setting in Section \ref{SectionAnalyticdRAndLocAn}, obtaining a generalization of equivariant $\wideparen{\n{D}}$-modules of Ardakov \cite{ArdakovEquivariantD}, and of the theory of solid locally analytic representations of $p$-adic Lie groups of  \cite{RRLocallyAnalytic,RJRCSolidLocAn2}.

\subsection*{Overview of the paper}

The body of the paper is divided in two main parts. First, we develop the theory of  derived Tate adic spaces and Tate stacks in Sections \ref{SectionAnalyticAdicSpaces} and \ref{SectionAnalyticAdicStacks}. The second part consists on Sections \ref{SectionCartierDualityVectorBundles}, \ref{SectionAnalyticDeRham} and \ref{SectionAnalyticdRAndLocAn}, where   we study different incarnations of Cartier duality of vector bundles, we define the analytic de Rham stack for Tate stacks over $\bb{Q}_p$, and finally we relate the theory of the analytic de Rham stack with the theory of locally analytic representations of $p$-adic Lie groups. 

\subsubsection*{\S \ref{SectionAnalyticAdicSpaces}  Derived  Tate adic spaces}

The theory of adic spaces of Huber \cite{HuberAdicSpaces} have been the language for non-archimedean analytic geometry  in  the last few decades, a weakness of this category are the restrictions imposed in the definition of complete Huber pairs. The main goal of this section is the introduction of  the $\infty$-category of  \textit{bounded affinoid rings} (Definition \ref{DefinitionBoundedAffinoidRings}),  generalizing the category of Tate Huber pairs, over which we can do both analytic and derived algebraic geometry. Similarly as for Huber pairs, one can construct a spectral space $\Spa \n{A}$ for any bounded affinoid ring $\n{A}$ (Definition \ref{DefinitionAdicSpectrumBoundedRing}),  generalizing the adic spectrum of an Huber ring. Maps between bounded affinoid rings will give rise to spectral maps of the adic spectra, and solid quasi-coherent modules will satisfy descent for the analytic topology. A bounded affinoid ring has the feature that any function is \textit{bounded} in the sense of the theorem down below. Moreover, one can define condensed and $\dagger$-nilradicals for these rings, consisting on ideals of uniformly nilpotent or overconvergently close to zero elements respectively.  The following  summarizes the main results (cf. Propositions \ref{PropRepresentedNilRadical}, \ref{PropositionConstructionsareSolid}, \ref{PropStabilityAffinoidRings}, \ref{PropAdicSpectralIsTopological} and \ref{PropAdicSpectrumInvariantReduced}).  

\begin{theorem}
\label{TheoremBoundedAffinoidIntro}
Let $R=\bb{Z}((\pi))$ be the Huber ring parametrizing pseudo-uniformizers in Tate Huber pairs. There is a full subcategory  $\AffRing^{b}_{R}\subset \AnRing_R$ of the $\infty$-category of analytic $R$-algebras, called the category of bounded affinoid rings,   containing fully faithfully the $1$-category of Tate Huber pairs over $R$. The category  $\AffRing^{b}_{R}$  is stable under small colimits in $\AnRing_{R}$.  Furthermore, let $\n{A}\in \AffRing^{b}_R$, the following hold

\begin{enumerate}
\item  There is an animated subring $\n{A}^{+}\subset \underline{\n{A}}(*)$ such that an $\underline{\n{A}}$-module is $\n{A}$-complete if and only if it is $\bb{Z}[a]_{\sol}$-complete for all $a\in \n{A}^+$. 

\item For any map $\bb{Z}[T]\to \n{A}$ of analytic rings, there is some $n\in \bb{N}$ and an extension to the Tate algebra  $R\langle \pi^{n} T\rangle \to \n{A}$. In other words, any element  $a\in \n{A}$ is \textit{bounded}. 

\item The ring $\n{A}$ has a condensed nil-radical $\Nil(\n{A})$ whose $S$ points (for $S$ profinite)  are given by maps $S\to \n{A}$ which are uniformly nilpotent. Furthermore,  the analytic ring structure on $\n{A}$ is already determined by the analytic ring structure of the quotient $\n{A}^{\red}:= \n{A}/\Nil(\n{A})$, and $\Nil(\n{A}^{\red})=0$.

\item The ring $\n{A}$ has  a $\dagger$-nil-radical $\Nil^{\dagger}(\n{A})$ whose $S$-points (for $S$ profinite)  are given by maps $S\to \n{A}$ which are overconvergently close to zero.  Furthermore, the analytic ring structure on $\n{A}$ is already determined by the analytic ring structure of the quotient $\n{A}^{\dagger-\red}:= \n{A}/\Nil^{\dagger}(\n{A})$, and $\Nil^{\dagger}(\n{A}^{\dagger-\red})=0$.

\item There is a spectral space $\Spa \n{A}$ endowed with an analytic topology with a basis given quasi-compact rational  subspaces, generalizing Huber's construction of  $\Spa(A,A^{+})$. Moreover, any map $\n{A}\to \n{B}$ of bounded affinoid rings gives rise a spectral morphism $\Spa \n{B}\to \Spa \n{A}$ preserving rational localizations. 

\item We have homeomorphisms of spectral spaces $\Spa \n{A}=\Spa \n{A}^{\red}= \Spa \n{A}^{\dagger-\red}$. 

\item The $\infty$-category $\Mod(\n{A})$ of solid $\n{A}$-modules satisfies descent for the analytic topology of $\Spa \n{A}$. 

\end{enumerate}  
\end{theorem}

Having stated some basic properties for the category of bounded affinoid rings one can formally defined the category of derived Tate adic space by gluing bounded affinoid rings along open covers (Definition \ref{DefinitionAnalyticDerivedAdicSpace}). 

\begin{definition}
We let $\Aff^{b}_R$ denote be the opposite category of bounded affinoid rings, an object in $\Aff^{b}_R$ is called a \textit{bounded affinoid space}. Given $\n{A}$ a bounded affinoid ring we let $\AnSpec \n{A}\in \Aff^{b}_{R}$ be its \textit{analytic spectrum}. A \textit{derived Tate adic space} is a sheaf on anima $X:\Aff^{b,\op}_{R}\to \Ani$ for the analytic topology of $\Aff^{b}_R$, such that $X$ admits an open analytic cover by bounded affinoid spaces. Given $X$ a  derived  Tate adic space we let $|X|=\varinjlim_{\AnSpec \n{A}\to X} |X|$ denote its underlying topological space.  We let $\AdicSp_{R}$ denote the $\infty$-category of derived Tate adic spaces over $R$. 
\end{definition}

The previous definition of derived Tate adic spaces is an extension of Huber's analytic adic spaces. The choice of the analytic topology to glue bounded affinoid rings is an arbitrary choice that was taken  to compare with the  classical theory. We will see in \S \ref{SectionAnalyticAdicStacks} that one can still do geometry in different kind of Grothendieck topologies, as long as one has descent for the six functor formalism of quasi-coherent sheaves.

\subsubsection*{\S \ref{SectionAnalyticAdicStacks} Tate stacks}

In this section we continue developing the  theory of derived Tate adic spaces. In \S \ref{SubsectionAbstractSix} we recall some language in the theory of abstract six functor formalisms of \cite{MannSix,MannSix2} and \cite{SixFunctorsScholze}. Then in \S \ref{SubsectionSolidStacks}, we use \cite[Theorem 4.20]{SixFunctorsScholze} to construct a very large six functor formalism on a category of analytic $\s{D}$-stacks on bounded affinoid rings (Definitions \ref{DefinitionSolidStacks} and \ref{DefinitionAdicStacks}).  We define different notions of morphisms of finite presentations in analytic rings  in \S \ref{SubsectionMorphismsFinitePresentation},  study basic properties of cotangent complexes in analytic rings in \S \ref{SubsectionCotangent}, and  introduce the notion of solid  smooth and \'etale maps of derived Tate adic spaces in \S \ref{SubsectionFormalleEtaleSmooth}. We obtain  an equivalent description of solid smooth and \'etale maps as formally smooth and \'etale maps of solid finite presentation (Theorem \ref{TheoFormalSmoothnesvsSmoothness} and Corollary \ref{CoroEtaleMapsStacks}).

\begin{theorem}
Let $f: X\to Y $ be a morphism of solid finite presentation of derived Tate adic spaces over $R=\bb{Z}((\pi))$. The following are equivalent:
\begin{enumerate}
\item The map $f$ is formally smooth (resp. \'etale).

\item The map $f$ is solid smooth (resp. \'etale), namely, locally in the analytic topology of $X$ and $Y$ the map $f$ is standard solid smooth (resp. standard solid \'etale).

\end{enumerate}
\end{theorem}

We continue by studying some properties of the categories of modules of derived Tate adic spaces with the goal of proving Serre duality (following an argument of  Clausen and Scholze by deformation to the normal cone). We have the following result (Theorem \ref{TheoSerreDuality}).

\begin{theorem}
Let $f:X\to Y$ be a morphism locally of solid finite presentation of derived Tate adic spaces. The following hold:

\begin{enumerate}

\item The map $f$ admits $!$-functors in the six functor formalisms of solid quasi-coherent sheaves.

\item If $f$ is solid smooth (resp. \'etale) then $f$ is cohomologically smooth (resp. cohomologically \'etale) for the six functor formalism of quasi-coherent sheaves. Furthermore, there is a natural identification $f^{!} 1_Y= \Omega^{d}_{X/Y}[d]$ where $d$ is the relative dimension of $f$, and $1_Y$ is the unit in the category of solid quasi-coherent sheaves on $Y$ (i.e. the structural sheaf).
\end{enumerate}
\end{theorem}

Finally, we introduce a new deformation condition called $\dagger$-formally smoothness and \'etaleness, related with liftings along $\dagger$-nilpotent ideals such as $\Nil^{\dagger}(\n{A})$. Then, we prove the following lifting property for solid smooth and solid \'etale maps (Proposition \ref{PropFormallyInftSmoothEtale}).

\begin{proposition}
\label{TheoFormallySmoothIntro}
Let $f:X\to Y$ be a solid smooth (resp \'etale) map of derived Tate adic spaces over $R$. Then $f$ is $\dagger$-formally smooth (resp. \'etale) locally in the analytic topology of $X$ and $Y$. 
\end{proposition}

In the next sections we apply all the previous theory on derived rigid geometry to study Cartier duality of vector bundles, and to construct a six functor formalism for analytic $D$-modules.

\subsubsection*{\S \ref{SectionCartierDualityVectorBundles} Cartier duality for vector bundles}

  With the introduction of derived Tate adic spaces, new commutative group objects appear in the nature  making possible  new incarnations of Cartier duality. In this section we do not pretend  to give a definition or an abstract set up for a Cartier duality theorem, instead  we explore new examples of Cartier duality arising from  analytic subspaces of vector bundles. For technical reasons, these Cartier duality isomorphisms are easier to describe if we restrict ourselves to analytic geometry over $\bb{Q}_p$, for simplicity let us even restrict ourselves to derived Tate adic spaces over $\bb{Q}_{p}$.  We start with the definition of the analytic incarnations of vector bundles (Construction \ref{ConstructionAnalytificationVB})

\begin{definition}
Let $X$ be a derived Tate adic space over $\bb{Q}_p$ and $\s{F}$  a vector  bundle over $X$ of rank $d$. We let $\bb{V}(\s{F})^{\an}$ be its geometric realization as an derived Tate adic space over $X$. Let $\iota: X\to \bb{V}(\s{F})^{\an}$ be the zero section, and let $\bb{V}(\s{F})^{\dagger}$ denote the overconvergent neighbourhood of zero. 
\end{definition}

So, if $X=\AnSpec \bb{Q}_p$ is a point, and $\s{E}$ is free of rank $1$, the space $\bb{V}(\s{F})^{\an}$ is isomorphic to the affine analytic line $\bb{A}^{1,\an}_{\bb{Q}_{p}}$. Similarly,   $\bb{V}(\s{F})^{\dagger}$ is nothing but the space $\bb{G}_{a,\bb{Q}_p}^{\dagger}$ given by the analytic spectrum of the algebra $\bb{Q}_{p}\{T\}^{\dagger}=\varinjlim_n \bb{Q}_p\langle \frac{T}{p^{n}}\rangle$ of functions that overconverge at $0\in \bb{A}^{1,\an}_{\bb{Q}_p}$. The following theorem describes the theory of six functors for analytic vector bundles and their classifying stacks (Proposition \ref{PropAlgebraicCartier2} and Theorem \ref{TheoremAnalyticCartierII}).

\begin{theorem}
Let $X$ be a derived Tate adic space over $\bb{Q}_p$ and let $\s{E}$ be a vector bundle over $X$ of rank $d$. 
\begin{enumerate}

\item Let $f: \bb{V}(\s{F})^{\an}\to X$, then $f$ is cohomologically smooth and there are natural equivalences $f^{!} 1_X = f^* \bigwedge^{d} \s{F}^{\vee} [d]$ and $f_!f^{!} 1_X=  \Sym^{\dagger}_X(\s{F})$, where $ \Sym^{\dagger}_X(\s{F})$ is the sheaf of functions of $\bb{V}(\s{F}^{\vee})^{\dagger}$.

\item  The map $f: \bb{V}(\s{F})^{\dagger}\to X$ satisfies $f_*=f_!$. 

\item Let $g: X/\bb{V}(\s{F})^{\an} \to X$. Then $g$ is cohomologically smooth with $g^{!} 1_X = g^{*} \bigwedge^{d} \s{F} [-d]$, and there is a natural equivalence $g_! \cong g_* [-2d]$.

\item  Let $g:X/\bb{V}(\s{F})^{\dagger}\to X$. Then $g$ is cohomologically smooth with $g^{!} 1_X = g^{*} \bigwedge^{d} \s{F} [d]$, and there is a natural equivalence $g_!\cong g_*$. 

\end{enumerate}
\end{theorem}

The analytic Cartier duality theorem is the following statement (Theorems \ref{TheoAlgebraicCartierDuality} and \ref{TheoremAnalyticCartierII}).

\begin{theorem}
\label{TheoCartierDualityIntro}
Let $X$ be a derived Tate adic space over $\bb{Q}_p$ and let $\s{F}$ be a vector bundle over $X$ of rank $d$. 

\begin{enumerate}

\item There is a bilinear morphism  $F:\bb{V}(\s{F})^{\dagger}\times (X/\bb{V}(\s{F}^{\vee})^{\an})\to B\bb{G}_m$ such that $F^{*}(\s{O}(1))$ is an isomorphism $\bb{V}(\s{F})^{\dagger}\xrightarrow{\sim} X/\bb{V}(\s{F}^{\vee})^{\an}$ in the category of Fourier-Moukai kernels for the six functor formalisms of solid quasi-coherent sheaves. Furthermore, the inverse of $F^{*}(\s{O}(1))$ is given by $F^{*}(\s{O}(-1))\otimes\bigwedge^{d}\s{F}^{\vee}[-d]$.

 \item  There is a bilinear morphism $G:\bb{V}(\s{F})^{\an}\times (X/\bb{V}(\s{F}^{\vee})^{\dagger})\to B\bb{G}_m$ such that $G^{*}(\s{O}(1))$ is an isomorphism $\bb{V}(\s{F})^{\an}\xrightarrow{\sim} X/\bb{V}(\s{F}^{\vee})^{\dagger}$ in the category of Fourier-Moukai kernels for the six functor formalisms of solid quasi-coherent sheaves. Furthermore, the inverse of $G^{*}(\s{O}(1))$ is given by $G^{*}(\s{O}(-1))\otimes\bigwedge^{d}\s{F}^{\vee}[d]$.

\end{enumerate}
\end{theorem}

\begin{corollary}
In the notation of Theorem \ref{TheoCartierDualityIntro},  there are natural equivalences of stable $\infty$-categories given by Fourier-Mukai transforms
\[
FM_1: \Mod_{\sol}(\bb{V}(\s{F})^{\dagger})\xrightarrow{\sim} \Mod_{\sol}(X/\bb{V}(\s{F}^{\vee})^{\an}). 
\]
and 
\[
FM_2: \Mod_{\sol}(\bb{V}(\s{F})^{\an})\xrightarrow{\sim} \Mod_{\sol}(X/\bb{V}(\s{F}^{\vee})^{\dagger}).
\]
\end{corollary}

We also review algebraic Cartier duality in Theorem \ref{TheoAlgebraicCartierDuality}, and show two other versions of analytic Cartier duality in Theorems \ref{TheoremCartierDualityAnalyticI} and \ref{TheoremCartierDualityAnalyticIII}.

\subsubsection*{\S \ref{SectionAnalyticDeRham} Algebraic and analytic de Rham stacks}

The first construction of the de Rham stack dates back to Simpson in his papers \cite{SimpsonDeRham,SimpsonTelemandeRham}. In this work we propose an analogue of this construction in analytic geometry, more precisely in rigid analytic geometry over $\bb{Q}_p$. We begin by extending the construction of the algebraic de Rham stack from algebraic geometry to condensed mathematics. Specialized to derived Tate adic spaces we get the following (Definition \ref{DefinitionAlgdeRhamStack}).

\begin{definition}[Algebraic de Rham stack]
Let  $R=\bb{Z}((\pi))$ and let $X$ be a derived Tate adic space over $R\otimes \bb{Q}$. The  algebraic de Rham prestack of $X$ is the presheaf on $\Aff^{b}_{R\otimes \bb{Q}}$ given by 
\[
X_{dR}^{\alg}(\n{A})=\varinjlim_{I\to \n{A}} X(\cone(I\to \n{A} ))
\]
where $I$ runs over all the uniformly nilpotent ideals of $\n{A}$. The de Rham stack of $X$ is the sheafification of the de Rham prestack in the $\s{D}$-topology, and we also denote it by $X_{dR}^{\alg}$.  Given a morphism $X\to Y$ of derived Tate adic spaces, the relative algebraic de Rham stack $X_{dR,Y}^{\alg}$ is the pullback 
\[
\begin{tikzcd}
X_{dR,Y}^{\alg} \ar[r] \ar[d] & Y \ar[d] \\
X_{dR}^{\alg} \ar[r] & Y_{dR}^{\alg}.
\end{tikzcd}
\]
We call $\Mod_{\sol}(X^{\alg}_{dR,Y})$ the category of algebraic $D_{X/Y}$-modules. 
\end{definition}

In analogy to the algebraic de Rham stack, and in view that there is a second kind of nil-radical in the category of bounded affinoid rings (Theorem \ref{TheoremBoundedAffinoidIntro} (4)), we define the analytic de Rham stack as follows (Definition \ref{DefinitionAnalyticDeRhamStack}).

\begin{definition}[Analytic de Rham stack]
Let $X$ be a derived Tate adic space over $\bb{Q}_p$, the analytic de Rham prestack is the presheaf on $\Aff^{b}_{\bb{Q}_p}$ defined by 
\[
X_{dR}(\n{A}):= X(\n{A}^{\dagger-\red}). 
\]
The analytic de Rham stack is the $\s{D}$-sheafification of the analytic de Rham prestack, and we also denote it by $X_{dR}$.  Let $f:X\to Y$ be a map of derived Tate adic spaces, we define the relative de Rham stack $X_{dR,Y}$ to be the pullback 
\[
\begin{tikzcd}
X_{dR,Y} \ar[r] \ar[d] & Y \ar[d] \\
X_{dR} \ar[r] & Y_{dR}.
\end{tikzcd}
\]
We call $\Mod_{\sol}(X_{dR,Y})$ the category of analytic $D_{X/Y}$-modules. 
\end{definition}

The main theorem on de Rham stacks is the following (see Corollaries \ref{CoroKashiwaraEquivalence} and \ref{CorollaryComparisonDeRhamCohomology},  Propositions \ref{PropAlgebraicDeRham} and \ref{PropAnDeRhamStack},  and Theorems \ref{TheoSixFunctorsAlgebraicDMod}, \ref{TheoHodgeFiltration}, \ref{TheoSixFunctorsanDmodules}, \ref{TheoPoincareDualityDmodules} and \ref{TheoremSurjectionXdRRigid}).

\begin{theorem}
\label{TheoSixFunctorDeRhamStackIntro}
 Let $f:X\to Y$ be a morphism of derived Tate adic spaces locally of solid finite presentation and write $f_{dR}^{\alg}: X_{dR}^{\alg}\to Y_{dR}^{\alg}$ and $f_{dR}: X_{dR}\to Y_{dR}$  for the associated maps at the level of stacks. 

\begin{enumerate}

\item The formation of $X\mapsto X_{dR}$ and $X\mapsto X^{\alg}_{dR}$ commutes with colimits and finite limits at the level of prestacks.

\item  The maps $f_{dR}^{\alg}$ and $f_{dR}$ admit $!$-functors.

\item Suppose that $X$ is a rigid space over a non-archimedean extension $K/\bb{Q}_p$, then the map $h:X\to X_{dR,K}$ is a $\s{D}$-cover. In particular, quasi-coherent sheaves on $X_{dR,K}$ descent along $h$. 

\item (Kashiwara equivalence) Let $X\to Y$ be a Zariski closed immersion of derived Tate adic spaces, and let $Y^{\dagger/X}$ be the overconvergent neighbourhood of $X$ in $Y$. Then there is an equivalence of analytic de Rham stacks $X_{dR}= Y^{\dagger/X}_{dR}$. In particular, analytic $D$-modules of $Y$ supported on $X$ are equivalent to analytic $D$-modules of $X$.

\item Suppose that $f$ is solid smooth (resp. \'etale), then the maps $f_{dR}^{\alg}$ and $f_{dR}$ are cohomologically smooth (resp. \'etale).

\item Suppose that $f$ is solid smooth, then we have natural equivalences $f_{dR}^{\alg,!} 1_{Y_{dR}^{\alg}}= 1_{X_{dR}^{\alg}}$ and $f_{dR}^{!} 1_{Y_{dR}} = 1_{X_{dR}}[2d]$ where $d$ is the relative dimension of $f$. 

\item Suppose that $f$ is solid smooth and consider the map $g: X_{dR,Y}^{\alg}\to X_{dR,Y}$, then $g$ admits $!$-functors, $g_*$ satisfies the projection formula, and there is a natural equivalence
\[
 1_{X_{dR,Y}}\xrightarrow{\sim} g_{*} 1_{X_{dR,Y}^{\alg}}.
\]
In particular, the pullback functor $g^{*}$ is fully faithful and induces an embedding of analytic $D_{X/Y}$-modules into algebraic $D_{X/Y}$-modules.

\item Suppose that $f$ is solid smooth, and denote $f_{dR,Y}^{\alg}:X_{dR,Y}^{\alg}\to Y$ and $f_{dR,Y}:X_{dR,Y}\to Y$ the natural maps. Then there are natural equivalences of de Rham cohomology 
\[
DR(X/Y):= f_{dR,Y,*}^{\alg} 1_{X_{dR,Y}^{alg}} = f_{dR,Y,*} 1_{X_{dR,Y}}.
\] 
Furthermore, $DR(X/Y)$ can be naturally promoted to a filtered object given by the Hodge filtration, with graded pieces
\[
\gr^{i} DR(X/Y)= \Omega^{i}_{X/Y}, 
\]
extending the Hodge filtration for smooth maps of rigid spaces. 
\end{enumerate}

\end{theorem}

\begin{remark}
In order to prove Theorem \ref{TheoSixFunctorDeRhamStackIntro} we need to consider a variation of the de Rham stacks given by the filtered de Rham stacks $X_{dR^+}^{\alg}$ and $X_{dR^+}$, see Definitions \ref{DefinitionAlgdeRhamStack} and \ref{DefinitionAnalyticDeRhamStack}. 
\end{remark}

\subsubsection*{\S \ref{SectionAnalyticdRAndLocAn} Analytic de Rham stack and locally analytic representations}

Finally we end with the relation between the analytic de Rham stack, the theory of solid locally analytic representations of $p$-adic Lie groups and a general notion of equivariant analytic $D$-module, generalizing definitions of \cite{ArdakovEquivariantD}.  

We first briefly discuss the relation with representation theory. Let $G$ be a $p$-adic Lie group, and denote by $G^{la}$ and $G^{sm}$ the analytic spaces defined by $G$ endowed with the sheaf of locally analytic and locally constant functions respectively. In \cite{RJRCSolidLocAn2} we proved that the category of solid locally analytic and smooth representations are given by the category of solid quasi-coherent sheaves of the classifying stacks $*/G^{la}$ and $*/G^{sm}$ respectively. A first relation between $p$-adic Lie groups, representation theory and the de Rham stacks is encoded in the following proposition (see Lemma \ref{LemmaRelationLieGroupsLocAn}).

\begin{proposition}
Let $G$ be a $p$-adic Lie group. There is a natural equivalence
\[
G^{la}_{dR}=G^{sm}.
\]
In particular, we have an equivalence of classifying stacks 
\[
(*/G^{la})_{dR}=*/G^{sm}.
\]
\end{proposition}

Next, both solid locally analytic representations and analytic $D$-modules extend  to a theory of equivariant  analytic $D$-modules. To motivate the definition let us make the following observation.

\begin{remark} Let $f:X\to Y$ be a solid smooth map of derived Tate adic spaces. By Proposition \ref{TheoFormallySmoothIntro} the map $h:X\to X_{dR,Y}$ is an epimorphism of $\s{D}$-stacks. The \v{C}ech nerve of  $h$ is given by the \textit{analytic de Rham groupoid}, whose $n$-th is the overconvergent neighbourhood of  the diagonal map $\Delta^{n+1}_{Y} X\to X^{\times_Y n+1}$, namely, the analytic space $(\Delta^{n+1}_{Y} X)^{\dagger} $ whose functions are given by functions of $X^{\times_Y n+1}$ that overconverge the locally closed subspace $|\Delta^{n+1}_{Y}|\subset |X^{\times_Y n+1}|$. This provides an equivalence of $\s{D}$-stacks
\[
X_{dR,Y}:= \varinjlim_{[n]\in \Delta^{\op}} (\Delta^{n+1}_{Y} X)^{\dagger}  
\]
\end{remark}

We extend the analytic de Rham groupoid to \textit{$\dagger$-smooth groupoids} in Definition \ref{DefinitionSmoothDaggerGroupoids}. Roughly speaking, these are groupoid objects that look like the overconvergent neighbourhoods of the zero sections  of vector bundles.  Prototipical examples of $\dagger$-smooth groupoids are constructed from Lie algebroids over rigid spaces as explained in Example \ref{ExampleDaggerGroupoids} (3).   We recall the definition of a normal map in  groupoids, and introduce equivariant analytic $D$-modules in great generality (Definition \ref{DefinitionEquivariantDmodules}).

\begin{definition}
Let $X$ be a derived Tate adic space over $\bb{Q}_p$ and $G$ a $p$-adic Lie group acting locally analytically on $X$. Let $\bb{H}^{\dagger}$ be a $\dagger$-smooth group over $X$ and let $\bb{H}^{\dagger}\to G^{la}\times X$ be a  map of groupoids with given normal quotient $G^{la}/\bb{H}^{la}$. We define the category of  analytic equivariant $D(G^{la}/\bb{H}^{\dagger})$-modules to be $\Mod(X/(G^{la}/\bb{H}^{\dagger}))$.
\end{definition}

Equivariant $D$-modules over solid smooth maps have a good cohomological behaviour (Theorem \ref{TheoDializingEquivariantDmod}): 

\begin{theorem}
Let $X\to Y$ be a solid smooth morphism of derived Tate adic spaces over $\bb{Q}_p$ of relative dimension $d$, and let $G$ be a $p$-adic Lie group of dimension $g$ acting locally analytically   on $X$ over $Y$.     Let $\bb{H}^{\dagger}$ be a $\dagger$-smooth group  over $X$ of relative dimension $e$, and  let $\bb{H}^{\dagger}\to G^{la}\times X$  be a  map of groupoids over $X$ with  given normal quotient $G^{la}/\bb{H}^{\dagger}$. Then $g: X/(G^{la}/\bb{H}^{\dagger})\to Y$ is cohomologically smooth and its underlying $G^{la}$-equivariant invertible sheaf is equivalent to 
\[
\Omega^{d}_{X/Y}[d]\otimes \bigwedge^{g} \Lie G [g]\otimes  \bigwedge^{e} \ell_{\bb{H}^{\dagger}/X} [-e]. 
\] 
\end{theorem}

\subsection*{Notations and conventions}
Throughout this paper we freely use the language of higher category theory and higher algebra of \cite{HigherTopos} and \cite{HigherAlgebra}, the theory of condensed mathematics of \cite{ClausenScholzeCondensed2019,ClauseScholzeAnalyticGeometry,CondensedComplex}, and the theory of abstract six functor formalisms of \cite{MannSix,MannSix2} and \cite{SixFunctorsScholze}. 

  To avoid any confusion,  $(\infty,1)$-categories will be called $\infty$-categories while classical categories will be called $1$-categories.  We let $\Cat_{\infty}$ denote the large $\infty$-category of $\infty$-categories, let $\Cat_{\infty}^{\colim}$ be the subcategory with  objects given by $\infty$-categories admitting small colimits and morphisms given by  colimit preserving functors, and  let $\Cat^{\ex}_{\infty}$ be the subcategory of stable $\infty$-categories with exact functors.   We  let $\n{P}r^L$ (resp. $\n{P}r^{R}$) be the $\infty$-category of presentable $\infty$-categories with colimit preserving functors (resp. accessible and limit preserving functors). Combining adjectives,  we let $\n{P}r^{L,\ex}$ denote the $\infty$-category of presentable stable $\infty$-categories. 
  
  Following Lurie, the previous categories have natural cartesian symmetric monoidal structures,  we have the following  translation of commutative algebra objects with respect to the cartesian product:

  \begin{itemize}

  \item[$\bullet$] $\CAlg(\Cat_{\infty})$  is naturally equivalent to the $\infty$-category of symmetric monoidal $\infty$-categories $\Cat^{\otimes}_{\infty}$.
  
\item[$\bullet$] $\CAlg(\Cat_{\infty}^{\colim})$ is the $\infty$-category of colimit preserving symmetric monoidal $\infty$-categories, i,.e. $\infty$-categories admitting small colimits, endowed  with a  symmetric monoidal structure that commutes with colimits in each variable, and  symmetric monoidal colimit preserving functors. 
  
\item[$\bullet$] $\CAlg(\Cat_{\infty}^{\ex})$ is the $\infty$-category of stable symmetric monoidal $\infty$-categories, i.e. stable  $\infty$-categories with a symmetric monoidal structure which is exact in each variable, and  symmetric monoidal exact functors.

\item[$\bullet$] $\CAlg(\Cat_{\infty}^{\colim,\ex})$ is the $\infty$-category of colimit preserving symmetric monoidal stable $\infty$-categories; it is the full subcategory of $\CAlg(\Cat_{\infty}^{\colim})$ with objects having a underlying  stable $\infty$-category. 
  
\item[$\bullet$] $\CAlg(\n{P}r^{L})$ is the $\infty$-category of presentably symmetric monoidal $\infty$-categories, i.e. presentable $\infty$-categories with a symmetric monoidal structure that commutes with colimits in each variable, and  symmetric monoidal colimit preserving functors.

\item[$\bullet$]  $\CAlg(\n{P}r^{L,\ex})$ is the $\infty$-category of presentably symmetric monoidal stable $\infty$-categories, i.e. presentable stable $\infty$-categories with a symmetric monoidal structure that commutes with colimits in each variable, and  symmetric monoidal colimit preserving functors. 
\end{itemize}

   Given an arrow $f:X\to Y$ in a pointed $\infty$-category $\n{C}$, we let $[X\to Y]$ and  $\cofib[X\to Y]$  denote the fiber and cofiber of $f$ respectively. The notion of descendable algebra in a  symmetric monoidal  stable $\infty$-category will be used repeatedly along the document, we send to \cite{MathewDescent} for its definition and main properties. 

We let $\Prof$ and $\Extdis$ be the sites of profinite and extremally disconnected sets with covers given by finite jointly  surjective maps.  Given an $\infty$-category $\n{C}$ with finite  products and small colimits,   we let $\Cond(\n{C})$ be the condensification of $\n{C}$, see \cite[\S 11.1]{ClauseScholzeAnalyticGeometry} and \cite[Definition 2.1.1]{MannSix}.  Let $\n{C}$ be a $1$-category that admits small colimits and that is generated by small colimits under its compact projective objects $\n{C}^{cp}$, we let $\Ani(\n{C})$ be the animation of $\n{C}$, see \cite[\S 11.4]{ClauseScholzeAnalyticGeometry}.

We shall write $\AnRing$ for the $\infty$-category of complete commutative  analytic (animated)  rings as in \cite[Definition 2.3.10]{MannSix}; unless otherwise specified all analytic rings will be assume to be objects in $\AnRing$. Given $\n{A}$ an analytic ring, we let $\underline{\n{A}}$  be its underlying condensed ring, and for $S\in \Extdis$ we let $\n{A}[S]$ be the free $\n{A}$-module generated by $S$, we also write $\AnRing_{\n{A}}$ for the slice $\infty$-category of analytic $\n{A}$-algebras.  We denote by   $\Mod_{\geq 0}(\n{A})$  the $\infty$-category of animated $\n{A}$-modules, and  let $\Mod(\n{A})$  be its stabilization.  Throughout this paper we use homotopical notation, so for a complex $M$ the fundamental group $\pi_i(M)$ is the same as the $(-i)$-th cohomology group $H^{-i}(M)$.  We shall write $\Mod^{\heartsuit}(\n{A})$ for  the heart of $\Mod(\n{A})$.

 Given $\n{A}$ an analytic ring, we shall write $\AniAlg_{\n{A}}$ for the $\infty$-category of animated $\n{A}$-algebras, namely, the category of condensed animated $\underline{\n{A}}$-algebras that are $\n{A}$-complete. Given $B$ an animated $\n{A}$-algebra, we let $B_{\n{A}/}$ denote the analytic ring obtained by restriction of analytic ring structure from $\n{A}$ to $B$, see \cite[Definition 2.3.13]{MannSix}. More generally, given $B$ an $\bb{E}_1$-algebra in $\Mod(\n{A})$, we let $B_{\n{A}/}$ be the analytic ring with underlying condensed ring $\underline{B}$, and whose category of left modules is given by $\Mod(B_{\n{A}/})= \Mod_{B}(\Mod(\n{A}))$, see Definition \ref{DefinitionAnalyticErings}.

In this paper all analytic rings are complete, and we always consider colimits as complete analytic rings unless otherwise specified. Let $\n{A}$ be an analytic ring, we will write $-\otimes_{\n{A}}-$ and  $\iHom_{\n{A}}(-,-)$ for the tensor product and the internal $\Hom$ space on $\Mod(\n{A})$, omitting in this way further decorations regarding derived functors.  In case we want to consider a classical tensor or Hom space for objects sitting in degree $0$, we will write $\pi_0(-\otimes_{\n{A}}-)$ and $\pi_0(\iHom_{\n{A}}(-,-))$ instead. For $\n{C}$ a  $1$-category with all small colimits and generated by compact  projective objects, an object $X$ in $\Ani(\n{C})$ is called \textit{static} if it belongs to the essential image of $\n{C} \to \Ani(\n{C})$.  We call an analytic ring $\n{A}$  static if $\underline{\n{A}}$ is a static  condensed animated  ring, i.e. a usual condensed ring sitting in degree $0$.

\subsection*{Acknowledgements}

This project has been the result of long conversations with Johannes Ansch\"utz,  Ko Aoki,  Arthur Cesar le Bras, Lue Pan,  Joaqu\'in Rodrigues Jacinto and Peter Scholze; very special thanks to all of them. I am particularly grateful with   Lucas Mann and  Konrad Zou for their patience in several discussions about higher category theory and abstract six functor formalisms. I hearty thank Grigory Andreychev, Konstantin Ardakov, Dustin Clausen, Gabriel Dospinescu, Akhil Mathew,  Riccardo Pengo, Alberto Vezzani and Bogdan Zavyalov for very fruitful conversations.  This paper is the culmination of the passage of the author in the Max Planck Institute for Mathematics in Bonn during the year 2022-2023, my heartfelt thanks to the institute for their hospitality and  support that made this work possible. This project has been partially done while the author was a Junior Fellow of the Simons Society of Fellows at Columbia university.


\section{ Derived Tate  adic spaces}
\label{SectionAnalyticAdicSpaces}

Clausen and Scholze's analytic geometry is a framework where classical algebraic, archimedean and non-archimedean geometries can be treated as equals. Throughout this paper we will focus on the non-archimedean side of the theory, namely the solid theory.  By taking as inspiration classical (derived) algebraic geometry (eg. \cite{LurieDerivedAlgebraic}), and Huber's theory of (analytic) adic spaces \cite{HuberEtaleCohomology}, we will introduce a category of derived Tate adic spaces\footnote{Following the conventions of Clausen-Scholze, we will replace the adjective \textit{analytic} on Huber rings by the adjective \textit{Tate}, meaning that we work with Huber rings admitting a pseudo-uniformizer.}.

As primary point, we need to introduce the categories of rings that serve as building blocks of  our theories. The first approximation will be modelled by analytic rings associated to generalized Huber pairs \cite[Definition 2.12.8]{MannSix}, called in this paper \textit{solid affinoid rings}, see Definition \ref{DefinitionSolidAffinoidRing}. Roughly speaking, the data of a solid affinoid ring is provided by a pair $(A,A^+)$, where $A$ is a solid animated ring, and $A^+\subset \pi_0(A)(*)$ is a discrete subring that determines which variables of $A$ are ``solid''.   

The next step towards  non-archimedean analytic geometry requires some technical constructions. In one hand,  we want to differentiate algebraic varieties from rigid spaces. On the other hand, we want to define a class of rings that mimics the relevant features of analytic complete Huber pairs $(A,A^+)$,  endowed with a fixed pseudo-uniformizer $\pi$. By \cite{Andreychev}, the $1$-category of complete Huber pairs embeds fully faithfully in the category of analytic rings. Since $(A,A^+)$ admits a pseudo-uniformizer, the subring $A^0$ of power bounded elements is an open subring of $A$. We can determine the objects in $A^0$ in the following way: consider the map $R:=\bb{Z}((\pi)) \to A$ defined by the pseudo-uniformizer $\pi$. An element $a \in A$ belongs to $A^0$ if and only if the map $R[T] \to A$ sending $T\mapsto a$ extends to $R\langle T \rangle \to A$, where $R\langle T \rangle $ is the Tate algebra of $R$. Using this observation, we are able to define a class of \textit{bounded affinoid rings} that  provides the building blocks for our non-archimedean geometry. 

After the introduction of the category of bounded affinoid rings, we extend the construction of the adic spectrum $\Spa \n{A}$ from Huber pairs to bounded affinoid rings, then, by gluing along rational covers, we define the category of  (analytic) derived  adic spaces, which is a large generalization of the classical $1$-category of (analytic) adic spaces endowed with a fix pseudo-uniformizer.

\subsection{Preliminaries}
\label{Subsection:Preliminaries}
In this section we address some technical results and definitions that will be used throughout the paper. The reader can skip it on a first reading and come back when the corresponding statement is referenced.

\subsubsection{Analytic $\bb{E}_{\infty}$-rings}  The definition of (complete) analytic ring and the main properties in \cite[Lecture XII]{ClauseScholzeAnalyticGeometry} can be extended to connective $\bb{E}_{\infty}$-condensed algebras instead of animated rings with minor changes in  the proofs.  Moreover, using \cite[Proposition 12.20]{ClausenScholzeCondensed2019} one can even extend the definition of analytic ring to general $\bb{E}_1$ and $\bb{E}_{\infty}$-algebras. In this paper we will essentially  only use the animated definition of analytic ring (\cite[Definition 2.3.10]{MannSix}). However, it is useful to have this slightly more general notion in mind, for example, when constructing idempotent algebras in the category of complete  modules of analytic rings.

\begin{definition}
\label{DefinitionAnalyticErings}
An analytic $\bb{E}_1$-ring $\n{A}$ is the data of a condensed $\bb{E}_1$-ring $\underline{\n{A}}$ together with a full subcategory $\Mod(\n{A})\subset \Mod(\underline{\n{A}})$ satisfying the following properties. 

\begin{enumerate}

\item $\Mod(\n{A})$ is stable under small limits and colimits in $\Mod(\underline{\n{A}})$.

\item For all $S\in \Extdis$ and $M\in \Mod(\n{A})$ the object $\iHom_{\underline{\n{A}}}(\underline{\n{A}}[S],M)$ belongs to $\Mod(\n{A})$.

\item The inclusion $\Mod(\n{A})\subset \Mod(\underline{\n{A}})$ has a  left adjoint $\n{A}\otimes_{\underline{\n{A}}}-$. 

\end{enumerate}

We say that an analytic ring is complete if the natural map $\underline{\n{A}}\to \n{A}[*]$ is an equivalence. An analytic $\bb{E}_{\infty}$-ring is an analytic $\bb{E}_1$-ring whose underlying ring has a structure of $\bb{E}_{\infty}$-ring.  A morphism $\n{A}\to \n{B}$ of analytic $\bb{E}_1$ (resp. $\bb{E}_{\infty}$)-rings  is a morphism of condensed rings such that the forgetful functor $\Mod(\underline{\n{B}})\to \Mod(\underline{\n{A}})$ sends $\Mod(\n{B})$ to $\Mod(\n{A})$.  We let  let $\AnCRing$ denote the $\infty$-category of  complete $\bb{E}_{\infty}$-analytic rings.

\end{definition}

\begin{remark}
Let $\n{A}$ be an analytic $\bb{E}_{\infty}$-ring, then $\Mod(\n{A})$ is naturally a symmetric monoidal category. Indeed, the same argument of \cite[Proposition 12.4]{ClauseScholzeAnalyticGeometry} shows that the kernel of $\n{A}\otimes_{\underline{\n{A}}}-$ is a tensor ideal. Furthermore, by the proof of  \cite[Proposition 2.3.8]{MannSix}, we have a natural transformation of functors $\Mod(\underline{(-)})\Rightarrow \Mod(-)$ given by  the analytification functors. 
\end{remark}

 The following lemma says that an analytic $\bb{E}_{\infty}$-ring is completely determined by its category of complete modules

\begin{lemma}
\label{LemmaFullyFaithfulInjectionAnRings}
Let $\underline{\bb{S}}$ be the sphere spectrum considered as  a  condensed spectrum, let $\Mod(\underline{\bb{S}})$ be the symmetric monoidal $\infty$-category of condensed spectra and  $\CAlg(\Cat_{\infty}^{\colim})_{\underline{\bb{S}}/}$ the  $\infty$-category of  colimit preserving symmetric monoidal $\infty$-categories tensored over $\Mod(\underline{\bb{S}})$. Then the functor $\Mod(-):\AnCRing\to \CAlg(\Cat_{\infty}^{\colim})_{\underline{\bb{S}}/}$ sending an analytic $\bb{E}_{\infty}$-ring $\n{A}$ to the  symmetric monoidal $\infty$-category $\Mod(\n{A})$ is a fully faithful embedding. 
\end{lemma}
\begin{proof}
 Consider  $\CAlg(\Mod(\underline{\bb{S}}))$  the $\infty$-category of $\bb{E}_{\infty}$-condensed rings and let 
\[
\Mod(\underline{(-)}): \CAlg(\Mod(\underline{\bb{S}})) \to \CAlg(\Cat_{\infty}^{\colim})_{\underline{\bb{S}}/}
\] 
be the functor sending a ring to its category of modules, by   \cite[Corollary 4.8.5.21]{HigherAlgebra}  this functor is fully faithful. Furthermore,  \cite[Proposition 2.3.8]{MannSix} provides a natural transformation $\Mod(\underline{(-)}) \Rightarrow \Mod(-)$  of functors $\AnCRing \to \CAlg(\Cat_{\infty}^{\colim})_{\underline{\bb{S}}/}$ given by the analytification functor $\n{A}\otimes_{\underline{\n{A}}}-: \Mod_{\underline{\n{A}}} \to \Mod_{\n{A}}$.

 Let $A,B\in \AnCRing$,  observe that any colimit preserving  $\Mod(\underline{\bb{S}})$-tensored symmetric monoidal morphism $f^*:\Mod(\n{A})\to \Mod(\n{B})$ is compatible with the natural morphisms 
\[
\begin{tikzcd}
\Mod(\n{A}) \ar[r, "f^*"]& \Mod(\n{B}) \\ 
\underline{\End}_{\Mod(\n{A})}(\n{A}[*])-\Mod(\Mod(\underline{\bb{S}})) \ar[u] \ar[r] & \underline{\End}_{\Mod(\n{B})}(\n{B}[*])-\Mod(\Mod(\underline{\bb{S}})) \ar[u],
\end{tikzcd}
\]
but  $\underline{\End}_{\Mod(\n{A})}(\n{A}[*])= \underline{\n{A}}$  and $\underline{\End}_{\Mod(\n{B})}(\n{B}[*])= \underline{\n{B}}$ since the analytic rings are complete. Then,  we have a natural  commutative diagram 
\[
\begin{tikzcd}
\Mod(\n{A}) \ar[r, "f^*"]& \Mod(\n{B}) \\ 
\Mod(\underline{\n{A}}) \ar[u, "\n{A}\otimes_{\underline{\n{A}}}-"] \ar[r, "\underline{\n{B}}\otimes_{\underline{\n{A}}}-"' ] & \Mod(\underline{\n{B}}) \ar[u, "\n{B}\otimes_{\underline{\n{B}}} -"'],
\end{tikzcd}
\]
which shows that $f^*$ is naturally equivalent to $\n{B}\otimes_{\n{A}}-$ by definition of the analytic base change. Therefore, by definition  of the mapping space of  analytic rings as a full subanima of the mapping space of the underlying condensed rings, cf. \cite[Lecture XII]{ClauseScholzeAnalyticGeometry} or \cite[Definition 2.3.1 (d)]{MannSix},  the mapping space from $\n{A}$ to $\n{B}$ in $\AnCRing$ is naturally equivalent to the mapping space from $\Mod(\n{A})$ to $\Mod(\n{B})$ in $\CAlg(\Cat_{\infty}^{\colim})_{\underline{\bb{S}}/}$, which finishes the proof. 
\end{proof}

\begin{remark}
\label{RemarkFullyFaithAnimatedAnalyticRing}
The previous lemma only applies for  analytic $\bb{E}_{\infty}$-rings and not for analytic animated rings. The  obstruction for the statement to hold for analytic animated rings is that the forgetful functor of animated rings towards $\bb{E}_{\infty}$-rings is not fully faithful. Nevertheless, the lemma holds for analytic animated rings over $\bb{Q}$.  In general, the functor $\Mod(-)$ is always conservative. 
\end{remark}

 \subsubsection{Generalized Huber pairs} For future reference we define generalized Huber pairs, see \cite[Definition 2.12.8]{MannSix}.  Let $\bb{Z}_{\sol}$ denote the analytic ring of solid integers, mapping a profinite set $S= \varprojlim_i S_i\in \Pro(\FinSet)$ to the condensed abelian group $\bb{Z}_{\sol}[S]= \varprojlim_{i} \bb{Z}[S_i]$. More generally, for $R$ a $\bb{Z}$-algebra of finite type we shall write $R_{\sol}$ for the analytic ring such that $R_{\sol}[S]= \varprojlim_i R[S_i]$, and for $R$ a discrete ring we set $R_{\sol}[S]= \varinjlim_{A \subset R} A_{\sol}[S]$ where $A$ runs over all the finitely generated subrings of $R$, cf. \cite[Examples 7.3]{ClausenScholzeCondensed2019}.

\begin{definition}
\label{DefinitionGeneralizedHuberPairs}

 A generalized Huber pair consists on a tuple $(A,S)$ with $A$  an animated $\bb{Z}_{\sol}$-algebra, and $S$ a set of elements $S\subset \pi_0(A)(*)$, such that $A$ is $\bb{Z}[X_s]$-solid for all $s\in S$, with $\bb{Z}[X_s]\to A$ a map sending $X_s\mapsto s$. We let $(A,S)_{\sol}$ denote the analytic ring $A_{\bb{Z}[X_{s}: s\in S]_{\sol}/}$. 

\end{definition}

\begin{remark}
In the notation of Definition \ref{DefinitionGeneralizedHuberPairs},  the analytic ring structure of $(A,S)_{\sol}$   only depends on the variables $S$, and not on the lifts $\bb{Z}[X_s]\to A$, thanks to \cite[Proposition 12.21]{ClauseScholzeAnalyticGeometry}.
\end{remark}

\subsection{Categorified locale}
\label{SubsectionCategorifiedLocale}

 In the following section we recall the formalism of categorified locales  of   \cite[Lectures V-VII]{CondensedComplex} and \cite{aoki2023sheavesspectrum}. The notion of categorified locale replaces the more classical definition of locally ringed space; the building blocks of analytic geometry are analytic rings, and these provide the data of a condensed ring and a category of complete modules. Thus, instead of gluing rings as in classical algebraic geometry we need to glue the categories of modules, the language of categorified locales formalizes this idea. Moreover, categorifies locales offer a clean understanding of open and closed  immersions from a six functor point of view, these notions will be repetitively used throughout the paper.

  Let $\CAlg(\Cat^{\colim,\ex}_{\infty})$ be the $\infty$-category of colimit preserving symetric monoidal stable $\infty$-categories, with morphisms denoted by pullback functors $f^{*}:C\to D$. In \S \ref{SectionAnalyticAdicStacks}  we shall restrict ourselves to the framework of presentably symmetric monoidal stable  $\infty$-categories  $\CAlg(\n{P}r^{L,\ex})$; as it is explain in   \cite{aoki2023sheavesspectrum}, this is not an important restriction since we will eventually take categories of $\kappa$-small condensed sets for some cut-off cardinal $\kappa$.  Given $C\in \CAlg(\Cat^{\colim,\ex}_{\infty})$, we let $\n{S}(C)$ denote  the  class of isomorphism classes  of idempotent algebras in $C$, that is, the class  of isomorphism classes of objects $A\in C$ endowed with a morphism from the tensor unit $1\to A$ such that the arrow 
 \[
 A\xrightarrow{1\otimes \id} A \otimes A
 \]
is an equivalence.  We endow $\n{S}(C)$ with a partial order as follows:  $A\leq A'$ if and only if  there is an arrow $A'\to A$ commuting with the unit maps.    By \cite[Theorem 3.13]{aoki2023sheavesspectrum}, if $C$ is presentably symmetric monoidal,  the category of idempotent algebras is in fact essentially small, so $\n{S}(C)$  defines a honest poset. 

\begin{prop}[{\cite[Proposition 5.3]{CondensedComplex}}]
\label{PropLocaleOperations}
The poset $\n{S}(C)$ is a locale with closed subspaces $Z\in \n{S}(C)$ defined by the isomorphism classes of idempotent algebras $A$. More explicitly, the following hold:
\begin{enumerate}
\item The ``empty subset'' corresponds to $0$.

\item The ``whole space'' corresponds to $1$.

\item The ``intersection'' $Z\cap Z'$ corresponds to $A\otimes A'$.

\item An ``arbitrary intersection'' $\bigcap_i Z_i$ corresponds to $\varinjlim_{i} A_i$. 

\item The ``union'' $Z\cup Z'$ corresponds to the fiber $B=[A\bigoplus A' \to A\otimes A']$ together with the unit $1\to B$ induced by the map $1 \xrightarrow{(1,-1)} A\oplus A'$.
\end{enumerate}
\end{prop}

Let $Z\in \n{S}(C)$ correspond to $A$. The closed subspace $Z$ has a natural category of \textit{modules supported on $Z$} defined by the  symmetric monoidal $\infty$-category  $C(Z):= \Mod_{A}(C)$  of $A$-modules in $C$. For $Z\in \n{S}(C)$, the category $C(Z)$ is a tensor ideal of $C$ stable under all limits and colimits. The \textit{frame} $\n{S}(C)^{\op}$ can be thought as  the open complements of the class of closed subspaces in $\n{S}(C)$.  Let $U\in \n{S}(C)^{\op}$ be the open complement of $Z\in \n{S}(C)$, we can define an $\infty$-category of modules on $U$ by taking the localization $C(U)= C/C(Z)$. One can explicitly define natural six functors associated to open and closed immersions. 

\begin{definition}
\label{DefinitionOpenClosedMonoidal}
Let $Z\in \n{S}(C)$ be a closed subspace with associated idempotent algebra $A$ and complementary open $U$. We define the following functors: 
\begin{enumerate}
\item The upper star functors $\iota^*_Z:C\to C(Z)$ and   $j^*_U:C\to C(U)$ given by $\iota^*_Z M = A\otimes M$ and the natural projection respectively. 

\item The lower star functors $\iota_{Z,*}: C(Z)\to C$ and $j_{U,*} :C(U)\to C$ given by the forgetful functor and $j_{U,*} j^*_UM= \iHom_C([1\to A], M)$ for $M\in C$. 

\item The upper shriek functors $\iota^!_Z : C\to C(Z)$ and $j^!_U: C\to C(U)$ given by $\iota^!_Z M=\iHom_C(A,M)$ and $j^!_U=j^*_U$ respectively. 

\item The lower shriek functors $\iota_{Z,!}: C(Z)\to C$ and $j_{U,!}: C(U)\to C$ given by $\iota_{Z,!}=\iota_{Z,*}$ and $j_! j^*_U M = [1\to A]\otimes M$ for $M\in C$ respectively. 
\end{enumerate}

An edge $f^*:C\to D$ in $\CAlg(\Cat^{\colim,\ex}_{\infty})$ is said an \textit{open (resp. closed) immersion} if it is equivalent to an edge of the form  $j_U^*: C\to C(U)$ (resp. $\iota_Z^*:C\to C(Z)$). 

\end{definition}

\begin{remark}
\label{RemarkExcisionSequences}
Let $M\in C$, by construction we have natural excision fiber sequences 
\[
j_!j^*M \to M \to  \iota_*\iota^* M 
\]
and 
\[
 \iota_*\iota^!M \to M \to j_*j^* M.
\]
\end{remark}

The following proposition tells us that the notions of open and closed immersions in $\CAlg(\Cat^{\colim,\ex}_{\infty})$ behave categorically as expected from a $6$-functors point of view.

\begin{prop}[{\cite[Proposition 6.5]{CondensedComplex}}]
\label{PropClosedOpenLocalizationsInftyCat}
Let $f^*:C\to D$ in $\CAlg(\Cat^{\colim,\ex}_{\infty})$. 

\begin{enumerate}
\item $f$ is a closed immersion if and only if $f^*$ has a fully faithful right adjoint $f_*$ which preserves colimits and satisfies  the projection formula 
\[
c\otimes f_* d \xrightarrow{\sim} f_*(f^*c\otimes  d)
\]
for $c\in C$ and $d\in D$. 

\item  $f$ is an open immersion if and only if $f^*$ has a fully faithful left adjoint $f_!: D\to C$ which satisfies the projection formula
\[
f_!(f^*c\otimes d) \xrightarrow{\sim} c\otimes f_! d
\]
for $c\in C$ and $d\in D$. 

\end{enumerate}
\end{prop}

Finally, one has the following theorem saying  that the  functor mapping $U\in \n{S}(C)$ to $C(U)$ is a sheaf for the natural topology of the locale. 

\begin{theorem}[{\cite[Theorem 6.7]{CondensedComplex}}]
\label{TheoLocaleTopology}
\begin{enumerate}

\item  There is a Grothendieck topology on $\CAlg(\Cat^{\colim,\ex}_{\infty})$ where the sieve coverings over $C$ are those which contain a set of open immersions whose corresponding open subsets cover $\n{S}(C)$.

\item   The identity functor $(\CAlg(\Cat^{\colim,\ex}_{\infty})^{\op} )^{\op} \to \CAlg(\Cat^{\colim,\ex}_{\infty})$ is a sheaf with respect to this Grothendieck topology.

\item The poset of open (resp. closed) immersions satisfies descent with respect to this Grothendieck topology.

\end{enumerate}
\end{theorem}

With the previous preparations done we can  define the $\infty$-category of categorified locale. 

\begin{definition}[{\cite[Definition 7.1]{CondensedComplex} and \cite[Definition 4.2]{aoki2023sheavesspectrum}}]
A categorified locale is a triple $(X,C,f)$ consisting on a locale $X$,  a  presentably symmetric monoidal stable $\infty$-category $C\in \CAlg(\Cat^{\colim,\ex}_{\infty})$, and a morphism of locales $f: \n{S}(C)\to X$. Morphisms of categorified locales $F:(X,C,f)\to (Y,C,g)$ consist on morphisms on the topological spaces, $F:X\to Y$ and morphisms of presentably symmetric monoidal categories  $f^*: D\to C$ commuting with the arrows $f:\n{S}(C)\to X$ and $g:\n{S}(D)\to Y$. 

 Given $C\in \CAlg(\Cat^{\colim,\ex}_{\infty})$ we let $\CatLoc_{C}$ be the $\infty$-category of $C$-tensored categorified locales, equivalently, the $\infty$-category of categorified locales  $(X,C',f)$ with $C'\in \CAlg(\Cat^{\colim,\ex}_{\infty})_{C/}$ and morphisms given by $C$-linear morphisms of categorified locales.
\end{definition}

We record the following lemma for future reference:

\begin{lemma}
\label{LemmaDualizableObjectslocale}
Let $(X,C,f)$ be a categorified locale. The category of dualizable (resp. invertible) objects on $C=C(X)$ is a sheaf on $X$.
\end{lemma}
\begin{proof}
By Lemma \cite[Lemma 6.2]{MannSix2}, an object $\n{L}\in C$ is dualizable if and only if the natural map $\n{L} \otimes \iHom_{C}(\n{L},1_{C}) \to \iHom_{C}(\n{L}, \n{L})$ is an equivalence. Moreover, it is invertible if in addition the natural map $1_{C} \to \iHom_{C}(\n{L}, \n{L})$ is an equivalence. For $U \subset X$ an open subspace, and objects $N,M\in C$, we have a natural equivalence
\[
j_U^* \iHom_{C}(N,M)= \iHom_{C(U)}(j_U^*N, j_U^* M),
\]
the lemma follows. 
\end{proof}

\subsection{Tate adic spaces as categorified locale}
\label{SubsectionClassicalAffinoid}

The goal of this section is to construct a categorified locale for classical Tate Huber pairs using the main results of \cite{Andreychev}, obtaining an analogue of the construction of categorified locales for complex analytic spaces of \cite{CondensedComplex}. In the following we only consider   sheafy  Tate Huber rings $(A,A^+)$ that admit a pseudo-uniformizer $\pi$, we  let $\Spa(A,A^+)$ denote the adic spectrum of equivalence classes of  continuous multiplicative  valuations $|-|_x:A \to \Gamma$ that satisfy $|a|_x\leq 1$ for all $a\in A^+$, cf. \cite{HuberEtaleCohomology}. Let us recall some basic properties of the adic spectrum: by \cite[Theorem 3.5]{HuberValuations}, $\Spa(A,A^+)$ is a spectral space with a basis of quasi-compact open subsets given by rational localizations $\{|f_i|\leq |g|\neq 0:  i=1,\ldots, d \} $ for $f_1,\ldots, f_d,g\in A$ elements generating $A$.  Furthermore, since $\{|f_i|\leq |g|\neq 0:  i=1,\ldots, d \}$ is quasi-compact, there is $n\in \bb{N}$ such that $\{|f_i|\leq |g|\neq 0:  i=1,\ldots, d \}\subset \{ |\pi^{n}| \leq |g|\neq 0 \}$, so that we can always assume that some $f_i$ is a pseudo-uniformizer of $A$. We have the following lemma

\begin{lemma}
\label{LemmaDevisageRationalLocalizations}
Let $U\subset \Spa(A,A^+)$ be a rational subset, then $U$ can be written as a composition of rational localizations of the form $\{|g|\leq 1\}$ and $\{1\leq|g|\}$.
\end{lemma}
\begin{proof}
Let us write $U=\{|f_i|\leq |g|\neq 0:  i=1,\ldots, d \}$, with $f_{d}= \pi^{n}$. Then $U$ is the composite of the rational localizations $\{ 1\leq |\pi^{-n}g|\}$ and $\{|f_i/g|\leq 1\}$. 
\end{proof}

In \cite[Theorem 4.1]{Andreychev}, Andreychev proved that the functor mapping a rational localization $U\subset \Spa(A,A^+)$ to the category of solid modules $\Mod((\s{O}(U),\s{O}^+(U))_{\sol})$ is in fact a sheaf on $\CAlg(\Cat^{\colim,\ex}_{\infty})$ (see Definition \ref{DefinitionGeneralizedHuberPairs} for the notion of  generalized Huber pair and the  construction of $(A,S)_{\sol}$). The next proposition says that  this functor can be upgraded to a categorified locale. 

\begin{prop}
\label{PropositionAffinoidCatLocale}
Let $(A,A^+)$ be an Tate Huber pair and let $X=\Spa(A,A^+)^{\op}$ be the poset of open subspaces of $\Spa(A,A^+)$. Consider the functor 
\[
\Mod_{X,\sol}(-): \Spa(A,A^+)^{\op} \to \CAlg(\Cat^{\colim,\ex}_{\infty})
\]
sending a rational localization $U$ to $\Mod((\s{O}(U), \s{O}^+(U))_{\sol})$. Then for any open $U\subset X$  the localization functor 
\[
j^*_U:\Mod((A,A^+)_{\sol}) \to \Mod_{X,\sol}(U)
\]
is an open localization in the sense of Proposition \ref{PropClosedOpenLocalizationsInftyCat} (2). 
\end{prop}
\begin{proof}
It suffices to prove the statement for rational localizations, by Lemma \ref{LemmaDevisageRationalLocalizations} we can even reduce to rational localizations of the form  $\{|g|\leq 1\}$  or $U=\{1\leq |g|\}$  for $g\in A$. Then, by \cite[Proposition 4.11]{Andreychev} we have 
\[
(\s{O}(U), \s{O}^+(U))_{\sol}=(A,A^+)_{\sol} \otimes_{(\bb{Z}[T],\bb{Z})_{\sol}} \bb{Z}[T]_{\sol} \mbox{ and }(\s{O}(U), \s{O}^+(U))_{\sol}= (A,A^+)_{\sol} \otimes_{(\bb{Z}[T], \bb{Z})_{\sol}} (\bb{Z}[T^{\pm 1}], \bb{Z}[T^{-1}])_{\sol}
\]
respectively, where $T$ is mapped to $g$ . Thus, it suffices to show that $(\bb{Z}[T], \bb{Z})_{\sol} \to \bb{Z}[T]_{\sol}$ and $(\bb{Z}[T], \bb{Z})_{\sol} \to (\bb{Z}[T^{\pm 1}], \bb{Z}[T^{-1}])_{\sol}$ define open localizations for their categories of modules.  By the proof of \cite[Theorem 8.1]{ClausenScholzeCondensed2019}, the former localization  is the complement of the idempotent $(\bb{Z}[T],\bb{Z})_{\sol}$-algebra $\bb{Z}((T^{-1}))= \bb{Z}[[T^{-1}]][T]$, and the last is the complement of the idempotent algebra $\bb{Z}[[T]]$, this ends the proof of the proposition. 
\end{proof}

\begin{definition}
\label{DefinitionLocaleClassicalAffinoid}
Let $(A,A^+)$ be an Tate Huber pair, we let $\Spa (A,A^+)_{\sol}$ denote the categorified locale $(\Spa(A,A^+), \Mod((A,A^+)_{\sol}))$ obtained by Proposition \ref{PropositionAffinoidCatLocale}. 
\end{definition}

\begin{corollary}
The functor $(A,A^+)\mapsto \Spa (A,A^+)_{\sol}$ extends to a conservative functor from  the $1$-category of analytic adic spaces to the $\infty$-category of categorified locales tensored over $\Mod(\bb{Z}_{\sol})$. Moreover, this functor is fully faithful when restricted to the full subcategory of analytic adic spaces over $\bb{Q}$. 
\end{corollary}
\begin{proof}
Let $X$ be an analytic adic space, and let $U_{\bullet}$ be an hypercover of $X$ by open affinoid subspaces. By Proposition \ref{PropositionAffinoidCatLocale}, we can  construct the simplicial categorified locale $(|U_{\bullet}|, \Mod_{U_{\bullet},\sol})$, taking geometric realizations we obtain a categorified locale $(|X|, \Mod_{X,\sol})$, where $|X|$ is the underlying space of $X$, $\Mod_{X,\sol}$ is the $\infty$-category of solid quasi-coherent sheaves on $X$, and $f: \n{S}(\Mod_{X,\sol})\to |X|$ is the geometric realization of the map $\n{S}(\Mod_{U_{\bullet},\sol}) \to U_{\bullet}$. It is easy to verify that this construction is  independent of the hypercover, so that it gives rise a well defined functor from adic spaces to categorified locales. The conservativity of the functor is clear by Remark \ref{RemarkFullyFaithAnimatedAnalyticRing} and Lemma \ref{LemmaFullyFaithfulInjectionAnRings}.   To prove the last statement about the fully faithful inclusion for adic spaces over $\bb{Q}$, one can reduce to Huber pairs $(A,A^+)$, where by \cite[Proposition 3.34]{Andreychev} and  Lemma \ref{LemmaFullyFaithfulInjectionAnRings} it suffices to show that the morphism of categorified locales $\n{S}(\Mod((A,A^+)_{\sol})) \to \Spa(A,A^+)$ is surjective, this will be proved independently in more generality  in  Proposition \ref{PropAdicSpectralIsTopological} (3). 
\end{proof}

\subsection{Some idempotent algebras}
\label{SubsectionIdempotentAlgebras}

We let $\Mod(\underline{\bb{Z}})$ denote the $\infty$-derived category of condensed abelian groups.  Let $\Cond(\AniRing)$ be the $\infty$-category of  condensed  animated rings, the forgetful functor $\Cond(\AniRing)\to \Mod_{\geq 0}(\underline{\bb{Z}} )$ has a left adjoint given by the symmetric group algebra $\Sym^{\bullet} M$. Moreover, for each $n\geq 0$ we have a symmetric power functors $\Sym^{n} M$ that are computed as the sheafification of $S\mapsto \Sym^{n} (M(S))$ for $S\in \Extdis$. Analogously, one has wedge products $\bigwedge^n M= \Sym^{n} M[1][-n]$ (which are given by the sheafification of $S\mapsto \bigwedge^n M(S)$), and divided powers functors $\Gamma^{n}(M)= \Sym^{n}(M[2])[-2n]$. 

For a free abelian group $F$, a concrete description of its $n$-th symmetric and divided power functor is given by the (co-)invariants of the symmetric group $\Sigma_n$ in its $n$-th fold tensor product respectively: 
\[
\Sym^n F = (F^{\otimes n})_{\Sigma_n} \mbox{ and } \Gamma^n F= (F^{\otimes n})^{\Sigma_n}. 
\]

Thus,   for $S\in \Extdis$, we have  explicit  descriptions $\Sym^{n} \bb{Z}[S] = \bb{Z}[S^n]_{\Sigma_n} = \bb{Z}[S^n_{\Sigma_n}]$ and $\Gamma^{n} \bb{Z}[S]= \bb{Z}[S^{n}]^{\Sigma_n}$, where  $S^{n}_{\Sigma_n}$ is the quotient space of $S^{n}$ by the  natural action of $\Sigma_n$.  In particular, the symmetric algebra of $\bb{Z}[S]$ is described as  $\Sym^{\bullet}(\bb{Z}[S])= \bb{Z}[\bb{N}[S]]$ where $\bb{N}[S]= \bigcup_{c} \bb{N}[S]_{\leq c} $ with $\bb{N}[S]_{\leq c}= \varprojlim_{i} \bb{N}[S_i]_{\leq c}$, and $\bb{N}[S_i]_{\leq c}$  being the space of  sequences $\sum_{s\in S_i}a_s s$ with $a_s\in \bb{N}$ and $\sum_{s\in S_i} a_s \leq c$, note that $\bb{N}[S]_{=c}= S^{c}_{\Sigma_c}$.

We want to describe explicitly the solidification of the symmetric powers, wedge products and divided power functors for the groups $\bb{Z}[S]$. 
 
 \begin{lemma}
 \label{LemmaExactKoszulDiagrams}
 Let $S$ be an extremally disconnected set, we have natural  exact sequences 
 \[
 \begin{tikzcd}
 0 \ar[r] & \bigwedge^n \bb{Z}[S] \ar[r] & \cdots \ar[r] & \Sym^{n-1} \bb{Z}[S] \otimes \bigwedge^1 \bb{Z}[S]  \ar[r] & \Sym^{n} \bb{Z}[S] \ar[r] & 0  \\ 
  0 \ar[r] & \Gamma^n \bb{Z}[S] \ar[r] & \cdots \ar[r] & \bigwedge^{n-1} \bb{Z}[S] \otimes \Gamma^1 \bb{Z}[S]  \ar[r] & \bigwedge^{n} \bb{Z}[S] \ar[r] & 0 .
 \end{tikzcd}
 \]
 \end{lemma}
\begin{proof}
By definition, $\bb{Z}[S]$ is the sheafification of  the presheaf mapping $T\in \Extdis$ to the free abelian group $\bb{Z}[S(T)]$. For a finite free $\bb{Z}$-module $F$ we have obvious exact sequences 
 \begin{equation}
 \label{ShortExactPDSYMWEDGE}
 \begin{tikzcd}
 0 \ar[r] & \bigwedge^n F \ar[r] & \cdots \ar[r] & \Sym^{n-1} F \otimes \bigwedge^1 F \ar[r] & \Sym^{n}F  \ar[r] & 0  \\ 
  0 \ar[r] & \Gamma^n F \ar[r] & \cdots \ar[r] & \bigwedge^{n-1}F \otimes \Gamma^1 F  \ar[r] & \bigwedge^{n} F \ar[r] & 0 .
 \end{tikzcd}
 \end{equation}
 where $\Sym^{i} F\otimes \bigwedge^j F \to  \Sym^{i+1} F \otimes \bigwedge^{j-1} F$ maps $(a_1 \odot \cdots \odot a_i) \otimes (b_1\wedge \cdots \wedge b_{j})\mapsto \sum_{k=1}^{j} (-1)^{k-1} (a_1\odot \cdots \odot a_{i}\odot b_{k}) \otimes (b_1\wedge \cdots \widehat{b_k}\wedge \cdots \wedge b_j) $, and the second sequence is obtained by taking duals to the first one evaluated at $F^{\vee}$. Note that both constructions are natural and covariant  for $F$, taking filtered colimits we obtain the same exact sequences for an arbitrary free $\bb{Z}$-module.  Taking $F= \bb{Z}[S(T)]$ and  sheafifications we get exact sequences as stated in the lemma. 
\end{proof}

\begin{definition}
For $\n{A}$ an analytic ring  and $M\in \Mod_{\geq 0}(\n{A})$,  let $\Sym^{\bullet}_{\n{A}} M$ be the left adjoint of the forgetful functor $\mathrm{AniAlg}_{\n{A}} \to \Mod_{\geq 0}(\n{A})$ from animated $\n{A}$-algebras towards connective $\n{A}$-modules. Equivalently, we let $\Sym^{\bullet}_{\n{A}}M = \bigoplus_{n\in \bb{N}} \Sym^{n}_{\n{A}} M$ and  $\Sym^{n}_{\n{A}} M =\n{A}\otimes_{\underline{\n{A}}} \Sym^n_{\underline{\n{A}}} M$, where $\Sym^n_{\underline{\n{A}}} M$ is the symmetric power as condensed $\underline{\n{A}}$-module, see \cite[Proposition 12.26]{ClauseScholzeAnalyticGeometry}. We denote the wedge  and divided power functors by $\bigwedge^{n}_{\n{A}} M = (\Sym^{n}_{\n{A}} M[1])[-n]$ and $\Gamma^n_{\n{A}} M = (\Sym^{n}_{\n{A}} M[2])[-2n]$. Finally, for $S$ an extremally disconnected set we write $\n{A}[\bb{N}[S]]:= \Sym_{\n{A}}^{\bullet} \n{A}[S]$. 
\end{definition}

\begin{lemma}
\label{LemmaKosulReolutionsSolidModules}
Let $S$ be a profinite set and  let $I$ be  an index set such that $C(S,\bb{Z})\cong \bigoplus_I \bb{Z} e_i$, so that  $\bb{Z}_{\sol}[S]\cong \prod_{I} \bb{Z} e_{i}^{\vee}$. The following hold: 
\begin{enumerate}
\item   $\Sym^{n}_{\bb{Z}_{\sol}} \bb{Z}_{\sol}[S]=\bb{Z}_{\sol}[S^n_{\Sigma_n}] \cong \prod_{\underline{\alpha}\in I^n_{\Sigma_n}} \bb{Z} (\odot_{i\in \underline{\alpha}} e_i^{\vee}) $, where $\odot$ is the symmetric tensor product. 

\item $\bigwedge^n_{\bb{Z}_{\sol}} \bb{Z}_{\sol}[S] \cong \prod_{\substack{ J\subset I \\ |J|=n}} \bb{Z} (\wedge_{j\in J} e_j^{\vee})$, where we have fixed a total order for $I$ in the wedge product.  

\item $\Gamma^n_{\bb{Z}_{\sol}} \bb{Z}_{\sol}[S]= \bb{Z}_{\sol}[S^{n}]^{\Sigma_n} \cong \prod_{\underline{\alpha}\in I^{n}_{\Sigma_n}} \bb{Z}(  (\odot_{i\in \underline{\alpha}} e_i ))^{\vee}$, where  $(-)^{\vee}$ is the dual basis.

\item The sequences of Lemma \ref{LemmaExactKoszulDiagrams} remain exact  after solidification. 
 \end{enumerate} 
\end{lemma}
\begin{proof}
We can assume without loss of generality that $S$ is extremally disconnected. The first equality in part (1) follows from the explicit description of the symmetric  power functor of the  free condensed abelian group generated by $S$, and the fact that $\bb{Z}_{\sol}[T]$ is the derived solidification of $\bb{Z}[T]$ for any profinite set $T$.  Then,  Lemma \ref{LemmaExactKoszulDiagrams} and an inductive argument show that  $\bigwedge^n_{\bb{Z}_{\sol}} \bb{Z}_{\sol}[S]$ and $\Gamma^n_{\bb{Z}_{\sol}} \bb{Z}_{\sol}[S]$ are compact $\bb{Z}_{\sol}$-modules for all $n\in \bb{N}$.  A compact $\bb{Z}_{\sol}$-module is reflexive, namely, it is a retract of a finite complex of compact projective modules $\bb{Z}_{\sol}[T]\cong \prod \bb{Z}$ and $\prod \bb{Z}$ is reflexive as solid $\bb{Z}_{\sol}$-module. Therefore, to compute  $\bigwedge^n_{\bb{Z}_{\sol}} \bb{Z}_{\sol}[S]$ and $\Gamma^n_{\bb{Z}_{\sol}} \bb{Z}_{\sol}[S]$ it suffices to compute their dual. But then, by taking duals of the sequences of Lemma \ref{LemmaExactKoszulDiagrams}  with $F=\bb{Z}[S]$ one obtains the analogue sequences
 \[
 \begin{tikzcd}
0 \ar[r] &  \Gamma^n C(S,\bb{Z}) \ar[r] & \cdots \ar[r] & \Gamma^{1} C(S, \bb{Z}) \otimes \bigwedge^{n-1} C(S,\bb{Z}) \ar[r] &  \bigwedge^n C(S,\bb{Z}) \ar[r] & 0 \\  
0 \ar[r] & \bigwedge^n C(S, \bb{Z}) \ar[r]& \cdots \ar[r] & \Sym^{n-1} C(S,\bb{Z})\otimes \bigwedge^1 C(S, \bb{Z}) \ar[r] & \Sym^n C(S,\bb{Z}) \ar[r] & 0  .  
 \end{tikzcd}
 \]
 By \cite[Theorem 5.4]{ClausenScholzeCondensed2019} the $\bb{Z}$-module $\Cont(S,\bb{Z})$ is a free abelian group, so that the previous sequences are actually exact. This gives the isomorphism $\Cont(S, \bb{Z})\cong \bigoplus_I \bb{Z} e_i$ we fixed in the proposition, proving that  $\Sym^n \Cont(S, \bb{Z})\cong \bigoplus_{\underline{\alpha}\in I^n_{\Sigma_n}} \bb{Z} (\odot_{i\in \underline{\alpha}} e_i)$, $\Gamma^n C(S,\bb{Z}) \cong \bigoplus_{\underline{\alpha}\in I^n_{\Sigma_n}} \bb{Z} (\odot_{i\in \underline{\alpha}} e_i^{\vee})^{\vee}$ and  $\bigwedge^n  \Cont(S, \bb{Z}) \cong \bigoplus_{\substack{J\subset I\\ |J|=n}} \bb{Z} (\wedge_{j\in J} e_j)$. Taking duals one deduces that the solidification of the exact sequences of Lemma  \ref{LemmaExactKoszulDiagrams} are still exact obtaining (4), and that the other explicit descriptions of (1)-(3) also hold. 
\end{proof}

\begin{corollary}
\label{CoroKoszulResolutionAlgebra}
Let $S$ be a profinite set, then the trivial $\bb{Z}_{\sol}[\bb{N}[S]]$-module $\bb{Z}$ has a long Koszul resolution
\begin{equation}
\label{eqLongKoszulresolution}
 \cdots \to  \bb{Z}_{\sol}[\bb{N}[S]] \otimes \bigwedge^2_{\bb{Z}_{\sol}} \bb{Z}_{\sol}[S] \to  \bb{Z}_{\sol}[\bb{N}[S]] \otimes \bigwedge^1_{\bb{Z}_{\sol}} \bb{Z}_{\sol}[S]\to  \bb{Z}_{\sol}[\bb{N}[S]] \to  \bb{Z}\to 0.
\end{equation}
Dually, we have a long co-Koszul resolution 
\begin{equation}
\label{eqLongcoKoszulResolution}
0 \to \bb{Z} \to \Gamma^{\bullet}_{\bb{Z}_{\sol}} \bb{Z}_{\sol}[S] \to   \Gamma^{\bullet}_{\bb{Z}_{\sol}} \bb{Z}_{\sol}[S] \otimes \bigwedge^1_{\bb{Z}_{\sol}} \bb{Z}_{\sol}[S] \to \Gamma^{\bullet}_{\bb{Z}_{\sol}} \bb{Z}_{\sol}[S] \otimes \bigwedge^2_{\bb{Z}_{\sol}} \bb{Z}_{\sol}[S] \to \cdots. 
\end{equation}
\end{corollary}
\begin{proof}
This follows from Lemma \ref{LemmaKosulReolutionsSolidModules} by taking the direct sums of the exact sequences 
\[
0\to \bigwedge^n_{\bb{Z}_{\sol}} \bb{Z}_{\sol}[S] \to  \cdots \to \Sym^{n-1}_{\bb{Z}_{\sol}} \bb{Z}_{\sol}[S] \otimes \bigwedge^1_{\bb{Z}_{\sol}} \bb{Z}_{\sol}[S] \to \Sym^{n}_{\bb{Z}_{\sol}} \bb{Z}_{\sol}[S] \to 0
\]
and 
\[
  0 \to \Gamma^n_{\bb{Z}_{\sol}} \bb{Z}_{\sol}[S] \to  \cdots \to  \bigwedge^{n-1}_{\bb{Z}_{\sol}} \bb{Z}_{\sol}[S] \otimes \Gamma^1_{\bb{Z}_{\sol}} \bb{Z}_{\sol}[S] \to \bigwedge^{n}_{\bb{Z}_{\sol}} \bb{Z}_{\sol}[S] \to 0.
\]
\end{proof}

After the previous preparations  we can  introduce some large idempotent algebras. We let $(R,R^+)=(\bb{Z}((\pi)), \bb{Z}[[\pi]])$ and $R_{\sol}=(R,R^+)_{\sol}$.  

\begin{definition}
\label{DefinitionIdempotentAlgebrasSym}
Let $S$ be a profinite set, we define the following objects: 
\begin{enumerate}

\item Let $I^{\bullet}_{S}$ be the natural ideal decreasing filtration of $\bb{Z}[\bb{N}[S]]$. We define $\bb{Z}[\bb{N}[S]]_n:= \bb{Z}[\bb{N}[S]]/I^n_{S}$. For $\n{A}$ an analytic ring we let $I^{\bullet}_{S,\n{A}}$ and  $\n{A}[\bb{N}[S]]_n$ be the base change of $I^{\bullet}_S$ and $\bb{Z}[\bb{N}[S]]_n$ to $\n{A}$.

\item We let $\bb{Z}_{\sol}[[\bb{N}[S]]]:= \varprojlim_{n} \bb{Z}_{\sol}[\bb{N}[S]]_n$.

\item  We let $R^+_{\sol}\langle \bb{N}[S] \rangle = \varprojlim_{n} (R^+/\pi^n)_{\sol} [\bb{N}[S]]$ and $ R_{\sol}\langle \bb{N}[S] \rangle= R^+_{\sol}\langle \bb{N}[S] \rangle [\frac{1}{\pi}]$.

\item We let $R_{\sol} \{ \bb{N}[S]\}^{\dagger} =\varinjlim_{n\to \infty} R_{\sol} \langle \bb{N}[\frac{S}{\pi^n}] \rangle$. 

\item More generally, for an analytic ring $\n{A}$ over $\bb{Z}_{\sol}$ (resp. over $R^+_{\sol}$ or $R_{\sol}$) we let $\n{A}[[ \bb{N}[S] ]]$, $ \n{A}\langle \bb{N}[S]\rangle$  and $\n{A}\{\bb{N}[S]\}^{\dagger}$ denote the base change of the constructions in (2)-(4) to $\n{A}$. 
\end{enumerate}
\end{definition}

\begin{warning}
The ring $\n{A}[[\bb{N}[S]]]$ is not in general equal to the limit $\varprojlim_{n} \n{A}[\bb{N}[S]]_n$. This holds for example if $\n{A}= B_{\sol}$ it the analytic ring of a finitely generated algebra over $\bb{Z}$, or if $\n{A}= R^+= \bb{Z}[[\pi]]$ is a power series ring.  
\end{warning}

\begin{lemma}
\label{LemmaIdempotentAlgebrasVerification}
The following hold: 
\begin{enumerate}

\item The ring $\bb{Z}_{\sol}[[\bb{N}[S]]]$ is an idempotent $\bb{Z}_{\sol}[\bb{N}[S]]$-algebra.

\item  The ring $R^+_{\sol} \langle \bb{N}[S] \rangle$ is an idempotent $R^+_{\sol}[\bb{N}[S]]$-algebra.

\item The rings $R_{\sol}\langle \bb{N}[S] \rangle$  and $R_{\sol}\{\bb{N}[S]\}^{\dagger}$ are idempotent $R_{\sol}[\bb{N}[S]]$-algebras.
\end{enumerate}
 Moreover they have a natural structure of co-commutative  Hopf algebras.
\end{lemma}
\begin{proof}
The ring $\bb{Z}_{\sol}[\bb{N}[S]]$ is naturally a co-commutative Hopf algebra since it correpresents the functor on solid animated $\bb{Z}$-algebras $A\mapsto A(S)$. Moreover, the co-multiplication map is induced from the map of solid   abelian groups 
\[
\bb{Z}_{\sol}[S] \to \bb{Z}_{\sol}[\bb{N}[S\bigsqcup S]]: \;\; s\mapsto s\otimes 1 + 1\otimes s,
\]
in particular it preserves the $I^{\bullet}_{S}$-filtration. Taking completions we see that $\bb{Z}_{\sol}[[\bb{N}[S]]]$ has a natural co-commutative Hopf algebra structure. The Hopf algebra structure for the other algebras is constructed by taking the base change from  $\bb{Z}_{\sol}$ to $R_{\sol}$ of the algebra $\bb{Z}_{\sol}[\bb{N}[S]]$,    taking $\pi$-adic completions, inverting $\pi$ and taking colimits along $R_{\sol}[\frac{S}{\pi^n}]\to R_{\sol}[\frac{S}{\pi^{n+1}}]$.

Now we prove  idempotency. Let $B$ denote $\bb{Z}_{\sol}[\bb{N}[S]]$ or its $R^+_{\sol}$ or $R_{\sol}$-base change,  and  let $C$ denote $\bb{Z}_{\sol}[[\bb{N}[S]]]$, $R^{(+)}_{\sol}\langle \bb{N}[S] \rangle$ or $R\{\bb{N}[S]\}^{\dagger}$. Let $A$ be $\bb{Z}$, $R^+$ or $R$ depending on the situation. Then both $B$ and $C$ are  $A$-linear Hopf algebras, and by \cite[Proposition 1.0.6]{RJRCSolidLocAn2}, to prove idempotency it suffices to show that $C\otimes_{(B,A)_{\sol}} A = A$.   By Corollary \ref{CoroKoszulResolutionAlgebra} the map  $C\otimes_{(B,A)_{\sol}}  A\to A$ is  represented by the long Koszul complex
\[
 \cdots \to  C \otimes_{A_{\sol}} \bigwedge^2_{A_{\sol}}A_{\sol}[S] \to  C \otimes_{A_{\sol}} \bigwedge^1_{A_{\sol}} A_{\sol}[S]\to C \to  A \to 0.
\]
In the case of (1) the previous sequence is exact since after taking graded pieces one recovers the long Koszul resolution of $\bb{Z}_{\sol}[\bb{N}[S]]$.  In the situation (2), exactness follows by taking the $\pi$-adic completion of the Koszul resolution of $R^+_{\sol}[\bb{N}[S]]\to R^+$. Finally, case (3) follows by inverting $\pi$ in (2) and taking colimits.  
\end{proof}

We address the following  technical lemma that will be used recurrently in the next section.

\begin{lemma}
\label{LemmaOverconvergentFunctionsIdeal}
Let $A_{\sol}= \bb{Z}[T_1,\ldots,T_n]_{\sol}$ be a solid polynomial algebra in $n$-variables, and let $B^{(+)}_{\sol}= A_{\sol}\otimes_{\bb{Z}_{\sol}} R^{(+)}_{\sol}$ be the associated solid Tate algebra over $R^{(+)}_{\sol}$.  Let $S_1$, $S_2$ and $S_3$ be profinite sets.

\begin{enumerate}

\item For any map $S_3 \to B^+_{\sol}\langle \bb{N}[S_1] \rangle$ the natural morphism $B^+_{\sol}[\bb{N}[S_3]] \to B^+_{\sol}\langle \bb{N}[S_1] \rangle$ extends uniquely to $B^+_{\sol} \langle \bb{N}[S_3] \rangle$. 

 \item  Consider the $B$-algebra $\s{T}=B_{\sol}\langle \bb{N}[S_1] \rangle \otimes_{B_{\sol}} B_{\sol}\{\bb{N}[S_2]\}^{\dagger}$ and let $\s{I}_2= \ker(B_{\sol}\{\bb{N}[S_2]\}^{\dagger}\to B)$ be the augmentation ideal of the second factor. For any map $S_3\to B_{\sol} \langle \bb{N}[S_1] \rangle \otimes_{B_{\sol}}\s{I}_2$ the natural morphism of algebras $B_{\sol}[\bb{N}[S_3]] \to \s{T}$ extends uniquely to $B_{\sol}\{\bb{N}[S_3]\}^{\dagger}$. 
 
\item Let $\s{I}= \ker(A_{\sol}[[\bb{N}[S_1]]]\to A)$ be the augmentation ideal. Then for all map $S_3\to \s{I}$, the natural morphism of algebras $ A_{\sol}[\bb{N}[S_3]] \to A_{\sol}[[\bb{N}[S_1]]]$ extends uniquely to $A_{\sol}[[\bb{N}[S_3]]]$. 
 
\end{enumerate} 
 
\end{lemma}
\begin{proof}
\begin{enumerate} 

\item This follows by taking $\pi$-adic completions. 

 \item Let us denote $ \s{S}_{i,n}= B^+_{\sol}\langle  \bb{N}[\frac{S_i}{\pi^n}] \rangle$ and let $\s{I}_{i,n}$ be the augmentation ideal of $\s{S}_{i,n}$. By construction of the algebras we can find $n$ big enough such that the image of $S_3$ in $\s{T}$ lands in  $\s{S}_{1,0}\otimes_{B^+_{\sol}} \s{I}_{2,n}$. This shows that for all $m\geq 0$ the image of $S_3$ in $\s{S}_{1,0}\otimes_{B^+_{\sol}} \s{S}_{2,n+m}$ lands in  $\pi^m \s{S}_{1,0}\otimes_{B^+_{\sol}} \s{S}_{2,n+m}$. Thus, dividing $S_3$ by $\pi^m$, by part (1) we get an arrow $\s{S}_{3,m} \to \s{S}_{1,0}\otimes_{B^+_{\sol}} \s{S}_{2,n+m}$. One gets (2) by inverting $\pi$ and taking colimits as $m\to \infty$.
 
\item  This follows from the fact that the filtration $I^{\bullet}_{S_1}$ of $A_{\sol}[\bb{N}[S_1]]$ is multiplicative.  
 
 \end{enumerate}
\end{proof}

\begin{corollary}
\label{CoroComoduleOverAplus}

The following hold

\begin{enumerate}
\item  The Hopf algebra $\bb{Z}_{\sol}[[\bb{N}[S]]]$  corepresents a module sheaf over the ring sheaf corepresented by $\bb{Z}[T]_{\sol}$. 

\item The Hopf algebra $R^+_{\sol}\langle \bb{N}[S] \rangle$ corepresents an algebra sheaf over the ring sheaf corepresented by $R^+\langle T \rangle_{\sol}:= R^+_{\sol} \otimes_{\bb{Z}_{\sol}} \bb{Z}[T]_{\sol}$.

\item The Hopf algebra $R_{\sol}\{\bb{N}[S]\}^{\dagger}$ corepresents a  module  sheaf over the ring sheaf corepresented by $R\langle T \rangle_{\sol}$. 
\end{enumerate}
\end{corollary}

\begin{proof}  The algebra $\bb{Z}_{\sol}[ \bb{N}[S]]$ corepresents an  algebra over the sheaf corepresented by $\bb{Z}[T]$, namely, for $\n{A}$ an analytic ring $\n{A}(S)$ is always an $\n{A}(*)$-algebra.   Its module action is given by the map 
\begin{equation}
\label{HopfAlgebraMAps1}
\bb{Z}_{\sol}[\bb{N}[S]]\to \bb{Z}_{\sol}[\bb{N}[S]] \otimes_{\bb{Z}} \bb{Z}[T]:\;\;  s\mapsto s\otimes t 
\end{equation}
which satisfies the diagrams of a comodule over a co-ring.  In particular, $\bb{Z}_{\sol}[S]$ lands in the ideal generated by the  augmentation ideal of $\bb{Z}_{\sol}[\bb{N}[S]]$. 

All the Hopf algebras on points (1)-(3) are idempotent over $\bb{Z}_{\sol}[\bb{N}[S]]$ or $R^+_{\sol}[\bb{N}[S]]$, also $\bb{Z}[T]_{\sol}$ is an idempotent analytic ring over $\bb{Z}[T]$. Thus, in order to show that the algebras of (1)-(3) correpresent algebras/modules over $\bb{Z}[T]_{\sol}$ or $R^+\langle T\rangle_{\sol}$, we only need to prove it for the  $\pi_0$ of the correpresented sheaves, and assume without loss of generality that $\n{A}$ is static.

Then, after taking base change of \eqref{HopfAlgebraMAps1} by $\bb{Z}[T]_{\sol}$ or $R^+[T]_{\sol}$, $\pi$-completions and colimits for points (2) and (3), Lemma \ref{LemmaOverconvergentFunctionsIdeal} implies that  the comodule (resp. co-algebra) diagrams of $\bb{Z}_{\sol}[\bb{N}[S]]$ over $\bb{Z}[T]$ can be extended to the corresponding diagrams in each point (1)-(3). 
\end{proof}

\subsection{Condensed Nil-radical}
\label{SubsectionCondensedNil-radical}

To motivate forthcoming constructions let us discuss the concept of nilpotency for condensed rings.  Let $A$ be a static condensed   ring over $\bb{Z}$,  since $A$ is a  sheaf on rings, the most natural definition of the nil-radical of $A$ consists on the ideal  $I$ whose values at $S$ are $\mathrm{nil}(A(S))$. Equivalently, we could  define 
\[
\mathrm{nil}(A)(S)= \varinjlim_n \Hom_{\CondRing_{\bb{Z}}}(\bb{Z}[T]/T^n, \mathrm{Cont}(S,A)).
\]
This definition only asks for a function $f:S \to A$ to be uniformly point-wise nilpotent, i.e  that there is $n\geq 0$ such that  $f(s)^n=0$  for all $s\in S$. However, we could also ask for a more uniform nilpotent condition, namely, that there is some $n>0$ such that for any familly of elements $s_1,\ldots, s_n\in S$, the product $f(s_1)\cdots f(s_n)$ vanishes. When $S$ is just a point $*$,  there is no difference between these two conditions. When $S=\{*,*\}$ is two points, there is a difference on the $n$-nilpotent elements, namely, one is correpresented by the ring $\bb{Z}[X,Y]/(X^n,Y^n)$ and the other by $\bb{Z}[X,Y]/(X,Y)^{n}$, yet both cofiltered systems  are cofinal. For a general profinite set both possible definitions of nilpotent elements differ:

\begin{example}
Let $S$ be a profinite set and let $\bb{Z}[\bb{N}[S]]$ be the symmetric algebra of $\bb{Z}[S]$. Let $n\geq 1$, consider the   map $\bb{Z}[S] \xrightarrow{\bb{Z}[\Delta]} \bb{Z}[ S^n_{\Sigma_n}] \subset \bb{Z}[\bb{N}[S]]$ with $\Delta:S\to S^{n}_{\Sigma_n}$ given by the diagonal map, and let $F:\bb{Z}[\bb{N}[S]]\to \bb{Z}[\bb{N}[S]]$ be  the induced map of algebras.  Then, the static ring that co-represents elements $f \in A(S)$ with $f^n=0$ is the algebra
\[
\n{R}= \pi_0(\bb{Z}[\bb{N}[S]]\otimes_{F,\bb{Z}[\bb{N}[S]]} \bb{Z}). 
\]
In particular, the $k$-th graded piece of $\n{R}$ for $k\geq n$ is non-zero and equal to the cokernel of $\bb{Z}[S] \otimes \bb{Z}[S^{k-n}_{\Sigma_{k-n}}] \xrightarrow{\Delta } \bb{Z}[S^k_{\Sigma_{k}}]$. On the other hand, the quotient 
\[
\bb{Z}[\bb{N}[S]]_n:= \bb{Z}[\bb{N}[S]]/I^n_S
\]
is the static ring that represents elements $f\in A(S)$ with the property that the $n$-th fold map $f^{\otimes n}:  S^n_{\Sigma_n}  \to A(S^n_{\Sigma_n})$ is zero.  In other words, it correpresents the maps   $f:S \to A$ such that $f(s_1)\cdots f(s_n)=0$ for any sequence $s_i\in S$. 
\end{example}

The previous discussion motivates the following definition
\begin{definition}
Let $A$ be a static condensed ring, we define the following presheaf on $\Extdis$: 
\[
\Nil'_{n}(A)(S):= \Hom_{\CondRing_{\bb{Z}}}(\bb{Z}[\bb{N}[S]]_{n} ,A).
\]
We let $\Nil'(A)$ denote the ind-presheaf $(\Nil'_{n}(A))_{n\in \bb{N}}$ on $\Extdis$. For an animated condensed ring $A$ we define $\Nil'_n(A)$ (resp. $\Nil'(A)$) to be the full condensed sub-anima of $A$ (resp. the constant ind-system of sub-anima of $A$) whose connected components are $\Nil'_n(\pi_0(A))$ (resp. $\Nil'(\pi_0(A))$). 
\end{definition}

An apparent disadvantage of the above definition is that the objects $\Nil'_{n}(A)(S)$ are not sheaves, the reason being that for $S$ and $S'$ extremally disconnected sets, the natural map  
\[
\bb{Z}[\bb{N}[S]]_n \otimes \bb{Z}[\bb{N}[S']]_{n}  \to \bb{Z}[\bb{N}[S\bigsqcup S']]_n
\]
is not an equivalence. However, in analogy to  the inclusions  $ (X,Y)^{2n}\subset (X^n,Y^n)\subset (X,Y)^{n}$ for two elements $X$ and $Y$ in a ring $R$, we have a factorization 
\[
\bb{Z}[ \bb{N}[S\bigsqcup S']]_{2n} \to \bb{Z}[\bb{N}[S]]_n \otimes \bb{Z}[\bb{N}[S']]_{n}  \to \bb{Z}[ \bb{N}[S\bigsqcup S']]_n,
\]
proving that the ind-system $\Nil'(A)$ is actually a sheaf. Furthermore, the formation of $\Nil'(A)$ is compatible with  analytic ring structures as follows: 
 
\begin{lemma}
Let $\n{A}$ be an analytic  ring, then $\varinjlim_{n} \Nil'_n(\underline{\n{A}})$ is a complete $\n{A}$-module. 
\end{lemma}
\begin{proof}
By \cite[Proposition 12.4]{ClauseScholzeAnalyticGeometry} it suffices to show that $\pi_0(\varinjlim_n \Nil'_n(\underline{\n{A}}))$ is $\n{A}$-complete, so we can assume that $\n{A}$ is static.  Let us first see that $\varinjlim_{n} \Nil'_n(\n{A})$ has a natural structure of $\underline{\n{A}}$-module. For this, it suffices to see that the pro-condensed sheaf $S \mapsto (\bb{Z}[\bb{N}[S]]_{n})_{n\in \bb{N}}$ is a condensed comodule  for the condensed co-ring  $S \mapsto \bb{Z}[\bb{N}[S]]$.  This follows from the fact that for profinite sets $S$, $S'$ and $S''$, and any map $S \to \bb{Z}[S'\times S'']$, we have a factorization  
\[
\bb{Z}[\bb{N}[S]]_{n}  \to \bb{Z}[\bb{N}[S']]\otimes_{\bb{Z}}  \bb{Z}[\bb{N}[S'']]_{n},
\]
as we have a natural map $ (S'\times S'')^n_{\Sigma_n}\to  S^{',n}_{\Sigma_n}\times S^{n}_{\Sigma_n}$.  Moreover, the fact that $\varinjlim_{n} \Nil'_n(\n{A})$ is $\n{A}$-complete follows by the same argument: any map $S' \to \n{A}[S'\times S'']$  induces a morphism
\[
\n{A}[\bb{N}[S]]_n \to \n{A}[\bb{N}[S']] \otimes_{\n{A}} \n{A}[\bb{N}[S'']]_{n}.
\] 
Therefore, any map $f:S \to \Nil'_n(\underline{\n{A}})$ can be lifted to a map $S\to \n{A}[S']$ for some extremally disconnected $S$,  so that we have a factorization $S \to \n{A}[\bb{N}[S']]_n \to \n{A}$, then  we can  extend $f$  to 
\[
\n{A}[\bb{N}[S]]_n \to \n{A}[\bb{N}[S']]_n \to \n{A},
\]
proving that $\varinjlim_n \Nil'_n(\n{A})$ is the image of maps $\bigoplus_i \pi_0(\n{A}[S]) \to \underline{\n{A}}$, so $\n{A}$-complete. 
\end{proof}

\begin{definition}
Let $\n{A}$ be an analytic ring, the \textit{condensed nil-radical} of $\n{A}$ is the $\n{A}$-analytic ideal 
\[
\Nil(\n{A})= \varinjlim_n \Nil_n'(\n{A}). 
\]
\end{definition}

Our next task is to show that the condensed nil-radical is corepresented by an explicit pro-system of analytic rings, this requires a slight modification of $\bb{Z}[\bb{N}[S]]_n$. 

\begin{definition} For $S\in \Extdis$ and $n\geq 0$  we let $\bb{Z}[\bb{N}[S]]^{\bb{L}}_{n}:= \bb{Z}[\bb{N}[S]] \otimes_{\bb{Z}[\bb{N}[S^n_{\Sigma_n}]]} \bb{Z}$. 
\end{definition}

\begin{lemma}
\label{LemmaHopfAlgebraNilpotent}
 Let $S\in \Extdis$, the pro condensed ring $(\bb{Z}[\bb{N}[S]]_n^{\bb{L}})_{n}$ has a natural structure of additive Hopf algebra such that  the natural map $\bb{Z}[\bb{N}[S]]\to (\bb{Z}[\bb{N}[S]]_n^{\bb{L}})_{n}$ is a morphism of Hopf algebras. 
\end{lemma}
\begin{proof}
The Hopf algebra structure of $\bb{Z}[\bb{N}[S]]$ is encoded in the cosimplicial ring $(\bb{Z}[\bb{N}[\bigsqcup_{i=1}^k S]])_{[k]\in \Delta}=(\bigotimes_{i=1}^k \bb{Z}[\bb{N}[S]])_{[k]\in \Delta}$ obtained by the  comultiplication map defined by $s \mapsto s\otimes 1+ 1 \otimes s$. Let us fix $m\geq 0$ and consider the truncation $(\bb{Z}[\bb{N}[\bigsqcup^k_{i=1} S]])_{[k]\in \Delta_{\leq m}}$. For a fix $k$ and any $l\geq k$ we have inclusions 
\begin{equation}
\label{eqComparisonSymPowers}
\bigoplus_{i=1}^k \bb{Z}[S^{ln}_{\Sigma^{ln}}] \subset \Sym^{ln} (\bigoplus_{i=1}^k\bb{Z}[S] ) \subset \Sym^l( \bigoplus_{i=1}^k \bb{Z}[S^n_{\Sigma^n}]).
\end{equation}
On the other hand, we have a $(\leq m)$-cosimplicial submodule $(\Sym^n (\bb{Z}[\bigsqcup_{i=1}^k S]))_{[k]\leq \Delta_{\leq m}} \subset (\bb{Z}[\bb{N}[ \bigsqcup_{i=1}^kS]])_{[k]\leq \Delta_{\leq m}}$, it induces a morphism of $(\leq m)$-cosimplicial algebras
\[
\left(\Sym^{\bullet} (\Sym^n( \bb{Z}[\bigsqcup_{i=1}^k S]))\right)_{[k]\in \Delta_{\leq m}} \to \left(\bb{Z}[\bb{N}[\bigsqcup_{i=1}^k S]]\right)_{[k]\leq \Delta_m},
\]
taking the push-out from the left term towards the constant $(\leq m)$-cosimplicial ring $(\bb{Z})_{[k]\in \Delta_{\leq m}}$ one gets a $(\leq m)$-cosimplicial ring 
\[
\left(\bb{Z}[\bb{N}[\bigsqcup_{i=1}^k S] ]^{\bb{L}}_{n}\right)_{[k]\in \Delta_{\leq m}}.
\]
Taking the pro-ring as $n\to \infty$, by \eqref{eqComparisonSymPowers} one gets a $(\leq m)$-cosimplicial pro-ring 
\[
\left(  \bigotimes_{i=1}^k ( \bb{Z}[\bb{N}[S]]^{\bb{L}}_n )_{n\in \bb{N}})\right )_{[k]\in \Delta_{\leq m}}.
\]
Taking colimits as $m\to \infty$, we get a cosimplicial pro-ring 
\begin{equation}
\label{eqHopfFormal}
\left(  \bigotimes_{i=1}^k ( \bb{Z}[\bb{N}[S]]^{\bb{L}}_n )_{n\in \bb{N}})\right)_{[k]\in \Delta}.
\end{equation}
By \cite[Proposition 6.1.2.6 (4)]{HigherTopos}, the cosimplicial ring \eqref{eqHopfFormal} pro-correpresents a group object in $\AnRing_{\bb{Z}}$, proving that it is in fact  a Hopf algebra.  By construction, it is clear that $\bb{Z}[\bb{N}[S]] \to (\bb{Z}[\bb{N}[S]]^{\bb{L}}_n)_{n\in \bb{N}}$ is a morphism of Hopf algebras, proving the lemma. 
\end{proof}

\begin{prop}
\label{PropRepresentedNilRadical}
Let $S\in \Extdis$, the functor $\n{A} \mapsto \Nil(\n{A})(S)$ is correpresented by the pro-condensed ring $(\bb{Z}[\bb{N}[S]]_n^{\bb{L}})_{n}$.
\end{prop}
\begin{proof}
We follow the same argument of \cite[Proposition 6.3.3]{gaitsgory2023dg}. Let $\widetilde{\Nil}$ be the functor 
\[
\widetilde{\Nil}(\n{A})(S):= \varinjlim_n \Map_{\AnRing_{\bb{Z}}}(\bb{Z}[\bb{N}[S]]_n^{\bb{L}}, \n{A}).
\]
By Lemma \ref{LemmaHopfAlgebraNilpotent}, $\widetilde{\Nil}(\n{A})(S)$ is naturally an animated $\bb{Z}$-module and the natural map $f:\widetilde{\Nil}(\n{A})(S) \to \underline{\n{A}}(S)$ is a morphism of animated $\bb{Z}$-modules. It suffices to show that $f$ is a fully faithful sub-anima with connected components $\pi_0(\Nil(\n{A})(S))$. The claim about $\pi_0$ is clear since $\pi_0(\bb{Z}[\bb{N}[S]]_n^{\bb{L}})= \bb{Z}[\bb{N}[S]]_n$. It is left to show that for all $i\geq 1$ the map $\pi_i(\widetilde{\Nil}(\n{A})(S)) \to \pi_i(\underline{\n{A}}(S))$ is an isomorphism. Let us denote $\widetilde{\Nil}_n(\n{A})(S):=\Map_{\AnRing_{\bb{Z}}}(\bb{Z}[\bb{N}[S]]^{\bb{L}}_n, \n{A})  $

  By definition, we have a cartesian square of anima
\[
\begin{tikzcd}
\widetilde{\Nil}_n(\n{A})(S) \ar[d] \ar[r] &  \n{A}(S)  \ar[d] \\
\{0\} \ar[r] & \n{A}(S^n_{\Sigma^n}) 
\end{tikzcd}
\]
where the map $\n{A}(S)\to \n{A}(S^n_{\Sigma^n})$ is induced by the map $S^{n}_{\Sigma^n} \to \bb{Z}[\bb{N}[S]]$.  Thus, by taking $0$ as marked point, we have a long exact sequence of homotopy groups for $i\geq 1$
\[
\pi_{i+1}(\n{A}(S^n_{\Sigma^n})) \to \pi_i(\widetilde{\Nil}_n(\n{A})(S)) \to \pi_{i}( \n{A}(S)) \to \pi_{i}(\n{A}(S^n_{\Sigma^n})).  
\]
For $m\geq 1$ we have a commutative triangle 
\[
\begin{tikzcd}
\n{A}(S) \ar[r] \ar[rd]&  \n{A}(S^{n}_{\Sigma^n})   \ar[d] \\ 
 &  \n{A}(S^{nm}_{\Sigma^{nm}})
 \end{tikzcd}
 \]
 induced by the map $S^{nm}_{\Sigma^{nm}} \to \bb{Z}[\bb{N}[S^n_{\Sigma^n}]] \to \bb{Z}[\bb{N}[S]]$. This gives rise a natural map $\widetilde{\Nil}_n(\n{A})(S) \to \widetilde{\Nil}_{nm}(\n{A})(S)$ defining a map of fiber sequences
 \[
 \begin{tikzcd}
 \widetilde{\Nil}_{n}(\n{A})(S) \ar[d]  \ar[r]& \n{A}(S) \ar[d, "\id"]  \ar[r]& \n{A}(S^n_{\Sigma^n})  \ar[d]\\ 
 \widetilde{\Nil}_{nm}(\n{A})(S) \ar[r] & \n{A}(S) \ar[r]& \n{A}(S^{nm}_{\Sigma^{nm}}). 
 \end{tikzcd}
 \]
 This induces a morphism of long exact sequence of homotopy groups
 \begin{equation}
 \label{DiagramCommutativeHomotopyGroups}
 \begin{tikzcd}
 \pi_{i+1}(\n{A}(S^n_{\Sigma^n})) \ar[r]\ar[d] & \pi_i(\widetilde{\Nil}_n(\n{A})(S)) \ar[d] \ar[r]& \pi_i(\n{A}(S))  \ar[d, "\id"] \ar[r]& \pi_{i}(\n{A}(S^n_{\Sigma^n})) \ar[d] \\ 
  \pi_{i+1}(\n{A}(S^{nm}_{\Sigma^{nm}}))  \ar[r]& \pi_i(\widetilde{\Nil}_{nm}(\n{A})(S)) \ar[r]& \pi_i(\n{A}(S))  \ar[r]& \pi_{i}(\n{A}(S^{nm}_{\Sigma^{nm}})).
 \end{tikzcd}
 \end{equation}
 We claim that the map $\pi_{i}(\n{A}(S^n_{\Sigma^n})) \to \pi_i(\n{A}(S^{nm}_{\Sigma^{nm}}))$ is zero for any $m\geq 2$. Indeed, it factors through the map 
 \[
 \n{A}(S^n_{\Sigma^n})\xrightarrow{\Delta} \bigoplus^{m}_{k=1} \n{A}(S^n_{\Sigma^n}) \to  \bigotimes_{k=1}^{m}\n{A}(S^n_{\Sigma^n}) \to \n{A}(S^{nm}_{\Sigma^{nm}}),
 \]
 and the induced arrow 
 \[
 \bigoplus_{k=1}^m \pi_i(\n{A}(S^n_{\Sigma^n})) \to \pi_i(\bigotimes_{k=1}^{m}\n{A}(S^n_{\Sigma^n}) )
 \]
 is zero for $i\geq 1$.  Taking colimits as $m\to \infty$ in \eqref{DiagramCommutativeHomotopyGroups} one finds that 
 \[
 \pi_i(\widetilde{\Nil}(\n{A})(S)) \xrightarrow{\sim} \pi_i(\n{A}(S))
 \]
 is an isomorphism, proving what we wanted. 
\end{proof}

With the previous proposition proven, we can define a stronger notion of nilpotent ideal. 

\begin{definition}
\label{DefUnifNilpotent}
Let $\n{A}\to \n{B}$ be a morphism of analytic rings surjective on $\pi_0$ with $\n{B}$ static and endowed with the induced analytic structure, let $I=[\underline{\n{A}}\to \underline{\n{B}}]$.  We say that $I$ is \textit{$n$-uniformly nilpotent}  if for any map $S \to I$ with $S$ an extremally disconnected set, there is an extension 
\[
\bb{Z}[\bb{N}[S]]^{\bb{L}}_n \to \n{A}. 
\]   
We say that $I$ is \textit{uniformly nilpotent} if it is $n$-uniformly nilpotent  for some $n\geq 1$. Finally, we say that $I$ is \textit{locally uniformly nilpotent} if for any map $f:S \to I$ there exists $n$ such that $f$ extends to $\bb{Z}[\bb{N}[S]]_n^{\bb{L}} \to \n{A}$. 
\end{definition}

\begin{remark}
Note that any uniformly nilpotent ideal of an analytic ring is also a nilpotent ideal as condensed ring. On the other hand,  by definition, $\Nil(\n{A})$ is a locally nilpotent ideal of $\n{A}$. 
\end{remark}

We finish this section by proving the invariance of solid structure under locally nilpotent ideals. 

\begin{prop}
\label{PropSolidInvarianceNilpotent}
Let $\n{A}$ be a solid affinoid ring and let $I \to \n{A}$ be a locally uniformly nilpotent ideal. Let $\n{B}= \n{A}/ I$. Then a map $\bb{Z}[T] \to \n{A}$ extends to $\bb{Z}[T]_{\sol}$ if and only if the composite $\bb{Z}[T] \to \n{B}$ does so. 
\end{prop}
\begin{proof}
The map $\bb{Z}[T] \to \n{A}$ extends to $\bb{Z}[T]_{\sol}$ if and only if $\bb{Z}((T^{-1}))\otimes_{\bb{Z}[T]} \n{A}=0$.  By hypothesis we know that $\bb{Z}((T^{-1})) \otimes_{\bb{Z}[T]} \n{B}=0$, this implies that 
\[
( \bb{Z}((T^{-1}))\otimes_{\bb{Z}[T]} I)\otimes_{\underline{\n{A}}} \n{A} = (\bb{Z}((T^{-1}))\otimes_{\bb{Z}[T]} \n{A}) [*].
\]
Therefore, $ (\bb{Z}((T^{-1}))\otimes_{\bb{Z}[T]} \n{A}) [*]$ is a locally uniformly nilpotent ideal when considered as ideal over itself, in particular the unit is nilpotent and so the ring must be zero proving what we wanted. 
\end{proof}

\subsection{Bounded affinoid rings and $\dagger$-nil-radicals}
\label{SubsectionBoundedRings}

Let $(R,R^+)= (\bb{Z}((\pi)), \bb{Z}[[\pi]])$ and $R_{\sol}=(R,R^+)_{\sol}$, we  denote by  $R^{(+)}\langle T_1,\ldots, T_n \rangle_{\sol}:= R^{(+)}_{\sol}\otimes_{\bb{Z}_{\sol}} \bb{Z}[T_1,\ldots, T_n]_{\sol}$ the solid Tate algebra over $R^{(+)}$ in $n$-variables. In the previous section we constructed a nilpotent radical for arbitrary analytic rings. The first motivation to introduce the category of bounded affinoid rings is the construction of a new nil-radical that will play a fundamental role in the definition of the analytic de Rham stack. This new nil-radical will measure elements $a \in A$ that are ``topologically zero'', namely, elements such that $|a| \leq |\pi^n|$ for all $n\in \bb{N}$.  The second motivation to define the bounded affinoid rings is to construct a category of rings that behaves as Tate affinoid algebras in classical rigid geometry, namely, algebras $A$ admitting some pseudo-uniformizer $\pi$ and some subring of ``power bounded functions'' $A^{0}$ with $A=A^{0}[\frac{1}{\pi}]$, such that any $a\in A^{0}$ satisfies the norm inequality $|a|\leq 1$ with respect to $\pi$ in a suitable sense. 

To make this idea precise we need some further definitions. 

\begin{definition}
\label{DefinitionBoundedNilpIdeal}
\begin{enumerate}

\item Let $\n{A}\in \AnRing_{\bb{Z}_{\sol}}$ be an analytic ring over $\bb{Z}_{\sol}$. The subring of \textit{$+$-bounded} or \textit{solid elements} is the discrete animated ring given by the mapping space
\[
\n{A}^+ = \Map_{\AnRing_{\bb{Z}_{\sol}}}(\bb{Z}[T]_{\sol}, \n{A}).
\]

\item Let $A \in \AniAlg_{\bb{Z}_{\sol}}$ be an animated algebra over $\bb{Z}_{\sol}$. We define the subgroup of  \textit{topologically nilpotent elements} to be the condensed animated abelian group mapping an extremally disconnected set $S$ to the anima 
\[
A^{00}(S):= \Map_{\AniAlg_{\bb{Z}_{\sol}}}(\bb{Z}_{\sol}[[\bb{N}[S]]], A). 
\]

\item Let $A\in \AniAlg_{R_{\sol}}$ be an animated algebra over $R_{\sol}$. We define its condensed subring of \textit{power bounded elements} to be the condensed animated ring with values at $S\in \Extdis$ given by  
\[
A^{0}(S)=\Map_{\AniAlg_{R_{\sol}}}(R_{\sol}\langle \bb{N}[S]\rangle, A).
\]

\item Let $A\in \AniAlg_{R_{\sol}}$, the  condensed $R$-subring of \textit{bounded elements} is defined as $A^{b}=A^{0}[\frac{1}{\pi}]$. 

\item Finally, let $A\in \AniAlg_{R_{\sol}}$, the \textit{$\dagger$-nil-radical ideal} is the condensed $A^{b}$-ideal  whose values at  $S\in \Extdis$ are 
\[
 \Nil^{\dagger}(A)(S)=\Map_{\AniAlg_{R_{\sol}}}(R_{\sol} \{\bb{N}[S]\}^{\dagger}, A). 
\]

\item For $\n{A}\in \AnRing_{\bb{Z}_{\sol}}$ we let $\n{A}^{00}$, $\n{A}^{0}$, $\n{A}^{b}$ and $\Nil^{\dagger}(\n{A})$ be as in  (2)-(5) for its underlying ring $\underline{\n{A}}$.
\end{enumerate}
\end{definition}

\begin{remark}
\label{RemarkCondensedFullAnimaIdempotent}
By definition, $\n{A}^+\subset \n{A}(*)$ is the full animated subring consisting on those connected components $a\in \n{A}(*)$ for which the induced map of analytic rings $\bb{Z}[a]\to \n{A}$ extends to $\bb{Z}[a]_{\sol} \to \n{A}$. Indeed,  the co-ring structure of $\bb{Z}[T]_{\sol}$ naturally induces a ring structure on $\pi_0(\n{A}^+)$, and we endow $\n{A}^+$ with an animated ring structure thanks to the following  cartesian diagram 
\[
\begin{tikzcd}
\n{A}^+ \ar[r] \ar[d] & \underline{\n{A}}(*) \ar[d] \\
\pi_0(\n{A}^+) \ar[r]  & \pi_0(\underline{\n{A}})(*).
\end{tikzcd}
\]
\end{remark}

\begin{remark}
\label{RemarkFullSubobjects}
The spaces $A^{00}$, $A^0$ and $\Nil^{\dagger}(A)$ are full condensed sub-anima of $A$. Indeed, they are condensed sheaves since for any of the  algebras $B(\bb{N}[S])$ as above we have $B(\bb{N}[S\bigsqcup S'])= B(\bb{N}[S]) \otimes B(\bb{N}[S'])$, and they are full condensed subanima since the algebras $B(\bb{N}[S])$ are idempotent over the corresponding free algebra generated by $S$.  Therefore, for $A^*$ representing any of the previous full condensed subanima of $A$, we have a cartesian square of anima 
\[
\begin{tikzcd}
A^* \ar[r] \ar[d] & A \ar[d] \\ 
\pi_0(A^*) \ar[r] & \pi_0(A).
\end{tikzcd}
\]
In particular, to endow $A^*$ with a natural animated abelian group, module or ring structure compatible with the map $A^*\to A$, it suffices to do so for $\pi_0(A^*)$.  Furthermore, Corollary \ref{CoroComoduleOverAplus} implies that these objects are naturally $\n{A}^+$-modules. Moreover, as $R^+_{\sol} \langle \bb{N}[S] \rangle$ is a co-ring algebra, $\n{A}^0$ is also a full condensed  animated subring of $\n{A}$. 
\end{remark}

\begin{remark}
\label{RemarkValuesAtProfinites}
In Definition \ref{DefinitionBoundedNilpIdeal} we restricted ourselves to define the $S$-valued points of different condensed objects attached to an animated solid ring $A$, for $S$ extremally disconnected. Since $\bb{Z}_{\sol}[S]$ is compact projective for $S$ an arbitrary profinite set, the  description of $S$-valued points of the objects in  Definition \ref{DefinitionBoundedNilpIdeal} also extends to $S$ profinite.  
\end{remark}

As a first reality check we prove that the objects $A^0$, $A^{00}$ and $\n{A}^+$ agree with the classical definitions for complete Tate algebras

\begin{lemma}
\label{ExampleHuberPairisComplete}
Let $(A,A^+)$ be a complete Tate Huber pair with pseudo-uniformizer $\pi$ and  set $\n{A}= (A,A^+)_{\sol}$. The condensed spaces $A^0$ and $A^{00}$ agree with the classical subspaces of $A$ of power bounded and topologically nilpotent elements. In addition, $\n{A}^+=A^+$. 
\end{lemma}
\begin{proof}
By \cite[Proposition 3.34]{Andreychev} one can recover the underlying discrete ring of $A^+$ simply as $\n{A}^+$. It is left to identity the condensed subobjects $A^0$ and $A^{00}$ for general $A$.  By definition, the underlying points of $A^0$ and $A^{00}$  consist on all the elements $a\in A$ for which the map $R^+[T]\to A$ extends to $R^+\langle T\rangle$ and $R^{+}[[T]]$  respectively. Then, by definition $A^{0}(*)$ is the subset of power-bounded elements, and $A^{00}(*)$ is the set of topologically nilpotent elements. Therefore, to prove the lemma it suffices to show that $A^{0}$ and $A^{00}$ are open subspaces of $A$, or equivalently, that $A/A^{00}$ is discrete. Let $A_0\subset A$ be a ring of definition of $A$, it will suffices to show that $\pi A_0\subset A^{00}$. Let $S$ be a profinite set and let $f:S\to \pi A_{0}$ be a map of condensed sets. Since $\pi A_0$ is $\pi$-adically complete, the map $f$ extends to a morphism  of rings
\[
R^{+}_{\sol}[[\bb{N}[S]]]\to A_0,
\]
this proves that $\pi A_0\subset A^{00}$, obtaining the claim. 
\end{proof}

  In order to  prove further properties of the condensed subspaces constructed previously, we need to introduce the category of solid rings that will serve as building blocks for the geometric theory treated in this paper. 

\begin{definition}
\label{DefinitionSolidAffinoidRing}
Let $\n{A}$ be an analytic ring over $\bb{Z}_{\sol}$, we say that $\n{A}$ is a \textit{solid  affinoid ring}  if the natural map $(\underline{\n{A}}, \pi_0(\n{A}^+))_{\sol} \to \n{A}$ is an equivalence of analytic rings. We let $\AffRing_{\bb{Z}_{\sol}} \subset \AnRing_{\bb{Z}_{\sol}}$ denote the full subcategory of solid  affinoid rings. Given $\n{A}$ a solid affinoid ring we let $\AffRing_{\n{A}}$ denote the slice category of solid affinoid $\n{A}$-algebras. 
\end{definition}

\begin{example}
\label{ExampleSolidRings}
\begin{enumerate}
\item Let $A$ be an animated discrete ring, by \cite[Proposition 3.34]{Andreychev} and \cite[Proposition 12.19]{ClauseScholzeAnalyticGeometry} solid affinoid ring structures on $A$ are in bijection with integrally closed  subrings $A^+\subset \pi_0(A)$ via the map $A^+\mapsto (A,A^+)_{\sol}$. If $A^+=A$ we denote $A_{\sol}=(A,A)_{\sol}$. 

\item An example of an analytic ring over $\bb{Z}_{\sol}$ that is not solid affinoid is the ring of ultra-solid rational numbers $\bb{Q
}_{\sol\sol}$ (construction due to Clausen and Scholze). It has by compact projective generators the $\bb{Q}$-vector spaces $\prod_I \bb{Q}$. In terms of locales, $\bb{Q}_{\sol\sol}$ is the open complement of $\bb{Z}_{\sol}$ associated to the idempotent solid algebra $\widehat{\bb{Z}}=\prod_p \bb{Z}_p$, we left the proof of this fact for a future work. 

\end{enumerate}
\end{example}

Polynomials algebras are the compact projective generators in the $\infty$-category of discrete commutative animated rings. Similarly,  one can explicitly provide a class of compact projective generators for the $\infty$-category of solid affinoid rings. 

\begin{prop}
\label{PropCompactProjectiveSolidAffinoid}
The $\infty$-category $\AffRing_{\bb{Z}_{\sol}}$ is stable under small colimits and finite products in $\AnRing_{\bb{Z}_{\sol}}$. Furthermore, it has a basis of compact projective generators given by the analytic rings $\bb{Z}[T_1,\ldots, T_n]_{\sol}[\bb{N}[S]]$, where $\{T_i\}_{i=1}^n$ is a finite set of variables, and $S$ is a profinite set.  Moreover, these rings are compact projective in $\AnRing_{\bb{Z}_{\sol}}$.
\end{prop}
\begin{proof}
It is clear that the category $\AffRing_{\bb{Z}_{\sol}}$ is generated by the rings $\bb{Z}_{\sol}[\bb{N}[S]]$ and $\bb{Z}[T]_{\sol}$ under colimits, namely, the rings $\bb{Z}_{\sol}[\bb{N}[S]]$ are generators of animated solid algebras and for any $\n{A}\in \AffRing_{\bb{Z}_{\sol}}$ we can write
\[
\n{A}=(\underline{\n{A}},\bb{Z} )_{\sol} \otimes_{\bb{Z}[\pi_0(\n{A}^+)]} \bb{Z}[\pi_0(\n{A}^+)]_{\sol}. 
\]
In particular, $\AffRing_{\bb{Z}_{\sol}}$ is stable under small colimits in $\AnRing_{\bb{Z}_{\sol}}$. Stability under finite products is clear since for $\n{A}$ and $\n{B}$ solid affinoid rings, one has $\n{A}\prod \n{B}= (\underline{\n{A}}\times \underline{\n{B}}, \n{A}^+\times \n{B}^+)_{\sol}$.

  It is left to see that the rings $\bb{Z}[\underline{T}]_{\sol}[\bb{N}[S]]$ are compact projective in the $\infty$-category of solid affinoid rings. Since the category of compact projective objecs is stable under finite coproducts, it suffices to show that $\bb{Z}[T]_{\sol}$ and $\bb{Z}_{\sol}[\bb{N}[S]]$ are compact projective. The ring $\bb{Z}_{\sol}[\bb{N}[S]]$ is clearly compact projective since it corepresents $\n{A}\mapsto \underline{\n{A}}(S)=R\Hom_{\bb{Z}_{\sol}}(\bb{Z}_{\sol}[S],\underline{\n{A}})$, and $\bb{Z}_{\sol}[S]$ is a  compact projective solid abelian group. It is left to show that $\bb{Z}[T]_{\sol}$ is compact projective  in $\AnRing_{\bb{Z}_{\sol}}$.
  
   Let $\{\n{A}_i\}_{i\in I}$ be a sifted diagram of analytic $\bb{Z}_{\sol}$-algebras with colimit $\n{A}$. We want to prove that the natural map 
\begin{equation}
\label{eqCompactTateAlgebra}
\varinjlim_{i} \Map_{\AnRing_{\bb{Z}_{\sol}}}(\bb{Z}[T]_{\sol}, \n{A}_i) \to \Map_{\AnRing_{\bb{Z}_{\sol}}}( \bb{Z}[T]_{\sol}, \n{A}) 
\end{equation}
is an equivalence. First,  note that both sides are full subanima of $\Map_{\AnRing_{\bb{Z}_{\sol}}}(R[T], \n{A})= \n{A}(*)$ as $\bb{Z}[T]_{\sol}$ is an idempotent $(\bb{Z}[T],\bb{Z})_{\sol}$-algebra, so it suffices to show that they have the same connected components. Let $f: \bb{Z}[T]_{\sol} \to \n{A}$ be a morphism of analytic rings, we want to show that $f$ factors through some $\n{A}_i$. As $\bb{Z}[T]$ is compact projective, we can find a lift $f_i:\bb{Z}[T]\to \n{A}_i$ to some $i$. The map $f_i$ extends to $\bb{Z}[T]_{\sol}$ if and only if $ \bb{Z}((T^{-1})) \otimes_{\bb{Z}_{\sol}} \n{A}_i=0$.   By hypothesis  $\bb{Z}((T^{-1})) \otimes_{(\bb{Z}[T],\bb{Z})_{\sol}} \n{A}=0$, and we have that 
\[
 \bb{Z}((T^{-1})) \otimes_{ (\bb{Z}[T],\bb{Z})_{\sol}} \n{A}=\varinjlim_i\bb{Z}((T^{-1})) \otimes_{ (\bb{Z}[T],\bb{Z})_{\sol}}  \n{A}_i.
\]
Then, there is some $i$ for which the unit of $\bb{Z}((T^{-1})) \otimes_{ (\bb{Z}[T],\bb{Z})_{\sol}}  \n{A}_i$ vanishes, implying that for any $i\to i'$ one has $\bb{Z}((T^{-1})) \otimes_{ (\bb{Z}[T],\bb{Z})_{\sol}}  \n{A}_{i'}=0$, this proves that $f_{i'}$ extends to $\bb{Z}[T]_{\sol}\to \n{A}_{i'}$ and that  \eqref{eqCompactTateAlgebra} is an equivalence. 
\end{proof}

We go back to Definition \ref{DefinitionBoundedNilpIdeal}, our next task is to show that the objects there constructed are complete and have the corresponding algebraic structure.

\begin{prop}
\label{PropositionConstructionsareSolid}
The following hold: 
\begin{enumerate}

\item  Let $\n{A}\in \AffRing_{\bb{Z}_{\sol}}$, then $\n{A}^{00}$ is a solid  $\n{A}^+_{\sol}$-module. 

\item Let $\n{A}\in \AffRing_{R^+_{\sol}}$, then $\n{A}^{0}$ is an animated  $\n{A}^{+}_{\sol}$-algebra. 

\item Let $\n{A}\in \AffRing_{R_{\sol}}$, then $\n{A}^b$ is a solid $\n{A}^+_{\sol}$-algebras  and $\Nil^{\dagger}(\n{A})$ a  solid $\n{A}^+_{\sol}$-module. Moreover,  $\Nil^{\dagger}(\n{A})$ has a natural structure of  $\n{A}^{b}$-module, defining a full subideal $\Nil^{\dagger}(\n{A})\subset \n{A}^b$. 

\end{enumerate}
\end{prop}
\begin{proof}
 By Remarks \ref{RemarkCondensedFullAnimaIdempotent} and \ref{RemarkFullSubobjects}, to prove that the objects in the proposition are solid $\n{A}^+_{\sol}$-modules and that have the claimed algebraic structures, we can take the $0$-truncation.   Therefore, we can assume without loss of generality that $\n{A}$ is static. Corollary \ref{CoroComoduleOverAplus} shows that any of the objects in (1)-(3) are $\n{A}^+$-modules, and that $\n{A}^0\subset \underline{\n{A}}$ is a subring. In particular, $\n{A}^+\to \n{A}^0\to \n{A}^b$ are morphisms of (static) commutative rings.  To prove $\n{A}^+_{\sol}$-completeness,   it suffices to do it for  $\n{A}^{00}$, $\n{A}^0$ and $\Nil^{\dagger}(\n{A})$, let  $\n{A}^*$ denote one of these condensed modules. We make the following conventions: 
 \begin{enumerate}
  \item We let $B$ denote $\bb{Z}$, $R^+$ or $R$ depending on the situation. 
  
  \item We take $I\subset \n{A}^+$ a finite set of variables, set $B_I=B \otimes_{\bb{Z}_{\sol}} \bb{Z}[T_I]_{\sol}$, and let  $B_I(\bb{N}[S])$  be the algebra $B_I[[\bb{N}[S]]]$, $B_I\langle \bb{N}[S] \rangle$ or $B_I\{\bb{N}[S]\}^{\dagger}$ depending on the situation. 
  \end{enumerate}

  Let $S$ be a profinite set and let $S\to \n{A}^* \subset \n{A}$ be a map of condensed sets, by definition it extends uniquely to a map $ B_I(\bb{N}[S]) \to \n{A}$, we claim that $B_{I}[S] \to \n{A}$ factors through $\n{A}^*$.   Suppose this holds, then $\n{A}^*$ would be the image of maps $\bigoplus_{S} B_{I}[S] \to \underline{\n{A}}$, proving that $\n{A}^*$ is $B_{I}$-complete, taking colimits along all  $I\subset \n{A}^+$, one gets that $\n{A}^*$ is indeed $\n{A}^{+}_{\sol}$-complete. Let $\bigsqcup_{i} S'_i \to B_{I}[S]$ be a surjection of condensed sets. Since $B_{I}[S]$ is in the augmentation ideal of $B_{I}(\bb{N}[S])$,  Lemma \ref{LemmaOverconvergentFunctionsIdeal} implies that  the map  $S'_i  \to B_{I}(\bb{N}[S])$ extends to  $B_{I}(\bb{N}[S'_i]) \to B_{I}(\bb{N}[S])$. Taking the composition we get maps $B_{I}(\bb{N}[S'_i]) \to \n{A}$ for all $i\in I$, that must send $S'_i$ to $\n{A}^*$ by definition. This shows that $B_{I}[S]$ is sent to $\n{A}^*$ proving the claim.

 Finally, suppose that $\n{A}$ is a solid affinoid ring over $R_{\sol}$, we want to prove that $\Nil^{\dagger}(\n{A})$ is naturally an $\n{A}^{b}$-module.  For this, by looking at the corresponding diagrams,  it suffices to prove that for all profinite set $S$ and all $n\in \bb{N}$,  the diagonal map $S \to S\times S$ induces a morphism of algebras 
 \[
 R\{\bb{N}[S]\}^{\dagger} \to R\langle \bb{N}[\pi^n S]\rangle \otimes_{R_{\sol}} R\{\bb{N}[S]\}^{\dagger},
 \]
 but this follows by Lemma \ref{LemmaOverconvergentFunctionsIdeal}.
\end{proof}

After the previous preparations we can define the desired category of bounded affinoid rings. 

\begin{definition}
\label{DefinitionBoundedAffinoidRings}
\begin{enumerate}

\item An animated $R_{\sol}$-algebra $A$ is \textit{bounded} if the natural map $A^{b}\to A$ is an equivalence.  We  let $\AniAlg^{b}_{R_{\sol}}$ be the full subcategory of $\AniAlg_{R_{\sol}}$ consisting on bounded animated $R_{\sol}$-algebras.

\item  Let $A$ be a bounded $R_{\sol}$-algebra, the cone $A^{\dagger-\red}$ of the map $\Nil^{\dagger}(A)\to A$  is called the \textit{$\dagger$-reduction} of $A$.  We say that $A$ is \textit{$\dagger$-reduced} if $\n{A}\to \n{A}^{\dagger-\red} $ is an equivalence. We let  $\AniAlg^{\dagger-\red}_{R_{\sol}}\subset \AniAlg_{R_{\sol}}$  be the full subcategory consisting on $\dagger$-reduced animated rings. 

\item A solid affinoid $R_{\sol}$-algebra is \textit{bounded} if its underlying condensed ring is bounded. We let $\AffRing^{b}_{R_{\sol}}\subset \AffRing_{R_{\sol}}$ be the full subcategory of bounded affinoid  $R_{\sol}$-algebras.  For $\n{A}\in \AffRing^{b}_{R_{\sol}}$ we let $\AffRing^b_{\n{A}}$ be the slice category of bounded affinoid $\n{A}$-algebras. 

\item Given $\n{A}$ a bounded affinoid ring, we let $\n{A}^{\dagger-\red}:= \underline{\n{A}}^{\dagger-\red}_{\n{A}/}$ be its $\dagger$-reduction, we say that $\n{A}$ is $\dagger$-reduced if the previous map is an equivalence. We let $\AffRing^{\dagger-\red}_{R_{\sol}} \subset \AffRing_{R_{\sol}}^b$ be the full subcategory of $\dagger$-reduced bounded affinoid rings. 
\end{enumerate}
\end{definition}

The following notation will be used throughout the  rest of the paper.

\begin{definition}
Let $\n{A}$ be a solid affinoid ring, we let $\n{A}[T]_{\sol}:= \n{A} \otimes_{\bb{Z}_{\sol}} \bb{Z}[T]_{\sol}$ be the \textit{solid polynomial algebra} over $\n{A}$. If $\n{A}$ is a bounded $R_{\sol}$-algebra, we write $\n{A}\langle T \rangle_{\sol}:= \n{A}[T]_{\sol}$ and call it the \textit{solid Tate algebra} over $\n{A}$. 
\end{definition}

We end this section by proving some permanence properties of the category of bounded affinoid rings, in particular that the $\dagger$-reduction is an idempotent functor.

\begin{lemma}
\label{LemmaTestBounded}
Let $A$ be an animated $R_{\sol}$-algebra. 

\begin{enumerate}

\item $A$ is bounded if and only if $\pi_0(A)$ is bounded.

\item  An animated  $R_{\sol}$-algebra is bounded if and only if there is a surjection $\bigoplus_{i} R_{\sol}[S_i]\to A$ of animated $R_{\sol}$-modules (i.e. surjection on $\pi_0$) with $S_i$ profinite,  such that each $R_{\sol}[S_i]\to A$ extends to an algebra morphism $R_{\sol}\langle \bb{N}[\pi^{n}S_i] \rangle\to A$ for some $n$ depending on $i$. 

\end{enumerate}
\end{lemma}
\begin{proof}
The first statement is clear since $A^b$ is a full condensed subanima of $A$. For the second statement, the hypothesis is clearly necessary, let us show that it is sufficient.  Let $A$ be a static algebra satisfying the hypothesis of the lemma and let $S\to A$ be a map from a profinite set.  By (1) we can assume that $A$ is static.  We can lift $S$ to a finite direct sum $\bigoplus_{i=1}^k R_{\sol}[S_i]$, after rescaling we can  even assume that it lands in $\bigoplus_{i=1}^k R^+_{\sol}[S_i]$ and that each map $R^+_{\sol}[S_i] \to A$ extends to $R^+_{\sol}\langle \bb{N}[S_i]\rangle$. Since 
\[
\bigotimes_{i=1}^k R^{+}\langle \bb{N}[S_i] \rangle = R^+\langle \bb{N}[\bigsqcup_{i=1}^k S_i]\rangle,
\]  the natural map 
\[
R^+_{\sol}[S]\to \bigoplus_{i=1}^k R^+_{\sol}[S_i]\to R^+_{\sol}\langle \bb{N}[\bigsqcup_{i=1}^k S _i]\rangle
\] can be extended to $R^+_{\sol}\langle \bb{N}[S] \rangle \to R^+_{\sol}\langle \bb{N}[\bigsqcup_{i=1}^k S_i]\rangle$, and  the map $S \to A$ extends to $R_{\sol} \langle \bb{N}[S] \rangle\to A$ proving that $A$ is bounded. 
\end{proof}

\begin{lemma}
\label{LemmaGenericFiberCompleteisBounded}
Let $A$ be a $\pi$-adically complete animated $R^{+}_{\sol}$-algebra, then $A[\frac{1}{\pi}]$ is a bounded subring. 
\end{lemma}
\begin{proof}
Let $S$ be profinite and $f:S\to A[\frac{1}{\pi}]$, after rescaling we can assume that $f$ factors through a map $f:S\to A$. Then, since $A$ is $\pi$-adically complete, we have an extension 
\[
R^+\langle \bb{N}[S]\rangle\to A
\]
and so a map $R\langle \bb{N}[S]\rangle\to A[\frac{1}{\pi}]$, proving that $A[\frac{1}{\pi}]$ is bounded as wanted.
\end{proof}

\begin{prop}
\label{PropStabilityAffinoidRings}
The following hold
\begin{enumerate}

\item The category of bounded animated $R_{\sol}$-algebras $\AniAlg^{b}_{R_\sol}$ is stable under all small colimits in $\AniAlg_{R_{\sol}}$.

\item The category of bounded animated $R_{\sol}$-algebras admits all limits. More precisely, let $\{A_i\}_{i\in I}$ be an $I$-diagram in $\AniAlg^b_{R_{\sol}}$, then its limit in $\AniAlg^b_{R_{\sol}}$ is given by the ``restricted limit''
\[
\varprojlim'_{I} A_i:= (\varprojlim_{i\in I} A_i^{0})[\frac{1}{\pi}]. 
\]

\item Let $A$ be an animated $R_{\sol}$-solid $R[T_1,\ldots, T_n]$-algebra whose underlying $R_{\sol}$-algebra is bounded. Then $A\otimes_{(R[T_1,\ldots, T_n], R^+)_{\sol}} R\langle T_1,\ldots T_n \rangle_{\sol}$ is bounded.

\item Let $\n{A}\to \n{B}$ be a morphism of bounded affinoid  $R_{\sol}$-algebras, let $C$ be an animated $\n{A}$-algebra whose underlying $R_{\sol}$-algebra is bounded, then $\n{B} \otimes_{\n{A}} C$ is a bounded algebra. 

\item More generally, the category of bounded affinoid rings  $\AffRing^{b}_{R_{\sol}}$ is stable under all colimits and finite products in $\AnRing_{R_{\sol}}$.

\item Let $\n{A}\to \n{B}$ be a morphism of bounded affinoid rings and let $\n{C}$ be a bounded  affinoid $\n{A}$-algebra. Then the natural map $\n{B}\otimes_{\n{A}} \Nil^{\dagger}(\n{C}) \to \n{B}\otimes_{\n{A}} \n{C}$ factors through $\Nil^{\dagger}(\n{B}\otimes_{\n{A}} \n{C})$. 

\end{enumerate}
\end{prop}
\begin{proof}
\begin{enumerate}

\item Let $A$ be a bounded $R_{\sol}$-algebra, $B$ and $C$ bounded $A$-algebras  and $D=B\otimes_{(A,R^+)_{\sol}} C$. The property of being a bounded algebra only depends on $\pi_0$, so we can assume that $A$, $B$ and $C$  are static and take $D=\pi_0(B\otimes_{(A,R^{+})_{\sol}} C)$ the non-derived pushout.  Let us take surjections $\bigoplus_{i} R_{\sol}[S_i]\to B$ and $\bigoplus_{j}R_{\sol}[S'_j]\to C$, then $\bigoplus_{i,j} R_{\sol}[S_i\times S_j']\to D$ is a surjection. By hypothesis the maps $R_{\sol}[S_i]\to B$ and $R_{\sol}[S_j']\to C$ extend naturally to morphisms of algebras $R_{\sol}\langle \bb{N}[S_i]\rangle\to B$ and $R_{\sol}\langle \bb{N}[S_i']\rangle\to C$ respectively (after rescaling the maps). This implies that $R_{\sol}[S_i\times S_j']\to D$ extends to a morphism of algebras $R_{\sol}\langle \bb{N}[ S_i\times S_j']\rangle \to D$ which by Lemma \ref{LemmaTestBounded} proves that $D$ is bounded.  Next we prove stability under sifted colimits. Let $\{A_i\}_i$ be a sifted diagram of  bounded animated $R_{\sol}$-algebras with colimit $A$, let $S\to A$ be a map from a profinite set, then $S$ lifts to some $S\to A_i$ and after rescaling it extends to  $R_{\sol}\langle \bb{N}[S] \rangle \to A_i$. Thus, $S\to A$ extends to $ R_{\sol}\langle \bb{N}[S] \rangle \to A$ proving that $A$ is bounded.

\item Let $\{A_i\}_{i\in I}$ be a diagram of bounded animated $R_{\sol}$-algebras, and let $B\in \AniAlg^b_{R_{\sol}}$. We need to show that $\varprojlim_i' A_i$ is bounded and that the natural map
\[
\Map_{\AniAlg_{R_{\sol}}}(B,\varprojlim_i' A)\to \varprojlim_i\Map_{\AniAlg_{R_{\sol}}}(B, A)
\]
is an equivalence. To see that $\varprojlim_i' A_i$ is bounded, note that for any profinite set $S$, a map $f:S\to \varprojlim_i A^0_i$ naturally extends to $R^+_{\sol}\langle \bb{N}[S]\rangle \to \varprojlim_i A_i^0$ as so does any projection to $A_i$. On the other hand, since $B$ is bounded, there is a natural equivalence of mapping spaces 
\begin{equation}
\label{eqBoundedLimit}
\Map_{\AniAlg_{R_{\sol}}}(B,A_i)= \Map_{\AniAlg_{R^+_{\sol}}}(B^{0}, A_i)=\Map_{\AniAlg_{R^+_{\sol}}}(B^{0}, A_i^{0}),
\end{equation}
where the second equivalence follows from the fact that any map of animated $R^+$-algebras $B^0\to A_i$ factors through $A_i^0$, and $A_i^0$ is a full subring of $A_i$.
Taking limits along $i$ we see that 
\[
\varprojlim_i\Map_{\AniAlg_{R_{\sol}}}(B,A_i)=\Map_{\AniAlg_{R^+_{\sol}}}(B^{0}, \varprojlim_i A_i^{0}).
\]
Note that $\varprojlim_i A_i^0$ is a full condensed subanima of $\varprojlim_i A_i$, namely, for finite limit it is  an equivalence,  and  cofiltered limits are left exact with respect to the natural $t$-structure. Then, to prove that \eqref{eqBoundedLimit} is an equivalence, it suffices to show that $(\varprojlim'_i A_i)^0=\varprojlim_i A_i^0$. It is clear that $\varprojlim_i A_i^0\subset (\varprojlim'_i A_i)^0$, conversely, given $S$ profinite and a map $S\to (\varprojlim'_i A_i)^0$, we  have an extension 
\[
R^+_{\sol}\langle \bb{N}[S] \rangle \to (\varprojlim'_i A_i)^0,
\]
and by composing with projections, maps $R^+_{\sol}\langle \bb{N}[S] \rangle\to A_i^0$, proving that we have a factorization
\[
R^+_{\sol}\langle \bb{N}[S] \rangle \to \varprojlim_i A_i^0
\]
as wanted.

\item Let $S$ be a profinite set and $S\to A$, after rescaling we can assume that it lifts canonically to $R_{\sol} \langle \bb{N}[S] \rangle\to A$. Then, by Lemma \ref{LemmaTestBounded} it suffices to prove that $R\langle T_1,\ldots, T_{s} \rangle_{\sol} \langle \bb{N}[S] \rangle$ is a bounded algebra, but this follows from Lemma \ref{LemmaGenericFiberCompleteisBounded}.

\item This is a direct consequence of  parts (1) and  (3), namely, we have that 
\[
\n{B}\otimes_{\n{A}} C = \varinjlim_{I \subset \n{B}^+} R\langle T_I \rangle_{\sol} \otimes_{(R[T_I],R^+)_{\sol}} (\underline{\n{B}} \otimes_{\n{A}_{R_{\sol}/
}} C),
\]
 where $I$ runs over all the finite subsets. The tensor product is a bounded algebra by part (1), the solidification is bounded by part (3) and the colimit is bounded by part (1) again.

\item  We need to prove that $\AffRing^{b}_{R_{\sol}}$ is stable under pushouts and sifted colimits. Let $\n{C} \leftarrow \n{A}\to \n{B}$ be a diagram in $\AffRing^{b}_{R_{\sol}}$, we want to prove that $\n{B}\otimes_{\n{A}} \n{C}$ is still in $\Aff_{R_{\sol}}^b$. Since $\n{A}=(\underline{\n{A}}, \n{A}^+)_{\sol}$ (resp. for $\n{B}$ and $\n{C}$), by  construction of the pushout, we have $\n{B}\otimes_{\n{A}} \n{C}=(\underline{\n{E}}, \n{B}^+\otimes_{\n{A}^+} \n{C}^+)_{\sol}$, where $\n{E}$ is the completion of $\underline{\n{B}}\otimes_{\n{A}} \underline{\n{C}}$ with respect to  the variables in $\n{B}^+\otimes_{\n{A}^+} \n{C}^+$. By parts (1) and (3) one deduces that $\n{E}$ is bounded, so that $\n{B}\otimes_{\n{A}} \n{C}$ is a bounded affinoid ring.  Let $\{\n{A}_i\}_{i\in I}$  be a sifted diagram in $\AffRing^b_{R_{\sol}}$ with colimit $\n{A}$. By Proposition \ref{PropCompactProjectiveSolidAffinoid} one has that $\n{A}^+ = \varinjlim_{i} \n{A}_i^+$, we find that 
\[
\n{A} =(\underline{\n{A}}, \n{A}^+)_{\sol} = \varinjlim_i (\underline{\n{A}_i}, \n{A}_i^+)_{\sol} 
\]
proving that $\n{A}$ is bounded affinoid. Finally, for stability under finite products, note that $(\underline{\n{A}},\n{A}^+)_{\sol} \times (\underline{\n{B}}, \n{B}^+)_{\sol}= (\underline{\n{A}}\times \underline{\n{B}}, \n{A}^+\times \n{B}^+)_{\sol}$.

\item  Given a map $\n{D}\to \n{D}'$ of bounded affinoid rings, we have an induced map $\Nil^{\dagger}(\n{D})\to \Nil^{\dagger}(\n{D}')$ on the $\dagger$-nil radical. Then, by Proposition \ref{PropositionConstructionsareSolid} (3) we can assume without loss of generality that $\n{A}=A$, $\n{B}=B$ and $\n{C}=C$ have the induced analytic structure from $R_{\sol}$. Since $\pi_0(B\otimes_{A} \Nil^{\dagger}(C))$ is a quotient of $\pi_0(B\otimes_{R_{\sol}}  \Nil^{\dagger}(C))$  we can further assume that $A=R$, and that $B$ and $C$ are static.  It suffices to prove that the image of $\pi_0(B\otimes_{R_{\sol}} \Nil^{\dagger}(C) )$ in $\pi_0(B\otimes_{R_{\sol}} C)$ lands in $\pi_0(\Nil^{\dagger}(B\otimes_{R_{\sol}} C))$.  Let $\bigoplus_{i\in I} R_{\sol}[S_i] \twoheadrightarrow B $ and $\bigoplus_{j\in J} R_{\sol}[T_j] \twoheadrightarrow \Nil^{\dagger}(C)$ be surjections, by hypothesis we have extensions to morphisms of algebras after rescaling  $R_{\sol}\langle \bb{N}[S_i] \rangle \to B$ and $R_{\sol}\{\bb{N}[T_j]\}^{\dagger}\to C$. We then have  induced maps 
\[
R_{\sol}\langle \bb{N}[S_i] \rangle \{\bb{N}[T_j]\}^{\dagger} \to B\otimes_{R_{\sol}} C
\]
such that the image of $\bigoplus_{i,j}R_{\sol}[S_i\times T_j]$ in $\pi_0$  is the image of $\pi_0(B\otimes_{R_{\sol}} \Nil^{\dagger}(C))$. By Lemma \ref{LemmaOverconvergentFunctionsIdeal} (2) we can extend the inclusion $R_{\sol}[S_i\times T_j]\to R_{\sol}\langle \bb{N}[S_i] \rangle \{\bb{N}[T_j]\}^{\dagger} $  to a morphism of algebras $R_{\sol}\{\bb{N}[S_i\times T_j]\}^{\dagger} \to  R_{\sol}\langle \bb{N}[S_i] \rangle \{\bb{N}[T_j]\}^{\dagger}$, proving that the map $R_{\sol}[S_i\times T_j] \to B \otimes_{R} C$ extends to $R_{\sol}\{\bb{N}[S_i\times T_j]\}^{\dagger}$, in particular its image in $\pi_0(B\otimes_{R_{\sol}}C)$  lands in  $\pi_0(\Nil^{\dagger} (B\otimes_{R_{\sol}}C))$ as wanted.  
\end{enumerate}
\end{proof}

\begin{example}
\label{ExamplePowerSeriesBoundedAlgebra}
In  Proposition \ref{PropCompactProjectiveSolidAffinoid} we provided a class of compact projective generators for solid affinoid rings, we next describe the power series developements of their $\pi$-completions.

 Let $S$ be a profinite set with $R_{\sol}[S] \cong \prod_{i\in I} R s_{i}$ and $d\geq 1$, let us consider the algebra $\n{A}= R^+_{\sol} \langle T_1,\ldots, T_d \rangle_{\sol} \langle \bb{N}[S] \rangle$. First, by Lemma \ref{LemmaKosulReolutionsSolidModules} we have that 
\[
R^+_{\sol} \langle \bb{N}[S] \rangle = \widehat{\bigoplus_{n\in \bb{N}}} \prod_{ \alpha \in  I^n_{\Sigma_n}} R^+ \mathbf{s}^{\alpha}.
\]

 By definition $\n{A}= R^+_{\sol}\langle \bb{N}[S] \rangle \otimes_{\bb{Z}_{\sol}} \bb{Z}[T_1,\ldots, T_d]_{\sol}$. This implies that for $S'$ another profinite set one has
 \[
 \begin{aligned}
 \n{A}[S'] &=  \widehat{\bigoplus}_{n\in \bb{N}} \prod_{\alpha \in I^n_{\Sigma_n}} ( R^+\langle  T_1,\ldots, T_d\rangle_{\sol}[S']) \bbf{s}^{\alpha} \\ 
 & =\widehat{\bigoplus}_{n\in \bb{N}} \prod_{\alpha \in I^n_{\Sigma_n}} \widehat{\bigoplus}_{\beta \in \bb{N}^{d}} R^+_{\sol}[S']  T^{\beta} \bbf{s}^{\alpha}. 
 \end{aligned}
 \]
In particular, we can write an element in $\n{A}[*]$ as a power series 
\[
 f(T, \bbf{s})=\sum_{\substack{\beta \in \bb{N}^d\\ \alpha \in I^n_{\Sigma_n}}} c_{\alpha,\beta} T^{\beta} \bbf{s}^{\alpha}
 \]
such that for any reduction modulo $\pi^c$, there is $N>>0$ such that $c_{\alpha,\beta}=0$ for $|\alpha|\geq N$, and for  each $\alpha$ there is $M_{\alpha}>>0$ such that $c_{\alpha,\beta}=0$ for $|\beta|\geq M_{\alpha}$.

By construction of $\AffRing^b_{R_{\sol}}$, the rings $R\langle \underline{T}\rangle_{\sol}\langle \bb{N}[S] \rangle$ form a class of (non-compact!) generators. Moreover, since $\AffRing^b_{R_{\sol}}\subset \AffRing_{R_{\sol}}$ is a full subcategory, being bounded is a property and not additional data on solid affinoid $R_{\sol}$-algebras.
\end{example}

In Proposition \ref{PropSolidInvarianceNilpotent}  we showed that the solid  affinoid structure of an analytic  ring was independent of the condensed nil-radical. The next result will prove an analogue statement when restricted to the category of bounded affinoid rings.

\begin{prop}
\label{PropLiftingOverconvergentAlgebrasLift}
\begin{enumerate}

\item Let $\n{A}$ be a bounded affinoid $R_{\sol}$-algebra and $f:R[T]\to \n{A}$ a morphism of analytic $R_{\sol}$-algebras.  Then $f$ extends to $R\langle T \rangle_{\sol}$ if and only if the induced map $R[T]\to \n{A}^{\dagger,\red}$ extends to $R\langle T\rangle_{\sol}$.

\item Let $\n{A}$ be a bounded  affinoid $R_{\sol}$-algebra and $R[T]\to \n{A}$ a morphism of analytic $R_{\sol}$-algebras.  The image of $T$ is invertible if and only if its image in $\n{A}^{\red-\dagger}$ is invertible. 

\item Let $A$ be a bounded $R$-algebra, $S$ a profinite set and  $S\to A$  a map.  Then $S$ extends to $R_{\sol}\langle \bb{N}[S]  \rangle$ if and only if the composite $S\to A^{\dagger-\red}$ does so. 

\item Let $A$ be a bounded $R$-algebra.  Then $ (A^{\dagger-\red})^{\dagger-\red}= A^{\dagger-\red}$. 

\end{enumerate}
\end{prop}

\begin{proof}

By Lemma \ref{LemmaTestBounded} (1) and \cite[Proposition 12.21]{ClauseScholzeAnalyticGeometry} we can assume that $A$ and $\n{A}$ are static rings. 

\begin{enumerate}

\item Let $\n{A}$ be a bounded affinoid ring over $R$.   We want to prove that a map $f:R[T]\to \n{A}$ of analytic rings extends to $R\langle T \rangle_{\sol}$ if and only if the composition $R[T]\to \n{A}^{\dagger-\red}$ does so. This condition is clearly necessary, let us show that it is sufficient.  Let $n>0$ be such that $f$ extends to  $B=R\langle \pi^nT \rangle_{\sol}$ and let $B_{\infty}= R^+\langle \pi^n T \rangle[[T^{-1}]][\frac{1}{\pi}]$.   Then $f$ extends to $R\langle T \rangle_{\sol}$ if  and only if $B_{\infty}\otimes_{B_{\sol}} \n{A}=0$, and this holds if and only if 
\[
(B_{\infty} \otimes_{B_{\sol}}  \n{A})[*]=0.
\]
We have a fiber sequence of $\n{A}$-modules
\[
( B_{\infty} \otimes_{B_{\sol}} \Nil^{\dagger}(\n{A}))\otimes_{\underline{\n{A}}} \n{A} \to (B_{\infty} \otimes_{B_{\sol}} \n{A})[*] \to( B_{\infty} \otimes_{B_{\sol}} \n{A}^{\dagger-\red})[*].
\]
Suppose that $B_{\infty} \otimes_{B_{\sol}} \n{A}^{\dagger-\red}=0$,  then $ (B_{\infty} \otimes_{B_{\sol}}  \Nil^{\dagger}(\n{A}))\otimes_{\underline{\n{A}}} \n{A} = (B_{\infty} \otimes_{B_{\sol}} \n{A}) [*]$.  The ring $B_{\infty}$ is bounded by Lemma \ref{LemmaGenericFiberCompleteisBounded},  and  by Proposition \ref{PropStabilityAffinoidRings} (6) the map $B_{\infty}\otimes_{B_{\sol}} \Nil^{\dagger}(\n{A}) \to B_{\infty}\otimes_{B_{\sol}} \n{A}$ lands in the $\dagger$-nil-radical of the tensor. This implies that the  map $R[T]\to B_{\infty}\otimes_{B_{\sol}} \n{A}$  sending $T\mapsto 1$ extends to $R\{T\}^{\dagger}$ which shows that $1=0$ as $R\{T\}^{\dagger}\otimes_{R[T],T\mapsto 1} R=0$, proving what we wanted.

\item Let $R[T]\to \n{A}$ be a morphism such that the composite $R[T]\to \n{A} \to \n{A}^{\dagger-\red}$ sends $T$ to an invertible element.  By hypothesis there is $a'\in \n{A}$ such that $aa'-1 \in  \Nil^{\dagger}(\n{A})$, as $T+1$ is invertible in $R\{T\}^{\dagger}$, we have that $aa'$ is invertible which implies that $a$ is invertible as we wanted.

\item  This follows a similar argument as parts (1) and (3). Let $S\to A$ be a map such that the composite $f:S\to A^{\dagger-\red}$ extends to $C=R_{\sol}\langle \bb{N}[S] \rangle$. We want to show that $f$ extends to $C$. As $A$ is bounded there is $n>>0$ such that $f$ extends to $B=R_{\sol}\langle \bb{N}[ \pi^n S] \rangle \to A$. Then, by the excision fiber sequences of Remark \ref{RemarkExcisionSequences}, $f$ extends to $C$ if and only if $D:=\iHom_{B}([B\to C], A)=0$, note that this $\iHom$ space is naturally an $\bb{E}_{\infty}$-algebra thanks to the formalism of locale and Definition \ref{DefinitionOpenClosedMonoidal}, namely, it is of the form $j_*j^* A$ for some open localization $j$ of $\Mod(B_{R_{\sol}/})$. By hypothesis we know that $\iHom_{B}([B\to C], A^{\dagger,\red})=0$, so  we have 
\begin{equation}
\label{eqLiftBoundedFunction}
\iHom_{B}([B\to C], \Nil^{\dagger}(\n{A}))= D.
\end{equation}
Let $S$ be a profinite set, we have  maps functorial on $R_{\sol}[S]$
\begin{equation}
\label{eqMapsIdempotentLocalization}
\begin{aligned}
\iHom_{R}(R_{\sol}[S], D) & =\iHom_{R}(R_{\sol}[S],  \iHom_{B}([B\to C],  \Nil^{\dagger}(A)))  \\ & = \iHom_{B}([B\to C], \iHom_{R}(R_{\sol}[S],  \Nil^{\dagger}(A))) \\ 
& = \iHom_{B}([B\to C],  \underline{\Map}_{\mathrm{AniRing}_{R}}( R_{\sol}\{\bb{N}[S]\}^{\dagger}, A)) \\ 
& \to \iHom_{B}([B\to C], \iHom_{R}(R_{\sol}\{\bb{N}[S]\}^{\dagger}, A))\\ 
& = \iHom_{R}(R_{\sol}\{\bb{N}[S]\}^{\dagger},\iHom_{B}([B\to C], A)) \\ 
& = \iHom_{R}(R_{\sol}\{\bb{N}[S]\}^{\dagger},D), 
\end{aligned}
\end{equation}
where $\underline{\Map}_{\mathrm{AniRing}_{R}}( R_{\sol}\{\bb{N}[S]\}^{\dagger}, A)$ is the condensed anima given by 
\[
\underline{\Map}_{\mathrm{AniRing}_{R}}( R_{\sol}\{\bb{N}[S]\}^{\dagger}, A)(S')= \Map_{\mathrm{AniRing}_{R}}( R_{\sol}\{\bb{N}[S \times S']\}^{\dagger}, A),
\]
that coincides with $\iHom_{R}(R_{\sol}[S], \Nil^{\dagger}(A))$.  This implies that any map $R_{\sol}[S] \to D$ can be naturally extended to a map $R_{\sol}\{\bb{N}[S]\}^{\dagger} \to D$. We claim that such a map induces an algebra homomorphism $R_{\sol}\{\bb{N}[S]\}^{\dagger}\to \pi_0(D)$. Suppose the claim holds, by taking the composite $R[T]\xrightarrow{T=1} R \xrightarrow{\mu} D$ where $\mu$ is the unit, the algebra morphism $R[T]\to \pi_0(D)$ extends to an algebra morphism $R\{T\}^{\dagger} \to \pi_0(D)$, which implies that $1=0$, this forces $\pi_0(D)=0$ and  $D=0$. 

Next, we prove the claim.  Let $S$ be a profinite set, let $f:R_{\sol}[S]\to D$ be a morphism of solid $R$-modules and  let  $g: R_{\sol}\{\bb{N}[S]\}^{\dagger} \to D$ be the  map constructed above. Let $\pi_0(g): R_{\sol}\{\bb{N}[S]\}^{\dagger} \to \pi_0(D)$ be the associated map on $\pi_0$. We want to prove that $\pi_0(g)$ is compatible with the multiplication  diagrams, for this, consider the map $R_{\sol}[S ]\oplus R_{\sol}[S] \xrightarrow{f\oplus f} D\oplus D \to D \otimes_{R_{\sol}} D$. By \eqref{eqMapsIdempotentLocalization} we have a natural map $R_{\sol}\{\bb{N}[S]\}\otimes_{R_{\sol}} R_{\sol}\{ \bb{N}[S] \} \xrightarrow{g\otimes g} D\otimes_{R_{\sol}} D$. Furthermore,  since we have a commutative diagram 
\[
\begin{tikzcd}
R_{\sol}[S] \oplus R_{\sol}[S]  \ar[r, "f\otimes 1\oplus 1\otimes f"]  \ar[d, "s_1+s_2"] & D\otimes_{R_{\sol}} D \ar[d, "m"] \\ 
R_{\sol}[S]  \ar[r, "f"]& D,
\end{tikzcd}
\]
we have an induced commutative diagram 
\[
\begin{tikzcd}
R_{\sol}\{\bb{N}[S]\}^{\dagger}\otimes_{R_{\sol}} R\{\bb{N}[S]\}^{\dagger} \ar[d, "m"] \ar[r, "g\otimes g"]& D \otimes_{R_{\sol}} D \ar[d, "m"] \\ 
R_{\sol}\{ \bb{N}[S] \}^{\dagger} \ar[r,"g"] &  D.
\end{tikzcd}
\]
Taking $\pi_0$ and knowing that $g\otimes g$ factors through $\pi_0(D)\otimes_{R_{\sol}} \pi_0(D) \to \pi_0(D\otimes_{R_{\sol}} D)$, one deduces that $\pi_0(g)$ is an algebra homomorphism. 
\item Finally, let $S$ be profinite, let $S \to \Nil^{\dagger}(A^{\dagger-\red})$ be a map, and take a lift   $S\to A$.  Then, $S \to  A^{\dagger-\red}$ extends to $B=R_{\sol}\{ \bb{N}[S]\}^{\dagger}$ by definition, and $S\to \n{A}$ extends to $B$ by part (4). This implies that the image of $S$ in $A$ is in its $\dagger$-nil-radical which shows that $S\to A^{\dagger-\red}$ is $0$, proving $\Nil^{\dagger}(A^{\dagger-\red})=0$ as wanted. 
\end{enumerate}
\end{proof}

\begin{remark}
We believe that the map $R_{\sol}\{\bb{N}[S]\}^{\dagger} \to D$ in the proof of part (3) of Proposition \ref{PropLiftingOverconvergentAlgebrasLift} can be  naturally promoted to a morphism of $\bb{E}_{\infty}$-rings. 
\end{remark}

The following lemma explains why classical Tate Huber pairs do not have many $\dagger$-nilpotent elements. 

\begin{lemma}
\label{LemmaVanishingNilradical}
Let $A$ be a solid animated $R_{\sol}$-algebra and suppose that $\pi_0(A^0)$ is $\pi$-adically separated. Then $\Nil^{\dagger}(A)=0$.
\end{lemma}
\begin{proof}
Let $S$ be a profinite set and  let $S \to \Nil^{\dagger}(A)$ be a map. For all $n\geq 1$ we  have that $\pi^{-n}S$ maps to $A^0$, which implies that $S$ is divisible by  $\pi^n$ for all $n\geq 0$ in $\pi_0(A^0)$. Then, as $\pi_0(A^0)$ is $\pi$-adically separated, the map $S \to \pi_0(A^0)$ must be zero proving the lemma. 
\end{proof}

\begin{example}
Let $A$ be a Tate algebra topologically of finite type over a non-archimidean field $K$, and let $\mathrm{nil}(A)$ be the classical nil-radical of $A$ seen as a closed ideal. Then the reduction $A^{\red}=A/\mathrm{nil}(A)$ is a Tate algebra topologically of finite type such that $A^{\red,0}$ is $\pi$-adically complete and separated (see \cite[Proposition 3.1.10]{BoschRigidFormalGeo}). Lemma \ref{LemmaVanishingNilradical} shows then that 
\[
\Nil^{\dagger}(A)=\mathrm{nil}(A),
\]
i.e. for classical Tate algebras the $\dagger$-nil-radical recovers the usual nil-radical of the ring. 
\end{example}

\begin{corollary}
\label{CorollaryDaggerNilRadical}
Let $\n{A}= R\langle X_1,\ldots, X_d \rangle_{\sol} \langle\bb{N}[S] \rangle \{\bb{N}[S']\}$ for profinite sets $S$ and $S'$. Let $I$ be the augmentation ideal of $R_{\sol}\{\bb{N}[S']\} \to R$. Then $\Nil^{\dagger}(\n{A})=  I \n{A}$.
\end{corollary}
\begin{proof}
The quotient $\n{A}/I\n{A}$ is isomorphic to $ \n{B}=R_{\sol} \langle X_1,\ldots, X_d\rangle_{\sol} \langle \bb{N}[S] \rangle$. It is easy to see that  $\n{B}^0= R^+\langle X_1,\ldots, X_d \rangle_{\sol}  \langle \bb{N}[S] \rangle$ and that it  is $\pi$-adically separated. By Lemma \ref{LemmaVanishingNilradical} we have $\Nil^{\dagger}(\n{B})=0$. This shows that $\Nil^{\dagger}(\n{A}) \subset I \n{A}$. On the other hand, Lemma \ref{LemmaOverconvergentFunctionsIdeal} (2) implies  that $I\n{A} \subset \Nil^{\dagger}(\n{A})$ which proves the equality. 
\end{proof}

Finally, the $\dagger$-nil-radical is related with the closure of ideals in classical Huber rings. 

\begin{corollary}
\label{CoroQuotientTateAlgebra}
Let $(A,A^+)$ be a classical Tate Huber pair, and let $I\subset A$ be a non-necessarily closed ideal in $A$ generated by its global sections. Let $\overline{I}$ be the closure of $I$ in $A$ and suppose that $(A/\overline{I})^0$ is $\pi$-adically separated. Then $\Nil^{\dagger}(A/I)=\overline{I}/I$ and $(A/I)^{\dagger-\red}=A/\overline{I}$. 
\end{corollary}
\begin{proof}
Since $(A/\overline{I})^0$ is $\pi$-adically separated, $A/\overline{I}$ is $\dagger$-reduced by Lemma \ref{LemmaVanishingNilradical}. This implies that $\Nil^{\dagger}(A/I)\subset \overline{I}/I$. Let $f:S\to \overline{I}$ be a map from a profinite set, we want to show that it extends to $R_{\sol}\{\bb{N}[S]\}^{\dagger}$. Let $A_0\subset A$ be a ring of definition, we can assume without loss of generality that $S$ lands in $A_0\cap \overline{I}$. By hypothesis, the subspace $A_0\subset I$ is dense in $A_0\cap \overline{I}$, then for any $n\geq 0$, we have that 
\[
\overline{I}+\pi^n A_0= I+\pi^n A_0.
\]
Thus, $\overline{I}+\pi^nA_0/I = I+\pi^n A_0/I\subset A/I$, and the image of $\pi^n A_0$ in $A/I$ contains $\overline{I}/I$ for all $n\geq 0$. Then, the composite map $S\to \overline{I}\to \overline{I}/I$ has a lift $S\to \pi^n A_0$, proving that we have a factorization 
\[
R_{\sol}\langle \bb{N}[\frac{S}{\pi^n}]\rangle \to A/I,
\]
taking colimits as $n\to \infty$ we get the desired map from $R_{\sol}\{ \bb{N}[\frac{S}{\pi^n}]\}^{\dagger}$, proving that $\underline{I}/I=\Nil^{\dagger}(A/I)$ as wanted. 
\end{proof}

\subsection{Adic spectrum and derived Tate adic spaces}
\label{SubsectionAnalyticAffinoid}

Let $(R,R^+)=(\bb{Z}((\pi)), \bb{Z}[[\pi]])$ and $R_{\sol}=(R,R^+)_{\sol}$.  Let $\AffRing^{b}_{R_{\sol}}$ be the $\infty$-category of bounded affinoid rings over $R_{\sol}$. Similarly as for Tate Huber pairs, given a bounded affinoid ring $\n{A}$ we want to construct  \textit{the adic spectrum} $|\Spa \n{A}|$, as well as a map of locales $\n{S}(\n{A}) \to |\Spa \n{A}|$ generalizing the one of   Definition \ref{DefinitionLocaleClassicalAffinoid}. Instead of trying to define this space using valuations, we  construct  it using the existing maps of locales for classical Huber rings. 

\begin{construction}
\label{ConstructionSpa}
Let $\n{A}\in \AffRing_{R_{\sol}}^b$ be a bounded affinoid $R_{\sol}$-algebra. For any finite set $I\subset A^{0}$ we have a morphism of analytic rings (depending on lifts) $(R\langle T_I \rangle,R^+)_{\sol} \to \n{A}$.  By Proposition \ref{PropositionAffinoidCatLocale} we have  maps of locales (independent of lifts)
\[
\n{S}(\n{A}) \to \n{S}((R\langle T_{I} \rangle, R^+)_{\sol}) \to \Spa(R\langle  T_I \rangle, R^++ R\langle  T_{I}\rangle^{00}).
\]
 Taking limits we  set $\n{T}_{\n{A}}:=   \varprojlim_{I\subset \pi_0(\n{A}^{0})} \Spa(R\langle  T_{I}\rangle, R^+ + R\langle  T_{I}\rangle^{00})$  and let
\[
\rho_{\n{A}}: \n{S}(\n{A}) \to  \n{T}_{\n{A}}
\]
be the associated map of locales. Note that the formation of both $\n{T}_{\n{A}}$ and $\rho_{\n{A}}$ are functorial on $\n{A}$ and only depend on $\pi_0(\n{A})$. 
\end{construction}

The following theorem  is the key input to define the adic spectrum of a bounded affinoid ring.

\begin{theorem}
\label{TheoExistenceSpaA}
Let $\n{A}\in \AffRing_{R_{\sol}}^{b}$. There is a maximal open subspace $U\subset \n{T}_{\n{A}}$ in the constructible topology such that $\rho_{\n{A}}$ factors through a map $\n{S}(\n{A})\to \n{T}_{\n{A}}\backslash U \to \n{T}_{\n{A}}$. Moreover, $\n{S}(\n{A})\to \n{T}_{\n{A}}\backslash U$ is surjective. 
\end{theorem}

\begin{definition}
\label{DefinitionAdicSpectrumBoundedRing}
The \textit{adic spectrum} of $\n{A}$ is the  space $|\Spa \n{A}|=\n{T}_{\n{A}}\backslash U$, with $U$  as in Theorem  \ref{TheoExistenceSpaA}. We let $\rho_{\n{A}}:\n{S}(\n{A})\to |\Spa \n{A}|$ be the associated maps of locales, and let $\Spa \n{A}$ denote the categorified locale $(|\Spa \n{A}|, \Mod(\n{A}), \rho_{\n{A}})$. 
\end{definition}

In order to prove Theorem \ref{TheoExistenceSpaA} we need some preparations.

\begin{lemma}
\label{LemmaIdempotentCompact}
Let $Z\subset \n{T}_{\n{A}}$ be a constructible closed subspace,  then the idempotent algebra $\n{A}(Z):=\rho_{\n{A}}^{-1}(Z)$ is a compact module in $\Mod(\n{A})$.
\end{lemma}
\begin{proof}
This follows from the fact that the complement $U$ of $Z$ is a finite union of rational affinoid localizations $U_i$ associated to analytic rings $\n{A}_i$, and that the forgetful functor $j_{i,*}: \Mod(\n{A}_i)\to \Mod(\n{A})$ commutes with  colimits. Indeed, the forgetful functor $j_*:\Mod(U)\to \Mod(\n{A})$ commutes with colimits, and it is given by 
\[
j_*j^*M=\iHom_{\n{A}}([\underline{\n{A}}\to \n{A}(Z)],M),
\]  
since $\underline{\n{A}}$ is compact one deduces that $\n{A}(Z)$ is compact. 
\end{proof}

\begin{lemma}
\label{LemmaExistenceStalks}
Let $x\in X=\n{T}_{\n{A}}$, then the constructible neighbourhoods $C$ of $x$ of the form 
\[
X\{f_i\leq g:\;\; i=1,\ldots, n\}\cap X\{ g < h \},
\] 
with $f_n=\pi$ and $h\in \pi_0(\n{A})$, are cofinal in all the constructible neighbourhoods of $x$.   We call such a constructible space $C$ a \textit{rational} constructible subspace of $X$.  
\end{lemma}
\begin{proof}
Since $\n{T}_{\n{A}}$ is a limit of spectra of Tate algebras over $R$, it suffices to prove the statement  for $X:=\Spa(R\langle  T_{I}\rangle, R^+ + R\langle  T_{I}\rangle^{00})$. It is clear that a basis of neighbourhoods of $x$ in $X$ for the adic topology are rational localizations of the form $X\{f_i\leq g: i=1,\ldots, n\}$ with $f_n=\pi$. Let $\{x\}_{\gen}$ be the space of generalizations of $x$ in $X$, then we can write 
\[
\{x\}_{\gen}=\bigcap_{x\in U\subset X} U
\]
where $U$ runs over all the open neighbourhoods of $x$. Then $\{x\}_{\gen}$ is a poset being homeomorphic to the adic spectrum of the residue field $\{x\}_{\gen}=\Spa(k(x), k(x)^+)$. But now, any rational subspace of $\{x\}_{\gen}$ is of the form $\{x\}_{\gen}\{ \tilde{h}\leq 1\}$ for some $\tilde{h}\in k(x)$, equivalently, any constructible closed subspace of $\{x\}_{\gen}$ is of the form $\{x\}_{\gen}\{1<\tilde{h}\}$ for $\tilde{h}\in k(x)$. Thus, we can find a neighbourhood $U=\{f_i\leq g \}$ of $x$, and a lift $h'$ of $\tilde{h}$ in $\s{O}(U)$ such that $\{x\}_{\gen }\cap U\{1 < h' \}=\{x\}_{\gen}\{1<\tilde{h}\}$. After multiplying $h'$ by a power of $g$, we can find an element $h\in R\langle T_I\rangle$ and an integer $n\in \bb{N}$ such that 
\[
U\cap X\{ g^n < h\}= U\{ 1< h'\}.
\]
Thus, after replacing $g$ by $g^n$ and $f_i$ by $f_i^n$, we have found an element $h$ such that 
\[
U\cap X\{g< h\}= U\{1<h'\}.
\]
The lemma follows from the previous construction, and the fact that 
\[
\{x\}= \bigcap_{Z\subset \{x\}_{\gen}} Z
\]
where $Z$ runs over the constructible closed subspaces. 
\end{proof}

\begin{lemma}
Let $C$ be a rational constructible subspace of $\n{T}_{\n{A}}$. Then the category $\Mod(C)$ obtained via $\rho^{-1}_{\n{A}}(C)$ defines a natural analytic ring structure for $\underline{\n{A}}$. We let $\n{A}_{\n{C}}$ denote the associated analytic ring.
\end{lemma}
\begin{proof}
We have a natural localization functor $f^*:\Mod(\n{A})\to \Mod(C)$ with  fully faithful right adjoint $f_*$. We let $\Mod(C)_{\geq 0}= f_* \Mod(C)\cap \Mod(\n{A})_{\geq 0}$.   By \cite[Proposition 12.20]{ClauseScholzeAnalyticGeometry} it suffices to show that $\Mod(C)_{\leq 0}$ is the category of complete modules of an analytic animated ring. Take any presentation $C=U\cap Z$ where $U=X\{f_i\leq g: i=1,\ldots, n\}$ with $f_n=\pi$, and $Z=X\{g< h\}$. Then $\n{A}_{U}$  is an analytic ring structure of $\underline{A}$, and $\n{A}(Z)$ is an idempotent algebra in $\Mod(\n{A})$.  Then, the category  $f_*\Mod(C)$ is the category  of $\n{A}_U\otimes_{\n{A}} \n{A}(Z)$-modules in $\Mod(\n{A}_U)$. But we can write 
\[
\n{A}_U\otimes_{\n{A}} \n{A}(Z)=\n{A}_U\otimes_{(\bb{Z}[T],\bb{Z})} \bb{Z}((T^{-1})),
\]
where $T$ is sent to $h/g$ in $\n{A}_{U}$. This last tensor is clearly an analytic animated ring, proving that $\Mod(C)_{\geq 0}$ is the the category of animated modules over $\n{A}_U\otimes_{\n{A}} \n{A}(Z)$. 
\end{proof}

\begin{definition}
\label{DefStalksAnalyticRings}
Let $x\in \n{T}_{\n{A}}$. 

\begin{enumerate}

\item The \textit{adic stalk} of $\n{A}$ at $x$ is the filtered colimit of analytic animated rings 
\[
\n{A}(x):=\varinjlim_{x\in U} \n{A}(U),
\]
where $U$ runs over all the  open rational neighbourhoods of $x$ in $\n{T}_{\n{A}}$.

\item The \textit{constructible stalk} of $\n{A}$ at $x$ is the filtered  colimit of analytic animated rings 
\[
\n{A}(x)_{\cons}:= \varinjlim_{x\in C} \n{A}_{C},
\]
where $C$ runs over all rational constructible neighbourhoods of $x$ as in Lemma \ref{LemmaExistenceStalks}. 
\end{enumerate}
\end{definition}

\begin{proof}[Proof of Theorem  \ref{TheoExistenceSpaA}]
We define $U$ as the set of $x\in \n{T}_{\n{A}}$ such that $\n{A}(x)_{\cons}=0$. To deduce the proposition it suffices to show the  following claim:
\begin{claim}
$U$ us  an open subspace in the constructible topology of $\n{T}_{\n{A}}$.   
\end{claim}

Suppose that the claim holds and let us write $|\Spa \n{A}|=\n{T}_{\n{A}}\backslash U$. We want to show that $\rho_{\n{A}}$ factors by a surjective map onto $|\Spa \n{A}|$. We have to prove the following:

\begin{itemize}

\item[(a)]  If $Z_1,Z_2$ are closed subspaces of $\n{T}_{\n{A}}$ such that $Z_1\cap |\Spa \n{A}|= Z_2\cap |\Spa \n{A}|$ then $\rho_{\n{A}}^{-1}(Z_1)=\rho_{\n{A}}^{-1}(Z_2)$.

\item[(b)] Let $Z_1$ and $Z_2$ be closed subspaces of $\n{T}_A$ such that $\rho_{\n{A}}^{-1}(Z_1)=\rho_{\n{A}}^{-1}(Z_2)$, then $Z_1\cap |\Spa \n{A}|= Z_2\cap |\Spa \n{A}|$. 

\end{itemize}

 We can assume without loss of generality that $Z_1\subset Z_2$. We first prove part (a). For a closed subspaces $Z\subset \n{T}_{\n{A}}$ we let $\n{A}(Z)=\rho_{\n{A}}^{-1}(Z)$ be its associated idempotent algebra in $\Mod(\n{A})$.   Let us write $Z_i= \bigcap_{j} C_{i,j}$ as an intersection of a filtered collection of constructible closed subspaces, we have that 
 \[
 \n{A}(Z)=\varinjlim_i \n{A}(C_{i,j}). 
 \] 
By Lemma \ref{LemmaIdempotentCompact} each $\n{A}$-module $\n{A}(C_{i,j})$ is compact. Then, by replacing $C_{1,j}$ with $C_{1,j}\cap C_{2,j}$, we can assume without loss of generality  the  $Z_i$ are constructible subspaces. We want to show that the natural map  $\n{A}(Z_2)\to \n{A}(Z_1)$ is an equivalence. By the claim, and the assumption of (a), we know that for all $x\in \n{T}_{\n{A}}$ the natural arrow 
\[
\n{A}(x)_{\cons}\otimes_{\n{A}} \n{A}(Z_2) \to \n{A}(x)_{\cons} \otimes_{\n{A}} \n{A}(Z_1).
\]
Indeed, if $x\in U$ then both terms are zero, and if $x\in |\Spa \n{A}|$ this follows from the fact that $Z_1\cap |\Spa \n{A}|= Z_2\cap |\Spa \n{A}|$ and that the $Z_i$ are constructible.   Since the algebras $\n{A}(Z_1)$ are compact $\n{A}$-modules, for each $x\in \n{T}_{\n{A}}$ there is a rational constructible neighbourhood $C_{x}$ such that $\n{A}_{C_x}\otimes_{\n{A}} \n{A}(Z_2)\to \n{A}_{C_x}\otimes_{\n{A}} \n{A}(Z_1)$ is an equivalence. Since $\n{T}_{\n{A}}$ is compact for the constructible topology, we can find a finite  cover $\{C_i\}$ by such $C_x$. But now the spaces $C_i$ are locally closed for the adic topology and their union is the whole $\n{T}_{\n{A}}$.  Therefore the localization functor 
\[
\Mod(\n{A})\to  \prod_{i} \Mod(\n{A}_{C_i})
\]
 is conservative, which proves that $\n{A}(Z_2)=\n{A}(Z_1)$ as wanted. 
 
 Next we prove part (b). Let $Z_1$ and $Z_2$ be closed subspaces of $\n{T}_{\n{A}}$ such that $\n{A}(Z_1)=\n{A}(Z_2)$. We can assume without loss of generality that $Z_1\subset Z_2$. Moreover, by writing $Z_i$ as colimits of constructible closed subspaces, by Lemma \ref{LemmaIdempotentCompact} we can even assume that $Z_1$ and $Z_2$ are constructible. By hypothesis, we know that for all $x\in \n{T}_{\n{A}}$ we have $\n{A}(x)_{\cons}\otimes_{\n{A}} \n{A}(Z_2)=\n{A}(x)_{\cons}\otimes_{\n{A}} \n{A}(Z_1)$, but the set of those $x$ such that $\n{A}(x)_{\cons}\otimes_{\n{A}} \n{A}(Z_i)\neq 0$ is precisely $Z_i\cap |\Spa \n{A}|$ thanks to the claim. One gets part (b).

Finally, we prove the claim. Let $x\in \n{T}_{\n{A}}$ be such that $\n{A}(x)_{\cons}=\varinjlim_{x\in C} \n{A}_{C}=0$, where $C$ runs over all the rational constructible neighbourhoods of $x$ in $\n{T}_{\n{A}}$. Since 
\[
\n{A}(x)_{\cons}[*]= \varinjlim_{x\in C} \n{A}_{C}[*],
\] 
there is some $C$ such that $\n{A}_C[*]=0$, so $C\subset U$, proving that $U$ is open in the constructible topology as wanted. 
\end{proof}

Our next task is to prove that the adic spectrum of a bounded affinoid ring  enjoys the same properties of adic spectra of Tate  Huber rings. More precisely, we shall prove the following:

\begin{prop}
\label{PropAdicSpectralIsTopological}

Let $\n{A}\to \n{B}$ be a morphism of bounded affinoid rings. 

\begin{enumerate}
\item $|\Spa \n{A}|$  is a spectral space and has a basis of qcqs open subspaces given by pullbacks of rational localizations of the adic spaces $\Spa(R\langle T_I \rangle, R^+)$ for some finite set  $I\subset \n{A}^0$.

\item The morphism $\n{A} \to \n{B}$  induces a spectral map $|\Spa \n{B}| \to |\Spa \n{A}|$. Moreover, the pullback of an open rational subspace is a rational subspace.

\item Let $\n{A}=(A,A^+)_{\sol}$ be the analytic ring associated to a Tate Huber pair, then the natural map 
\[
|\Spa \n{A}| \to |\Spa(A,A^+)|
\]
is a homeomorphism. 
\end{enumerate}
\end{prop}

\begin{proof}

\begin{enumerate}

\item By Theorem \ref{TheoExistenceSpaA} we know that the space $| \Spa \n{A}|$ is pro-constructible in $\n{T}_{\n{A}}$, so an spectral space. Since $\n{T}_{\n{A}}$ has a basis given by rational localizations, the same holds for $|\Spa \n{A}|$. 

\item   Let $\n{A}\to \n{B}$ be a morphism in $\AffRing^{b}_{R_{\sol}}$. We have a commutative diagram 
\[
\begin{tikzcd}
{|\Spa \n{B}|} \ar[r]  \ar[d] & \n{T}_{\n{B}} \ar[d] \\ 
{|\Spa \n{A}|} \ar[r] & \n{T}_{\n{A}},
\end{tikzcd}
\]
namely, the fiber of $x\in \n{T}_{\n{A}} $ in $|\Spa \n{B}|$ is given by analytic ring  $\n{A}(x)_{\cons}\otimes_{\n{A}} \n{B}$, and this vanishes if $\n{A}(x)_{\cons}=0$.  The right vertical arrow is spectral and the horizontal arrows are pro-constructible immersions, this implies that the left vertical arrow is spectral. It is clear that the inverse image of a  rational localization  is again a  rational localization.

\item Let $\n{A}=(A,A^+)_{\sol}$ be the analytic ring attached to an Tate Huber pair.  Let us write $\n{T}_{\n{A}}= \varprojlim_{I} \Spa(R\langle T_I \rangle, R^+ +R\langle T_I \rangle^{00})$, it is easy to see that  $\Spa(A,A^+) \to \n{T}_{\n{A}}$ is a pro-constructible immersion. Since the map $\rho_{\n{A}}:\n{S}((A,A^+)_{\sol})\to |\Spa \n{A}|$ is surjective, we  have immersions $|\Spa \n{A}| \to \Spa (A,A^+) \to \n{T}_{\n{A}}$. We are left to show that the map of locales $\n{S}(\n{A}) \to \Spa(A,A^+)$ is surjective. Let $x\in \Spa(A,A^+)$, we have a  map of affinoid rings $(A,A^+)\to (\kappa(x), \kappa(x)^+)$ and an induced map of topological spaces $|\Spa((\kappa(x),\kappa(x)^+ )_{\sol})| \to \Spa(\kappa(x),\kappa(x)^+ )$, thus one can reduce to the case of an affinoid field $\Spa(K,K^+)$. Then the open subsets of $\Spa(K,K^+)$ form a totally ordered set, and the connected constructible subspaces of $ \Spa(K,K^+)$ containing the generic point are in bijection with open integrally closed subrings $K^+\subset \tilde{K}^+\subset \n{O}_K$, with $\n{O}_K$ the valuation ring of $K$. On the other hand, the  functor $(A,A^+)\mapsto (A,A^+)_{\sol}$ is a fully faithful embedding of Huber pairs in analytic rings by \cite[Proposition 3.34]{Andreychev}. This shows that $|\Spa (K,K^+)_{\sol}| \to \Spa(K,K^+)$ must be a bijection which finishes the proof. 
\end{enumerate}
\end{proof}

\begin{remark}
\label{RemarkEmbeddingBoundedAffinoid}
By Lemma \ref{LemmaFullyFaithfulInjectionAnRings} and Theorem  \ref{TheoExistenceSpaA}, the functor $\AffRing^{b}_{R_{\sol}} \to \CatLoc_{\Spa R_{\sol}}$ sending $\n{A}$ to $\Spa \n{A}$ is  fully faithful  when restricted to bounded affinoid $R_{\sol}$-algebras over $\bb{Q}$. In particular, after specializing to $R\to \bb{Q}_p$ for any prime number $p$, we have a fully faithful embedding  $\AffRing^{b}_{\bb{Q}_p,\sol}\to \CatLoc_{\Spa \bb{Q}_{p,\sol}}$ from bounded affinoid $\bb{Q}_{p}$-algebras to categorified locales over $\Spa \bb{Q}_{p,\sol}$. 
\end{remark}

Thanks to the $\dagger$-nilradical we can define residue fields for both the analytic and constructible topologies of $|\Spa \n{A}|$.

\begin{definition}
Let $\n{A}\in \AffRing^{b}_{R_{\sol}}$, and let $x\in |\Spa \n{A}|$
\begin{enumerate}

\item The \textit{residue field}  of $\n{A}$ at $x$ is defined as the  $\dagger$-reduced quotient $\kappa(x):=\n{A}(x)^{\dagger-\red}$.

\item The \textit{constructible residue field} of $\n{A}$ at $x$ is the  $\dagger$-reduced quotient $\kappa(x)_{\cons}= \n{A}(x)_{\cons} ^{\dagger-\red}$. 

\end{enumerate}
\end{definition}

Next, we prove that the underlying rings of  the previous  residue fields are honest fields. We need the following lemma.

\begin{lemma}
\label{LemmaStalksAtPoints}
Let $\n{A}$ be a bounded affinoid ring. 
\begin{enumerate}
\item  The following are equivalent 
\begin{itemize}
\item[(a)]  The open subsets of $|\Spa \n{A}|$ form a totally ordered set. 

\item[(b)] There is a unique closed point in $|\Spa \n{A}|$. 

\item[(c)] For any $f,g\in \n{A}\backslash \Nil^{\dagger}(\n{A})$ either $\{|f| \leq |g|\neq 0\}= |\Spa \n{A}|$ or $\{|g|\leq |f|\neq 0\}=|\Spa \n{A}|$. 
\end{itemize}
 Moreover, adic stalks of bounded affinoid rings satisfy these equivalent properties.

\item  The following are equivalent 
\begin{itemize}
\item[(a)] $|\Spa \n{A}|$ is a point.

\item[(b)] $|\Spa \n{A}|$ has a unique closed point and for any $f\in \n{A}$ we have $f\in \n{A}^+$ or $f$ is invertible and $f^{-1}\in \n{A}^{00}$.
\end{itemize}
Moreover, constructible stalks of bounded affinoid rings satisfy these equivalent properties.
\end{enumerate}

\end{lemma}
\begin{proof}
\begin{enumerate} 

\item Suppose that the open subsets of $|\Spa \n{A}|$ form a total order. By taking complements, the closed subspaces also form a total order. Let $\s{I}$ be the total ordered family of non-empty  closed subspaces. Then, since $|\Spa \n{A}|$ is constructible,  by Zorn's lemma  one has that $Z= \bigcap_{C\in \s{I}} C$ is the minimal non-empty closed subspace of $|\Spa \n{A}|$.  The space $Z$ is pro-constructible, so it is spectral. Suppose that $Z$ has more than two points, as it is a $T_0$-topological space,  there is a non-empty properly contained closed subspace in $Z$ which is a contradiction with the fact that it is the minimal closed subspace of $|\Spa \n{A}|$. Thus, $Z$ is a point showing that (a) implies (b). 

Suppose that $|\Spa \n{A}|$ has a unique closed point $x$. As $|\Spa \n{A}|$ is spectral,  the unique open subset of $|\Spa \n{A}|$ containing $x$ is $|\Spa \n{A}|$.  Let $f\in \n{A}$, if for all $n\in \bb{N}$ the open set $\{|f|\leq |\pi^n|\}$ contains $x$, then the map $R[T]\to \n{A}$ defined by $f$ extends to $R\{T\}^{\dagger} \to \n{A}$ proving that $f\in \Nil^{\dagger}(\n{A})$. Thus, for $f\in \n{A} \backslash \Nil^{\dagger}(\n{A})$ there is some $n\in \bb{N}$ such that $\{|f|\leq |\pi^n|\}$ does not contain $x$, which implies that $\{|\pi^n|\leq |f|\}$ does contain $x$ and therefore that $\{|\pi^n|\leq |f|\}= |\Spa \n{A}|$.  In particular such an $f$ must be invertible.  Now let $f,g\in \n{A} \backslash \Nil^{\dagger}(\n{A})$, then the open sets $\{|f/g|\leq 1\}$ and $\{|g/f|\leq 1\}$ form an open cover of $|\Spa \n{A}|$, in particular $x$ belongs to one of them, which shows that either $\{|f|\leq |g|\neq 0\}=|\Spa\n{A}|$ or $\{|g|\leq |f|\neq 0\}=|\Spa \n{A}|$ as wanted. 

Now suppose that (c) holds. It suffices to show that the poset of open rational subspaces forms a total order.  Let $f\in \n{A}\backslash  \Nil^{\dagger}(\n{A})$. Then there is some $n\in \bb{N}$ such that $\{|f|\leq |\pi^n|\}\neq |\Spa \n{A}|$, by hypothesis this implies that $\{|\pi^n|\leq |f|\neq 0\}=|\Spa \n{A}|$ proving that $f$ is invertible. We define the following partial order in $\n{A}\backslash  \Nil^{\dagger}(\n{A})$:  we say that $|f|\leq |g| $ if $\{|f|\leq |g|\neq 0\}=|\Spa \n{A}|$.  By hypothesis, given two elements $f,g\in \n{A}\backslash \Nil^{\dagger}(\n{A})$ we have either $|f|\leq |g|$ or $|g|\leq |f|$.   Let $U \subset \Spa \n{A}$ be a rational set of the form $\{|f_i|\leq |g|\neq 0:\; i=1,\ldots, d\}$ with $f_d = \pi^n$ for some $n\in \bb{N}$. If $U$ is non-empty then $g\notin \Nil^{\dagger}(\n{A})$, in particular it is invertible and by taking $h_i=f_i/g$ we can write $U=\{|h_i|\leq 1: i=1,\ldots,d\}$. Let $h$ be one of the $h_i$ with maximal norm $|h|$, then $U=\{|h|\leq 1 \}$.  Now, if $U=\{|h|\leq 1\}$ and $V=\{|g|\leq 1\}$, as we have either $|h|\leq |g|$ or $|g|\leq |h|$, then $U\subset V$ or $V\subset U$ proving that the rational open subspaces of $\Spa \n{A}$ form a total order. 

Finally the last assertion about stalks of bounded affinoid rings hold since property (c) can be easily verified by construction. 

\item   Suppose that $|\Spa \n{A}|$ is a point and let $f\in \n{A}$, we have an induced map $|\Spa \n{A}| \to \Spa( R\langle T \rangle, R^++R\langle T \rangle^{00})$ sending $T$ to $f$. We can write $\Spa( R\langle T \rangle, R^++R\langle T \rangle^{00})=\{|T|\leq 1\}\bigsqcup \{|T|>1\}$. Since $|\Spa \n{A}|$ is a point it must land in one and only one term of the disjoint union, which translates in property (b) by definition of $\n{A}^+$ and $\n{A}^{00}$.

Conversely, suppose that (b) holds. By the proof of part (1), all the rational subspaces of $|\Spa \n{A}|$ are of the form $\{|f|\leq 1\}$ for some $f\in \n{A}$. But then, if $\{|f|\leq 1\}$ does not contain the maximal point of $|\Spa \n{A}|$, one has that $f\notin \n{A}^+$, which implies that the complement $\{|f^{-1}|<1\}= |\Spa \n{A}|$, i.e. that $\{|f|\leq 1\}=\emptyset$. This shows that $|\Spa \n{A}|$ has the trivial topology, and being a spectral space with a unique closed point it must consist on a single point.

Finally, the last assertion about constructible stalks holds since property (b) can be easily verified by construction. 
\end{enumerate}

\end{proof}

\begin{lemma}
Let $\{\n{A}_i\}_{i}$ be a sifted diagram of bounded affinoid rings with colimit $\n{A}$, then the natural map $|\Spa \n{A}| \to \varprojlim_i |\Spa \n{A}_i|$ is a homeomorphism.
\end{lemma}
\begin{proof}
We have a natural map $f:|\Spa \n{A}| \to \varprojlim_{i} |\Spa \n{A}_{i}|$. Since $\n{A}[*]= \varinjlim_{i} \n{A}_i[*]$, any rational localization of $|\Spa \n{A}|$ arises as the pullback of a rational localization of some $|\Spa \n{A}_i|$. In particular, any constructible set of $|\Spa \n{A}|$ is the pullback of some constructible set of some $|\Spa \n{A}_i|$. Thus, it suffices to show that $f$ is a bijection, this can be proved using the constructible topology. Let $x_i\in |\Spa \n{A}_i|$ be a compatible sequence of points, and let $\n{A}_{i}(x_i)_{\cons}$ be the constructible stalk of $\n{A}_i$ at $x_i$. Then we have a map 
\[
\n{A} \to \varinjlim_{i} \n{A}_{i}(x_i)_{\cons}
\]
where the right term is non-zero as none of the analytic rings are zero. By Lemma \ref{LemmaStalksAtPoints} (2.b), the adic spectrum of  $ \varinjlim_{i} \n{A}_{i}(x_i)_{\cons}$ is a point. But  $\varinjlim_{i} \n{A}_{i}(x_i)_{\cons}$ is also the fiber of $f$ along the sequence $(x_i)_{i}$, this shows that $f^{-1}((x_i)_i)=\{x\}$ is a point, proving that $f$ is indeed a bijection. 
\end{proof}

\begin{lemma}
\label{LemmaQuotientRing}
Let $\n{A}\to \n{B}$ be a morphism in $\AffRing^{b}_{R_{\sol}}$ such that $\n{B}= \underline{\n{B}}_{\n{A}/}$ and that $\pi_0 \underline{\n{A}} \to \pi_0 \underline{\n{B}}$ is surjective. Then $F: |\Spa \n{B}| \to |\Spa \n{A}|$ is an  immersion. 
\end{lemma}
\begin{proof}
Let $x\in \Spa \n{A}$, then $\n{B}\otimes_{\n{A}}  \n{A} (x)_{\cons}$ is either $0$ or one has a surjection on $\pi_0$ of $\n{A}(x)_{\cons} \to \n{B}\otimes_{\n{A}} \n{A} (x)_{\cons}$.  One easily verifies that the tensor satisfies the condition (2.b) of Lemma \ref{LemmaStalksAtPoints}, this shows that the fiber $F^{-1}(x)$ is either empty or a point. Furthermore, any element $g\in \n{B}$ can be lifted to an element $\tilde{g}\in \n{A}$, this implies that a rational subspace of $|\Spa \n{B}|$ arises as the pullback of a rational subspace of $|\Spa \n{A}|$, and that $F$ is an immersion. 
\end{proof}

\begin{prop}
\label{PropAdicSpectrumInvariantReduced}
Let $\n{A}$ be a bounded affinoid ring. Then the natural map $|\Spa \n{A}| \to |\Spa \n{A}|^{\dagger-\red}$ is a homeomorphism preserving rational localizations. Moreover, for any $x\in |\Spa \n{A}|$ the underlying discrete rings of $\kappa(x)$ and $\kappa_{\cons}(x)$ are fields.
\end{prop}
\begin{proof}
 By Lemma \ref{LemmaQuotientRing} we have an immersion $|\Spa \n{A}^{\dagger-\red}| \to |\Spa \n{A}|$. It suffices to show that it is bijective. But the constructible residue field of $\n{A}$ at $x$ factors through $\n{A}^{\dagger-\red}$, proving  the claim. Finally, the fact that the underlying discrete rings of  $\kappa(x)$ and $\kappa_{\cons}(x)$ are fields follows from Lemma \ref{LemmaStalksAtPoints}. 
\end{proof}

\begin{cor}
\label{PropZariskiCloseImmersion1}
 Let $\n{A}\to \n{B}$ be as in Lemma \ref{LemmaQuotientRing}.  If $\pi_0(I)$ is generated by its discrete points $\pi_0(I(*))$, then  the image of $F$ is the Zariski closed subspace  $\{|f|=0 : f\in I\}$.
\end{cor}
\begin{proof}
We can assume that both rings are static.  Then  $Z = \bigcap_{f\in I} \{|f|=0\}\subset |\Spa \n{A}|$ corresponds to the analytic ring $\n{A}_{I^{\dagger}}= \varinjlim_{\substack{f\in I \\ n\in \bb{N}}} \n{A}\langle \frac{f}{\pi^n} \rangle$. Therefore, the map $ \n{A} \to (\n{A}_{I^{\dagger}})^{\dagger-\red}$ factors through $\n{B}$ proving that the image of $|\Spa \n{B}| $ in $|\Spa \n{A}|$ is $Z$ by Proposition \ref{PropAdicSpectrumInvariantReduced}. 
\end{proof}

We do not know if a morphism $\n{A} \to \n{B}$ in $\AffRing^{b}_{R_{\sol}}$ that is surjective on $\pi_0$ induces a closed immersion in the underlying adic spaces. Nevertheless, it defines a closed subspace in a suitable quotient of the locale $S(\n{A})$.

\begin{definition}
\label{DefDaggerLocale}
Let $\n{A}\in \AffRing^{b}_{R_{\sol}}$, we let $|\Spa^{\dagger}\n{A}|$ denote the quotient of $S(\n{A})$ consisting of the idempotent algebras generated under arbitrary intersections and finite unions by iterations of idempotent  algebras of the form $\n{A}\otimes_{\bb{Z}[T]}\bb{Z}[[T]]$ and $\n{A}\otimes_{\bb{Z}[T]} \bb{Z}((T^{-1}))$  for some $\bb{Z}[T]\to \n{A}$, and algebras $\n{A}\otimes_{R_{\sol}[\bb{N}[S]]} R_{\sol}\langle \bb{N}[S]\rangle$ for a map from a profinite set $S\to \n{A}$. 
\end{definition}

\begin{lemma}
The map of locales $S(\n{A}) \to |\Spa \n{A}|$ factors as a map
\[
S(\n{A}) \to |\Spa^{\dagger}(\n{A})| \to |\Spa \n{A}|.
\]
\end{lemma}
 \begin{proof}
 By Lemma \ref{LemmaDevisageRationalLocalizations} the open subsets of $|\Spa \n{A}|$ are generated by  composite of   subspaces  of the form $\{1\leq |g|\}$ and $\{ |g| \leq 1\}$. The complement of these spaces  correspond to the idempotent algebras $\n{A} \otimes_{\bb{Z}[T]} \bb{Z}[[T]]$ and $\n{A}\otimes_{\bb{Z}[T]} \bb{Z}((T^{-1}))$ respectively. The lemma follows by Theorem \ref{TheoExistenceSpaA} and the definition of $|\Spa^{\dagger} \n{A}|$.  
 \end{proof}

 \begin{lemma}
 \label{LemmaInvarianceSpaOfNilreduction}
 Let $\n{A} \in \AffRing^b_{R_{\sol}}$, then the natural map 
 \[
 |\Spa^{\dagger} \n{A}^{\dagger-\red}| \to  |\Spa^{\dagger} \n{A}|
 \]
 is an isomorphism of locales.
 \end{lemma}
\begin{proof}
This follows from the invariance of localizations of the form $\bb{Z}[T] \to \bb{Z}[T]_{\sol}$ and $ R_{\sol}[\bb{N}[S]] \to R_{\sol}\langle \bb{N}[S]  \rangle$ under the $\dagger$-nil-radical of Proposition \ref{PropLiftingOverconvergentAlgebrasLift}. 
\end{proof}

\begin{proposition}
\label{PropositionCLosedImageLocale}
Let $\n{A} \to \n{B}$ be a map of bounded affinoid rings such that $\n{B}$ has the induced analytic structure and that is surjective on $\pi_0$. Then the natural map 
\[
|\Spa^{\dagger} \n{B}| \to |\Spa^{\dagger} \n{A}|
\]
is a closed immersion of locales. 
\end{proposition} 
\begin{proof}
By lemma \eqref{LemmaInvarianceSpaOfNilreduction} it suffices to construct an idempotent algebra $\n{A}'$ in $|\Spa^{\dagger} \n{A}|$ such that we have a factorization $\n{A}\to \n{A}' \to \n{B}$ and that $\n{A}^{',\dagger-\red}= \n{B}^{\dagger-\red}$. Let $I=[\n{A} \to\n{B}]$ be the fiber, for any  profinite set $S$ and any map $S\to \pi_0(I)$ let us consider the base change $\n{A}\otimes_{R_{\sol}[\bb{N}[S]]} R_{\sol}\{\bb{N}[S]\}^{\dagger}$, and let $\n{A}'$ be the colimit of all such algebras. Then, by construction, the map $\n{A}\to \n{A}'$ sends $I$ to  the $\dagger$-nil-radical $\Nil^{\dagger}(\n{A}')$ of $\n{A}'$. This shows that 
\[
\n{A}^{',\dagger-\red}= \n{B}^{\dagger-\red}
\]
as wanted. 
\end{proof}

\subsubsection{Derived Tate adic spaces} We end this section with the definition of   derived  Tate   adic spaces.

\begin{definition}
\begin{enumerate}
\item We let $\Aff_{\bb{Z}_{\sol}}:= \AffRing_{\bb{Z}_{\sol}}^{\op}$ be the $\infty$-category of \textit{solid affinoid spaces}. For a ring $\n{A} \in \AffRing_{\bb{Z}_{\sol}}$, we  let $\Aff_{\n{A}}$ be the $\infty$-category of solid affinoid spaces over $\n{A}$.  We also denote by $\AnSpec \n{A}$  the representable presheaf on anima over $\Aff_{\bb{Z}_{\sol}}$ defined by $\n{A}$, we call $ \AnSpec \n{A}$ the \textit{analytic spectrum} of $\n{A}$.  

 \item We let $\Aff^{b}_{R_{\sol}}:= \AffRing^{b,\op}_{R_{\sol}}$  be the category of \textit{bounded affinoid spaces} over $R_{\sol}$.   The \textit{analytic topology} in $\Aff^{b}_{R_{\sol}}$ is the Grothendieck topology defined by open affinoid coverings of $\Spa \n{A}$. 
\end{enumerate}
\end{definition}

\begin{lemma}
\label{PropLocaleTopIsSubcanonial}
The analytic topology of $\Aff^{b}_{R_{\sol}}$ is subcanonical. 
\end{lemma}
\begin{proof}
Let $\SpecAn \n{B}\in \Aff^{b}_{R_{\sol}}$, we want to prove that the functor $\Map_{\Aff^b_{R_{\sol}}}(-, \SpecAn \n{B})=\Map_{\AffRing^{b}_{R_{\sol}}}(\n{B},-)$ satisfies descent for the analytic topology of $\Aff^{b}_{R_{\sol}}$. Given $\SpecAn \n{A}$ a bounded affinoid ring, by definition of  the category of analytic rings,  $\Map_{\AffRing_{R_{\sol}}}(\n{B}, \n{A})$ is the full subanima of $\Map_{\AniRing_{R_{\sol}}}(\underline{\n{B}}, \underline{\n{A}})$ whose connected components are those arrows $f:\underline{\n{B}} \to  \underline{\n{A}}$ such that any $\n{A}$-complete module is $\n{B}$-complete. Now let $\{\n{A}_i\}_{i=1}^n$ be an analytic cover of $\n{A}$, let $\n{C}= \prod_{i=1}^n \n{A}_i$ and let $\n{C}^{n}$ be the $n$-th fold tensor product of $\n{C}$ over $\n{A}$. The maps $\{\Spa \n{A}_i\}_{i}$ form an open cover of  the categorified locale $\Spa \n{A}$ and   by  Theorem \ref{TheoLocaleTopology} one has descent  of modules
\begin{equation}
\label{eqDescentAnalyticCover}
\Mod(\n{A}) \to \varprojlim_{[n]\in \Delta} \Mod(\n{C}^{n+1}). 
\end{equation}
In particular, the natural map $\underline{\n{A}} = \varprojlim_{[n]\in \Delta} \underline{\n{C}}^{n+1}$ is an equivalence. Thus, the map 
\[
\Map_{\AffRing^{b}_{R_{\sol}}}(\n{B},\n{A}) \to \varprojlim_{[n]\in \Delta} \Map_{\AffRing^b_{R_{\sol}}}(\n{B}, \n{C}^{n+1})
\]
is a fully faithful embedding, and to prove that it is an equivalence it suffices to check that it is essentially surjective, but this follows from \eqref{eqDescentAnalyticCover} and the fact that $\Mod(\n{B})$ is stable under limits and colimits in $\Mod(\underline{\n{B}})$. 
\end{proof}

\begin{definition}
\label{DefinitionAnalyticDerivedAdicSpace}
We let $\Sh_{\an}(\Aff^{b}_{R_{\sol}})$  denote the sheaves on anima of bounded affinoid spaces with respect to the analytic topology.   A \textit{derived  Tate adic space over $R_{\sol}$} (or more shortly, a derived  adic space) is a sheaf $X\in \Sh_{\an}(\Aff^{b}_{R_\sol})$ that admits an open analytic cover by representable sheaves.  We let $\AdicSp_{R_{\sol}}$ be the full $\infty$-subcategory  of $\Sh_{\an}(\Aff^{b}_{R_{\sol}})$ consisting on   derived   adic spaces over $R_{\sol}$.

Given $X$ a derived Tate adic space, let $\Mod_{X,\sol}=\varprojlim_{\SpecAn \n{A}\to X}\Mod(\n{A})$ be its $\infty$-category of solid quasi-coherent sheaves  on $X$, and let   $|X|= \varinjlim_{\SpecAn \n{A}\to X} |\SpecAn \n{A}|$ be  its associated topological space. We let   $X_{adic}:=(|X|, \Mod_{X,\sol})$ denote the categorified locale of $X$ obtained as the colimit of the categorified locales on bounded affinoid spaces mapping to $X$ 
\end{definition}

The following  corollary follows from the definitions and Lemma \ref{LemmaFullyFaithfulInjectionAnRings}.

\begin{corollary}
Let $X$ be a  derived Tate adic space, $\{U_{i}\}_{i}$ an analytic  open cover of $X$ by affinoid spaces, and $\Mod_{X,\sol}=\varprojlim_{i} \Mod_{U_i,\sol}$. Then $|X|= \varinjlim_{i} |U_{i}|$ is a locally spectral space. Morphisms of derived Tate adic spaces $X\to Y$ induce morphisms of locally spectral spaces $|X|\to |Y|$. When restricted to derived Tate adic spaces over $\bb{Q}\otimes R_{\sol}$, the functor $X\mapsto X_{adic}$ from derived Tate adic spaces to categorified locales over $\bb{Q}\otimes R_{\sol}$ is fully faithful (eg. for derived Tate adic spaces over $\bb{Q}_p$).  
\end{corollary}

\subsubsection{Analytification functor}

We finish this section by defining an analytification functor. Let $\PSh(\Aff_{R_{\sol}})$ be the category of presheaves on anima of solid affinoid rings over $R_{\sol}$. 

\begin{definition}
\label{DefinitionAnalytificationFunctor}
We define the analytification functor $(-)^{\an}$  to be  the composite $\PSh(\Aff_{R_{\sol}})\xrightarrow{k^{*}} \PSh(\Aff^{b}_{R_{\sol}})\to \Sh_{\an}(\Aff^{b}_{R_{\sol}})$, where the first is the restriction along the inclusion $k:\Aff^{b}_{R_{\sol}}\to \Aff_{R^{\sol}}$, and the second is sheafification. 
\end{definition}

By Proposition \ref{PropCompactProjectiveSolidAffinoid}, the category $\AffRing_{\bb{Z}_{\sol}}$ of solid affinoid rings is generated by the compact projective objects $\bb{Z}[\underline{T}]_{\sol} [\bb{N}[S]]$, where $\underline{T}$ is a finite set of variables, and $S$ a profinite set. Therefore, the analytification functor $(-)^{\an}$ is the left Kan extension of its restriction to the objects $R\langle\underline{T} \rangle_{\sol} [\bb{N}[S]]$. These are computed as follows:

\begin{lemma}
\label{LemmaAnalytificationFunctor}
Let $\bb{D}_{R}:=\AnSpec R\langle  T \rangle_{\sol}$ be the unit affinoid disc. For $S$ a profinite set let us write $\bb{A}^{\alg}_{R,S}:=\AnSpec(R_{\sol}[\bb{N}[S] ])$ and  $\bb{A}_{R,S}:=\bigcup_{n\in\bb{N}} \AnSpec(R_{\sol}\langle \bb{N}[\pi^{n} S] \rangle) $. Then there is a natural equivalence
\[
(\bb{D}_{R}^{n}\times \bb{A}_{R,S}^{\alg} )^{\an}= \bb{D}_{R}^{n}\times \bb{A}_{R,S}^{\an}.
\]
\end{lemma}
\begin{proof}
This follows from the fact that both $k^*$ and sheafification commute with finite limits, that $\bb{D}_R^n$ is already a bounded affinoid space, and that  $\bb{A}^{\an}_{R,S}$ represents the functor on $\Aff^{b}_{R_{\sol}}$ given by $\n{A}\mapsto \underline{\n{A}}(S)$. 
\end{proof}

\section{Tate stacks}
\label{SectionAnalyticAdicStacks}

In this section we introduce a geometric framework to do derived rigid geometry. Following the theory of analytic stacks of Clausen and Scholze, we use the  abstract $6$-functor formalisms of Mann \cite{MannSix,MannSix2}, revisited in \cite{zavyalov2023poincare} and \cite{SixFunctorsScholze}, to construct  very general categories of solid and Tate stacks. We discuss other classical geometric objects and features like finitely presented morphisms of derived Tate adic spaces, the theory of the cotangent complex for analytic rings, formally \'etale and smooth morphisms, and  Serre duality. Finally, we introduce   new deformation properties for morphisms, called $\dagger$-formally smoothness and \'etaleness, that will be key in the theory of the analytic de Rham stack. 

\subsection{Recollections on abstract six functor formalisms}
\label{SubsectionAbstractSix}

In this section we briefly recall the definition of a six functor formalism and some of its most important features, we follow \cite{MannSix,MannSix2} and \cite{SixFunctorsScholze}.

\subsubsection{Abstract six functor formalisms}  A \textit{geometric set up} is a pair $(\n{C},E)$ consisting on an $\infty$-category $\n{C}$ and a collection $E$ of homotopy classes of edges in $\n{C}$ such that $E$ contains all isomorphisms, and  is stable under compositions and pullbacks. Throughout this section we  assume that $\n{C}$ admits finite limits. 

Provided the data $(\n{C},E)$, one constructs a symmetric monoidal $\infty$-category of correspondences $\Corr(\n{C}, E)$, see \cite[Definition A.5.4]{MannSix} and \cite[Definition 2.3]{SixFunctorsScholze}.  In a more instructive way, the homotopy category of $\Corr(C,E)$ has the following description: the objects of $\Corr(\n{C},E)$ are the objects of $\n{C}$, the symmetric monoidal structure is given by direct products,  an arrow from $X$ to $Y$ is a correspondence 
\[
\begin{tikzcd}
& W\ar[ld,"f"'] \ar[rd,"g"] & \\ 
X & & Y
\end{tikzcd}
\]
with $g\in E$, and the composition of two arrows is given by  the outer correspondence of the following diagram
\[
\begin{tikzcd} 
& & W\times_{Y} W' \ar[ld] \ar[rd] & & \\  & W\ar[ld] \ar[rd] & & W' \ar[ld] \ar[rd] \\ X & & Y & & Z.
\end{tikzcd}
\]

\begin{definition}[{\cite[Definition A.5.6]{MannSix}}]
A $3$-functor formalism  (or a pre $6$-functor formalism) on $(\n{C},E)$ is a lax symmetric monoidal functor 
\[
\s{D}:\Corr(\n{C},E)\to \Cat_{\infty}
\]
where $\Cat_{\infty}$ is endowed with the cartesian symmetric monoidal structure. 
\end{definition}

As it is explain in the paragraph after \cite[Definition 2.4]{SixFunctorsScholze}, the data of a $3$-functor formalism encodes a functor $\s{D}: \n{C} \to \Cat^{\otimes}_{\infty}$ from $\n{C}$ to symmetric monoidal $\infty$-categories, the pullback functors $f^*$, and the lower shriek functors $f_!$, in such a way that a diagram 
\[
\begin{tikzcd}
& W\ar[ld,"f"'] \ar[rd,"g"] & \\ 
X & & Y
\end{tikzcd}
\]
is sent to the functor $g_!f^*: \s{D}(X)\to \s{D}(Y)$.

\begin{definition}
A $6$-functor formalism is a $3$-functor formalism for which the symmetric monoidal categories $\s{D}(X)$ for $X\in \n{C}$ are closed, and the functors $f^*$ and $f_!$ have right adjoints $f_*$ and $f^!$ respectively. 
\end{definition}

\begin{remark}[Dual $3$-functor formalism]
\label{RemarkDualSixFunctor}
From the datum of a $3$-functor formalism $\s{D}:\Corr(\n{C},E)\to \Cat_{\infty}$ it is possible to construct a \textit{dual}  $3$-functor formalism $\s{D}^{op}$ as in \cite[Remark 6.3]{SixFunctorsScholze}. Concretely, $\s{D}^{op}$ is constructed as the composite $\Corr(\n{C},E)\xrightarrow{\s{D}} \Cat_{\infty} \xrightarrow{(-)^{op}} \Cat_{\infty}$, see  \cite[Remark 2.4.2.7]{HigherAlgebra}. At the level of objects, it maps $X$ to the opposite symmetric monoidal category $\s{D}(X)^{op}$.
\end{remark} 

The following  lemma allows us to construct six functors by taking a precomposition. 

\begin{lemma}
\label{LemmaFunctoriality6Functors}
Let $(\n{C},E)$ be a geometric set up and let $\n{C}'$ be an $\infty$-category with finite limits.   Let $F: \n{C}' \to  \n{C}$ be a functor preserving final objects and cartesian squares.   Let $E'$ be the class of edges $s$ in $\n{C}'$ such that $F(s)\in E$, then $(\n{C}',E')$ is a geometric set up and $F$ induces  a natural symmetric monoidal functor
\[
\Corr(F):\Corr(\n{C}',E')\to \Corr(\n{C},E).
\]
In particular, if $\s{D}: \Corr(\n{C},E) \to \Cat_{\infty}$ is a $3$-functor formalism then $\s{D}\circ \Corr(F): \Corr(\n{C}',E')\to \Cat_{\infty}$ is also a $3$-functor formalism.
\end{lemma}
\begin{proof}
We use the notation of \cite[Definitions A.5.2 and A.5.4]{MannSix}. First, note that the class of arrows $E'$ is stable under compositions and pullbacks since $F$ preserves cartesian squares.  Let us first see that the hypothesis imply that there is a natural functor $\Corr(F): \Corr(\n{C}',E')\to \Corr(\n{C}, E)$. Indeed by construction, $\Corr(\n{C},E)$ is the simplicial subset of $B(\n{C})$ whose $n$-cells are maps $C(\Delta^n) \to \n{C}$ sending vertical edges to $E$ and exact squares to pullback squares. Since $F:\n{C}' \to \n{C}$ preserves cartesian squares, the restriction of $B(F): B(\n{C}')\to B(\n{C})$ to $\Corr(\n{C}',E')$ lands in $\Corr(\n{C},E)$ as wanted. For the symmetric monoidal structure, by \cite[Definition A.5.4]{MannSix} one has 
\[
\Corr(\n{C},E)^{\otimes}=\Corr(( \n{C}^{\op,\bigsqcup})^{\op},E^{-}),
\]
where $\n{C}^{\op,\bigsqcup}$ is the symmetric monoidal structure define by co-products, and the class of edges $E^-$ are those living over $\id:\langle n \rangle \to \langle n \rangle$ for $n\in \bb{N}$ of the form $f: (Y_j)_{1\leq j\leq n} \to (X_i)_{1\leq i\leq n}$ where $Y_i \to X_i$ is in $E$. Then, since $F$ preserves final objects and cartesian squares, it sends co-products to co-products in the opposite categories, so that we have a symmetric monoidal functor 
\[
(F^{\op,\bigsqcup})^{\op}: (\n{C}^{'\op,\bigsqcup})^{\op} \to (\n{C}^{\op,\bigsqcup})^{\op}.
\]
Moreover, by definition the functor $(F^{\op,\bigsqcup})^{\op}$ still sends the edges $E^{'-}$ to $E^{-}$. Then, to finish the proof, we need to see that the natural functor 
\[
B((F^{\op,\bigsqcup})^{\op}): B(  (\n{C}^{'op,\bigsqcup})^{\op}) \to  B((\n{C}^{op,\bigsqcup})^{\op})
\]
restricts to a functor in the correspondence categories. This follows from the fact that  $(F^{\op,\bigsqcup})^{\op}$ still preserves cartesian diagrams and that it sends $E^{'-}$ to $E^{-}$.   
\end{proof}

One of the major contributions of \cite{MannSix} is the construction of $6$-functor formalisms from a minimal amount of data that is of easy access in practice, namely, we are usually given  a functor $\s{D}: \n{C}\to \Cat^{\otimes}_{\infty}$ with values in symmetric monoidal (stable)  $\infty$-categories, and two classes of \textit{\'etale} and \textit{proper} maps $I$ and $P$.  It turns out that  if  the data $(\n{C},\s{D}, I ,P)$ satisfies a minimal set of expected properties, one can construct a $6$-functor formalism  for $(C,E)$ in such a way that all element in $E$ is written as $p\circ j$ with $p\in P$ and $j\in I$, that for $f\in P$ one has $f_*=f_!$, and that for $f\in I$ one has $f^!=f^*$. For the precise statement see \cite[Proposition A.5.10]{MannSix}.

On the other hand, the results of Mann permit the extension of a six functor formalism on $(\n{C}, E)$ to a very general class of arrows in a  suitable category of sheaves on anima of $\n{C}$. To state such an extension theorem we need some definitions, we refer to \cite[Appendix of Lecture IV]{SixFunctorsScholze} for more details.

   Let $(\n{C}, E)$ be a geometric set up and suppose that the six functor formalism $\s{D}:\Corr(\n{C},E)\to \n{P}r^{L,\ex}$ takes  values in presentably  stable  $\infty$-categories.    Let $\widetilde{\n{C}}^{psh}$ be the $\infty$-category of presheaves of anima of $\n{C}$, and $\s{D}: \Corr(\widetilde{\n{C}}^{psh}, \widetilde{E}^0)\to \n{P}r^{L,ex}$ its natural extension to a six-functor formalism on presheaves of anima, where $\widetilde{E}^0$ are the arrows whose pullbacks to $\n{C}$ are representable in $E$ (cf. \cite[Proposition A.5.16]{MannSix}).

\begin{definition}[{\cite[Definition 4.14]{SixFunctorsScholze}}]
\label{DefinitionSixFunctorDtop}

Consider $\{f_i:X_i\to Y\}$ a family of objects in $\n{C}$. 

\begin{enumerate}

\item The maps $f_i$ form a \textit{canonical cover} if for all $Z\in \n{C}$ an any $Y'\to Y$ in $\n{C}$ with pullback $f_i': X_i'\to Y'$, the functor $\Hom_{\n{C}}(-,Z)$ satisfies descent along  $\{f_i'\}$. 

\item  The maps $f_i$ satisfy \textit{universal $*$-descent} if for all pullbacks $\{f_i':X_i'\to Y'\}$ along a map $Y'\to Y$ in $\n{C}$,  the functor $\s{D}^*$ satisfy  descent along $\{f_i'\}$ (i.e. where the transition maps are given by $f^*$-maps).

\item Assume all $f_i$ are in $E$. The maps $f_i$ satisfy \textit{universal $!$-descent} if for all pullbacks  $\{f_i': X_i'\to Y'\}$ along $Y'\to Y$ from a presheaf on anima on $\n{C}$, the functor $\s{D}^!$ satisfies descent along $\{f_i'\}$ (i.e. where the transition maps are given by $f^!$-maps). 

\end{enumerate}

 A \textit{$\s{D}$-cover} is a family  $\{f_i: X_i\to Y\}$ of objects in $E$ such that they form a cover in the canonical topology, satisfy universal $*$-descent, and satisfy universal $!$-descent.  The \textit{$\s{D}$-topology} on $\n{C}$ is the topology generated by  $\s{D}$-covers,  we let $\widetilde{\n{C}}$ denote the $\infty$-category of sheaves on $\n{C}$ for the $\s{D}$-topology.    
\end{definition}
 
Let $\widetilde{E}^0$ the class of arrows in $\widetilde{\n{C}}$  represented by arrows in $E$. As we saw above, the six functor formalism of $(\n{C},E)$ extends to $(\widetilde{\n{C}}, \widetilde{E}^0)$, we want to use the theory of \cite[Appendix A.5]{MannSix} to enlarge the class of arrows $\widetilde{E}^0$ by localizing the target and the source of a map, this leads to the following definition. 

\begin{definition}[{\cite[Definition 4.18]{SixFunctorsScholze}}]
\label{DefinitionStabilityE}
 Let $\widetilde{E}^0\subset \widetilde{E}$ be a class of morphisms in $\widetilde{\n{C}}$ that is stable under pullbacks and compositions. 
 
 \begin{enumerate}
\item The class $\widetilde{E}$ is \textit{stable under disjoint unions} if whenever $f_i :\widetilde{X}_i\to \widetilde{Y}$ are  morphisms in $\widetilde{E}$ then $\bigsqcup_i f_i: \bigsqcup_i \widetilde{X}_i \to \widetilde{Y}$ is in $\widetilde{E}$. 

\item The class $\widetilde{E}$ is \textit{local on the target} if whenever $\widetilde{f}: \widetilde{X}\to \widetilde{Y}$ is a morphism in $\n{C}$ such that  for all $Y\in \n{C}$ with map $Y\to \widetilde{Y}$, the pullback $\widetilde{X}\times_{\widetilde{Y}} Y \to Y$ is in $\widetilde{E}$, then $f\in \widetilde{E}$.  

\item Assume that the six functors of  $(\widetilde{\n{C}},\widetilde{E}^0)$ extend  uniquely to $(\widetilde{\n{C}},\widetilde{E})$. The class $\widetilde{E}$ is \textit{local on the source} if whenever $\widetilde{f}: \widetilde{X}\to \widetilde{Y}$ is a morphism in $\widetilde{C}$ such that there is some map $\widetilde{g}: \widetilde{X}'\to \widetilde{X}$ in $\widetilde{E}$ that is of universal $!$-descent, and such that $\widetilde{f}\circ \widetilde{g}$ lies in $\widetilde{E}$, then $\widetilde{f}\in \widetilde{E}$.

\item  Assume that the six functors of  $(\widetilde{\n{C}},\widetilde{E}^0)$ extend  uniquely to $(\widetilde{\n{C}},\widetilde{E})$. The class $\widetilde{E}$ is \textit{tame} if whenever $Y\in \n{C}$ and $\widetilde{f}: \widetilde{X}\to Y$ is a map in $\widetilde{E}$, then there are morphisms $h_i:X_i\to Y$ in $E$ and a  morphism $\bigsqcup_i X_i\to \widetilde{X}$ over $Y$ that lies in  $\widetilde{E}$ and is of universal $!$-descent. 
 
 \end{enumerate}
\end{definition}

\begin{theorem}[{\cite[Theorem 4.20]{SixFunctorsScholze}}]
\label{TheoSixFunctorsScholze}
There is a minimal collection of morphisms $\widetilde{E}^0\subset \widetilde{E}$ of $\widetilde{\n{C}}$ such that $\s{D}$ extends uniquely from $(\widetilde{\n{C}}, \widetilde{E}^0)$ to $(\widetilde{\n{C}},\widetilde{E})$, and such that $\widetilde{E}$ is stable under disjoint unions, local on the  target, local on the source, and tame. 
\end{theorem}

\subsubsection{The Lu-Zheng $2$-category}

Let $(\n{C},E)$ be a geometric set up  with finite limits and $\s{D}$ a six functor formalism on $(\n{C},E)$. We  assume that $\s{D}:\Corr(\n{C},E)\to \n{P}r^{L,\ex}$ takes values in presentable stable $\infty$-categories.  Another important tool in the theory of six-functor formalisms is the $2$-category constructed by Lu-Zheng \cite{lu_zheng_2022} which encodes the Fourier-Mukai kernels between  objects $X$ and $Y$ in $\n{C}$ living over a base  $S$.

\begin{definition}[{\cite[Definition 7.1]{MannSix2}}]
Let $S\in \n{C}$,  the \textit{Lu-Zheng category} $\LZ_{\s{D},S}$ of $(\n{C},E)$ (relative to $\s{D}$ and $S$), is the $2$-category with objects given by arrows $X\to S$ in $E$, for each pair of objects $X,Y$ a $1$-category of functors $\Hom_{\LZ,S}(X,Y)=\s{D}(X\times_S Y)$.  The identity functor in $\s{D}(X\times_S X)$ is given by $\Delta_{!} 1_{X}$ where $\Delta: X\to X\times_S X$ is the diagonal map.  For a triple of objects $X,Y,Z$ the composite transformations
\[
 \Hom_{\LZ,S}(Y,Z) \times \Hom_{\LZ,S}(X,Y) \to  \Hom_{\LZ,S}(X,Z)
\]
are given by the Fourier-Mukai transform $M\star N = \pi_{1,3,!}(\pi_{1,2}^*N\otimes \pi_{2,3}^* M)$ for $N\in \s{D}(X\times_S Y)$, $M\in \s{D}(Y\times_S Z)$ and $\pi_{i,j}$ the corresponding projection of $X\times_{S} Y \times_{S} Z$.
\end{definition}

\begin{remark}
In \cite[Proposition 2.2.6]{zavyalov2023poincare}, Zavyalov shows that the Lu-Zheng category has a natural $(\infty,2)$-categorical enhancement. 
\end{remark}

With the help of the Lu-Zheng category one defines smooth and proper objects, cf. \cite[Definition 6.1]{SixFunctorsScholze}.

\begin{definition}
\label{DefinitionSmoothEtaleProper}
Let $(\n{C},E)$ be a six functor formalism, $S\in \n{C}$ and $f:X\to S$ an arrow in $E$. 
\begin{enumerate}
\item An object $M \in \s{D}(X)=\Hom_{\LZ,S}(X,S) $ is called $f$-\textit{smooth} if it is a left adjoint in $\LZ_{\s{D},S}$.

\item  An object $M\in \s{D}(X)=\Hom_{\LZ,S}(S,X)$ is called $f$-\textit{proper} if it is a left adjoint in  $\LZ_{\s{D},S}$.  

\end{enumerate}
\end{definition}

The following proposition provides different equivalent characterizations of $f$-smooth and $f$-proper objects. 

\begin{prop}
\label{PropSmoothProperObjects}
Let $f:X\to S$ be an arrow in $E$, $p_{i}:X\times_S X\to X$ the projection maps and $\Delta: X\to X\times_{S} X$ the diagonal map. Let $Q\in \s{D}(X)$.

\begin{enumerate}
\item  Let $\n{D}_f(Q)=\iHom_{X}(Q,f^! 1_{S})$. The following are equivalent

\begin{itemize}
\item[(a)]  $Q$ is $f$-smooth.

\item[(b)] The natural map $p_1^* D_{f}(Q)\otimes p_2^*Q\to \iHom_{X\times_{S} X}(p_1^* Q, p_2^! Q)$ is an equivalence.

\item[(c)] For all $g:S'\to S$ with pullback $f': X'\to S'$ and projection map $g:X'\to X$ the following natural functors are equivalences
\[
\begin{gathered}
\n{D}_{f'}(g^{'*}Q)\otimes f^{'*} \to \iHom_{X'}(g^{'*}Q, f^{'!}), \\ 
g^{'*}\iHom_{X}(Q, f^!) \to \iHom_{X'}(g^{'*}Q, f^{'!}g^*).
\end{gathered}
\]
\end{itemize}

If these conditions holds then $\n{D}_{f}(Q)$ is also $f$-smooth with right adjoint $Q$.

\item Let $\n{P}_{f}(Q)= p_{2,*}( \iHom_{X\times_{S} X}(p_{1}^* Q,\Delta_{!} 1_{X}))$. The following are equivalent

\begin{itemize}
\item[(a)] $Q$ is $f$-proper.
 
\item[(b)] The natural map $f_!(Q\otimes \n{P}_{f}(Q)) \to f_* \iHom_{X}(Q,Q)$ is an equivalence. 

\item[(c)]  For all $g:S'\to S$ with pullback $f': X'\to S'$ and projection map $g:X'\to X$ the following natural functors are equivalences
\[
\begin{gathered}
f'_!(-\otimes  \n{P}_{f'}(g^{'*}Q)) \to f'_* \iHom_{X'}(g^{'*}Q,-), \\
g^{*}  f_*\iHom_{X}(Q, -) \to f'_{*}\iHom_{X'}(g^{'*}Q, g^{'*}(-)).  
\end{gathered}
\]

\end{itemize}

If these conditions holds then $\n{P}_{f}(Q)$ is also $f$-proper with right adjoint $Q$. 

\end{enumerate}

\end{prop}
\begin{proof}
The point (1) is precisely \cite[Proposition 7.7]{MannSix2}. For point (2), the equivalence between (a) and (b) is \cite[Proposition 6.9]{SixFunctorsScholze}. The implication (c) to (b) follows by taking $g=\id_{S}$ and evaluating the first equivalence at $Q$. For (a), (b) implies (c), consider the natural transformation $\LZ_{\s{D},S}\to \LZ_{\s{D},S'}$ obtained by taking pullback along $g$, then $g^{*}$ preserves adjunctions  which implies that $g^{'*}Q$ is $f'$-proper with  dual $g^{'*} \n{P}_{f}(Q)$. By \cite[Proposition 6.8 (3)]{SixFunctorsScholze} there is a natural equivalence $g^{'*} \n{P}_f(Q)\cong \n{P}_{f'}(g^{'*} Q)$.  Now, the adjunction between  $g^{'*}Q$ and $\n{P}_{f'}(g^{'*}Q)$ shows  that the functor $f'^* \otimes g^{'*}Q: \s{D}(S')\to \s{D}(X') $ has by right adjoint the functor $f_!(-\otimes \n{P}_{f'}(g^{'*}Q))$, but the first has also by right adjoint 
the functor $f_*'\iHom_{X'}(g^{'*} Q,-)$, which provides the first equivalence of functors.   The second isomorphism follows from the first, proper base change, and the  natural identification $g^{'*} \n{P}_{f}(Q)\cong \n{P}_{f'}(g^{'*}Q)$. 
\end{proof}

\begin{prop}[{\cite[Proposition 7.11]{MannSix2}}]
\label{PropCompositionProperSmoothObjects}
Let $f:Y\to X$ and $g: Z\to Y$ be maps in $E$, let $P\in \s{D}(Y)$ and $Q\in \s{D}(Z)$. 
\begin{enumerate}

\item  If $P$ is $f$-smooth and $Q$ is $g$-smooth then $Q\otimes g^*P$ is $(f\circ g)$-smooth, and the natural map 
\[
g^* \n{D}_{f}(P)\otimes \n{D}_{g}(Q)\to \n{D}_{f\circ g} (Q\otimes g^* P)
\]
is an equivalence. 

\item  If $P$ is $f$-proper and $Q$ is $g$-proper then $Q\otimes g^*P$ is $(f\circ g)$-proper, and there is a natural equivalence
\[
\n{P}_{f\circ g} (Q\otimes g^* P) \cong g^* \n{P}_{f}(P)\otimes \n{P}_{g}(Q).
\]
\end{enumerate}
\end{prop}
\begin{proof}
In \textit{loc. cit.} it is shown part (1), the same argument using the dual six functors  $\s{D}^{\op}$ recovers  part (2), see Remark \ref{RemarkDualSixFunctor}.
\end{proof}

\begin{remark}
The equivalence in (2) of Proposition \ref{PropCompositionProperSmoothObjects} is not very explicit, it is obtained from a very involved adjunction in the Lu-Zheng category.  
\end{remark}

 \begin{prop}[Local on the target and stable by base change]
\label{PropLocalTargetBaseChange}
Consider a cartesian square
\[
\begin{tikzcd}
X' \ar[d, "f'"] \ar[r, "g'"] & X \ar[d,  "f"] \\ 
S'\ar[r, "g"] & S
\end{tikzcd}
\]
with $f\in E$, and let $Q\in \s{D}(X)$.  The following hold: 

\begin{enumerate}
\item If $Q$ is $f$-smooth (resp. $f$-proper) then $g^{'*}Q$ is $f'$-smooth (resp. $f'$-proper).

\item If $g$ satisfies universal $*$-descent  and $g^{'*} Q$ is $f'$-smooth (resp. $f'$-proper) then $Q$ is $f$-smooth (resp. $f$-proper).
\end{enumerate}
\end{prop}
\begin{proof}
For smooth objects this is  \cite[Corollary 7.8]{MannSix2}, the same proof applies in the abstract context. Note that the only property of a $v$-cover that is used in \textit{loc. cit.} is that it satisfies universal $*$-descent which holds in our case by hypothesis.  The case for proper objects follows by dual arguments in the sense of  Remark \ref{RemarkDualSixFunctor} as we describe next: point (1) follows from the fact that the natural transformation of $2$-categories $ g^*:\LZ_{\s{D},S}\to \LZ_{\s{D},S'}$ preserves adjunctions. For part (2), by Proposition \ref{PropSmoothProperObjects} (2.b) it suffices to show that the natural transformation
\[
f_!(-\otimes \n{P}_{f}(Q))\to f_* \iHom_{X}(Q,-)
\] 
is an equivalence. Let $S'_{\bullet}$ be the \v{C}ech nerve of $g$, and $X'_{\bullet}$ its pullback to $X$. By universal $*$-descent we have natural equivalences 
\[
\s{D}(S) \to \varprojlim_{[n]\in \Delta} \s{D}(S'_{n}) \mbox{ and } \s{D}(X) \to \varprojlim_{[n]\in \Delta} \s{D}(X'_{n}).
\]
 Consider the functor 
\[
\s{D}(X'_{\bullet})\to  \s{D}(S'_{\bullet}) \;\;\; M_{\bullet} \mapsto f'_{\bullet,*} \iHom_{X'_{\bullet}}(g^{'*}_{\bullet}Q,M_{\bullet}).
\]
By the second equivalence in Proposition \ref{PropSmoothProperObjects} (2.c), it preserves cocartesian sections, so it descends to a functor 
\[
\s{D}(X)\to \s{D}(S),
\]
by looking at left adjoints one shows that  this functor is actually equal to $f_* \iHom_X(Q,-)$.  It follows that the natural map 
\[
g^{'*}f_* \iHom_X(Q,-)\to f^{'}_* \iHom_{X'}(g^{'*}Q,-)
\]
is an equivalence of functors. On the other hand, since $g^{'*} Q$ is $f'$-proper, the simplicial object $\n{P}_{f'_{\bullet}}(g^{'*}_{\bullet} Q)$ is a cocartesian section of $\s{D}(X'_{\bullet})$ and it descents to an object $B\in \s{D}(X)$. Moreover, the isomorphism of functors
\[
f'_{\bullet, !}(-\otimes \n{P}_{f'_{\bullet}}(g^{'*}_{\bullet} Q))\to  f'_{\bullet,*}\iHom_{X'_{\bullet}}(g^{'*}_{\bullet}Q,-)
\]provided by Proposition \ref{PropSmoothProperObjects} (2.c) descents to an equivalence of functors 
\[
f_!(-\otimes B) \to f_*\iHom_{X}(Q,-).
\]
Note that the previous equivalence passes through any base change $S''\to S$. Taking the base change along $X\to S$ and evaluating at $\Delta_! 1_{X}$ one gets that $B= \n{P}_{f}(Q)$ and that the previous arrow is the natural one coming from the adjunction of \cite[Proposition 6.9]{SixFunctorsScholze}, this finishes the proof.  
\end{proof}

The following are two practical ways to construct $\s{D}$-covers in a six functor formalism, they correspond to smooth and proper descent respectively.

\begin{prop}[{\cite[Proposition 6.18]{SixFunctorsScholze}}]
\label{PropositionSmoothDescent}
Let $f:X\to Y$ be a morphism in $E$ such that $1_X$ is $f$-smooth. Then 
\[
f^*: \s{D}(Y)\to \s{D}(X)
\]
is conservative if and only if the natural map 
\[
\varinjlim_{[n]\in \Delta^{\op}} f_{n+1,!}f^{n+1,!}(1_Y) \to 1_Y
\]
is an isomorphism (where $f_{n+1}: X^{n+1/Y}\to Y$ is the $n+1$-th fold fiber product), and this condition passes to any base change.  In that case, the pullback functors 
\[
(f^*_{n+1})_n: \s{D}(Y) \to \varprojlim_{[n]\in \Delta} \s{D}(X^{n+1/Y}) \mbox{ and } (f^!_{n+1})_n: \s{D}(Y) \to \varprojlim_{[n]\in \Delta} \s{D}(X^{n+1/Y})
\]
are equivalences. In particular, if $f$ is a canonical cover, then it is a $\s{D}$-cover, and of universal $*$ and $!$-descent. 
\end{prop}

\begin{prop}[{\cite[Proposition 6.19]{SixFunctorsScholze}}]
\label{PropositionProperDescent}
  Let $f: X\to Y$ be an arrow in $E$ such that $1_X$ is $f$-proper, with $f_n: X^{n/Y}\to Y$ the $n$-th fold fiber product. Assume that  the map 
\[
1_Y \to ``{\varprojlim_{[n]\in \Delta}}" f_{n+1,*} 1_{X^{n+1/Y}}
\] 
is an isomorphism in $\mathrm{Pro}(\s{D}(Y))$; equivalently, $f_* 1_X \in \mathrm{CAlg}(\s{D}(Y))$ is descendable. 

Then the pullback functors 
\[
(f^*_{n+1})_n: \s{D}(Y) \to \varprojlim_{[n]\in \Delta} \s{D}(X^{n+1/Y}) \mbox{ and } (f^!_{n+1})_n: \s{D}(Y) \to \varprojlim_{[n]\in \Delta} \s{D}(X^{n+1/Y})
\]
are equivalences. In particular, if $f$ is a canonical cover, then it is a $\s{D}$-cover, and of universal $*$ and $!$-descent.
\end{prop}

\subsubsection{Cohomologically smooth and co-smooth maps}

Let $\s{D}$  be a six functor formalism on $(\n{C},E)$ taking values in  presentable stable $\infty$-categories, and  suppose that $\n{C}$ admits finite limits. One of the main advantages of the abstract six functor formalisms is that one can axiomatize cohomological properties of smooth and proper maps in algebraic geometry via the Lu-Zheng category. We follow \cite[\S 8]{MannSix2} and \cite[Lecture V]{SixFunctorsScholze} for the definition of cohomologically smooth maps. For the replacement of proper maps,  we will use a weaker notion suggested at the beginning of \cite[\S 9]{MannSix2} which we shall call \textit{co-smooth maps}, this is the same as  cohomologically smooth  for the dual six functor formalism $\s{D}^{\op}$.

\begin{definition}
\label{DefinitionCohoSmooth}
An arrow $f: X\to Y$ in $E$ is \textit{cohomologically smooth} if $1_{X}$ is $f$-smooth and $\n{D}_{f}(1_{X})=f^! 1_{Y}$ is invertible.  Similarly, $f$ is \textit{cohomologically co-smooth} if $1_{X}$ is $f$-proper and $\n{P}_{f}(1_{X})$ is invertible. 
\end{definition}

Translating the definition of cohomologically smooth and co-smooth maps  from the Lu-Zheng category to functors, a smooth map $f:X\to Y$ gives a natural equivalence of functors $ f^! 1_Y \otimes f^* \xrightarrow{\sim} f^!$ while a co-smooth map gives a natural equivalence of functors $f_!(- \otimes \n{P}_{f}(1_X)) \xrightarrow{\sim } f_*$. Among smooth and co-smooth maps, there are two special families consisting on \'etale and proper maps.

\begin{definition}[{\cite[Definitions 6.10 and 6.12]{SixFunctorsScholze}}]
\label{DefinitionCohoEtaleProper}
Let $f:Y\to X$ be an $n$-truncated map in $E$.

\begin{enumerate}

\item We say that $f$ is \textit{cohomologically proper} if $\Delta_{f}$ is cohomologically proper or an isomorphism, and if $1_Y$ is $f$-proper.  

\item We say that $f$ is \textit{cohomologically \'etale} if $\Delta_f$ is cohomologically \'etale or an isomorphism, and if $1_{Y}$ is $f$-smooth. 

\end{enumerate}
\end{definition}

In coherent cohomology, the maps $f:Y \to X$ that are proper in a suitable geometric sense are far from being $n$-truncated. However, one still would expect to have identifications $f_!= f_*$, and that the functor $f_*$ preserves ``coherent'' sheaves. Moreover, if the six functor formalism arises from a geometric decomposition $(I,P)$ one would expect that the arrows in $I$ and $P$ are \'etale and proper in a suitable sense respectively.  As a replacement of cohomologically \'etale and proper maps we define the following weaker notion:

\begin{definition}
\label{DefinitionWeaklyProper}
Consider and arrow $f: Y\to X$  in $E$ and  let $\Delta_{f}: Y \to Y\times_X Y$. We say that $f$ is  \textit{weakly cohomologically proper} (resp. \'etale) if the following hold: 
\begin{enumerate}

\item $1_{Y}$ is $\Delta_{f}$-proper (resp. $f$-\'etale) and there is a (non-canonical) equivalence $\n{P}_{\Delta_{f}}(1_Y) \simeq 1_{Y}$ (resp. an equivalence $\n{D}_{\Delta_{f}}(1_Y)\simeq 1_Y$). 
\item $1_{Y}$ is $f$-proper (resp. $f$-\'etale). 

\end{enumerate}
\end{definition}

\begin{remark}
Note that, under the hypothesis of the definition,  the equivalence $\n{P}_{\Delta_f}(1_{Y})\simeq  1_{Y}$, induces an equivalence $\n{P}_{f}(1_Y) \simeq  1_{Y}$. In particular, we have  (non-natural) isomorphisms of functors $\Delta_{f,!}\simeq \Delta_{f,*}$ and $f_! \simeq f_*$. Similarly for weakly cohomological \'etale maps. 

\end{remark}

The following lemma gives a way to construct weakly cohomologically \'etale and proper maps from  $I$ and $P$.

\begin{lemma}
\label{LemmaRepresentableInP}
Suppose that the six functor formalism on $(\n{C},E)$ arises from a suitable decomposition $(I,P)$. Let $f: \widetilde{X}\to \widetilde{S}$ be a map in $\widetilde{\n{C}}$  that is, locally in the $\s{D}$-topology on $\widetilde{S}$, representable by an arrow in $I$ (resp. $P$). Then $f$ is weakly cohomologically \'etale (resp. proper). 
\end{lemma}
\begin{proof}
 By Proposition \ref{PropLocalTargetBaseChange} (2) it suffices to prove that an arrow of $I$ (resp. $P$) satisfies the conclusion of the lemma, this follows from the proof of \cite[Propositions 6.11 and 6.13]{SixFunctorsScholze} since diagonal maps of arrows in $I$ are in $I$ (resp. $P$), and by construction $f^*=f^!$ for $f\in I$ (resp. $f_*=f_!$ for $f\in P$). 
\end{proof}

\begin{remark}
If the category $\n{C}$ is $n$-truncated then the maps in $I$ and $P$ are cohomologically \'etale and proper as in Definition \ref{DefinitionCohoEtaleProper}. For a general $\n{C}$, it is not clear to the author how to define a natural notion of cohomologically \'etale and proper maps that contains the $n$-truncated ones, and the arrows of $I$ and $P$ respectively.  
\end{remark}

\begin{lemma}[{\cite[Lemma 8.7]{MannSix2}}]
\label{LemmaStabilitycoSmoothCompositionPullback}
Cohomologically smooth and co-smooth maps are stable under pullbacks and compositions, and $*$-local on the target. 
\end{lemma}
\begin{proof}
Stability under pullbacks and  local in the target for universal $*$-descent maps  follows from Proposition \ref{PropLocalTargetBaseChange}. Stability under composition follows from Proposition \ref{PropCompositionProperSmoothObjects}. 
\end{proof}

  One of the interest of cohomologically smooth and co-smooth maps is that they will preserve smooth or proper objects in a suitable sense.

\begin{prop}[Local on the source and stable under composition]
\label{PropLocalSourceTarget}
Let $f: Y\to X$ and $g: Z\to Y$ be maps in $E$ and let $P\in \s{D}(Y)$. The following hold: 

\begin{enumerate}

\item Assume that $g$ is smooth.
\begin{itemize}

\item[(i)] If $P$ is $f$-smooth then $g^* P$ is $(f\circ g)$-smooth with dual $g^* \n{D}_{f}(P)\otimes g^! 1_{Y}$. 

\item[(ii)] If $g$ is a $\s{D}$-cover   and $g^* P$ is $(f\circ g)$-smooth then $P$ is $f$-smooth.

\end{itemize}

\item Assume that $g$ is co-smooth.

\begin{itemize}

\item[(i)] If $P$ is $f$-proper then $g^* P$ is $(f\circ g)$-proper with dual $g^* \n{P}_{f}(P)\otimes \n{P}_{g}(1_{Z})$.

\item[(ii)] If $g$ is a $\s{D}$-cover satisfying the hypothesis of Proposition \ref{PropositionProperDescent} and $g^* P$ is $(f\circ g)$-proper then $P$ is $f$-proper.
\end{itemize}

\end{enumerate}
\end{prop}
\begin{proof}
Part (1) is \cite[Proposition 8.6]{MannSix2}.  Part (2.i) follows from Proposition \ref{PropCompositionProperSmoothObjects}. For part (2.ii), we cannot directly dualize the argument of \textit{loc. cit} since the functors $f^*$ and $f_!$ in $\s{D}^{\op}$ might not have right adjoints, instead we will make use of the descendability property of $g$.

By Proposition \ref{PropSmoothProperObjects} (2.c), it suffices to prove that there is some $Q\in \s{D}(Z)$ and a natural equivalance of functors 
\[
f_!(-\otimes Q)\to f_*\iHom_{Y}(P,-),
\] 
that holds after any base change $X'\to X$; one then necessarily has that $Q= \n{P}_{f}(P)$ by taking the pullback along $Y\to X$ and evaluating at $\Delta_! 1_{Y}$.   Let $g_{\bullet}:Z_{\bullet}\to Z$ be the \v{C}ech nerve of $Z$ over $Y$. By Lemma \ref{LemmaStabilitycoSmoothCompositionPullback}, co-smooth covers are stable under composition and pullbacks, in particular, any projection $Z_{n} \to Z$ is co-smooth and by part (2.i) $g_{n}^* P \in \s{D}(Z_{n})$ is $(f\circ g_{n})$-proper over $X$. Furthermore, part (2.i) also gives rise a  cosimplicial object $( \n{P}_{f\circ g_{\bullet}} (g^*_{\bullet} P )\otimes \n{P}_{g_{\bullet}}(1_{Z_{\bullet}})^{-1})_{[n]\in \Delta}$ which is a co-cartesian section  in $\s{D}(Z_{\bullet})$, defining an object $Q\in \s{D}(Z)$.  Therefore, the functor 
\begin{equation}
\label{equationLocalSource1}
g_{\bullet}^* P \otimes (f\circ g_{\bullet} )^* : \s{D}(X)\to \s{D}(Z_{\bullet})
\end{equation}
descends to the functor $P\otimes f^*: \s{D}(X)\to \s{D}(Y)$. For each $[n]\in \Delta$, the functor  $g_{n}^*P \otimes (f\circ g_n)^*$ has by right adjoint $  (f\circ g_n)_*\iHom_{Z_n}( g_n^* P,- )$ which by Proposition \ref{PropSmoothProperObjects} (2.c) is naturally isomorphic to 
\[
(f\circ g_{n})_! ( - \otimes \n{D}_{(f\circ g_n)}(g_n^* P)) =  (f\circ g_{n})_!  (- \otimes g_n^* Q \otimes \n{P}_{g_n}(1_{Z_n})) = f_!(g_{n,*}(-) \otimes Q ).
\]
 Therefore, \eqref{equationLocalSource1} has  by right adjoint the totalization
\[
\varprojlim_{[n]\in \Delta} f_!( g_{n,*}(-) \otimes Q ). 
\]
Evaluating at a co-cartesian section $M_{n}$ of $\s{D}(Z_{\bullet})$, the proof of \cite[Proposition 6.19]{SixFunctorsScholze} shows that the Pro-system $(g_{n,*}M_n)_{[n]\in \Delta}$ is pro-constant, this implies that 
\[
\varprojlim_{[n]\in \Delta} f_! (g_{n,*} M_n \otimes Q) \cong  f_! (\varprojlim_{[n]\in \Delta}(g_{n,*}M_n)\otimes Q).
\]
Therefore,  $P\otimes f^*$ has by right adjoint $f_!(-\otimes Q)$, which provides the natural equivalence 
\begin{equation}
\label{eqLocalSource2}
f_!(-\otimes Q)  \xrightarrow{\sim } f_* \iHom_{Y}(P, -).
\end{equation}
Finally, note that the formation of \eqref{eqLocalSource2} is natural with respect to base change  $X'\to X$ by construction. This finishes the proof. 
\end{proof}

\begin{definition}
Let $f: Y\to X$ be an arrow in $E$. We say that $f$ is a  \textit{smooth (resp. descendable) $\s{D}$-cover} if it is smooth (resp. co-smooth), it is a canonical cover, and it satisfies the hypothesis of Proposition \ref{PropositionSmoothDescent} (resp. of Proposition \ref{PropositionProperDescent}). 
\end{definition}

We deduce the following corollary from Proposition \ref{PropLocalSourceTarget}.

\begin{corollary}
\label{CorollaryDescentSmoothProperCovers}
Being cohomomologically smooth  is smooth $\s{D}$-local on the source. Analogously, being cohomologically co-smooth is   descendably $\s{D}$-local on the source. 
\end{corollary}

Smooth and descendable $\s{D}$-covers provides a description of the coefficients of the quotient as modules and comodules respectively.

\begin{prop}
\label{PropCoverModuleComoduleDescription}
Let $f:Y\to X$ be an arrow in $E$. 

\begin{enumerate}

\item Suppose that $f$ is a smooth $\s{D}$-cover. Then there is a natural equivalence of stable $\infty$-categories
\[
\s{D}(X)\xrightarrow{\sim} \Mod_{f^!f_! }(\s{D}(Y)),
\]
where the monad $f^!f_!$ naturally belongs to $\Alg((\End^{L}_{\s{D}(X)}( \s{D}(Y)))$.

\item Suppose that $f$ is a descendable $\s{D}$-cover. Then there is a natural equivalence of $\infty$-categories
\[
\s{D}(X)\xrightarrow{\sim} \mathrm{CoMod}_{f^* f_* }(\s{D}(Y)),
\]
where the comonad $f^*f_*$ naturally belongs to $\mathrm{CoAlg}((\End^{L}_{\s{D}(X)}( \s{D}(Y)))$. 
\end{enumerate}

Moreover, if $f$ has a retraction $g:X\to Y$ then the monad $f^!f_!$ in (1) arises from an object $\n{D}_{f}\in \Alg(\End_{\s{D}(Y)}(\s{D}(Y)))=\Alg(\s{D}(Y))$, and the comonad $f^*f_*$ arises from an object $\n{C}_f\in \mathrm{CoAlg}(\End^L_{\s{D}(Y)}(\s{D}(Y)))=\mathrm{CoAlg}(\s{D}(Y))$.  

\end{prop}
\begin{proof}
In the case of (1), both functors $f^!$ and $f_!$ are linear over $\s{D}(X)$, namely $f^!=f^*\otimes f^* 1_X$ and $f^*$ is $\s{D}(X)$-linear being symmetric monoidal, and $f_!$ is $\s{D}(X)$-linear by the projection formula. This implies that modules of the monad $f^! f_!$ arises from an object $f^!f_!\in \Alg(\End^{L}_{\s{D}(X)}(\s{D}(Y)))$.  

 Similarly, in the case of (2), we have a natural  equivalence $f_* = f_!(-\otimes \n{P}_f(1_Y))$, proving that $f_*$ satisfies the projection formula so that is $\s{D}(X)$-linear. Therefore, the comonad $f^*f_*$ arises from an object  $f^*f_* \in \mathrm{CoAlg}(\End^{L}_{\s{D}(X)}( \s{D}(Y)))$.
 
In order to prove the proposition we only need to show that the functor $f^!$ is monadic in the case (1), and that $f^*$ is comonadic in the case (2), see \cite[Theorem 4.7.3.5]{HigherAlgebra}. It is clear that both functors are conservative in both situations. In the case of (1), the functor $f^!$ already preserves colimits being isomorphic to $f^*\otimes f^! 1_X$, and the monadicity theorem can be applied. In the case of (2),  in order to apply the comonadicity theorem we need to show that $f^*$ preserves $f^*$-split totalizations.
 
 By Proposition \ref{PropositionProperDescent} we have that 
 \[
 \s{D}(X)=\varprojlim_{[n]\in \Delta} \s{D}(Y^{n+1/X})
 \]
 along pullback maps. Let $(M_m)_{[m]\in \Delta}$ be a cosimplicial object in $\s{D}(X)$ whose pullback to $\s{D}(Y)$ is split. Then, for all $n\geq 0$, the pullback of $(M_m)_{[m]\in \Delta}$ to $\s{D}(Y^{n+1/X})$ is split with limit $N_{n+1}$. Moreover, the object $(N_{n+1})_{[n]\in \Delta}$ is a cocartesian section in $\s{D}(Y^{\bullet+1/X})$ because of the splitting, and it  defines an object $N$ in the limit $\s{D}(X)=\varprojlim_{[n]\in \Delta} \s{D}(Y^{n+1/X})$. We deduce that 
 \[
 \begin{aligned}
 \varprojlim_{[m]\in \Delta} M_m & = \varprojlim_{[m]\in \Delta}  \varprojlim_{[n]\in \Delta} f_{n+1,*} f_{n+1}^* M_{m}  \\
 				& =  \varprojlim_{[n]\in \Delta} f_{n+1,*} \varprojlim_{[m]\in \Delta}   f_{n+1}^* M_{m} \\
 				& =  \varprojlim_{[n]\in \Delta} f_{n+1,*} N_{n+1} \\
 				& = N.
 \end{aligned}
 \]
 We deduce that 
 \[
 f^*(\varprojlim_{[m]\in \Delta} M_m )= f^* N = N_{1}= \varprojlim_{[m]\in \Delta}   f^* M_{m},
 \]
 proving that $f^*$ preserves $f^*$-split totalizations.

 Finally, if $f:Y\to X$ has a retraction $g:Y\to X$, the functors $f^!$ and $f_!$ are $\s{D}(Y)$-linear in the case of (1), and the functors $f^*$ and $f_*$ are $\s{D}(Y)$-linear in the case of (2). This shows that the monad and comonad $f^!f_!$ and $f^*f_*$ in (1) and (2) respectively, arise from objects $\n{D}_{f}\in \Alg(\End_{\s{D}(Y)}^L(\s{D}(Y)))=\Alg(\s{D}(Y))$ and  $\n{C}_f\in \mathrm{CoAlg}(\End^L_{\s{D}(Y)}(\s{D}(Y)))=\mathrm{CoAlg}(\s{D}(Y))$ respectively (see \cite[Lemma A.4.7]{MannSix} for the natural equivalence $\End_{\s{D}(Y)}^L(\s{D}(Y))=\s{D}(Y)$).
\end{proof}

\begin{remark}
\label{RemarkComoduleDescription}
In the part (2) of Proposition \ref{PropCoverModuleComoduleDescription} the only important conditions for the statement to hold is that $f_*$ satisfies the projection formula, and that $f$ satisfies universal $*$-descent. 
\end{remark}

We end this section by recalling how smooth and proper objects are preserved by lower shrieck functors under suitable hypothesis.  

\begin{prop}
\label{PropPreserveAdjointsComposition}
Let $f: Y\to X$ and $g:Z\to Y$ be maps in $E$. Let $P,Q\in \s{D}(Z)$. The following hold
\begin{enumerate}
\item If $P$ is $(f\circ g)$-smooth and $Q$ is $g$-proper then $g_*\iHom_{Z}(Q,P)=g_!(\n{P}_{g}(Q)\otimes P)$ is $f$-smooth.

\item If $P$ is $(f\circ g)$-proper and $Q$ is $g$-smooth then $g_!(Q\otimes P)$ is $f$-proper. 

\end{enumerate}
\end{prop}
\begin{proof}
Part (1) is \cite[Proposition 7.13]{MannSix2}, part (2) is proven with the same argument that we recall down below: 

One has a morphisms of $2$-categories $\iota_!:\LZ_{\s{D},Y}\to \LZ_{\s{D},X}$ mapping $[W\to Y]$ to $[W\to X]$ and $M\in \Hom_{\LZ,Y}(W,V)=\s{D}(W\times_Y V)$ to  $\iota_!M\in \Hom_{\LZ,X}(W,V)=\s{D}(W\times_X V)$ where $\iota: W\times_Y V \to W\times_{X} V$. By hypothesis $P\in \Hom_{\LZ,X}(X,Z)$ is a left adjoint and $Q\in \Hom_{\LZ,Y}(Z,Y)$ is a left adjoint. Then, since $\iota_!$ sends left adjoints to left adjoints, one has that 
\[
\iota_! Q \star P= \pi_{Y,!}( \pi_{Z}^* P\otimes \iota_! Q)=\pi_{Y,!}\iota_!( \iota^* \pi_Z^* P \otimes Q )= g_!(Q\otimes P)
\]
is a left adjoint in $\Hom_{\LZ,X}(X,Y)$, proving what we wanted. 
\end{proof}

\begin{prop}[{\cite[Proposition 9.10]{MannSix} }]
Let $f: Y\to X$ and $g: Z\to Y$ be maps in $E$, and $P\in \s{D}(Y)$. The following hold: 
\begin{enumerate}

\item If $g$ is cohomologically co-smooth and $P$ is $(f\circ g)$-smooth then $g_*P$ and $g_! P$ are $f$-smooth.

\item If $g$ is cohomologically smooth and $P$ is $(f\circ g)$-proper then $g_! P$ is $f$-proper. 

\end{enumerate}
\end{prop}
\begin{proof}
This follows from Proposition \ref{PropPreserveAdjointsComposition} by taking $Q= 1_{Z}$ or $Q= \n{P}_f(1_Z)$ for point (1), and taking $Q=1_{Z}$ for point (2). 
\end{proof}

We end this section with a couple of lemmas that will be useful later. 

\begin{lemma}
\label{LemmaSmoothvsHom}
Let $f:Y\to X$, then there is  a natural equivalence 
\[
f^! \iHom_{X}(\s{F},\s{G})\cong \iHom_{Y}(f^*\s{F}, f^!\s{G}).
\]
In particular, if $f$ is cohomological smooth we have that 
\[
f^*\iHom_{X}(\s{F}, \s{G})\cong \iHom_{Y}(f^* \s{F}, f^* \s{G}).
\]
\end{lemma}
\begin{proof}
This follows from the adjunctions:
\[
\begin{aligned}
\Hom_{Y}(\s{H}, f^! \iHom_{X}(\s{F}, \s{G})) & \cong \Hom_{X}(f_! \s{H}, \iHom_{X}(\s{F}, \s{G})) \\ 
& \cong \Hom_{X}(f_!\s{H} \otimes \s{F}, \s{G}) \\
& \cong \Hom_{X}(f_! (\s{H} \otimes f^* \s{F}), \s{G}) \\ 
&\cong= \Hom_{Y}(\s{H} \otimes f^* \s{F}, f^! \s{G}) \\
& \cong \Hom_{Y}(\s{H} ,\iHom_{Y}(f^* \s{F}, f^! \s{G})). 
\end{aligned}
\]
The claim about cohomologically smooth maps follows from the fact that $f^! = f^* \otimes f^! 1_{X}$ and that $f^! 1_{X}$ is an invertible object in $\s{D}(Y)$. 
\end{proof}

\begin{lemma}
Let $\s{F} \in \s{D}(X)$ and let $f: Y\to X$ be an arrow in $E$. 

\begin{enumerate}
 \item If $f$ is a smooth cover and $f^* \s{F}$ is $f$-smooth, then $\s{F}$ is dualizable. 
 
 \item If $f$ is a descendable cover and $f^* \s{F}$ is $f$-proper, then $\s{F}$ is dualizable. 
\end{enumerate} 
\end{lemma}
\begin{proof}
This follows from Proposition \ref{PropLocalSourceTarget} as being $\id_X$-smooth or proper is equivalent to being a dualizable object in $\s{D}(X)$. 
\end{proof}

\subsection{Solid and Tate stacks}
\label{SubsectionSolidStacks}

In this section we explain how the  theory of  abstract six functor formalisms of   \cite{MannSix, MannSix2}  and \cite{SixFunctorsScholze}  provides a very general six functor formalism of solid quasi-coherent sheaves for stacks.  Throughout this section we fix an uncountable cutoff cardinal $\kappa$  as in \cite[Definition 2.9.11]{MannSix}, in real world applications the construction of  the six functors down below will be independent of $\kappa$ large enough. 

 By \cite[Lemma 2.9.12]{MannSix}, if $(A,A^+)$ is a discrete animated Huber ring, then the forgetful functor $\Mod((A,A^+)_{\sol})\to \Mod(\underline{A})$ preserves $\kappa$-small objects. Therefore, if $B\in \Mod((A,A^+)_{\sol})$ is a $\kappa$-small algebra, the forgetful functor $\Mod((B,A^+)_{\sol})\to \Mod(\underline{B})$ also preserves $\kappa$-small objects. From now on we will work with $\kappa$-small condensed sets.  Recall that a solid affinoid ring $\n{A}$ is an analytic $\bb{Z}_{\sol}$-algebra such that the natural map $(\underline{\n{A}}, \n{A}^+)_{\sol}\to \n{A}$ is an equivalence.  We let $\AffRing_{\bb{Z}_{\sol},\kappa}$ denote the $\infty$-category of solid affinoid rings $\n{A}$ with $\underline{\n{A}}$ being a $\kappa$-small condensed set, we let $\Aff_{\bb{Z}_{\sol},\kappa}$ denote its opposite category of $\kappa$-small solid affinoid spaces, we also let $\SpecAn \n{A} \in \Aff_{\bb{Z}_\sol,\kappa}$ denote the analytic spectrum of the solid affinoid ring $\n{A}$.

We recall  some basic properties of the categories of $\kappa$-small complete modules of analytic rings.

\begin{prop}[{\cite[Proposition 2.3.9]{MannSix}}]
Let $\AnRing_{\kappa}$ be the full subcategory of analytic rings $\n{A}$ with $\kappa$-small underlying condensed rings, and let $\Mod(\n{A})_{\kappa}$ be the full subcategory of $\Mod(\n{A})$ generated under sifted colimits  by the objects  $\n{A}[S]$  with $S$ a $\kappa$-small extremally disconnected set.   The functor  $\Mod(-):\AnRing \to \CAlg(\Cat^{\colim,\ex}_{\infty})$ of complete modules restricts to a functor 
\[
\Mod(-)_{\kappa}: \AnRing_{\kappa} \to \CAlg(\mathcal{P}r^{L,\ex}),
\]
of $\kappa$-small complete modules. In other words, for $\n{A}$ a $\kappa$-small analytic ring, $\Mod(\n{A})_{\kappa}$ is a  presentably symmetric monoidal  stable $\infty$-category, and for a map $\n{A}\to \n{B}$ of $\kappa$-small analytic rings,  the  base change $\n{B}\otimes_{\n{A}}-: \Mod(\n{A}) \to \Mod(\n{B})$ preserves $\kappa$-small analytic modules. 
\end{prop}

We now want to define a six functor formalism for solid affinoid rings, and then  apply \cite[Theorem 4.20]{SixFunctorsScholze} to construct a very large six functor formalism for suitable stacks over $\Aff_{\bb{Z}_{\sol},\kappa}$.  For this, by   \cite[Proposition A.5.10]{MannSix}, all we need is a minimal amount of data consisting on \textit{\'etale} and \textit{proper} arrows $(I,P)$ in $\Aff_{\bb{Z}_{\sol},\kappa}$ satisfying some minimal properties, cf. Definition A.5.9 of \textit{loc. cit}. The following definition is due to  Clausen and Scholze.

\begin{definition}
\label{DefinitionSuitableDecompositionSolid}
We denote $\n{C}= \Aff_{\bb{Z}_{\sol},\kappa}$. 

\begin{enumerate}
\item Let $I$ be the family of arrows in $\n{C}$ consisting on morphisms $f:\SpecAn \n{B}\to \SpecAn \n{A}$ such that $f^*: \Mod(\n{A})\to \Mod(\n{B})$ is an open localization in the sense of Definition \ref{DefinitionOpenClosedMonoidal}, and such that the  associated idempotent algebra $D$ lies in $\Mod(\n{A})_{\kappa}$.

\item  Let $P$ be the family of arrows in $\n{C}$ consisting on morphisms $f:\SpecAn \n{B} \to \SpecAn \n{A}$ such that  $\n{B}=\n{B}_{\n{A}/}$ is induced from $\n{A}$. 

\item We let $E$ be the family of arrows in $\n{C}$  of the form  $f\circ i$ with $i\in I$ and $f\in P$. 

\end{enumerate}

\end{definition}

We first need to check that $(\n{C},E)$ is a geometric set up. 

\begin{lemma}
 The class of arrows $I$, $P$ and  $E$ are stable under composition and pullbacks in $\n{C}$. 
\end{lemma}
\begin{proof}

The stability under pullbacks and compositions for the class $P$ is obvious, for the class $I$ follows from Theorem \ref{TheoLocaleTopology}. Stability under pullbacks of the class $E$ follows from the stability for $I$ and $P$,   we are left to prove stability under composition for $E$. Consider two maps of analytic rings $f:\n{A} \to \n{B}$ and $g:\n{B} \to \n{D}$. Suppose that we have factorizations $f= j_1 \circ p_1$ with $p_1:  \n{A} \to \n{A}' $ an induced analytic ring and  $j_1: \n{A}' \to \n{B}$ an open immersion. Similarly, suppose that $g=  j_2\circ p_2$ with $p_2: \n{B} \to \n{B}'$ and $j_2: \n{B}' \to \n{D}$.  Then we can write $g\circ f$ as the composite 
\[
\n{A} \to \n{B}'_{\n{A}/} \to \n{B}'_{\n{B}/} \to \n{D},
\]
the first arrow is in $P$ by definition, the second and third arrows are open immersions, so it is their composite. This proves that $E$ is stable under composition as wanted. Moreover, the idempotent algebra associated to the map $ \n{B}'_{\n{A}/} \to \n{D}$ is a "union" in the sense of Proposition \ref{PropLocaleOperations} (5) of two $\kappa$-small $\n{B}$-algebras, so it is $\kappa$-small. 
\end{proof}

\begin{remark}
\label{RemarkInternalHomAnRing}
Given $\SpecAn \n{A}\in \Aff_{\bb{Z}_{\sol},\kappa}$, the stable $\infty$-category $\Mod(\n{A})_{\kappa}$ is closed by the adjoint functor theorem.  Indeed, the inclusion $\Mod(\n{A})_{\kappa} \to \Mod(\n{A})$ has a right adjoint $(-)_{\kappa}$ given by taking the underlying $\kappa$-small set, and the internal $\iHom$ of $\Mod(\n{A})_{\kappa}$ is equal to $\iHom_{\n{A}}(-,-)_{\kappa}$.   This implies that  both internal $\iHom$'s could differ for general objects $N,M\in \Mod(\n{A})_{\kappa}$. However, after taking some big enough   cardinal $\kappa'>\kappa$ one has that 
\[
\iHom_{\n{A}}(N,M)_{\kappa'}= \iHom_{\n{A}}(N,M).
\]
Actually, the proof of \cite[Proposition 2.1.11 (2)]{MannSix} shows that the choice of the cardinal $\kappa'$ only depends on $N$: write $N=\varinjlim_{I} \n{A}[S_i]$ as a small colimit of compact projective generators with $S_i$ a $\kappa$-small extremally disconnected set. Then, 
\[
\iHom_{\n{A}}(N,M)= \varprojlim_{I} \iHom_{\n{A}}(\n{A}[S], M).
\]
Thus, after taking $\kappa'$ big enough such that $|I|<\kappa'$ and $\Mod(\n{A})_{\kappa'} \to \Mod(\n{A})$ commutes with $|I|$-small limits, one is reduced to prove that $\iHom_{\n{A}}(\n{A}[S],M)$ is $\kappa'$-small for all $\kappa$-small extremally disconnected set $S$  and $M\in \Mod(\n{A})_{\kappa}$. Writing $M$ as a colimit of compact projective generators one just needs to take $\kappa'$ such  that $\iHom_{\n{A}}(\n{A}[S], \n{A}[S'])$ is $\kappa'$-small for all $\kappa$-small extremally disconnected sets  $S$ and $S'$. 
\end{remark}

\begin{lemma}
\label{LemmaSixFunctorBabycase}
Keep the notation of Definition \ref{DefinitionSuitableDecompositionSolid}.  The pair $(I,P)$ is a suitable decomposition  of the geometric set up $(\n{C},E)$. Moreover, it satisfies the criteria of \cite[Proposition A.5.10]{MannSix}, so that  the functor $\Mod(-)_{\kappa}: \n{C} \to \CAlg(\n{P}r^{L,\ex})$ enhances to  a six functor formalism
\[
\s{D}= \Mod(-)_{\kappa}: \Corr(\n{C}, E) \to \n{P}r^{L,\ex}.
\]
Furthermore, for any arrow $f:\AnSpec \n{B} \to \AnSpec\n{A}$ the functors $f^*$ and $f_*$ are independent of $\kappa$. For any $f\in E$, the functor  $f_{!}$ is  independent of $\kappa$  and   there is some $\kappa'>\kappa$ such that for all $\kappa''\geq \kappa'>\kappa$ the restriction of the functor $f^!$  from $\kappa''$-small modules to $\kappa$-small modules stabilizes.  
\end{lemma}
\begin{proof}
First, we check that the conditions of a \textit{suitable decomposition} hold, cf. \cite[Definition A.5.9]{MannSix}. By definition, the objects in $E$ are compositions $p\circ j$ with  $j\in I$ and $p\in P$, so property (a) in \textit{loc. cit.} holds. Next, if $f: \n{A} \to \n{B}$ is an object in $I$, then $f^*: \Mod(\n{A})\to \Mod(\n{B})$ is an open immersion in the sense of Definition \ref{DefinitionOpenClosedMonoidal}. Then, $\n{A} \to \n{B}$ is a localization of analytic rings in the sense that $\n{B}\otimes_{\n{A}} \n{B}=\n{B}$, and a morphism $\n{A}\to \n{D}$ extends to $\n{B}\to \n{D}$ if and only if the natural map $\n{D}\to \n{B}\otimes_{\n{A}} \n{D}$ is an equivalence, this shows that $f: \SpecAn \n{B} \to \SpecAn \n{A}$ is $-1$-truncated, which implies condition (b). Finally, it is easy to check that $I$ and $P$ contain the identity maps and that satisfy the two-out-of-three property (use  Theorem \ref{TheoLocaleTopology} for $I$), this gives (c) and (d) in \cite[Definition A.5.9]{MannSix}.

 For the existence of a $3$-functor formalism we need to check the following conditions: 
\begin{itemize}

\item[(i)] for $[j: \AnSpec \n{B} \to  \AnSpec \n{A} ]\in I$ the following hold
\begin{itemize}

\item $j^*$ admits a left adjoint $j_!$

\item $j_!$ satisfies the proper base change.

\item  $j_!$ satisfies  the projection formula. 

\end{itemize}

\item[(ii)] for $f: \AnSpec \n{B} \to \AnSpec \n{A}$ the following hold
\begin{itemize}

\item $f^*$ admits a colimit preserving right adjoint $f_*$. 

\item $f_*$ satisfies proper base change

\item $f_*$ satisfies the projection formula. 

\end{itemize}

\item[(iii)] For every cartesian diagram 
\[
\begin{tikzcd}
U' \ar[r, "j'"] \ar[d,"f'"] & X' \ar[d, "f"] \\ 
U \ar[r, "j"] & X
\end{tikzcd}
\]
in $\n{C}$ such that $j\in I$ and $f\in P$, the natural map $  j_!f'_* \xrightarrow{\sim}  f_*j'_!$ is an isomorphism.

\end{itemize}

We will prove the properties  for the derived categories $\Mod(\n{A})$ and then show that they preserve $\kappa$-small objects.   Part (i) is a consequence of Proposition \ref{PropClosedOpenLocalizationsInftyCat}. For part (ii),  $f_*$ is the forgetful functor, then the projection formula is clear as $f$ is defined by a morphism of analytic rings $\n{A}\to \n{B}$ where $\n{B}$  has the induced structure of $\n{A}$. It is left to check   that $f_*$ satisfies proper base change, but if $\n{A}\to \n{D}$ is another morphism of analytic rings, then $\n{B}\otimes^{L}_{\n{A}} \n{D} = (\underline{\n{B}}\otimes_{\n{A}} \n{D})_{\n{D}/}$ and the proper base change formula is clear.  Finally we prove (iii), let $X$, $X'$, $U$ and $U'$ be the analytic spectrum of $\n{A}$, $\n{A}'$, $\n{B}$ and $\n{B}'$. Then we have that $\n{A}'=\n{A}'_{\n{A}/}$, and that $\n{B}'=\n{A}'\otimes_{\n{A}} \n{B}$. Let $D\in \Mod(\n{A})$ be the idempotent algebra that complements $U$ in $X$, then $D'= \n{A}' \otimes_{\n{A}} D$ is the idempotent algebra  that complements $U'$ in $X'$. Let $M\in \Mod(\n{B}')$, by definition we have that 
\[
j_!f'_* M = [\n{A}\to D]\otimes_{\n{A}} M 
\]  
and 
\[
f_* j'_! M =  [\n{A'}\to D']\otimes_{\n{A'}} M,
\]
but $[\n{A}'\to D']= [\n{A}\to D]\otimes_{\n{A}} \underline{\n{A'}}$ so that 
\[
\begin{aligned}
 [\n{A'}\to D']\otimes_{\n{A'}} M  & = [\n{A}\to D]\otimes_{\n{A}} \underline{\n{A}'} \otimes_{\underline{\n{A}'}_{\n{A}/}} M \\ 
 & = [\n{A}\to D] \otimes_{\n{A}} M
 \end{aligned}
\]
proving that the natural map $j_!f'_*\to f_*j'_!$ is an equivalence.  

Finally, we need to show that the functors $f_*$ and $f_!$ are independent of $\kappa$, and that $f^!$ is stabilized for $\kappa'>>\kappa$ large enough. The claim about $f_*$ follows from \cite[Lemma 2.9.12]{MannSix} and the discussion at the beginning of the section. For the functors $f_!$ for $f\in E$, it suffices to  prove it for $f\in I$ or $f\in P$. If $f\in P$ then $f_!=f_*$ and we are done, if $f\in I$  and $f: \n{A}\to \n{B}$, then $f_! = [\n{A}\to D]\otimes_{\n{A}}-$ for $D$ the idempotent algebra that complements $\n{B}$. By hypothesis $D$ is a $\kappa$-small algebra, which implies that the tensor product $[\n{A}\to D]\otimes_{\n{A}}-$ preserves $\kappa$-small objects as wanted. For the stability of $f^!$, we can assume that $f\in I$ or $f\in P$, in the first case $f^!=f^*$ and we are done, in the second case $f$ corresponds to a map $\n{A}\to \n{B}_{\n{A}/}$ and $f^! = \iHom_{\n{A}}(\n{B},-)$. Then the stability of $f^!$ for large enough $\kappa'$ follows by   Remark \ref{RemarkInternalHomAnRing}. 
\end{proof}

With the minimalistic $6$-functor formalism for solid affinoid spaces $\n{C}$ we can create a very large class of stacks $\widetilde{\n{C}}$, and a large class of arrows  $\widetilde{E}$ as in Theorem \ref{TheoSixFunctorsScholze}. One has the following corollary.

\begin{cor}
\label{CorollaryStacksSixFunctors}
Let $\n{C}=\Aff_{\bb{Z}_{\sol},\kappa}$ be the  category of $\kappa$-small solid affinoid spaces. Let $E$ be as in Definition \ref{DefinitionSuitableDecompositionSolid}.  Let $\widetilde{\n{C}}= \Sh_{\s{D}}(\Aff_{\bb{Z}_{\sol},\kappa})$ be the $\infty$-category of sheaves on anima for the $\s{D}$-topology where $\s{D}= \Mod_{(-),\kappa}$. Then there is a minimal class of arrows $\widetilde{E}$ in $\widetilde{\n{C}}$  containing the arrows represented in $E$ such that  the six functor formalism $(\n{C},E)$ obtained from Lemma \ref{LemmaSixFunctorBabycase} extends uniquely to $(\widetilde{\n{C}}, \widetilde{E})$, and such that $\widetilde{E}$ is stable under disjoint unions, local on the target, local on the source, and tame. 
\end{cor}

\begin{definition}
\label{DefinitionSolidStacks}
With the notation of Corollary \ref{CorollaryStacksSixFunctors}, we call $\Sh_{\s{D}}(\Aff_{\bb{Z}_{\sol},\kappa})$ the $\infty$-category of \textit{$\kappa$-small solid $\s{D}$-stacks}. If $\kappa$ is omitted in the notation we write instead $\Sh_{\s{D}}(\Aff_{\bb{Z}_{\sol}})$ and call it the $\infty$-category of \textit{solid $\s{D}$-stacks}. 
\end{definition}

\begin{remark}
The six functor formalism for solid quasi-coherent sheaves constructed before depends on the cardinal $\kappa$, in particular the functors $f_*$, $f_!$ and $f^!$ might depend on $\kappa$. Nevertheless, in practice we will always have formulas for these functors that will make them independent in large enough cardinals. 
\end{remark}

Next, we prove that the locale topology of Theorem \ref{TheoLocaleTopology} gives rise to cohomologically proper and \'etale $\s{D}$-covers. 

\begin{lemma}
\label{LemmaOpenClosedCoverLocaleCoho}
The following hold:

\begin{enumerate}

\item Let $f: \SpecAn \n{B} \to \SpecAn \n{A}$. If $f$ is open in the associated locale then $f$ is cohomologically \'etale. Similarly, if $\n{B}_{\n{A}/}= \n{B}$ and $\n{B}$ is an idempotent algebra over $\n{A}$, then $f$ is cohomologically proper.  

\item Let $\SpecAn \n{A}\in \Aff_{\bb{Z}_{\sol}}$ and let $\{ f_i:\SpecAn \n{B}_i \to \SpecAn \n{A}\}_{i=1}^n$ be a collection of morphisms of solid affinoid rings. If $\{f_i\}_{i\in I}$ is an open  cover of locales then $\bigsqcup f_i $ is a smooth $\s{D}$-cover. Similarly, if $\{f_{i}\}_{i\in I}$ is a closed cover of locales then $\bigsqcup_{i} f_i$ is a descendable $\s{D}$-cover. 

\end{enumerate} 
\end{lemma}
\begin{proof}
\begin{enumerate}

\item Suppose that $f$ is open in the associated locale, by definition $f\in I$ and $f^*= f^!$. Moreover, $f$ is $-1$-truncated as $\n{B}\otimes_{\n{A}} \n{B}= \n{B}$, this shows that $f$ is cohomologically \'etale. Similarly, if $\n{B}=\n{B}_{\n{A}/}$ is an idempotent algebra over $\n{A}$, then $f$ is $-1$-truncated and we have $f_*=f_!$, this implies that $f$ is cohomologically proper.

\item Let $F= \bigcup_i f_i : \bigcup_{i} \SpecAn \n{B}_i \to \SpecAn \n{A}$ be a finite cover.  In the case that the $\{f_i\}_{i\in I}$ form an open cover of the locale, by  Proposition \ref{PropLocaleTopIsSubcanonial} adapted to $\Aff_{\bb{Z}_{\sol}}$, the family $\{f_i\}$ form a canonical cover, and the pullback along $F$ is conservative. Then the conditions of Proposition \ref{PropositionSmoothDescent} hold and $F$ is a smooth $\s{D}$-cover.  Similarly, if $\{f_i\}_{i\in I}$ is a closed cover on the locales, then it is refined by an open cover and therefore it defines a subcanonical cover. By Proposition \ref{PropositionProperDescent}, we are left to prove that $F_* 1 = \prod_i \n{B}_i $  is a descendable $\n{A}$-algebra. But by definition of closed covering in the locales, we know that $\underline{\n{A}}$ is equal to the ``union''  of the algebras $\n{B}_i$, which clearly belongs to the thick tensor ideal generated by $F_* 1$ in $\Mod(\n{A})$. This proves the lemma. 
\end{enumerate}
\end{proof}

We finish this section with the definition of  Tate stacks.

\begin{definition}
\label{DefinitionAdicStacks}
Let $R_{\sol}=(R,R^+)_{\sol}=(\bb{Z}((\pi)), \bb{Z}[[\pi]])_{\sol}$, and let $\Aff^{b}_{R_{\sol},\kappa}$ be the $\infty$-category of $\kappa$-small bounded affinoid spaces. The $\infty$-category of \textit{$\kappa$-small Tate stacks} over $R_{\sol}$ is  the category $\Sh_{\s{D}}(\Aff^{b}_{R_{\sol}, \kappa})$ of sheaves on anima of  $\Aff^{b}_{R_{\sol},\kappa}$ for the $\s{D}$-topology, with $\s{D}=\Mod(-)_{\kappa}$. If $\kappa$ is omitted in the notation we simply call $\Sh_{\s{D}}(\Aff^{b}_{R_{\sol}})$ the category of  \textit{Tate stacks}. 
\end{definition}

The following lemma gives a sufficient criteria for the existence of $!$-functors for a morphism of solid stacks.

\begin{lemma}
Let $f:X\to Y$ be a map in $\Sh_{\s{D}}(\Aff_{\bb{Z}_{\sol}})$ such that there is an epimorphism $\bigsqcup_I \AnSpec \n{A}_i\to Y$ with $\n{A}_i$ solid affinoid rings, such that for all pullback $X_{i} \to \AnSpec  \n{A}_i$ there is a $\s{D}$-cover $\bigsqcup_{j} \AnSpec \n{B}_{i,j}\to X_i$, such that the maps $\n{A}_i\to \n{B}_{i,j}$ factor through maps
\[
\n{A}_i \to \n{A}[T_1,\ldots, T_{d}]_{\sol} \to \n{B}_{i,j},
\]
where $\n{A}[ T_1,\ldots, T_{d}]_{\sol}\to \n{B}_{i,j}$ has the induced analytic structure. Then $f\in \widetilde{E}$ has $!$-functors for the theory of solid quasi-coherent sheaves.  The same holds for $\s{D}$-stacks over $\Aff^{b}_{R_{\sol}}$.
\end{lemma}
\begin{proof}
By construction, the category $\widetilde{E}$ of maps admitting $!$-functors is stable under disjoint unions, local on the target and local on the source, thus it suffices to show that each map $\AnSpec \n{B}_{i,j}\to \AnSpec \n{A}_i$ has $!$-functors. Since $\widetilde{E}$ is stable under compositions, it suffices to see that $\n{A}\to \n{A}[ T_1,\ldots, T_d]_{\sol}$ has $!$-functors, for which is enough to show that $\bb{Z}\to \bb{Z}[T]_{\sol}$ has  $!$-functors. But we can write $\bb{Z}\to \bb{Z}[T]\to \bb{Z}[T]_{\sol}$ where the first arrow has the induced analytic structure, and the second is an open immersion of locales (see Proposition \ref{PropositionAffinoidCatLocale}), thus the composite has $!$-functors proving what we wanted. 
\end{proof}

\subsection{Morphisms of finite presentation}
\label{SubsectionMorphismsFinitePresentation}

In applications we find different definitions of morphisms of finite presentation depending on the geometry we are studying. In this section we explain a way to treat some formal properties of any of these situations in a more axiomatic way. We shall restrict ourselves to the case of solid affinoid rings. 

\begin{definition}
\label{DefinitionCoordinateTheory}
Let $\n{A}$ be a  solid affinoid ring and let $\n{A}[T]$ be the polynomial algebra. A \textit{coordinate theory} over $\n{A}$ is an idempotent map of solid $\n{A}$-algebras $f:\n{A}[T] \to \n{A}(T)$ such that 
\begin{itemize}
\item[(i)] $\n{A}(T)$ is an animated ring stack in $\Aff_{\n{A}}$, i.e. the functor correpresented by $\n{A}(T)$ has a given enhancement in animated rings,  and $f$ is a morphism of animated ring stacks over $\n{A}$.

\item[(ii)] The natural map $\AnSpec \n{A}(T) \bigsqcup \AnSpec \n{A}(T^{-1}) \to \bb{P}^1_{\n{A}}$ is a $\s{D}$-cover. 
\end{itemize}

We define $\n{A}(T^{\pm 1}):= \n{A}(T)[T^{-1}]\otimes_{\n{A}[T^{\pm 1}]} \n{A}(T^{-1})[T]$. 
\end{definition}

Let us give different examples of coordinate theories that occur in practice: 

\begin{example}
\label{ExampleCoordinateTheories}
\begin{enumerate}

\item  Of course the trivial example is the identity map $\bb{Z}[T] \to \bb{Z}[T]$, in this case the ``coordinate" is the classical one from algebraic geometry. A more interesting example is the solidification functor $\bb{Z}[T] \to \bb{Z}[T]_{\sol}$, here we think of $T$ as the ``solid coordinate''.

\item In rigid geometry we have at least two examples: the first one is given by $(\bb{Q}_p[T],\bb{Z}_p)_{\sol} \to (\bb{Q}_p \langle T \rangle, \bb{Z}_p)$, the second one by $(\bb{Q}_p[T], \bb{Z}_p)_{\sol} \to \bb{Q}_p\langle T \rangle_{\sol}$.  The first coordinate is the adic compactification of the affinoid disc, the last is parametrized by the algebra $\bb{Q}_p\langle T \rangle_{\sol}$. Note that $\bb{Q}_p\langle T \rangle_{\sol}= \bb{Q}_p \otimes_{\bb{Z}_{\sol}} \bb{Z}[T]_{\sol}$, in general, the base change of a coordinate theory is a coordinate theory.

\item Let $\bb{Q}_p\langle T \rangle^{\dagger}:= \varinjlim_{\epsilon \to 0^+} \bb{Q}_p \langle p^{\epsilon} T \rangle$, this "coordinate" is the one used in the theory of dagger spaces over $\bb{Q}_p$. 

\end{enumerate}
\end{example}

We now define  rational localizations and morphims of (almost) finite presentation. 

\begin{definition}
\label{DefinitionGeneralFinitePresentation}
Let $\n{A}$ be a solid affinoid ring and $\n{A}[T] \to \n{A}(T)$  a coordinate theory over $\n{A}$.

\begin{enumerate}

\item A morphism $\n{A} \to \n{B}$ is an \textit{$\n{A}(T)$-rational localization} if it is a composite of morphisms of the form $\n{A} \to \n{A} \otimes_{\n{A}[T]} \n{A}(T)$ or $\n{A}\otimes_{\n{A}[T]} \n{A}(T^{-1})[T]$.  The \textit{$\n{A}(T)$-analytic topology} on $\Aff_{\n{A}}$ is the Grothendieck topology with covers given by $\s{D}$-covers consisting on finite disjoint unions of $\n{A}(T)$-rational localizations. We let $\Sh_{\n{A}(T)}(\Aff_{\n{A}})$ denote the category of $\n{A}(T)$-analytic sheaves on anima.  

\item A morphism $f:\n{A} \to \n{B}$ is of \textit{$\n{A}(T)$-finite presentation} if it belong to the smallest category of $\n{A}$-algebras stable under finite colimits and containing $\n{A}(T)$. We say that $f$ is of \textit{local  $\n{A}(T)$-finite presentation} if it is a retract of a morphism of $\n{A}(T)$-finite presentation. If $\n{A}(T)$ is clear from the context we simply say that $f$ is of (local)  finite presentation. 

\item An \textit{$\n{A}(T)$-adic space} is an object in $\Sh_{\n{A}(T)}(\Aff_{\n{A}})$ which is representable by an affinoid ring locally in the $\n{A}(T)$-analytic topology. 

\item A morphism $X \to Y$ of $\n{A}(T)$-adic spaces is   \textit{locally  of (local) finite presentation} if it is of (local) finite presentation locally in the $\n{A}(T)$-analytic topology of $X$ and $Y$. 

\end{enumerate}

\end{definition}

\begin{remark}
We use the name ``local of finite presentation'' for what \cite{LurieDerivedAlgebraic} calls ``locally of finite presentation''. The reason for this difference is to avoid properties on spaces  such as ``locally of locally of finite presentation'' which might be confusing. 
\end{remark}

\begin{remark}
\label{RemarkNonTrivialCoverCoordinatetheory}
Condition (ii) in Definition \ref{DefinitionCoordinateTheory} guarantees that we have non trivial rational covers, namely, if $b \in \n{B}$ is depicted from a map $\n{A}[T] \to \n{B}$, the localizations 
\[
\n{B} \to \n{B}(g)= \n{B} \otimes_{\n{A}[T]} \n{A}(T) \mbox{ and } \n{B} \to \n{B}(\frac{1}{g})= \n{B} \otimes_{\n{A}[T]} \n{A}(T^{-1})[T]
\]
form a $\s{D}$-cover of $\AnSpec \n{B}$. 
\end{remark}

Next, given a suitable six functor formalism on solid prestacks $\PSh(\Aff_{\n{A}})$, we want to give a simple criteria for morphisms locally of finite presentation of $\n{A}(T)$-adic spaces to admit $!$-functors.  Let $X\in \Sh_{\s{D}}(\Aff_{\bb{Z}_{\sol}})$ be a solid $\s{D}$-stack, suppose we are given a finite limit preserving functor $F: \PSh(\Aff_{\n{A}}) \to \Sh_{\s{D}}(\Aff_{\bb{Z}_{\sol}})_{/X}$. Let $E'$ be the  class of edges in $\PSh(\Aff_{\n{A}})$ such that $\sigma \in E'$ if and only if $F(\sigma) \in \widetilde{E}$, by Lemma \ref{LemmaFunctoriality6Functors} we have an induced $6$-functor formalism $\s{D}'$ on $(\PSh(\Aff_{\n{A}}), E')$, we let $\Sh_{\s{D}'}(\Aff_{\n{A}})$ denote its category of $\s{D}'$-stacks and $\widetilde{E}'$ the class of arrows of Theorem \ref{TheoSixFunctorsScholze}.

\begin{prop} 
\label{PropositionDevisageSixFiniteType}
Let us keep the previous notation. Suppose that the following condition hold
\begin{enumerate}

\item $F$ preserves coproducts. 

\item $F$ sends $\n{A}(T)$-analytic covers to $\s{D}$-covers.

\item The images by $F$ of $\AnSpec \n{A} \to \AnSpec \n{A}(T)$ and  $\AnSpec \n{A}(T) \to \AnSpec \n{A}$ are in $\widetilde{E}$, i.e. they admit $!$-functors.

\item For all $n\geq 1$  the image by $F$ of the map $\AnSpec \n{A} \to \AnSpec (\Sym_{\n{A}}^\bullet \n{A}[n] )$ is in $\widetilde{E}$. 

\end{enumerate}

Then $\n{A}(T)$-analytic covers are $\s{D}'$-covers, and any map in $\Sh_{\s{D}'}(\Aff_{\n{A}})$ representable, locally in the $\s{D}'$-topology, by morphisms of $\n{A}(T)$-adic spaces locally of  finite presentation is in the class $\widetilde{E'}$ of morphisms admitting $!$-functors. 
\end{prop}
\begin{proof}
Since $F$  sends $\n{A}(T)$-analytic covers to $\s{D}$-covers, rational localizations belong to $E'$ and admit $!$-functors. Indeed, by Remark \ref{RemarkNonTrivialCoverCoordinatetheory} any $\n{A}(T)$-rational localization belongs to an analytic $\n{A}(T)$-cover, and since $F$ preserves disjoint unions, the image under $F$ of rational localizations must admit $!$-functors.   Furthermore, since $F$ preserve co-products, disjoint union of rational localizations also admit $!$-functors. This implies that  analytic $\n{A}(T)$-covers of $\n{A}(T)$-adic spaces are $\s{D}'$-covers.  Let $f:Y \to Z$ be a map of $\s{D}'$-stacks over $\n{A}$ representable by a morphism of $\n{A}(T)$-adic spaces  locally of    finite presentation. Since the class $\widetilde{E'}$ is both local on the target and the source,  to show that $f$ admits $!$-functors it suffices to treat the case of a morphism of algebras $\n{B} \to \n{D}$ of finite presentation.  Now,   since $\n{A}[T] \to \n{A}(T)$ is idempotent, we have
\[
\n{A} \otimes_{\n{A}(T)} \n{A} = \n{A} \otimes_{\n{A}[T]} \n{A} = \Sym^{\bullet}_{\n{A}} \n{A}[1].
\] 
Then, as $\Sym^{\bullet}_{\n{A}} \n{A}[n+1]= \n{A} \otimes_{\Sym^{\bullet}_{\n{A}} \n{A}[n]} \n{A}$ for all $n\geq 1$, all the morphisms $\Sym^{\bullet}_{\n{A}} \n{A}[n] \to\n{A}$ and $\n{A}\to \Sym^{\bullet}_{\n{A}} \n{A}[n] $ are of $\n{A}(T)$-finite presentation. Moreover, since $F$ preserves finite limits, the map $\AnSpec (\Sym^{\bullet}_{\n{A}} \n{A}[n]) \to \AnSpec \n{A}$ belongs to $E'$. Therefore, to show that a morphism of finite presentation is in the class $\widetilde{E'}$, it suffices to see that it is constructed by composites of pushouts along the maps of the form 
\begin{itemize}
\item $\AnSpec \n{A} \to \AnSpec \Sym^{\bullet}_{\n{A}} \n{A}[n]$,
\item $ \AnSpec \Sym^{\bullet}_{\n{A}} \n{A}[n]\to \AnSpec \n{A}$,
\item $\AnSpec \n{A}(T) \to \n{A}$,

\item $\AnSpec \n{A} \to \AnSpec \n{A}(T)$.  
\end{itemize}
The claim follows by the following lemma:

\begin{lemma}
Let $\bb{I}:=(I_{n})_{n\in \bb{N}}$ be a sequence of finite sets with almost all $I_n$ empty, let $\n{A}(T_{\bb{I}})$ denote the algebra
\[
\n{A}(T_{\bb{I}}):= \n{A}(T_i:i\in I_0) \otimes_{\n{A}} \Sym_{\n{A}}^{\bullet} (\n{A}^{\oplus I_1}[1]) \otimes_{\n{A}} \cdots \otimes_{\n{A}} \Sym_{\n{A}}^{\bullet} (\n{A}^{\oplus I_n} [n] ) \otimes_{\n{A}} \cdots .
\]
Let  $(f_{n,j})_{j\in J_n}$ be elements in $\pi_n(\underline{\n{A}}(T_{\bb{I}}))$ such that $J_n$ is empty for almost all $n$, and such that  $f_{0,j}$ extends to a map $\n{A}[X]\to \n{A}(X)$.  Write $\bb{J}=(J_n)_{n\in\bb{N}}$. We denote by  $\n{A}(T_{\bb{I}})/^{\bb{L}} (f_{\bb{J}})$  the pushout of the diagram 
\[
\begin{tikzcd}
\n{A}(T_{\bb{J}}) \ar[r, "f"] \ar[d] & \n{A}(T_{\bb{I}}) \\
\n{A} .
\end{tikzcd}
\]
Then, any $\n{A}$-algebra of $\n{A}(T)$-finite presentation is isomorphic to a composite of algebras of the form $\n{A}(T_{\bb{I}})/^{\bb{L}} (f_{\bb{J}})$. 
\end{lemma}
\begin{proof}
We need to show that the category $\s{C}_{\n{A}}$ of  $\n{A}$-algebras constructed as composites of  algebras of the form $\n{A}(T_{\bb{I}})/^{\bb{L}} (f_{\bb{J}})$ is stable under finite colimits. By \cite[Corollary 4.4.2.4]{HigherTopos} it suffices to show that it is stable under pushouts.  Note that if $\n{B}\in \s{C}_{\n{A}}$ then the objects of $\s{C}_{\n{B}}$ are in $\s{C}_{\n{A}}$.  Consider a diagram $ \n{C} \leftarrow \n{B} \to \n{D}$ in $\s{C}_{\n{A}}$, we can write 
\[
\n{C}\otimes_{\n{B}} \n{D} = (\n{C}\otimes_{\n{A}} \n{D})\otimes_{\n{B}\otimes_{\n{A}} \n{B}} \n{B}. 
\]
It is clear that $\n{C}\otimes_{\n{A}} \n{D}\in \s{C}_{\n{A}}$.  We claim that the multiplication map $\n{B}\otimes_{\n{A}} \n{B}\to\n{B}$ is in $\s{C}_{\n{B}}$, if this holds then $\n{C}\otimes_{\n{B}} \n{D} \in \s{C}_{\n{C}\otimes_{\n{A}} \n{D}}$ whose objects are in $\s{C}_{\n{A}}$. Let us write $\n{B}=\n{A}(T_{\bb{I}})/^{\bb{L}} (f_{\bb{J}})$, then 
\[
\n{B}\otimes_{\n{A}} \n{B} = \n{A}(T_{\bb{I}},S_{\bb{I}})/^{\bb{L}} (f_{\bb{J}}(T), g_{\bb{J}}(S)).
\]
Since $\n{A}(T)$ is a animated ring stack, the maps $\n{A}[T_i-S_i]\to \n{B}\otimes_{\n{A}} \n{B} $ naturally extend to $\n{A}(T_i-S_i)\to \n{B}\otimes_{\n{A}} \n{B}$. We have that
\[
\n{B}\otimes_{\n{A}} \n{B} /^{\bb{L}}(T_{\bb{I}}-S_{\bb{I}}) = \n{B}/^{\bb{L}} (0_{\bb{J}}),
\]
where $0_{\bb{J}}$ is the sequence of $|J_n|$-zeros in $\pi_n(\n{B})$ for $n\in \bb{N}\geq 0$. But for any ring $\n{C}$ we have that  $\n{C}/^{\bb{L}} (0_n) = \n{C}(T[n+1])$ is the free algebra over $\n{C}$ with one generator in degree $n+1$. Thus, we can find elements $g_{n+1,J_n}$ in $\pi_{n+1}(\n{B}/^{\bb{L}} (0_{\bb{J}}))$ such that 
\[
\left(\n{B}/^{\bb{L}} (0_{\bb{J}})\right)/^{\bb{L}} (g_{\bb{J}}) =\n{B},
\]
proving the claim. 
\end{proof}

\end{proof}

\subsection{The cotangent complex}
\label{SubsectionCotangent}

In this section we briefly discuss some basic properties of cotangent complexes for prestacks on analytic rings.   We will follow mutatis mutandis \cite[\S 3.2]{LurieDerivedAlgebraic}. 

 We let $\PSh(\AnRing^{\op})$ be the $\infty$-category of presheaves of anima on $\AnRing^{\op}$. Given $\n{A}\in \AnRing$ we let $\SpecAn \n{A}$ denote its representable presheaf that we  refer as the \textit{analytic spectrum of $\n{A}$}. As it is standard, we let 
\[
\Mod(-): \PSh(\AnRing^{\op}) \to \CAlg(\Cat^{\colim,\ex}_{\infty})
\]
denote the right Kan extension of the functor of complete modules of analytic rings.

\begin{definition}[{\cite[Definition 3.2.5]{LurieDerivedAlgebraic}}]
    Let $M\in \Mod(\n{A})$, we say that $M$ is \textit{almost connective} if $M[n]$ is connective for some $n\geq 0$. Let $\n{F}: \AnRing\to \Ani$ be a  functor and $M\in \Mod(\n{F})$ a quasi-coherent complex, we say that $M$ is \textit{locally almost connective} if for all  analytic ring $\n{A}$ and all $\eta\in \n{F}(\n{A})$, $M(\eta)\in \Mod(\n{A})$ is almost connective. 
\end{definition}

Let $\n{A}$ be an analytic ring and $M$ a connective $\n{A}$-module. We have an adjunction $F: \AniAlg_{\underline{\n{A}}} \to \Mod_{\geq 0}(\underline{\n{A}}): \Sym^{\bullet}_{\underline{\n{A}}}$ between animated algebras over $\underline{\n{A}}$ and connective $\underline{\n{A}}$-modules. Given $M\in \Mod_{ \geq 0}(\underline{\n{A}})$, we can form the  $\underline{\n{A}}$-algebra $\underline{\n{A}}\oplus M$ as a condensed animated ring, which is obtained by forgetting the terms of degree $\geq 2$ in $\Sym^{\bullet}_{\underline{\n{A}}} M$.  

Since $M$ is  a nilpotent ideal of $\underline{\n{A}}\oplus M$, by Proposition 12.23 of \cite{ClauseScholzeAnalyticGeometry} the analytic ring structures on $ \underline{\n{A}}\oplus M$ and $\underline{\n{A}}$ are in bijection, and we have that 
\[
(\underline{\n{A}}\oplus M)_{\n{A}/}[*] = \n{A}\oplus \n{A}\otimes_{\underline{\n{A}}} M.
\] Given $M\in\Mod_{\geq 0}(\n{A})$ we shall denote by $\n{A}\oplus M$ the \textit{trivial square-zero extension} of $\underline{\n{A}}$ by $M$ endowed with the analytic ring structure arising from $\n{A}$.  Let $\n{F}: \AnRing \to \Ani$ be a presheaf, we say that $\n{F}$ admits an \textit{absolute cotangent complex} if  there exists a locally almost connective quasi-coherent complex $\bb{L}_{\n{F}}$ of $\n{F}$ such that the functor mapping a triple $(\n{A},\eta,M)$ consisting of $\n{A}\in \AnRing$, $\eta\in \n{F}(\n{A})$ and $M\in \Mod_{\geq 0}(\n{A})$ to the fiber product of
\[
\begin{tikzcd}
 & \eta \ar[d] \\
\n{F}(\n{A}\oplus M) \ar[r] & \n{F}(\n{A})
\end{tikzcd}
\]
is correpresented  by $ \bb{L}_{\n{F}}(\eta)$. 

One deduces easily the following property:

\begin{proposition}[{\cite[Proposition 3.2.9]{LurieDerivedAlgebraic}}]
\label{PropositionColimitsCotangentComplex}
    Let $\{\n{F}_i\}_{i\in I}$ be a diagram of functors $\n{F}_i: \AnRing \to \Ani$. Suppose that each $\n{F}_i$ has an absolute cotangent complex $\bb{L}_{i} \in \Mod(\n{F}_{i})$. Let $\n{F} = \varprojlim_{i} \n{F}_i$ and $\bb{L}= \varinjlim_{i} \bb{L}_i|_{\n{F}}$. Then $\bb{L}$ is an absolute cotangent complex for $\n{F}$ provided that it is locally almost connective. 
\end{proposition}

More generally, let $\n{F},\n{G}: \AnRing \to \Ani$ be two functors and let $\varphi:\n{F}\to \n{G}$ be a natural transformation. We say that the morphism $\varphi$ has a \textit{relative cotangent complex} if there is a locally almost connective quasi-coherent complex $\bb{L}_{\n{F}/\n{G}}\in \Mod(\n{F})$ such that for all $\n{A}\in \AnRing$, $\eta\in \n{F}(\n{A})$ and any connective $\n{A}$-module $M$, the object $\bb{L}_{\n{F}/\n{G}}(\eta)$ correpresents the fiber product
\[
\begin{tikzcd}
    & \eta \ar[d] \\
  \n{F}(\n{A} \oplus M)   \ar[r] & \n{F}(\n{A})\times_{\n{G}(\n{A})} \n{G}(\n{A}\oplus M)
\end{tikzcd}
\]
where the map from $\eta$ to $\n{G}(\n{A} \oplus M)$ is induced by the evaluation of $\n{G}$ at the zero section $\n{A}\to \n{A} \oplus M$. 

The relative cotangent complexes satisfy the following  properties, whose same proofs also hold in our context.

\begin{proposition}[{\cite[Proposition 3.2.10]{LurieDerivedAlgebraic}}]
    Let $\n{F}\to \n{G}$ be a natural transformation of functor from $\AnRing$ to $\Ani$. Suppose that $\bb{L}_{\n{F}/\n{G}}$ exists,  let $\n{G}'\to \n{G}$ be a natural transformation and  $\n{F'}=\n{F}\times_{\n{G}} \n{G}'$. Then $\bb{L}_{\n{F}/\n{G}}|_{\n{F}'}$ is the cotangent complex of $\n{F'}\to \n{G}'$. 
\end{proposition}

\begin{proposition}[{\cite[Proposition 3.2.12]{LurieDerivedAlgebraic}}]
\label{PropLurieDerived3.2.12}
    Let $\n{F}\to \n{F}'\to \n{F}''$ be a sequence of natural transformations of functors. Suppose that  there exists a cotangent complex $\bb{L}_{\n{F}'/\n{F}''}$.  Then there is an exact triangle 
    \[
    \bb{L}_{\n{F'}/ \n{F}''}|_{\n{F}} \to \bb{L}_{\n{F}/\n{F}''} \to \bb{L}_{\n{F}/\n{F}'}
    \]
in the sense that if either the second or the third term exists, then so does the other and there is a triangle above. 
\end{proposition}

The following  is the analogue to Proposition  3.2.24 of \cite{LurieDerivedAlgebraic} regarding the existence of relative cotangent complexes for morphisms of analytic rings. 

\begin{proposition}
    Let $f:\n{A} \to \n{B}$ be a morphism in $\AnRing$ and let $f':\SpecAn \n{B} \to \SpecAn \n{A}$ be the associated natural transformation at the level of presheaves. Then there exists a cotangent complex $\bb{L}_{\SpecAn \n{B}/ \SpecAn \n{A}}$ that is associated to the  $\n{B}$-module $\bb{L}_{\n{B}/\n{A}}$. Furthermore, if  $\bb{L}_{\underline{\n{B}}/ \underline{\n{A}}}$ is the cotangent complex of the map of  underlying  condensed rings, then 
    \[
    \bb{L}_{\n{B}/\n{A}} = \n{B} \otimes_{\underline{\n{B}}} \bb{L}_{\underline{\n{B}}/ \underline{\n{A}}}. 
    \]
\end{proposition}
\begin{proof}
    By Proposition \ref{PropLurieDerived3.2.12} it is enough to show that the functor $\SpecAn \n{A}$ has an absolute cotangent complex given by $\bb{L}_{\n{A}}=   \n{A} \otimes_{\underline{\n{A}}}\bb{L}_{\underline{\n{A}}}$. Let $\n{A}\to \n{C}$ be a morphism of analytic rings and let $M$ be a connective $\n{C}$-module. We want to describe the space 
    \begin{equation}
    \label{eqSectionsSquareZero}
    \Map_{\AnRing_{ /\n{C}} }(\n{A}, \n{C} \oplus M)
    \end{equation}
   in terms of the cotangent complex of $\underline{\n{A}}$. Since the analytic ring structure of $\n{C}\oplus M$ only depends on $\n{C}$, and we have already fixed a morphism of analytic rings $\n{A} \to \n{C}$, the above space is equivalent to the space
   \[
   \Hom_{\Cond(\AniRing)/\underline{\n{C}}}(\underline{\n{A}}, \underline{\n{C}} \oplus M)
   \]
   of morphisms of condensed animated  rings over $\underline{\n{C}}$.  Therefore, \eqref{eqSectionsSquareZero} is naturally equivalent to 
   \[
   \Hom_{\Mod_{ \geq 0}(\underline{\n{A}})}(\bb{L}_{\underline{\n{A}}}, M) = \Hom_{\Mod_{ \geq 0}(\n{A})}( \n{A} \otimes_{\underline{\n{A}}}\bb{L}_{\underline{\n{A}}},M)
   \]
   proving that $\bb{L}_{\n{A}}= \n{A} \otimes_{\underline{\n{A}}}\bb{L}_{\underline{\n{A}}}$ is an absolute cotangent complex for $\SpecAn \n{A}$. 
\end{proof}

\begin{remark}
    Let $\n{A} \to \n{B}$ and $\n{A}\to \n{C}$ be two morphisms of analytic rings. The base change property of the cotangent complex  is now given by 
    \[
    \bb{L}_{\n{C}/\n{A}} \otimes_{\n{C}} (\n{C}\otimes_{\n{A}} \n{B}) = \bb{L}_{\n{C} \otimes_{\n{A}} \n{B} / \n{B}}.
    \]
    If $\n{A}\to \n{B}$ is steady (\cite[Definition 12.13]{ClauseScholzeAnalyticGeometry} and \cite[Definition 2.3.16]{MannSix}) then it can be written simply as $\bb{L}_{\n{C}/\n{A}} \otimes_{\n{A}} \n{B}= \bb{L}_{\n{C} \otimes_{\n{A}} \n{B} / \n{B}}$. 
\end{remark}

We deduce the following construction that helps to identify the $\pi_0$ of a relative cotangent complex of rings with continuous  K\"ahler differentials.

\begin{proposition}[{\cite[Proposition 3.2.16]{LurieDerivedAlgebraic}}]
\label{PropAproxCotangent}
    Let $\n{A}\to \n{B}$ be a morphism of analytic rings, and let $K$ be the cone of this map seen as an object in $\Mod(\n{A})$. Then there is a natural map $\phi: K\otimes_{\n{A}} \n{B} \to  \bb{L}_{\n{B}/\n{A}}$. Moreover, if $f$ is $n$-connected for $n\geq 0$, then $\phi$ is $(n+2)$-connected. 
\end{proposition}
\begin{proof}
    By \textit{loc. cit} we have a map of the underlying condensed rings 
    \[
    K\otimes_{\underline{\n{A}}} \underline{\n{B}} \to \bb{L}_{\underline{\n{B}}/\underline{\n{A}}}
    \]
    satisfying the prescribed properties of the proposition. Since the analytification functor $\n{B} \otimes_{\underline{\n{B}}}-$ preserves connective objects, after $\n{B}$-analytification one has a natural map $\phi$ as stated satisfying the same connectness properties. 
\end{proof}

\begin{corollary}[{\cite[Corollary 3.2.17]{LurieDerivedAlgebraic}}]
\label{CoroEquiCotAndpi0}
 A morphism of analytic rings $\n{A}\to \n{B}$ is an equivalence if and only if $\pi_0(\n{A}) \to \pi_0(\n{B})$ is an isomorphism of analytic rings  and $\bb{L}_{\n{B}/\n{A}}=0$. 
\end{corollary}
\begin{proof}
    By \cite[Proposition 12.21]{ClauseScholzeAnalyticGeometry} the analytic ring structures of $\n{A}$ and $\pi_0(\n{A})$ are in bijection, so we only need to check that the underlying condensed rings are isomorphic. But then, the same argument of {\cite[Corollary 3.2.17]{LurieDerivedAlgebraic}} can be applied to deduce the equivalence. 
\end{proof}

\begin{remark}
Note that the hypothesis in the corollary asks for the map on connected components to be an isomorphism of \textbf{analytic rings}. If $\n{A} \to \n{A}'$ is a morphism of analytic rings with same underlying condensed ring then $\bb{L}_{\n{A}/\n{A}'}=0$. 
\end{remark}

\begin{example}
\label{ExampleCotangentComplexes}
    \begin{enumerate}
        \item The cotangent complex of a discrete affinoid  ring $(A,A^+)_{\sol}$ is simply the cotangent complex of the underlying condensed discrete ring.

        \item Let $f:(A,A^+)\to (B,B^+)$  be a  morphism of Huber pairs, by Proposition \eqref{PropAproxCotangent} we can compute the continuous K\"ahler differentials of $f$ as 
        \[
        \pi_0( \bb{L}_{(B,B^+)_{\sol}/(A,A^+)})= \Omega^1_{(B,B^+)/(A,A^+)} = \n{I} /\n{I}^2
        \]
        where $\n{I}$ is the augmentation ideal of $\pi_0(B\otimes_{(A,A^+)_{\sol}} B) \to B$, note that this tensor product coincides with the solid tensor product over $(A,\bb{Z})_{\sol}$ which is equal to the classical completed tensor product of Huber rings. 

        \item Let $A$ be a $I$-adically complete $I$-torsion bounded ring where $I$ is a finitely generated ideal of $A$. We have that $A\otimes_{\bb{Z}_{\sol}} \bb{Z}[T]_{\sol} = (A\langle T \rangle, \bb{Z}[T])_{\sol}$ where $A\langle T \rangle$ is the $I$-adic completion of the polynomial algebra. Then the cotangent complex
        \[
        \bb{L}_{(A\langle T \rangle, \bb{Z}[T])_{\sol}/ (A, \bb{Z})_{\sol}} = \bb{L}_{\bb{Z}[T]_{\sol}/ \bb{Z}}  \otimes_{\bb{Z}[T]_{\sol}} (A\langle T \rangle,\bb{Z}[T])_{\sol} \cong   A\langle T \rangle \cdot dT
        \]
        is just the usual continuous K\"ahler differentials.

        \item Let $\n{A}\to \n{C}$  be an idempotent morphism or analytic rings, i.e. such that $\n{C}\otimes_{\n{A}} \n{C}= \n{C}$. Then $\bb{L}_{\n{C}/\n{A}}=0$. Indeed, we have that
        \[
        \bb{L}_{\n{C}/\n{A}}=(\n{C}\otimes_{\n{A}} \n{C})\otimes_{\n{C}} \bb{L}_{\n{C}/\n{A}} = \bb{L}_{\n{C}\otimes_{\n{A}} \n{C}/\n{C}}= \bb{L}_{\n{C}/\n{C}}=0. 
        \]
  
    \end{enumerate}

\end{example}

Now let us restrict to morphisms of solid affinoid rings. 

\begin{definition}
\label{DefinitionSolidFinitePresentation}
We say that a morphism $f:\n{A}\to \n{B}$ of solid affinoid rings  is of \textit{solid finite presentation} if it belongs to the smallest category containing $\n{A}[ T ]_{\sol}:=\n{A}\otimes_{\bb{Z}_{\sol}} \bb{Z}[T]_{\sol}$ and stable under finite colimits (see Definition \ref{DefinitionGeneralFinitePresentation} and Example \ref{ExampleCoordinateTheories}). If the underlying rings of  $\n{A}$ and $\n{B}$ are static, we say that $f$ is  of \textit{solid finite presentation as static rings} if $\n{B}$ is a quotient of $\n{A}[ T_1,\ldots, T_d]_{\sol}$ by a finitely generated ideal. 
\end{definition}

One easily deduces the following lemma thanks to Proposition \ref{PropositionColimitsCotangentComplex} and the computation of the cotangent complex $\bb{L}_{\bb{Z}[T]_{\sol}/\bb{Z}}=\bb{Z}[T]dT$. 
\begin{lemma}
 Let $\n{A}$ be a solid affinoid ring, and let $\n{B}$ be an $\n{A}$-algebra of solid finite presentation. Then $\bb{L}_{\n{B}/\n{A}}$ is a finitely presented $\n{B}$-module, in particular it is discrete. 
\end{lemma}

\subsection{Solid  \'etale and smooth maps}
\label{SubsectionFormalleEtaleSmooth}

In this section we review the definition of formally smooth and formally \'etale maps in the  $\infty$-categories of presheaves on  bounded affinoid rings over $R_{\sol}=(R,R^+)_{\sol}=(\bb{Z}((\pi)), \bb{Z}[[\pi]])_{\sol}$. We will characterize formally smooth maps of solid finite presentation in more geometric way, in analogy to classical algebraic geometry. Let us first adapt the definition of small extensions to our setting. 

\begin{definition}[{\cite[Definition 3.3.1]{LurieDerivedAlgebraic}}]
  Let $\n{A}\to \n{B}$ be a morphism of analytic rings and $M$ a $\n{B}$-module. A small extension of $\n{B}$ by $M$ over $\n{A}$ is an analytic $\n{A}$-algebra $\widetilde{\n{B}}$   with a morphism $\widetilde{\n{B}}\to \n{B}$ whose underlying condensed ring is a small extension  of $\underline{\n{B}}$ by $M$ as $\underline{\n{A}}$-animated algebra. 

\end{definition}

\begin{remark}
    Note that the map $\pi_0(\underline{\widetilde{\n{B}}}) \to \pi_0(\underline{\n{B}})$ is a square zero extension, so the analytic ring structure of $\widetilde{\n{B}}$ is uniquely determined by that of $\n{B}$ by \cite[Proposition 12.23]{ClauseScholzeAnalyticGeometry}.  Therefore, the $\infty$-category of square zero extensions of $\n{B}$ over $\n{A}$ is the full subcategory of square zero-extensions $\widetilde{B}$ of  $\underline{\n{B}}$ over $\underline{\n{A}}$ as condensed rings such that the fiber $[\widetilde{B}\to \underline{\n{B}}]$ is in $\Mod(\n{B}) \subset \Mod(\underline{\n{B}})$. 
\end{remark}

\begin{definition}[{\cite[Definition 3.4.1]{LurieDerivedAlgebraic}}]
    Let $\n{F}: \AnRing_{\bb{Z}} \to \Ani$ be a functor, we say that $\n{F}$ is \textit{nilcomplete} if for all $\n{A}\in \AnRing$ the natural map $\n{F}(\n{A}) \to \varprojlim_n \n{F}(\tau_{\leq n}\n{A})$ is an equivalence. 

    We say that $\n{F}$ is \textit{infinitesimally cohesive} if for all small extension $\widetilde{\n{A}}$ of $\n{A}$ by an $\n{A}$-module $M$, the natural map
    \[
    \n{F}(\widetilde{\n{A}})\to \n{F}(\n{A})\times_{\n{F}(\n{A} \oplus M[1])} \n{F}(\n{A}) 
    \]
    is an equivalence. 
\end{definition}

\begin{definition}[{\cite[Definition 3.4.3]{LurieDerivedAlgebraic}}]
    Let $T:\n{F} \to \n{F}'$ be a natural transformation of functors $\n{F},\n{F}':\AnRing \to \Ani$.  We say that $T$ is 
    \begin{enumerate}
        \item \textit{weakly formally smooth} if it has a relative cotangent complex $\bb{L}_{\n{F}/\n{F}'}$  which is the dual of a connective (discrete) perfect complex. 

        \item \textit{formally smooth} if it is weakly formally smooth, nilcomplete and infinitesimally cohesive. 

        \item \textit{formally \'etale} if it is formally smooth and $\bb{L}_{\n{F}/\n{F}'}=0$.

    \end{enumerate}
\end{definition}

We can define  solid smooth and \'etale morphisms as follows.

\begin{definition}
\label{DefinitionSolidSmoothetale}
\begin{enumerate}
   \item  A morphism $f:\n{A}\to \n{B}$ of solid affinoid rings  is \textit{solid smooth (resp. \'etale)} if it is formally smooth (resp. formally \'etale) and of solid finite presentation. If $\n{A}$ and $\n{B}$ are static we say that $f$ is  \textit{solid smooth or \'etale as static rings} if it is of finite presentation as static rings (Definition \ref{DefinitionSolidFinitePresentation}) and formally smooth or \'etale with respect to square-zero extensions of static rings. 
   
   \item A morphism $f:\n{A} \to \n{B}$ of solid affinoid rings  is \textit{standard solid smooth} if $\n{B}= \n{A}[T_1,\ldots, T_n]_{\sol} /^{\bb{L}} (f_1,\ldots, f_k)$ for some sequence of elements $f_i \in \pi_0(\n{A}[T_1,\ldots, T_n]_{\sol} ) $  such that $\det((\frac{\partial f_i}{\partial T_j})_{1 \leq i,j\leq k})$ is invertible in $\n{B}$, it is \textit{standard solid \'etale} if it is standard solid smooth with $n=k$.  If $\n{A}$ and $\n{B}$ are static we say that $f$ is \textit{standard solid smooth or \'etale as static rings} if it is the $\pi_0$ of a standard solid smooth or \'etale algebra over $\n{A}$.
    \end{enumerate}
\end{definition}

Our main goal is to show that (1) and (2) in Definition \ref{DefinitionSolidSmoothetale} are equivalent after taking suitable rational covers. If $\n{A}$ is discrete,  any finitely presented $\n{A}$-algebra is discrete and different characterizations of (classical) smooth morphisms can be deduced from  \cite[Proposition 3.4.9]{LurieDerivedAlgebraic}. The main case of interest for us is when $\n{A}$ is a  bounded affinoid algebra over $R_{\sol}$. We have the following theorem.

\begin{theorem}
\label{TheoFormalSmoothnesvsSmoothness}
   Let $\n{A}$ be a bounded affinoid algebra over $R_{\sol}$.   Let $T:\n{A}\to \n{B}$ be a morphism of bounded affinoid $R_{\sol}$-algebras. The following are equivalent
    \begin{enumerate}
        \item $T$ is formally smooth and $\n{B}$ is of solid finite presentation over $\n{A}$. 
        
        \item $T$ is formally smooth and $\pi_0 \n{B}$ is  of solid finite presentation over $\pi_0 \n{A}$ as static rings. 
        
        \item $T$ is, locally in the analytic topology of $\n{B}$, a standard solid smooth $\n{A}$-algebra. 
    \end{enumerate}
   
   Furthermore, let $\underline{\n{A}}$ be nuclear over $R_{\sol}$,  and $f:\n{A}\to \n{B}$ a morphism  of the form  $\n{B}=\n{A}\langle T_1,\ldots, T_d \rangle_{\sol}/^{\bb{L}}(f_1,\ldots,f_d)$ with $f_i\in \pi_0(\n{A}\langle T_1,\ldots, T_d \rangle_{\sol})$.  Let $\n{D}=R\langle X_{n}: n\in \bb{N} \rangle $, then there is a map $\n{D} \to \n{A}$ and elements $g_i \in \n{D}\langle T_1,\ldots, T_d \rangle$ mapping to $f_i$ in $\n{A}\langle T_1,\ldots, T_d\rangle_{\sol}$.  In particular, $f$ is the pushout of a morphism of solid finite presentation over $\n{D}$. Moreover,   if $f$ is standard solid smooth (resp. \'etale)  we can can take the $g_i$ to define a  standard solid smooth (resp. \'etale) algebra over $\n{D}$.  Finally,  let $K$ be a non-archimedean field, then a classical standard solid smooth algebra over   $K\langle X_{n}: n\in \bb{N} \rangle$ is already derived.  In particular, solid smooth and \'etale maps of nuclear analytic $K$-algebras arise, locally in the analytic topology, as pushouts of classical  smooth and \'etale maps from  sous-perfectoid rings. 
\end{theorem} 

In order to prove the theorem we need some  standard preparations in commutative algebra. 

\begin{lemma}
\label{LemmaFormSmooth1}
    Let $\n{A}\in \AffRing_{R_{\sol}}^b$ and let $M$ be a finite  projective connective $\n{A}$-module together with a surjection $\bigoplus_{i=1}^k \underline{\n{A}}e_i \to M$ (i.e. a surjection on $\pi_0$). Let $E\subset \{1,\ldots, k\}$, there is a maximal analytic open subspace $U_E\subset \Spa(\n{A})$ such that $\{e_i: \; i\in E\}$ is a basis of $M$. Moreover, the open subsets $U_E$ are Zariski open and cover $\Spa(\n{A})$. 
\end{lemma}
\begin{proof}
    Let $\Spa(\n{A}) \to \Spec(\pi_0(\n{A})(*))$ be the natural map that sends a basic open Zariski $U_g^{\zar}=\{g\neq 0\}$ to the analytic set $U_g= \bigcup_{n\in \bb{N}} \{ |\pi^{n}| \leq |g| \neq 0\}$.  By  {\cite[\href{https://stacks.math.columbia.edu/tag/00O0}{Lemma 00O0}]{stacks-project}} there is an open cover $\{U_E^{\zar}\}_{E\subset \{1,\ldots, k\}}$ of $\Spec(\pi_0(\n{A}(*)))$ such that $U_E^{\zar}$ is the locus where the map $\bigoplus_{i\in E} \pi_0(\n{A})(*) e_i\to \pi_0(M)(*)$ is an isomorphism. This implies that over $U_{E}^{\zar}$, the map $\pi: \bigoplus_{i\in E}\n{A}(*) e_i \to M(*)$ is an equivalence and the same holds for $\bigoplus_{i\in E}\n{A} \to M$ since both are discrete $\n{A}$-modules. By taking the analytification $U_E\subset \Spa(\n{A})$ of $U_{E}^{\zar}$ we get the lemma. 
\end{proof}

\begin{lemma}
\label{LemmaFormSmooth2}
     Let $\n{A}$ be a static  bounded  $R_{\sol}$-algebra. Let  $\n{B}$ be a solid finitely presented $\n{A}$-algebra as static rings.  Suppose we have a presentation $\n{B}= \pi_0(\n{A} \langle T_1,\ldots, T_n\rangle_{\sol})/I$ with $I$ finitely generated and $I/I^{2}$ a free $\n{B}$-module. Then $\n{B}$ has a (non-derived) presentation of the form $\n{B}= \pi_0(\n{A}\langle S_1,\ldots, S_l \rangle_{\sol})/ (f_1,\ldots, f_c)$ such that $(f_1,\ldots, f_c)/(f_1,\ldots, f_c)^{2}$ is a free $\n{B}$-module with basis $(f_1,\ldots, f_c)$.   
\end{lemma}
\begin{proof}
    This is \cite[\href{https://stacks.math.columbia.edu/tag/07CF}{Lemma 07CF}]{stacks-project}, we will see that the same proof can be adapted in this setting. Let $f_1,\ldots, f_c\in I(*)$ be such that they form a basis of $I/I^{2}$. By Nakayama's lemma there is $g\in 1+I(*)$ such that $g\cdot I \subset (f_1,\ldots, f_c)$: this holds for the underlying ring $\pi_0(\n{A}\langle T_1,\ldots, T_n \rangle_{\sol}(*))$, but  $I$ and $I^2$ are finitely generated, then by taking the non-derived tensor $\pi_0(\n{A}\langle T_1,\ldots, T_d\rangle  \otimes_{\n{A}\langle T_1,\ldots, T_d\rangle(*)}-)$ the same holds for the condensed ideal  $I$. Then, $I[\frac{1}{g}]$ is generated by $f_1,\ldots, f_c$ and we can write
    \[
    \n{B}=\pi_0( \n{A} \langle T_1,\ldots, T_n \rangle_{\sol}[T_{n+1}] )/(f_1,\ldots, f_c, gT_{n+1}-1).
    \]
   where the ideal $J=(f_1,\ldots, f_c, gT_{n+1}-1)$ satisfies that $\{f_1,\ldots, f_c, gT_n-1\}$ is a basis for $J/J^2$.  On the other hand,  $g$ maps to $1$ in $\n{B}$ and we can localize at the open $\{ |g|\geq 1\}$.  We obtain a presentation 
   \[
   \n{B}= \pi_0(\n{A}\langle T_1,\ldots, T_n , T_{n+1}\rangle_{\sol}) / (f_1,\ldots, f_c, gT_{n+1}-1)
   \]
   with kernel $J'$ satisfying $J'/J^{'^2}= J/J^2\otimes_{(\bb{Z}[T],\bb{Z})} \bb{Z}[T]_{\sol} = J/J^{2}$,  where the solidification is with respect to $T=g^{-1}$. This proves the lemma. 
\end{proof}

\begin{lemma}
\label{LemmaFormSmooth3}
    Let $\n{A} \to \n{B}$ be a map of static bounded  affinoid $R_{\sol}$-algebras  and let $I\subset \n{B}$ be a finitely generated ideal. Set $\n{B}'= \n{B}/I^{n+1}$.  The map $\Omega^1_{\n{B}/\n{A}} \to \Omega^1_{\n{B}'/\n{A}}$ induces an isomorphism of non-derived tensors
    \[
    \pi_0(\Omega^1_{\n{B}/\n{A}} \otimes_{\n{B}} \n{B}/I^{n}) \xrightarrow{\sim} \pi_0( \Omega^1_{\n{B}'/\n{A}}\otimes_{\n{B}'} \n{B}/I^{n}).
    \]
\end{lemma}
\begin{proof}
    By \cite[\href{https://stacks.math.columbia.edu/tag/02HQ}{Lemma 02HQ}]{stacks-project} we know that this is true for the differentials of the underline condensed rings. The lemma follows by taking the analytification with respect to $\n{B}$. 
\end{proof}

\begin{lemma}
\label{LemmaFormSmooth4}
    Let $\n{A} \to \n{B}$ be a morphism of static bounded affinoid $R_{\sol}$-algebras  that is  of solid finite presentation as static rings.  Let  $P=\pi_0(\n{A} \langle T_1,\ldots, T_n\rangle_{\sol})$ and write  $\n{B}= \pi_0(\n{A} \langle T_1, \ldots, T_n\rangle_{\sol})/I$   with $I$ a finitely generated ideal. Then the sequence 
    \begin{equation}
    \label{eqShortExactSmooth}
    0 \to I/I^2 \to \Omega^1_{P/\n{A}} \otimes_{P} \n{B} \to \Omega^1_{\n{B}/\n{A}}\to 0.
    \end{equation}
    is exact and split if and only if $\n{A}\to \n{B}$ is formally smooth when restricted to static analytic rings.
\end{lemma}
\begin{proof}
    This is  the analogue of  \cite[\href{https://stacks.math.columbia.edu/tag/031I}{Lemma 031I}]{stacks-project}, let us see that the same argument works.  First, by invariance of analytic ring structures under nilpotent thickenings \cite[Proposition 12.23]{ClauseScholzeAnalyticGeometry}, and the fact that $\n{A}\langle T_1,\ldots, T_n \rangle_{\sol} = \n{A}[T_1,\ldots, T_n]_{\sol}$,   it is clear that the Tate ring $\n{A}\langle T_1,\ldots, T_d\rangle_{\sol}$ is formally smooth over $\n{A}$. Thus, one can easily check that $\n{B}$ is formally smooth  over $\n{A}$ (as static rings) if and only if there is a section of algebras $\n{B} \to P/I^{2}$. 

    Now, if $\n{B}$ is formally smooth we can find a split as above, this provides a section of the map $\Omega^1_{P/\n{A}}\otimes_{\n{P}} B \to \Omega^1_{\n{B}/\n{A}}$ and by applying \cite[\href{https://stacks.math.columbia.edu/tag/02HP}{Lemma 02HP}]{stacks-project} one gets that the sequence \eqref{eqShortExactSmooth} is also exact (this argument uses Lemma \ref{LemmaFormSmooth3}).   Conversely, suppose that the above sequence is exact and split. We want to construct a section of $P/I^{2}\to \n{B}$.  Let $\sigma: \Omega^1_{\n{B}/\n{A}} \to \Omega^1_{P/\n{A}}\otimes_{P} \n{B}$ be a section, and let us take $a_i\in I$ such that $d a_{i} = dT_{i}- \sigma(d\overline{T_i})\in \Omega^1_{P/\n{A}}\otimes_{P} \n{B}$. Consider the map  $f:P\to P/I^2$  sending  $T_i$ to $T_i-a_i$.  We claim that $f$ factors through $\n{B}$ providing the desired split. Since $I$ is finitely generated, it is enough to show that for any $b\in I$ we have $f(b)=0$. By exactness of the sequence it suffices to show that $d(b(T_i-a_i))=0$, but we have that 
    \[
    d(b(T_i-a_i))= d(b -\sum_{i=1}^{k} (\frac{\partial b}{\partial T_i}  )a_i) = \sum_{i=1}^{k} (\frac{\partial b}{\partial T_i}  ) \sigma (d\overline{T_i}) = \sigma(db)=0.
    \]
\end{proof}

\begin{lemma}[{\cite[\href{https://stacks.math.columbia.edu/tag/00TA}{Lemma 00TA}]{stacks-project}}]
\label{LemmaSmoothAsStandarSmooth}
    Let $\n{A}\to \n{B}$ be a solid smooth morphism of bounded affinoid rings. Then there is a finite analytic affinoid cover $\{\Spa \n{B}_i\}_{i}$ of $\Spa \n{B}$ such that the composite map $\n{A}\to \n{B}_i$ is standard solid smooth. 
\end{lemma}
\begin{proof}
    Let us write $\pi_0(\n{B})= \pi_0(\n{A}\langle T_1,\ldots, T_n\rangle_{\sol})/(f_1,\ldots, f_k)$, and  $I=(f_1,\ldots, f_k)$. Since $\n{A}\to \n{B}$ is formally smooth, the morphism $\pi_0(\n{A})\to \pi_0(\n{B})$ is formally smooth as static rings. By Lemma \ref{LemmaFormSmooth4}  we have a split short exact sequence  \eqref{eqShortExactSmooth}. In particular, $I/I^2$ is a projective $\pi_0(\n{B})$-module. By Lemma \ref{LemmaFormSmooth1} we can find a finite analytic Zariski cover of $\Spa\n{B}$ of the form $\{U_{g}\}_{g\in \pi_0(\n{B})}$ such that the module $I/I^2$ restricted to $U_g$ is free. Taking an open cover by affinoids $\{ |\pi^n|\leq g \neq 0\}$ we find that 
    \[
    \pi_0(\n{B}(\frac{\pi^n}{g})) = \pi_0(\n{A} \langle T_1,\ldots, T_n, T_{n+1} \rangle_{\sol})/ (f_1,\ldots, f_k, gT_{n+1}-\pi^n)
    \]
    is a presentation with kernel $J$ such that $J/J^2 = \pi_0(I/I^2\otimes_{\n{B}} \n{B}(\frac{\pi^n}{g})) \oplus (dT_{n+1}-\pi^n)  \pi_0(\n{B}(\frac{\pi^n}{g}))$ is a free $\pi_0(\n{B}(\frac{\pi^n}{g}))$-module. Thus without loss of generality we can assume that $I/I^2$ is free. Then by Lemma \ref{LemmaFormSmooth2} we have a presentation 
    \[
    \pi_0(\n{B})= \pi_0(\n{A}\langle T_1,\ldots, T_n\rangle_{\sol})/(f_1,\ldots, f_c)
    \]
    such that the elements $f_1,\ldots, f_c$ form a basis of $(f_1,\ldots, f_c)/(f_1,\ldots, f_c)^2$. By arguing as in \cite[\href{https://stacks.math.columbia.edu/tag/00TA}{Lemma 00TA}]{stacks-project} we can find a Zariski cover of the form $\{U_g\}_{g}$ such that the composition $I/I^2\to \bigoplus_{i=1}^n \pi_0(\n{B}) dT_i \to \bigoplus_{i=1}^{c} \pi_0(\n{B})d T_i $ is an isomorphism (after reordering the variables for each open $U_g$). Thus, we find covers of the form 
    \[
    \pi_0(\n{B}(\frac{\pi^n}{g})) = \pi_0(\n{A} \langle  T_1, \ldots, T_n , T_{n+1}\rangle_{\sol})/(f_1,\ldots, f_c, gT_{n+1}-\pi^{n}).
    \]
Reordering the variables, we get a standard solid smooth presentation of $\pi_0(\n{B}(\frac{\pi^n}{g}))$ as static rings. Thus, we can assume that the map $\pi_0(\n{A}) \to \pi_0(\n{B})$ is standard solid smooth as static rings: $\pi_0(\n{B})= \pi_0(\n{A} \langle T_1,\ldots, T_n \rangle_{\sol})/(f_1,\ldots, f_c)$ with $\det(\frac{\partial f_i}{\partial T_j})_{1\leq i,j\leq c}$ invertible in $\pi_0(\n{B})$. Consider $\n{B}'=\n{A} \langle T_1,\ldots, T_n \rangle_{\sol}/^{\bb{L}}(f_1,\ldots, f_c)$. We can lift the elements $\overline{T_i}\in \n{B}$ to a map of rings $\n{A}[T_1,\ldots, T_n]\to \n{B}$ that can be completed to $\n{A}\langle T_1,\ldots, T_n \rangle_{\sol}$. Moreover, the elements $f_i$ are mapped to $0$ in $\pi_0(\n{B})$ so that this map extends to a morphism of analytic rings $\n{B}'\to \n{B}$ which is an isomorphism on $\pi_0$. On the other hand, it is clear that the cotangent complex of $\bb{L}_{\n{B}'/\n{A}}$ is free and generated by $dT_{c+1},\ldots, dT_{n}$. Thus, the morphism of cotangent complexes 
\[
\bb{L}_{\n{B}'/\n{A}}\otimes_{\n{B}'}\n{B} \to \bb{L}_{\n{B}/\n{A}}
\]
is a surjective morphism of projective $\n{B}$-modules that is an isomorphism on $\pi_0$, then it must be an isomorphism. One gets that  $\bb{L}_{\n{B}/\n{B}'}=0$ and by Corollary \ref{CoroEquiCotAndpi0} we must have an equivalence $\n{B}' \cong \n{B}$. 
\end{proof}

\begin{corollary}
\label{CoroFormallySmoothStuff}
    A morphism of bounded affinoid rings $\n{A}\to \n{B}$ is solid smooth if and only if locally on the analytic topology of $\n{B}$ it factors as a composition $\n{A} \to \n{A} \langle T_1,\ldots, T_s \rangle_{\sol} \to \n{B}$ where the second arrow is standard solid \'etale. 
\end{corollary}
\begin{proof}
Suppose $f:\n{A}\to \n{B}$ is solid smooth.  By Lemma \ref{LemmaSmoothAsStandarSmooth}, locally on $\n{B}$ we can write $\n{B}= \n{A} \langle T_1,\ldots, T_n \rangle_{\sol}/^{\bb{L}}(f_1,\ldots, f_c)$ as a standard solid smooth map. It is clear that $\n{A} \langle T_{c+1} ,\ldots, T_{n} \rangle_{\sol} \to \n{B}$ is standard solid \'etale. 

Conversely, suppose that $f$ is, locally in the analytic topology, standard solid smooth. Since rational localizations are  standard solid \'etale maps, it suffices to show that standard solid smooth (resp. \'etale) morphisms is formally smooth (resp. formally \'etale). By invariance of analytic ring structure under nilpotent thickenings, one is even reduced to show that standard solid \'etale morphisms are formally \'etale. But the standard  computation of the cotangent complex (using $\bb{L}_{\bb{Z}[T]_{\sol}/\bb{Z}}=\bb{Z}[T]dT$) shows that if $f$ is standard solid \'etale  then $\bb{L}_{\n{B}/\n{A}}=0$ proving that it is formally \'etale.
\end{proof}

\begin{lemma}
\label{LemmaStandarSmoothRegular}
Let $(K,K^+)$ be a non-archimedean field,  $X=\Spa(A,A^+)$  an affinoid sous-perfectoid space over $(K,K^+)$ and  $B= A\langle T_1,\ldots, T_n \rangle/(f_1,\ldots, f_n)$ a  (non-derived) standard solid \'etale extension of $A$. Then the sequence $(f_1,\ldots, f_n)$ is regular, i.e. the natural map of condensed anima 
\begin{equation}
\label{eqKoszulregularEtale}
\Kos(A\langle T_1,\ldots, T_n \rangle; f_1,\ldots, f_n) \to B
\end{equation}
is an equivalence. 
\end{lemma}
\begin{proof}
  By the open mapping theorem it suffices to show that  \eqref{eqKoszulregularEtale} is an equivalence for the underlying sets. Without loss of generality we can take $f_1,\ldots, f_n\in A^+\langle T_1,\ldots T_n\rangle$, then we have to prove that the $\pi_i$ of the Koszul complex 
  \begin{equation}
  \label{eqKoszulregularEtale2}
  \Kos(A^+\langle T_1,\ldots, T_n\rangle; f_1,\ldots, f_n)
  \end{equation}
  have bounded torsion for $i>0$. Since $A$ is sous-perfectoid, and the terms of the Koszul complex are free Banach $A^+$-modules, we can assume that it is a perfectoid ring. Moreover, the Koszul complex \eqref{eqKoszulregularEtale2} is $\pi$-adically complete, so by $v$-descent we can even assume that $A$ is totally disconnected. Since solid almost $\pi$-adically complete modules glue in the analytic topology of $X$ (cf. \cite{MannSix}), it suffices to prove the claim locally on $X$. 
  
  Let $x\in \Spa(A,A^+)$ be a point with residue field $(\kappa(x), \kappa(x)^+)$. Then $B\otimes_{A} \kappa(x)$ is a finite \'etale extension of $\kappa(x)$, so a finite product of  $\kappa(x)$'s since it is algebraically closed. Thus, since $\Spa(\kappa(x), \kappa(x)^+)= \varprojlim_{x\in U\subset X} U$, there is an open neighbourhood $x\in U_x$ such that $B \otimes_{A} \s{O}(U)$ is just a finite product  of copies of $A$. Then, after localizing $B$, we can assume that it is equal to $\prod_{s} A$.  In this case, there are almost idempotent elements $e_1,\ldots, e_s \in B^+$ corresponding to the projections on each component of the product, and the localization $B(\frac{1}{e_i})=B\langle T \rangle/(T-e_i)$ corresponds to taking the $i$-th component. Therefore, we can even assume that $B= A$, in this case the Koszul-regularity follows from \cite[Lemma IV.4.16]{FarguesScholze}.
\end{proof}

\begin{lemma}
\label{LemmaFinitenessAnalytictoAdic}
    Let $\n{A}$ be a nuclear bounded $R_{\sol}$-algebra.    Let $\n{B}$ be a solid finitely presented algebra over $\n{A}$ of the form $\n{A} \langle T_1,\ldots, T_d \rangle_{\sol}/^{\bb{L}}(f_1, \cdots, f_k)$ with $f_i\in \pi_0(\n{A}\langle  T_1,\ldots, T_d\rangle_{\sol})$. Let $C= R\langle X_{n}: n\in \bb{N} \rangle$ be the Tate algebra over $R$ in countably many variables. Then there is a map $(C,R^+)_{\sol}\to \n{A}$ and elements $g_1,\ldots, g_k\in C\langle  T_1,\ldots, T_d\rangle$ mapping to $f_1,\ldots, f_k$. In particular, we have  an equivalence 
    \[
    \n{B} = C \langle T_1,\ldots, T_d \rangle_{\sol}/^{\bb{L}}(g_1,\ldots, g_k) \otimes_{(C,R^{+})_{\sol}} \n{A}. 
    \]
\end{lemma}

\begin{proof}
   Write  $\pi_0(\n{A})$ as a quotient of compact projective generators
    \[
    \bigoplus_{I} R_{\sol}[S_i] \to \pi_{0}(\underline{\n{A}}).
    \]
    Since $\underline{\n{A}}$ is nuclear, we have a surjection 
    \[
    \bigoplus_{I} R_{\sol}[S_i] \otimes_{R_{\sol}} R\langle T_1,\ldots, T_d\rangle \to \pi_{0}(\underline{\n{A}} \langle  T_1, \ldots, T_d\rangle).
    \]
    Let $f_1,\ldots, f_k\in \pi_{0}(\underline{\n{A}} \langle  T_1, \ldots, T_d\rangle)$, by taking finite disjoint unions of the $S_i$'s if necessary, we can find a profinite set $S$ and a lift  $\widetilde{f}_i \in R_{\sol}[S] \langle T_1,\ldots, T_d  \rangle$ of the $f_i$. Then, we can find a quasi-finitely generated subalgebra $E\subset R^+$ and a quasi-finitely generated $E$-module $M\subset R$ such that 
    \[
    \widetilde{f}_i \in M_{\sol}[S] \langle T_1,\ldots, T_d \rangle
    \]
    for all $i$ (cf. \cite[\S 3.1]{Andreychev}). We can write $f_{i}= \sum_{\alpha \in \bb{N}^d}  \pi^{r_{\alpha}}a_{i,\alpha} \underline{T}^{\alpha}$ with $r_{\alpha}\to \infty$ as $|\alpha| \to \infty$ and the $a_{i,\alpha}$ converging to $0$ in $M_{\sol}[S]$. 
    Let $\NS^{d}(R)\cong R^+[[T_1,\ldots, T_d]][\frac{1}{\pi}]$ be the space of null-$d$-sequences of $R_{\sol}$, i.e.  $\NS^{d}( R_{\sol}) = R_{\sol}[\bb{N}^{d}\cup \{\infty\}]/(\infty)$.  Then, we have a map 
   \[
    \bigoplus_{i=1}^{k}\NS^{d}(R_{\sol}) \to R_{\sol}[S]
   \]
   defined by the $d$ null-sequences $(a_{i,n})_{n\in \bb{N}^d}$ for $i=1,\ldots, k$.   Since $\n{A}$ is bounded, by multiplying the nullsequences by a power of $\pi$ we can assume that the composition $\bigsqcup_i \bb{N}\cup\{\infty\} \to \bigoplus_{i=1}^{k} \NS^{d}(R_{\sol})\to R_{\sol}[S] \to \n{A}$ lands in $\n{A}^{0}$. Thus, by definition of $\n{A}^{0}$, it lifts to a morphism of algebras 
   \[
     \n{R}= \bigotimes_{i=1}^{k} R_{\sol} \langle \bb{N}[ \bb{N}^{d} \cup \infty]   \rangle /(\infty) \to \underline{\n{A}}.
   \]
   Let $C=R\langle X_{i,\alpha}: 1\leq i\leq d, \;\; \alpha \in \bb{N}^d \rangle$, we can take the composite map 
   \[
   C \to  \n{R} \to \underline{\n{A}}
   \]
   mapping $X_{i,\alpha}$ to $a_{i,\alpha}$. Then letting $g_i= \sum_{\alpha\in \bb{N}}  \pi^{r_{\alpha}}X_{i,\alpha} \underline{T}^{\alpha}  \in C\langle  T_1,\ldots, T_d\rangle$, we see that $g_i$ maps to $f_k$ as wanted. 
\end{proof}

\begin{remark}
If $\n{A}$ admits a surjection from a nuclear $R_{\sol}$-algebra then the conclusion of Lemma \ref{LemmaFinitenessAnalytictoAdic} holds. Indeed,  one can lift a finitely presented algebra  of $\n{A}$ to the nuclear algebra and then apply the lemma. 
\end{remark}

\begin{proof}[Proof of Theorem \ref{TheoFormalSmoothnesvsSmoothness}]
The equivalences between (1)-(3) follow from Lemma \ref{LemmaSmoothAsStandarSmooth} and Corollary \ref{CoroFormallySmoothStuff}. The second statement for nuclear affinoid rings follows from Lemmas  \ref{LemmaStandarSmoothRegular} and \ref{LemmaFinitenessAnalytictoAdic}. 
\end{proof}

\begin{corollary}[{\cite[Corollary 3.4.10]{LurieDerivedAlgebraic}}]
\label{CoroEtaleMapsStacks}
    Let $T:\n{A} \to \n{B}$ be a morphism of bounded affinoid rings. The following are equivalent
    \begin{enumerate}
        \item $T$ is formally \'etale and of solid finite presentation.
        
        \item $T$ is formally \'etale and  $\pi_0(\n{B})$ is of solid finite presentation as static rings over $\pi_0(\n{A})$.

        \item $T$ is, locally on the analytic topology of  $\n{B}$, a standard solid \'etale map over $\n{A}$.
    \end{enumerate}
\end{corollary}


\subsection{Derived rigid geometry}
\label{SubsectionDerivedRigidGeometry}

In this section we study properties of solid smooth and \'etale morphisms of derived Tate adic spaces. We will show that these morphisms are cohomologically smooth and \'etale respectively, for the solid quasi-coherent six functor formalism. Finally, following \cite{CondensedComplex}, we give a proof of Serre duality for solid smooth maps, by identifying the dualizing sheaf $f^! 1$ with the canonical line bundle.

\subsubsection{Zariski closed immersions} Let us begin with a brief discussion of Zariski closed immersions and affinoid morphisms. 

\begin{definition}
\label{DefinitionAffinoidZariskiClosed}

\begin{enumerate}
\item  Let $f: X\to Y$ be a morphism of   derived Tate adic spaces.  We say that $f$ is \textit{affinoid} for the analytic topology if there is an open affinoid cover $\{U_i \}_{i}$  of $Y$ such that $V_i=X\times_{Y} U_i $ is an affinoid analytic ring. We say that $f$ is \textit{strictly affinoid} if in addition $V_i$ has the induced analytic structure from $U_i$.

\item Let $f: X\to Y$ be a morphism of   derived Tate adic spaces. We say that $f$ is a \textit{Zariski closed immersion} for the analytic topology if it is strictly affinoid, and we can find a cover  as before such that the map of animated condensed  rings $\s{O}(U_i)\to \s{O}(V_i)$ is surjective on $\pi_0$.

\item More generally, let $f:X \to Y$ be a map of Tate stacks, we say that $f$ is affinoid (resp. strictly affinoid, resp. Zariski closed immersion) in the $\s{D}$-topology,  if $f$ is of the form $\AnSpec \n{B} \to \AnSpec \n{A}$ (resp. $\AnSpec \n{B}_{\n{A}/} \to \AnSpec \n{A}$, resp. surjective on $\pi_0$) locally in the $\s{D}$-topology of $Y$. We say that $\s{O}_X$ is an analytic $\s{O}_Y$-algebra and that $X = \AnSpec_Y \s{O}_X$ is the relative analytic spectrum of $\s{O}_X$ over $Y$. 

\end{enumerate}
\end{definition}

\begin{lemma}
\label{PropQuotientfinitePresentation}
Let $\n{A} \to \n{B}$ be a morphism of bounded affinoid rings. Suppose that the following hold:
\begin{itemize}

\item $\n{A} \to \n{B}$ has the induced analytic structure and is surjective on $\pi_0$.

\item $\n{B}$ is a retract of an algebra of the form $\n{A}\langle \underline{T} \rangle_{\sol}/^{\bb{L}}(f_1,\ldots, f_d)$ for $f_i\in \pi_0(\n{A} \langle \underline{T} \rangle_{\sol})$. 

\end{itemize}

Then $\n{B}$ is a dualizable $\n{A}$-module. 

\end{lemma}
\begin{proof}
By hypothesis $\n{B}$ is a retract of an algebra of the form $\n{C}:=\n{A} \langle \underline{T}\rangle_{\sol}/^{\bb{L}} (f_i)$ for some finite set of variables $\{\underline{T}\}$ and a finite sequence $(f_i)_{i}$ in $\pi_0(\n{A})$. We can even assume that $\pi_0(\n{B})=\pi_0(\n{C})$ by killing additional  elements. Then, it suffices to show that $\n{C}$ is a dualizable $\n{A}$-module, and we can take $\n{B}=\n{C}$. Let $a_i\in \n{A}$ be a lift of the variables $T_i$, since $\n{A}$ is bounded there is some $k\geq 0$ such that the map $\n{A}[T_i] \to \n{A}$ sending $T_i$ to $a_i$ extends to $\n{A} \langle\pi^k \underline{T} \rangle_{\sol} \to \n{A}$.  By Lemma \ref{LemmaDualizableObjectslocale}, it suffices to prove that $\n{B}$ is dualizable locally in the open topology of the locale $S(\n{A})$. Actually, we will show that $\n{B}$ is dualizable locally in the topology of the locale $\Spa^{\dagger} \n{A}$ of Definition \ref{DefDaggerLocale}. By Proposition \ref{PropositionCLosedImageLocale} we have a closed immersion $\Spa^{\dagger} \n{B} \to \Spa^{\dagger} \n{A}$, let $U$ be the  open complement.  Let $V$ be the open subspace of $\Spa^{\dagger} \n{A}$ that corresponds to the open localization  $\n{A}\to \n{A} \otimes_{\n{A}\langle \pi^k  \underline{T} \rangle_{\sol}} \n{A} \langle T \rangle_{\sol}$. Then $V$ contains $\Spa^{\dagger} \n{B}$ and we have that $V\cup U = \Spa^{\dagger} \n{A}$. The localization of $\n{B}$ at $U$ is zero  by construction. On the other hand, the localization at $V$ of $\Spa^{\dagger} \n{A}$ is defined by an analytic ring $\n{A}'$ such that we have a commutative diagram 
\[
\begin{tikzcd}
\n{A}' \ar[r] &  \n{B} \\
\n{A}' \langle \underline{T} \rangle_{\sol} \ar[u] \ar[ur]
\end{tikzcd}
\] 
Since $\n{B}$ is an $\n{A}'$-module, we also have that 
\[
\n{B}=\n{B} \otimes_{\n{A}} \n{A}' = \n{A}'\langle T\rangle_{\sol}/^{\bb{L}}(f_i).
\]
Let $c_i$ be the image of $f_i$ in $\pi_0(\n{A}'\langle \underline{T} \rangle)$, then we have a retraction 
\[
\n{B}= \n{A}'\langle \underline{T} \rangle_{\sol}/^{\bb{L}}(f_i) \to \n{A}'/^{\bb{L}} (c_i) \to \n{B}
\]
proving that $\n{B}$ is a perfect $\n{A}'$-complex, so dualizable. 
\end{proof}

\begin{remark}
The proof of Lemma \ref{PropQuotientfinitePresentation} shows that $\n{B}$ is actually a perfect complex locally in the topology of $\Spa^{\dagger}(\n{A})$. However, this does not necessarily imply that it is a perfect complex over $\n{A}$, only a dualizable sheaf. If the ring $\n{A}$ is Fredholm (cf. \cite[Definition 9.7]{CondensedComplex}) then any dualizable $\n{A}$-module is perfect so $\n{B}$ would be  perfect as well. 
\end{remark}

\begin{remark}
We do not know how to prove that if $X\to Y$ is an affinoid (resp. strictly affinoid) map of derived Tate adic spaces such that $Y$ is affinoid, then $X$ is affinoid. One of the main obstacles is that it is not clear (and probably unlikely) whether the category of animated  algebras over an analytic ring satisfy analytic descent. Similarly, if $f$ is a Zariski closed immersion, even if we assume that both $X$ and $Y$ are affinoids, we do not know if the map on $\pi_0$ is surjective (the problem here is the lack of flatness for rational localizations). 
\end{remark}

\begin{example}
The hypothesis that $\n{B}$ is a retract of an algebra of finite presentation obtained by killing some $0$-cells of a solid Tate algebra is actually necessary. For example, $ \bb{Z} \to  \Sym_{\bb{Z}}^{\bullet} (\bb{Z}[2])$ is a Zariski closed immersion but $\Sym_{\bb{Z}}^{\bullet} (\bb{Z}[2])$ is not a perfect $\bb{Z}$-algebra since $\Sym^{n}_{\bb{Z}} (\bb{Z}[2])= (\Gamma^{n}_{\bb{Z}}\bb{Z}) [2n] \cong \bb{Z}[2n]$. 
\end{example}

\subsubsection{Solid \'etale and smooth maps}

\begin{definition}
\label{DefinitionDerivedGeometryProperties}
Let $f:X\to Y$ be a morphism of    derived Tate adic spaces over $\n{A}$. We define the following notions for $f$.
\begin{enumerate}
\item $f$ is  \textit{locally of (local) solid  finite presentation} if locally in the analytic topology of $X$ and $Y$, it is a morphism of (local) solid finite presentation of bounded affinoid algebras. We say that it is of   \textit{(local) solid finite presentation} if it is locally of finite presentation and qcqs for the analytic topology (cf. Definition \ref{DefinitionGeneralFinitePresentation}).

\item $f$ is \textit{solid smooth (resp. \'etale)} if it is,  locally  in the analytic topology of $X$ and $Y$,  a  solid smooth (resp. \'etale) morphism of bounded affinoid algebras.

\end{enumerate}
\end{definition}

In order to relate Zariski closed immerions of solid finite presentation with conormal cones we need the following lemmas.

\begin{lemma}
\label{LemmaSelfTesnsorQuotient}
Let $A\to B$ be a map of animated commutative rings that is surjective in $\pi_0$. Then $B\otimes_{A} B = \Sym^{\bullet}_{B} \bb{L}_{B/A}$.
\end{lemma}
\begin{proof}
By \cite[Theorem 2.23]{MaoCrystalline} the surjection $A\to B$ can be viewed as an animated pair $I\to A$ where $I=[A\to B]$ is the fiber. Then, since all the formulas commute with sifted colimits, one is reduced to prove the lemma for a animated pair of the form $(\underline{Y})\to \bb{Z}[\underline{X}, \underline{Y}]$ with $\underline{Y}$ and $\underline{X}$ sets of variables. The lemma follows by taking the Koszul resolution.  
\end{proof}

\begin{lemma}
\label{LemmaPerfectDiagonal}
Let $\n{B} \to \n{C}$ be a morphism of solid affinoid rings of the form $\n{C}=\n{B} [\underline{T} ]_{\sol}/^{\bb{L}}(f_i)$ with $f_i\in \pi_0(\n{B}[ \underline{T} ]_{\sol})$. Then $\n{C}$ is a perfect $\n{C}\otimes_{\n{B}} \n{C}$-module.
\end{lemma}
\begin{proof}
 We have that 
\[
\n{C}\otimes_{\n{B}} \n{C} = \n{B} [\underline{T}, \underline{S}]_{\sol}/^{\bb{L}}(f_i(T), f_i(S)),
\]
and the multiplication map   $\n{C}\otimes_{\n{B}} \n{C} \to \n{C}$ factors through 
\[
\n{B} [\underline{T}, \underline{S}]_{\sol}/^{\bb{L}}(f_i(T), f_i(S), T_i-S_i) \to \n{C}.
\]
But we have that 
\[
\n{B} [\underline{T}, \underline{S}]_{\sol}/^{\bb{L}}(f_i(T), f_i(S), T_i-S_i)  = \n{B}[\underline{T}]_{\sol}/^{\bb{L}} (f_i, \underline{0}_i)= \n{C}/^{\bb{L}}(\underline{0}_i), 
\]
proving that $\n{C}$ is a perfect complex over $\n{C} \otimes_{\n{A}} \n{C}$ (here we use that the $0$'s in the derived quotient are in degree $0$, so that it is indeed a perfect complex being represeted by a  Koszul complex).
\end{proof}

\begin{prop}
\label{PropositionLCI}
Let $X$ and $Y$ be derived Tate adic spaces over $\n{A}$ and  let $f:X \to Y$ be a  Zariski closed immersion of solid finite presentation with $\s{O}_X$ a perfect $\s{O}_Y$-module locally in the analytic topology of $Y$. The following hold

\begin{enumerate}

\item $1_{X}$ is $f$-smooth with dual given by $\iHom_{\s{O}_Y}(\s{O}_X, \s{O}_Y)$.

 \item If $\n{N}^{\vee}_{X/Y}=\bb{L}_{X/Y}[-1]$ is locally  free, then  $f^! 1_{Y}= (\bigwedge^{d} \n{N}_{X/Y})[-d]$ where $\n{N}_{X/Y}$ is the $\s{O}_X$-dual of $\n{N}^{\vee}_{X/Y}$, and $d$ is the locally constant rank of $\n{N}_{X/Y}$. In particular $f$ is cohomologically smooth. 
 \end{enumerate}
\end{prop}
\begin{proof}

\begin{enumerate}

\item  By hypothesis, $\s{O}_X$ is a dualizable $\s{O}_Y$-module, we also have a natural  identification $f^!\cong  \iHom_{\s{O}_Y}(\s{O}_X,-)$.   Thus,  the $f$-smooth dual of $\s{O}_X$ is $\iHom_{\s{O}_Y}(\s{O}_X, \s{O}_Y)$ which is a dualizable $\s{O}_Y$-complex. This implies that the formation of $f^!$ and $\n{D}_{f}(1_{X})$ commutes with any base change $Y' \to Y$, by Proposition \ref{PropSmoothProperObjects} (1) one deduces that $1_{X}$ is $f$-smooth. 

\item To prove part (2) it suffices to show that $f^! 1_Y = (\bigwedge^d \n{N}_{X/Y})[-d]$. As $\s{O}_X$ is a perfect $\s{O}_Y$-module locally in the analytic topology, the ideal $I=[\s{O}_Y \to \s{O}_X]$ is discrete relative to $\s{O}_Y$ locally in the analytic topology, and by Lemma \ref{LemmaSelfTesnsorQuotient} the space $X\times_{Y} X$ is given by the  relative analytic spectrum  over $Y$ of the (locally in the analytic topology) animated algebra $\s{O}_Z:= \Sym_{\s{O}_X}^{\bullet} \n{N}^{\vee}_{X/Y}[1]=\bigoplus_{i=0}^{d} \bigwedge^{i} \n{N}_{X/Y}^{\vee}[i]$. Consider the diagram 
\[
\begin{tikzcd}
X \ar[r, "\Delta"]& X\times_{Y} X \ar[r, "\pi_2"]\ar[d, "\pi_1"] & X \ar[d, "f"] \\ 
	& X \ar[r, "f"] & Y
\end{tikzcd}
\]
We find that $f^! 1_{Y} = \Delta^* \pi_1^* f^! {1_Y}= \Delta^* \pi_2^! 1_{X}$. But then 
\[
\begin{aligned}
\pi_{2}^! 1_{X} & = \iHom_{\s{O}_X}(\s{O}_Z, \s{O}_X) \\ 
& = \bigoplus_{i=0}^d \bigwedge^i \n{N}_{X/Y}[-i] \\ 
& = \s{O}_{Z} \otimes_{\s{O}_X} \bigwedge^d \n{N}_{X/Y}[-d] 
\end{aligned}
\]
which shows that $f^!1_{Y}=\Delta^*( \s{O}_{Z} \otimes_{\s{O}_X} \bigwedge^d \n{N}_{X/Y}[-d])  = \bigwedge^d \n{N}_{X/Y}[-d]  $ as wanted.
\end{enumerate}
\end{proof}

\begin{cor}
\label{CoroSmoothImmersion}
Let $S$ be a derived  adic space over $\n{A}$, let $X$ and $Y$ be solid smooth derived Tate adic spaces over $S$ and $f:X\to Y$ a Zariski closed immersion with $\s{O}_X$ a perfect $\s{O}_Y$-complex locally in the analytic topology of $Y$. Then $\n{N}^{\vee}_{X/Y}=\bb{L}_{X/Y}[-1]$ is a locally  free sheaf over $X$ for the analytic topology.  In particular, $f$ is cohomologically smooth. 
\end{cor}
\begin{proof}
We prove that $\bb{L}_{X/Y}[-1]$ is a locally  free sheaf over $X$. For this, we can assume without loss of generality that $S$, $X$ and $Y$ are affinoid with rings $\n{B}$, $\n{C}$ and $\n{D}$ respectively. Furthermore, by Theorem \ref{TheoFormalSmoothnesvsSmoothness}, we can even assume that  $\n{D}$ is a  standard solid smooth over $\n{B}$, so that  $\bb{L}_{\n{D}/\n{B}}$ is a free sheaf of constant rank. We have a fiber sequence of cotangent complexes 
\[
 \bb{L}_{\n{C}/\n{D}}[-1] \to \n{C} \otimes_{\n{D}} \bb{L}_{\n{D}/\n{B}} \to \bb{L}_{\n{C}/\n{B}} \xrightarrow{+}.  
\] 
Since both $\bb{L}_{\n{C}/\n{B}}$ and $\n{C}\otimes_{\n{D}} \bb{L}_{\n{D}/\n{B}}$ are free $\n{C}$-modules, and $\n{C}\otimes_{\n{D}} \bb{L}_{\n{D}/\n{B}} \to \bb{L}_{\n{C}/\n{B}}$ is surjective on $\pi_0$,  $\bb{L}_{\n{C}/\n{D}}[-1]$ is a projective $\n{C}$-module that is free locally in the analytic topology of $X$ by Lemma \ref{LemmaFormSmooth1}.  The corollary follows by Proposition \ref{PropositionLCI}.
\end{proof}

Our next goal is to prove Serre duality in derived rigid geometry following the methods of \cite[Lecture XIII]{CondensedComplex}. First, we need to prove that solid smooth and \'etale maps are indeed cohomologically smooth and \'etale.  

\begin{lemma}
\label{LemmaEtaleStandardDiagonal}
Let $\n{B}\to \n{C}$ be a standard solid \'etale map of bounded affinoid rings, then the multiplication map $\n{C}\otimes_{\n{B}} \n{C} \to \n{C}$ defines an analytic open immersion at the level of affinoid spaces. 
\end{lemma}
\begin{proof}
By hypothesis we can write $\n{C}= \n{B}\langle T_1,\ldots, T_d \rangle /^{\bb{L}} (f_1,\ldots, f_d)$ with $\det( \frac{\partial f_i}{\partial T_j})_{i,j}$ a unit in $\n{C}$ that we can assume solid.  Moreover, writing $\pi_0(\n{B})$ as a filtered  colimit of quotients of bounded algebras generated by extremally disconnected sets, we can assume that $\n{B}= R_{\sol}\langle \underline{X} \rangle_{\sol} \langle \bb{N}[K] \rangle$ with $\underline{X}$ a finite set of variables and $K$ a profinite set.  Then  $\n{C}$ is of the form
\[
\n{C}= R\langle \underline{X}, T_1,\ldots, T_d \rangle_{\sol} \langle \bb{N}[K] \rangle /^{\bb{L}}(f_1,\ldots, f_{d}). 
\]
Let $g= \det (\frac{\partial f_i}{\partial T_j})_{i,j}$, we can assume that  $f_i\in R^+\langle \underline{X}, T_1,\ldots, T_d \rangle_{\sol} \langle \bb{N}[K] \rangle $ for all $i$,  $|g|\leq 1$ and that $|\pi^{n}| \leq |g|$ for some fixed $n$. We claim that the multiplication map 
\[
R\langle \underline{X}, T_1,\ldots, T_d ,S_1,\ldots, S_d \rangle_{\sol} \langle \bb{N}[K] \rangle /^{\bb{L}}(f_1(T),\ldots, f_{d}(T), f_1(S), \ldots, f_d(S)) \langle \frac{T_i-S_i}{ \pi^{2n+1}} \rangle_{\sol} \to \n{B}
\]
is an isomorphism.   By Corollary \ref{CoroEquiCotAndpi0}, since  the relative cotangent complex vanishes, it suffices to prove that it is an isomorphism on $\pi_0$. For this, let $\n{D}= \n{C} \langle S_i, \frac{S_i-T_i}{\pi^{2n+1}}:i=1,\ldots, d \rangle_{\sol}$, we want to show that the ideals generated by $(S_i-T_i)_{i=1}^d$  and  $(f_i(S))_{i=1}^d$  on $\n{C}$ are the same. We have Taylor series expansions 
\[
f_{i}(S) = f_{i}(T+(S-T)) = \sum_{\alpha \in \bb{N}^d} f_i^{[\alpha]}(T) (S-T)^{\alpha} 
\]
where $f^{[\alpha]}$ is the $\alpha$-th $PD$-derivative of $f$ with respect to the variables $T$. Indeed, as $(S-T)$ is divisible by $\pi^{2n+1}$, the series converges by the explicit growth conditions of Example \ref{ExamplePowerSeriesBoundedAlgebra}.  As $f(T)=0$ in $\pi_0(\n{C})$, we have that 
\[
f_i(S)= \sum_{k=1}^d (\partial_{T_k} f_i)(T) (S_k-T_k)+ h_i(S-T).
\]
where $h_i(S-T)$ has bounded coefficients in $\n{C}$ and monomials of degree $\geq 2$.   Let $(b_{i,j})_{i,j}$ be the inverse of the matrix $((\partial_{T_j} f_i) (T))_{i,j}$, then we can write 
\begin{equation}
\label{eqRecurrenceSolveTS}
S_i-T_i= \sum_{j=1}^d b_{i,j} (f_i(S)-h_i(S-T)).  
\end{equation}
By hypothesis, $|b_{i,j}|\leq |g^{-1}| \leq | \pi^{-n}|$, so that $| b_{i,j} h_i| \leq \pi^{n+1}$ in $\n{C}$,  this implies that $\sum_{j=1}^d b_{i,j} f_{i}(S)$ is bounded by $\pi^{n+1}$ in $\n{C}$. Iterating the equation \eqref{eqRecurrenceSolveTS} one finds that 
\[
S_i-T_i= \sum_{1\leq |\alpha|\leq 2} c_{i,\alpha} \underline{f}(S)^{\alpha} + h^{(2)}_i(S-T)  
\]
where the $c_{i,\alpha} \in \n{C}$ satisfy that $|c_{i,\alpha}|\leq |\pi^{|\alpha|(n+1)}|$, and $h^{(2)}_i (S-T)$ has bounded coefficients in $\n{C}$ with monomials of degree $\geq 4$.  An inductive hypothesis let us write 
\[
S_i-T_i= \sum_{1 \leq |\alpha| \leq 2^{k}} c_{i,\alpha} \underline{f}(S)^{\alpha}+ h_{i}^{(k)}(S-T)
\]
where $|c_{i,\alpha}|\leq |\pi^{|\alpha|(n+1)}|$ and $h_i^{(n)}(S-T)$ has bounded coefficients in $\n{C}$ with monomials of degree $\geq 2^{k}$. Taking limits as $k\to \infty$ we get that  $S_i-T_i$ belong to the ideal generated by the $f_{j}(S)$ for all $i$ , which finishes the proof. 
\end{proof}

\begin{lemma}
\label{LemmaCohoSmoothAffineLine}
The map of analytic rings $f:\bb{Z}_{\sol}\to \bb{Z}[T]_{\sol}$ is cohomologically smooth. 
\end{lemma}
\begin{proof}
This is a consequence of \cite[Theorem 8.1]{ClausenScholzeCondensed2019}. Consider the  compactification $\bb{Z}_{\sol} \xrightarrow{g} (\bb{Z}[T], \bb{Z})_{\sol} \xrightarrow{j} \bb{Z}[T]_{\sol}$. Let us describe explicitly the shriek functors of  $f$. Recall that $(\bb{Z}[T], \bb{Z})_{\sol} \to \bb{Z}[T]_{\sol}$ is the open localization complement to the idempotent algebra $\bb{Z}((T^{-1}))$. Then, we have descriptions  for $M\in \Mod( (\bb{Z}[T],\bb{Z})_{\sol})$
\[
j_! j^*M =  [\bb{Z}[T] \to \bb{Z}((T^{-1}))]\otimes_{\bb{Z}[T]}M \mbox{ and } j_* j^*M = \iHom_{\bb{Z}[T]}([\bb{Z}[T] \to \bb{Z}((T^{-1}))], M),
\]
notice that the fiber  $[\bb{Z}[T] \to \bb{Z}((T^{-1}))]$ is isomorphic to 
\[
\bb{Z}[T]^{\vee}[-1]=\iHom_{\bb{Z}}(\bb{Z}[T], \bb{Z})[-1]=\iHom_{\bb{Z}}(\bb{Z}[T], \bb{Z})[-1]
\] as $\bb{Z}[T]$-module. Therefore, the funtor $f^!$ is isomorphic to 
\[
f^! N\simeq \iHom_{\bb{Z}}( \bb{Z}[T]^{\vee}, N  )[1],
\]
but $\iHom_{\bb{Z}}( \bb{Z}[T]^{\vee}, -)=f^*$, namely  $\bb{Z}[T]^{\vee}$ is a compact projective $\bb{Z}$-module, both functors commute with limits and $\iHom_{\bb{Z}}(\bb{Z}[T]^{\vee}, \bb{Z})= \bb{Z}[T]$. In particular, one has that $f^! \bb{Z}\simeq \bb{Z}[T][1]$. Notice that the previous hold for any base ring $A$ with $A$ a finitely generated $\bb{Z}$-algebra.  Now, take $Y= \AnSpec \bb{Z}[T]_{\sol}$, $X= \AnSpec \bb{Z}_{\sol}$ and consider the cartesian square 
\[
\begin{tikzcd}
X\times_Y \ar[r, "p_1"]  \ar[d, "p_2"']X & X \ar[d, "f"] \\ 
X \ar[r, "f"] & Y
\end{tikzcd}
\] 
Then, under the identification $f^! \bb{Z}\simeq \bb{Z}[T][1]$, one has the natural map $p_1^* (\bb{Z}[T][1]) \to p_2^! \bb{Z}[T]$ is an equivalence (both being equal to $\bb{Z}[T_1,T_2][1]$). This implies that $f$ is cohomologically smooth by Proposition \ref{PropSmoothProperObjects} (1.b). 
\end{proof}

\begin{proposition}
\label{PropGeoSmoothIsSmooth}
Let $f:X \to S$ be a solid smooth (resp. \'etale) morphism of derived Tate adic spaces over $\n{A}$. Then $f$ is cohomologically smooth (resp. \'etale).
\end{proposition}
\begin{proof}
First, by Lemma \ref{LemmaCohoSmoothAffineLine}, the map $\bb{Z}_{\sol}\to \bb{Z}[T]_{\sol}$ is cohomologically smooth.  Since cohomologically smooth maps are stable under composition and base change, this implies that any map $\n{A} \to \n{A} \langle T_1,\ldots, T_d\rangle_{\sol}$ of bounded affinoid rings is cohomologically smooth. By Theorem \ref{TheoFormalSmoothnesvsSmoothness}, any solid smooth map factors, locally in the analytic topology, as a composite of a standard solid \'etale map and  the projection of an affinoid disc. Therefore, for the first assertion it suffices to see that a solid \'etale map is cohomologically \'etale. Let us assume $f$ solid \'etale, by Lemma \ref{LemmaEtaleStandardDiagonal} the diagonal $\Delta_{f}$ is an open embedding, so it is $-1$-truncated and by Lemma \ref{LemmaRepresentableInP} it is cohomologically \'etale. Thus, by Definition \ref{DefinitionCohoEtaleProper},  it suffices to see that $1_{X}$ is a $f$-smooth object. We can argue locally in the analytic topology of both $X$ and $S$ and assume that both are affinoid. Then, we can find a Zariski closed embedding $\iota: X\to S\times \bb{D}^{d}_{R}$ into some affinoid disc over $S$ such that $\s{O}_X$ is a perfect $\s{O}_{S\times \bb{D}^{d}_R}$-module. By Corollary \ref{CoroSmoothImmersion}  the map $\iota$  is cohomologically smooth, as $S\times \bb{D}^d_R \to S$ is cohomologically smooth  one deduces that $X\to S$  is cohomologically smooth which in particular implies that $1_{X}$ is $f$-smooth as wanted. 

\end{proof}

A first application of the previous proposition are some classical facts about \'etale maps

\begin{proposition}
\label{PropMapsEtalemaps}
Let $S=\AnSpec \n{A}$ be a bounded affinoid space and let $X=\AnSpec \n{C}$ and $Y=\AnSpec \n{B}$ be  bounded affinoid spaces whose rings of functions are given by  $\n{B}=\n{A}\langle \underline{T} \rangle/^{\bb{L}}(f_i)$ and $\n{C}=\n{A}\langle \underline{T'} \rangle/^{\bb{L}}(g_i)$ with $f_i\in \pi_0(\n{A} \langle \underline{T} \rangle )$ and $g_i\in \pi_0(\n{A} \langle \underline{T'} \rangle ) $.   Suppose that  we have maps $f:X\to Y$ over $S$. The following hold
\begin{enumerate}

\item $f$ is  of local solid finite presentation, i.e. a retract of a morphism of solid finite presentation.

\item  If $X$ and $Y$ are solid \'etale over $S$ then $f$ is solid \'etale.

\item  If $f$ is solid \'etale and a Zariski closed embedding  with $\s{O}_X$ a perfect $\s{O}_Y$-module, then it is a rational open subspace associated to an open and closed subspace of $|Y|$. 

\end{enumerate}
\end{proposition}
\begin{proof}
\begin{enumerate}
\item The algebra  $\n{C}$ is a retract of $\n{B}\langle \underline{T}'\rangle_{\sol}/^{\bb{L}}(g_i)$, this shows   that it is of local solid finite presentation over $\n{B}$.

\item By (1) we know that $f$ is of local solid finite presentation. On the other hand, since $X$ and $Y$ are \'etale over $S$, the fiber sequence of cotangent complexes shows that $\bb{L}_{X/Y}=0$,  so that $f$ is formally \'etale. One deduces that $f$ is  solid \'etale by Corollary \ref{CoroEtaleMapsStacks}.

\item  By Proposition \ref{PropositionLCI} and Lemma \ref{LemmaSelfTesnsorQuotient}, one deduces that $X\times_Y X = X$, which implies that $f$ is $-1$-truncated so an immersion. Since $X$ has the induced analytic structure from $Y$, one deduces that $\s{O}_X \otimes_{\s{O}_Y} \s{O}_X= \s{O}_X$ is an idempotent $\s{O}_Y$-algebra, and that $X$ defines a closed subspace in the locale of $\Mod_{\sol}(Y)$. Since $f$ is \'etale and a closed immersion, one has that $f^!= f^*$ and $f_!= f_*$, which by Proposition \ref{PropClosedOpenLocalizationsInftyCat} implies that $f$ also defines an open embedding in the locale. Let $\n{C}$ be the open and closed complement of $\s{O}_{X}$ in the locale of $\Mod_{\sol}(Y)$, then we have that $\s{O}_Y = \s{O}_X \oplus \n{C}$ as $\bb{E}_{\infty}$-algebras. In particular, $\n{C}$ is a locally connective $\bb{E}_{\infty}$-algebra in the analytic topology of $Y$.

 Note that the notion of being a bounded affinoid algebra only depends on $\pi_0$, in particular it is also a well defined notion for connective analytic $\bb{E}_{\infty}$-algebras over $R_{\sol}$. In particular, Proposition \ref{PropLiftingOverconvergentAlgebrasLift} also holds for $\n{C}$. Now, let $I=[\s{O}_Y \to \s{O}_X]$ be the ideal defining $X$, by hypothesis $I$ is a perfect $\s{O}_Y$-module, so its $\pi_0$ is a module generated by its global sections at the point, and by Proposition \ref{PropZariskiCloseImmersion1} the idempotent $\s{O}_Y$-algebra $\s{O}_Y \{I\}^{\dagger}$ is associated to the closed subspace $|X|\subset |Y|$.  But we have $I\cong\n{C}$, so   $\n{D}=\s{O}_{Y} \{I\}^{\dagger} \otimes_{\s{O}_Y} \n{C}$ is equal to its $\dagger$-nil-radical  which implies that it is $0$ as $T-1$ is invertible in $R\{T\}^{\dagger}$ and the map $R[T]\to \n{D}$ mapping $T$ to $1$ extends to the overconvergent power series. The previous reasoning shows that in fact 
 \[
 \s{O}_X= \s{O}_Y\{I\}^{\dagger}.
 \]
Now,  since $\s{O}_X$ is $\s{O}_Y$-perfect,   one has that $\s{O}_X = \s{O}_Y \langle \frac{I}{\pi^n}\rangle_{\sol}$ for some $n>>0$, and $X$ defines both an open and a closed subspace arising from the underlying space $|Y|$.

\end{enumerate}
\end{proof}

\subsubsection{Serre duality} We want to prove the following theorem. 

\begin{theorem}[Serre duality]
\label{TheoSerreDuality}
Let $f:X\to S$ be a solid smooth morphism of  derived Tate adic  spaces.   Then $f$ is cohomologically smooth and there is a natural identification $f^! 1_{S}= \Omega^{d}_{X/S}[d]$, where $d$ is the locally constant relative dimension of $f$, and $\Omega^d_{X/S}:= \bigwedge^d \bb{L}_{X/S}$ is the determinant of the (locally free) cotangent complex. 

\end{theorem}

We have already proved the first part of the theorem in Proposition \ref{PropGeoSmoothIsSmooth}, the rest of the proof will follow  the same steps of  \cite[Theorem 13.6]{CondensedComplex} using the deformation to the normal cone.

\begin{remark} 
\label{RemarkDeformationNormalCone}
Note that, if $\n{A} \to \n{B}$ of a morphism solid finite presentation which is surjective on $\pi_0$ and such that $\n{B}$ is a perfect $\n{A}$-module, the ideal $I=[\n{A}\to \n{B}]$   is a  discrete $\n{A}$-module, i.e. it arises as base change from an ideal of the animated ring $\n{A}(*)$. Therefore,  \cite[Construction 13.4]{CondensedComplex} of deformation to the normal cone applies in our setting, by taking base change of the construction at the level of underlying discrete rings and taking analytifications as derived Tate adic spaces, see Definition \ref{DefinitionAnalytificationFunctor}. The result is a  map $\widetilde{Y} \to Y= \AnSpec \n{A}$  locally of solid finite presentation. More generally, if $X\to Y$ is a Zariski closed immersion of solid finite presentation with $\s{O}_Y$ a perfect $\s{O}_X$-module in the analytic topology, then the deformation to the normal cone glues to a morphism locally of solid finite presentation 
\[
\widetilde{X}= X \times \bb{P}^1\to  \widetilde{Y} \to X \times \bb{P}^1. 
\]
In addition, this construction only requires $X \to Y$ to be a Zariski closed immersion locally in the analytic topology on $Y$, namely, taking $U\subset Y$ such that $X\to U$ is Zariski closed, one can glue $\widetilde{U}$ and $Y \backslash X$ along the complement of the exceptional divisor.

  On the other hand,  for $\n{B}$ a solid smooth $\n{A}$-algebra, by Lemma \ref{LemmaPerfectDiagonal} we know that the multiplication map  $\n{B} \otimes_{\n{A}} \n{B} \to \n{B}$ realizes $\n{B}$ as a perfect $\n{B} \otimes_{\n{A}} \n{B}$-module, allowing the construction of the deformation of the normal cone for any diagonal embedding $\Delta: X \to X\times_{S} X$ for any solid smooth map $X \to S$ .   
 \end{remark}

\begin{proof}[Proof of  Theorem \ref{TheoSerreDuality}]
Let $f: X \to S$ be a solid smooth morphism of derived Tate adic spaces,  consider the diagonal map $\Delta_f : X \to X\times_S X=:Y$, and let $\pi_i: X\times_S X \to X$ denote the projection maps.   By Lemma \ref{LemmaEtaleStandardDiagonal} the map $\Delta_f$ is a locally Zariski closed immersion such that $\s{O}_X$ is a perfect $\s{O}_U$-complex for some open neighbourhood $U\subset X\times_S X$ containing $\Delta_f(X)$.  By smooth base change,  we have a natural isomorphism 
\[
f^! 1_{S}= \Delta_f^* \pi_1^* f^! 1_{S} = \Delta_f^* \pi_2^! 1_{X}.
\]
Therefore, it suffices to prove the statement for the projection $\pi_2: X\times_S X \to X$, or more generally, that when we have a section $s:S \to X$ such that $\s{O}_S$ is a perfect complex in an open subspace $U\subset X$ for the analytic topology, there is a natural equivalence $s^*\Omega^d_{X/S}[d]= s^* f^! 1_{S}$.  Consider the deformation to the normal cone of $s$
\[
\widetilde{f}: \widetilde{X} \to \widetilde{S}=S \times \bb{P}^1
\]
together with the section $\widetilde{s}: \widetilde{S} \to \widetilde{X}$. Over   $\bb{P}^1\backslash \{0\}$ the section $\widetilde{s}$ is isomorphic to the base change of $S\to X$, and the fiber at $0$ is the zero section of the analytification of the  normal cone of $s$ (see Definition \ref{DefinitionAnalytificationFunctor} and Construction \ref{ConstructionAnalytificationVB}). 

The pullback functor $\pi^*: \Mod_{\sol}(S) \to \Mod_{\sol}(\bb{P}^1_S)$ is fully faithful. Indeed, the map $\pi:\bb{P}^1_S \to S$ is weakly cohomologically proper being the base change of $\bb{P}^1_{\bb{Z}} \to \SpecAn \bb{Z}_{\sol}$, and this last being the glueing of $\AnSpec (\bb{Z}[T],\bb{Z})_{\sol}$ and $\AnSpec (\bb{Z}[T^{-1}],\bb{Z})_{\sol}$ along the torus $\AnSpec (\bb{Z}[T^{\pm}],\bb{Z})_{\sol}$.  Thus, by projection formula  and proper base change it suffices to show that $\pi_* 1_{\bb{P}^1_{\bb{Z}}}= \bb{Z}$ which is classical. We make the following claim:

\begin{claim}
 The sheaf $\widetilde{s}^* \widetilde{f}^!   \s{O}_{\widetilde{S}}(d)$ belongs to the essential image of $\pi^*$, where $ \s{O}_{\widetilde{S}}(d)$ is the $d$-th Serre twist of $\widetilde{S}= \bb{P}^1_S$.
\end{claim}

 Suppose this holds true, and let $\iota_{0}:S \to \bb{P}^1_{S}$ and $\iota_{\infty}: S \to \bb{P}^1_{S}$ the $0$ and $\infty$-sections. Then we have natural isomorphisms 
\begin{equation}
\label{eqTwoPullbackCanonicalNormalCone}
s^*f^! 1_{S}\cong \iota_{\infty}^*  \widetilde{s}^* \widetilde{f}^!   \s{O}_{\widetilde{S}}(d) \cong \iota_0^* \widetilde{s}^* \widetilde{f}^! \s{O}_{\widetilde{S}}(d) \cong \bar{\iota}_{0}^* p^! 1_{S}
\end{equation}
where $p: \n{N}_{X/S}^{\an} \to S$ is the (analytic) normal cone of $s$ and $\bar{\iota}_0: S \to \n{N}_{X/S}^{\an}$ is the zero section. 

\begin{proof}[Proof of the Claim] The formation $\widetilde{X}$ is local on $S$ and the section $s:S\to X$, thus, by taking rational covers, we can assume that $S=\AnSpec \n{A}$ and that $X= \AnSpec \n{B}$ is standard solid \'etale over $S$. Write $\n{B}= \n{A}\langle T_1,\ldots, T_{d+c} \rangle_{\sol}/^{\bb{L}}(f_1,\ldots, f_{c})$ with $g=\det (\frac{\partial f_i}{\partial T_j})_{1\leq i,j\leq c}$ invertible. The last $d$ coordinates produce a standard solid \'etale map 
\[
g:X \to \bb{D}^{d}_{S}. 
\]
Thus, the composite $s'= g\circ s:S \to \bb{D}^{d}_{S}$ produces a section. Consider the cartesian square 
\[
\begin{tikzcd}
 S' \ar[r] \ar[d,"s''"] & S \ar[d, "s'"] \\ 
 X\ar[r, "g"] & \bb{D}^{d}_S,
\end{tikzcd}
\]
Then,  the section $S \to X$ produces a retract  $S \xrightarrow{r} S'$ which is necessarily a Zariski closed immersion, and by Proposition \ref{PropMapsEtalemaps} (3), it is actually a closed and open immersion associated to a closed and open subspace of the underlying adic space. Summarizing, we have the diagram
\[
\begin{tikzcd}
 & X \ar[r, "g"] \ar[dl, bend right, "f"'] & \bb{D}_{S}^d \ar[d,"p"'] \\
 S \ar[ur, "s"'] \ar[r,"r"]& S' \ar[r] \ar[u, "s''"'] & S  \ar[u, bend right, "s'"']
\end{tikzcd}
\]
where the square is cartesian and $r$  is an open and closed immersion. Writting $S' = S \sqcup S''$, and  replacing $X$ with a neighbourhood  of $s''$ of the form $X_1\sqcup X_2$ such that $X_1 \cap s''(S')= S$ and $X_2\cap s''(S')= S''$, we can  assume that $S=S'$. 

 We have a diagram of deformations to normal cone 
\[
\begin{tikzcd}
& \bb{P}^1_S  \ar[ld, "\widetilde{s}"'] \ar[rd, "\widetilde{s'}"]& \\  
 \widetilde{X} \ar[rr, "\widetilde{g}"] \ar[d] \ar[ddr, "\widetilde{f}"', bend right =100] & &  \widetilde{\bb{D}^d}_S \ar[d] \ar[ddl, bend left =100, "\widetilde{p}"]\\
\bb{P}^1_X  \ar[rr, "g"] \ar[rd] &   &   \bb{P}^1_{\bb{D}^d_{S}} \ar[ld] \\ 
&\bb{P}^1_S & 
\end{tikzcd}
\]
where the middle square is cartesian. Indeed, this follows from \cite[Proposition 13.3]{CondensedComplex}   (see also \cite[Corollary 3.54]{MaoCrystalline})  since for a surjection  $A \to B$ with kernel $I=[A\to B]$, the $I$-adic filtration $(I^n)_{n\in \bb{N}}$ is compatible with base change along $A$. In particular, $\widetilde{g}$ is solid \'etale and $\widetilde{g}^!= \widetilde{g}^*$. Therefore, we find natural equivalences of functors
\[
\widetilde{s}^* \widetilde{f}^! = \widetilde{s}^* \widetilde{g}^* \widetilde{p}^! = \widetilde{s'}^* \widetilde{p}^!,
\]
this reduces the claim to the case of a disc $X=\bb{D}^d_{S}$. By a change of coordinates, we can even assume that the section $s:S \to \bb{D}^d_{S}$ is given by the zero section. Thus, by base change we can further reduce to the algebraic statement of $\bb{A}^1_{\bb{Z}}= \Spec \bb{Z}[T_1,\ldots, T_d]$ with the zero section, where this is classical and follows by an explicit computation. 
\end{proof}

Let $p:\n{N}_{X/S} \to S$ be the normal cone of the section $s$, and $\bar{\iota}_0: S \to \n{N}_{X/S} \to S$. To end the proof we need to show that there is a natural equivalence 
\[
\overline{\iota}_0^* p^! 1_S = s^* \Omega^d_{X/S}[d]. 
\]
It suffices to show more generally that for a vector bundle $\n{E}$ of rank $d$  over $S$, with analytic geometric realization $q:\underline{\n{E}}^{\an} \to S$ and zero section $\iota: S \to \underline{\n{E}}^{\an}$, there is a natural equivalence 
\[
\iota^* q^! 1_{S} = \bigwedge^d \n{E}^{\vee}[d].
\]
The functor mapping $[\n{E} \to S]$ to $\iota^* q^! 1_{S}[-d]$ defines a map $*/ \GL_d \to \bb{G}_m$ of stacks.  Equivalently, it defines a line bundle  over $*/ \GL_d$ seen as a stack in  the analytic topology of $\Aff^{b}_{R_{\sol}}$. Thus, to identify this object it suffices to compute it for the standard vector bundle  of rank $d$  over $*/\GL_d$, this is proven  independently in Proposition \ref{PropositionDAggerCartier2}.
\end{proof}

 We finish this section by describing the smooth  objects of solid smooth  maps  for the six functors of solid quasi-coherent sheaves.

\begin{prop}
Let $f:X\to Y$ be a solid smooth  morphism of derived  adic spaces over $\n{A}$. Then an object $P\in \Mod_{\sol}(X)$ is $f$-smooth  if and only if it is dualizable.
\end{prop}
\begin{proof}
Suppose that $f$ is solid smooth, and let $P\in \Mod_{\sol}(X)$. Being $f$-smooth is a local property in the analytic topology, we can then assume that both $X$ and $Y$ are affinoids. Consider the diagram 
\[
\begin{tikzcd}
X\ar[r,"\Delta"] & X\times_Y X \ar[r, "\pi_2"] \ar[d, "\pi_1"] & X \ar[d, "f"] \\
 & X \ar[r, "f"] & Y.
\end{tikzcd}
\]
By Proposition \ref{PropositionLCI}, $\Delta$ is cohomologically smooth, then $P=\Delta^* \pi_1^* P $ is $\id_{X}$-smooth which is the same as dualizable. Conversely, let $P$ be  dualizable. We then have that $\n{D}_{f}(P)= \iHom_{X}(P, f^! 1_Y)= f^! 1_Y \otimes P^{\vee}$ is dualizable and that the natural map 
\[
\pi_1^* \n{D}_{f}(P)\otimes \pi_2^* P \to \iHom_{X\times_Y X}(\pi_1^*P, \pi_2^! P)
\]
is an isomorphism, then $P$ is $f$-smooth by Proposition \ref{PropLocalSourceTarget} (1.b).
\end{proof}


\subsection{Formally overconvergent \'etale and smooth maps}
\label{SubsectionFormallyOverconvergentEtaleSmooth}

In this final section we introduce two new deformation properties that will  play a fundamental role in the definition of the analytic de Rham stack. 

\begin{definition}
Let $\n{A} \in \Aff^{b}_{R_{\sol}}$, a \textit{$\dagger$-nilpotent ideal} of $\n{A}$ is a full sub $\n{A}$-module $I\subset \n{A}$  contained in $\Nil^{\dagger}(\n{A})$. 
\end{definition}

\begin{definition}
\label{DefinitionDaggerFormalSmoothEtale}
    Let $T:\n{F} \to \n{F}'$ be a natural transformation of functors $\n{F},\n{F}' : \Aff^{b}_{R_{\sol}} \to \Ani$. We say that $T$ is  \textit{$\dagger$-formally smooth} (resp.  \textit{$\dagger$-formally \'etale}) if it is formally smooth (resp. formally \'etale) and for all $\n{A}\in  \Aff^{b}_{R_{\sol}}$, and all $\dagger$-nilpotent ideal  $I$ of $\n{A}$, the natural map of anima
    \[
    \n{F}(\n{A}) \to \n{F}(\n{A}/I) \times_{\n{F}'(\n{A}/I)} \n{F}'(\n{A})
    \]
    is surjective (resp. an equivalence). 
\end{definition}

\begin{remark}
 Since the underlying ring of $\n{A}/I$ sits in degree $0$, to check that a formally smooth (resp. \'etale) functor $T:\n{F} \to \n{F}'$ is  $\dagger$-formally smooth (resp. \'etale) it is enough to take $\n{A}$ an static bounded affinoid ring. 
\end{remark}

\begin{prop}
\begin{enumerate}
\item A composition of  $\dagger$-formally  smooth morphisms is   $\dagger$-formally  smooth.

\item If $\{\n{F}_i\to \n{F}\}_{i\in I}$ is a  cofiltered diagram of   $\dagger$-formally smooth functors  with each arrow $\n{F}_i\to \n{F}_{j}$ formally \'etale, then $\n{F}'=\varprojlim_{i}\n{F}_i \to \n{F}$ is  $\dagger$-formally  smooth.

\item A pullback of   $\dagger$-formally smooth maps is  $\dagger
$-formally smooth. 
\end{enumerate}

Similar statements hold for $\dagger$-formally  \'etale. 

\end{prop}
\begin{proof}
    Parts (1) and (3) are proved in the same way as for formally smooth maps. For part (2), note that since the transition maps of the cofiltered limit are formally \'etale, the cotangent complex $\bb{L}_{\varprojlim_{i} \n{F}_i/ \n{F} }$ is still the dual of a connective perfect  complex. It is also clear that the map $\varprojlim_i\n{F}_i\to \n{F}$ is nilcomplete and infinitesimally cohesive since limits commute with limits.  The $\dagger$-formally smooth condition also passes through the limit. 
\end{proof}

The main reason to define these overconvergent deformation properties is that they hold for solid smooth and \'etale maps.

\begin{prop}
\label{PropFormallyInftSmoothEtale}
    A solid \'etale morphism of bounded affinoid rings is  $\dagger$-formally \'etale.  A  solid smooth morphism of affinoid rings is   $\dagger$-formal smooth locally in the analytic topology. Moreover, a standard solid smooth morphism of bounded affinoid rings is $\dagger$-formally smooth.  
\end{prop}
\begin{proof}
We can assume without loss of generality that $\n{A}\to \n{B}$ is standard solid smooth or standard solid \'etale, namely, rational localizations are also written as composite of standard solid \'etale maps. First, let us show that $\n{A}\to \n{A}\langle T \rangle_{\sol}$ is  $\dagger$-formally smooth, it suffices to prove that $R_{\sol}\to R\langle T \rangle_{\sol}$ is  $\dagger$-formally smooth, but this follows from Proposition \ref{PropLiftingOverconvergentAlgebrasLift} (1)  and the fact that $R[T]$ is a projective animated $R$-algebra. Indeed, if $I \subset \n{A}$ is a $\dagger$-nilpotent ideal, then $(\n{A}/I)^{\dagger-\red}= \n{A}^{\dagger-\red}$ and a map $\bb{Z}[T] \to \n{A}$ extends to $\bb{Z}[T]_{\sol}$ if and only if it does for $\n{A}/I$.

We are  left to show that a standard solid \'etale morphism of bounded algebras is  $\dagger$-formally \'etale.  By writing $\pi_0(\n{A})$ as a sifted colimit of quotients of algebras of the form $R\langle \underline{X} \rangle_{\sol} \langle \bb{N}[S] \rangle$, we can assume that $\n{A}= R\langle X_1,\ldots, X_s\rangle_{\sol}\langle \bb{N}[S] \rangle$ for $S$ a profinite set and a finite set of variables $X_i$, and that $\n{B}= \n{A}\langle T_1,\ldots, T_d \rangle/^{\bb{L}}(f_1,\ldots, f_{d})$ with $a=\det(\frac{\partial f_i}{ \partial T_j})$ a unit. We can also assume that all the $f_i$ are of norm $\leq 1$ and that $ |\pi^{k}| \leq |a| \leq 1 $ for some $k\geq 0$.  By Lemma  \ref{LemmaEtaleStandardDiagonal} the map $\SpecAn \n{B}\to \SpecAn \n{A}$ is $0$-truncated, then we only need to prove the existence and uniqueness of lifts at the level of points.  Let $\n{D}$ be a bounded affinoid  ring and $I\subset  \n{D}$ a $\dagger$-nilpotent ideal, consider a solid commutative diagram 
\[
\begin{tikzcd}
\n{A} \ar[d] \ar[r] & \n{B} \ar[d] \ar[dl, dashed]  \\ 
\n{D} \ar[r] &  \n{D}/I.
\end{tikzcd}
\]
 We want to see that there is a unique dashed arrow $\n{B} \to \n{D}$ making the  diagram commutative.  By Proposition \ref{PropLiftingOverconvergentAlgebrasLift} we can find a lift $F:\n{A}\langle  T_1,\ldots, T_d\rangle_{\sol} \to \n{D}$ such that $f_i(\underline{T}) \in I \subset \Nil^{\dagger}(\n{D})$. Therefore, the map $F$ extends to a map 
 \[
 \n{A}\langle T_1,\ldots, T_n\rangle_{\sol} \{S_1,\ldots, S_n\}^{\dagger}/^{\bb{L}}(f_i-S_i)\to \n{D}.
 \]
By Lemma \ref{LemmaDaggerCompletionPowerSeries}  down below,  we  have an equivalence  of $\n{A} \{S_1,\ldots, S_d\}$-algebras
\[
\n{A}\langle T_1,\ldots, T_n\rangle_{\sol} \{S_1,\ldots, S_n\}^{\dagger}/^{\bb{L}}(f_i-S_i) \cong \n{B}\{S_1,\ldots, S_d\}
\] 
which shows that there is a lift $\n{B} \to \n{D}$ over $\n{D}/I$. Suppose we have two lifts $f_1,f_2:\n{B} \to \n{D}/I$. Then they extend to a map $\n{B}\otimes_{\n{A}} \n{B} \to \n{D}$, write 
\[
f:\n{B}\otimes_{\n{A}} \n{B} =\n{A}\langle T_1,\ldots, T_d, S_1,\ldots, S_d \rangle_{\sol}/(f_i(T),f_i(S)),
\]
then the differences $T_i-S_i$ are sent to $I$, and the map $f$ factors through  the overconvergent diagonal 
\[
\n{B}\otimes_{\n{A}} \n{B}\{T_i-S_i\}^{\dagger} \to \n{D},
\]
but the proof of Lemma \ref{LemmaEtaleStandardDiagonal} implies  that $\n{B}\otimes_{\n{A}} \n{B}\{T_i-S_i\}^{\dagger}=\n{B}$ proving the uniqueness. 
\end{proof}

The following lemma was used in the previous proposition.

 \begin{lemma}
 \label{LemmaDaggerCompletionPowerSeries}
Let $\n{A}$ be a bounded affinoid ring, $\n{D}= \n{A} \langle T_1,\ldots, T_d\rangle_{\sol}$ a solid Tate algebra over $\n{A}$ in $d$-variables, and $\n{B}= \n{D}/^{\bb{L}}(f_{1},\ldots, f_d)$ a standard solid \'etale algebra over $\n{A}$. Let $\n{C}= \n{D}\{f_i:i=1,\ldots,d\}^{\dagger}$ denote the idempotent algebra associated to the closed subspace $\Spa \n{B}\subset \Spa \n{D}$, i,e, the base change 
\[
\n{C} =\n{D}\otimes_{R_{\sol}[S_1,\ldots, S_d]} R\{S_1,\ldots, S_d\}^{\dagger}
\] 
mapping $S_d\mapsto f_d$. Then there is an isomorphism of $\n{A}\{S_1,\ldots, S_d\}$-algebras 
\[
\n{B}\{S_1,\ldots, S_d\} \cong \n{C}.
\]
\end{lemma}
 \begin{proof}
By writting $\pi_0(\n{A})$ as a sifted colimit of quotients of algebras of the form $R\langle  \underline{X} \rangle_{\sol} \langle \bb{N}[K] \rangle$ for finite set of variables $\underline{X}$ and profinite sets $K$, we can assume without loss of generality that $\n{A}=R\langle  \underline{X} \rangle_{\sol} \langle \bb{N}[K] \rangle$, that the $f_i$ have norm $\leq 1$ and that $g= \det( \frac{\partial f_i}{\partial T_j})_{1\leq i,j\leq d}$ satisfies $ |\pi^k| \leq |g|\leq 1$ for some $k\geq 0$. Let $\n{A}^0= R^+\langle \underline{X}\rangle_{\sol} \langle  \bb{N}[K] \rangle$, $\n{D}^0= \n{A}^0 \langle T_1,\ldots, T_d \rangle_{\sol}$ and  $\n{B}^0=\n{D}^0/^{\bb{L}}(f_1,\ldots, f_d)$. For $n \geq 0$ let $\n{C}_n^
0= \n{D}^0 \langle \frac{S_1}{\pi^n}, \ldots, \frac{S_d}{\pi^n} \rangle/^{\bb{L}}(f_i(T)-S_i)$ and set $\n{D}_n= \n{D}_n^0 [\frac{1}{\pi}]$. The explicit Koszul resolution of the cotangent complex $\bb{L}_{\n{B}^0/\n{A}^0}$ shows that multiplication by $\pi^{k}$ is homotopic to $0$. On the other hand, for $n\geq 2k+1$ the map 
\[
\n{D}^0 \to \n{C}^0_n/^{\bb{L}} \pi^{n}
\]
factors trough $\n{B}^0 \to \n{C}^0_{n}/^{\bb{L}}\pi^n$, namely, $S_i= \pi^{n} \frac{S_i}{\pi^n}$ vanishes in the quotient.  By lemma \ref{LemmaHenselQuantitative} down below we have a lift $\n{B}^0 \to  \n{C}^0_{n}$, and by the uniqueness of lifts shown at the end of Proposition \ref{PropFormallyInftSmoothEtale} (which is independent of the existence of a lift),  we have a natural lift  in generic fibers $\n{B} \to \n{C}$ independent of $n$. Now, let us fix some $n\geq 2k+1$ and a lift $\n{B}^0 \to \n{C}^0_n$, and for all $m\geq n$ take $\n{B}^0\to \n{C}^0_m$ to be the composite of $\n{B}^0\to \n{C}^0_n \to \n{C}^0_m$.  We can extend these maps to  morphisms $\n{B}^0\langle \frac{S_1}{\pi^m},\ldots, \frac{S_d}{\pi^m} \rangle \to \n{C}^0_n$  of $\n{A}^0 \langle \frac{S_1}{\pi^m}, \ldots \frac{S_d}{\pi^m} \rangle$-algebras. For $m'>m\geq n$  these algebras factor through 
\[
\n{B}^0\langle \frac{S_i}{\pi^m} \rangle \to \n{B}^0[[ \frac{S_i}{\pi^m}]] \to \n{B}^0 \langle \frac{S_i}{\pi^{m'}} \rangle
\]
and 
\[
\n{C}^0_m \to \n{D}^0[[\frac{S_i}{\pi^m}]]/^{\bb{L}}(f_i(T)-S_i) \to \n{C}^0_{m'}. 
\]
Then, we have a map of $\n{A}^0[[\frac{S_i}{\pi^{m}}]]$-algebras
\begin{equation}
\label{eqMapAlgebras1}
\n{B}^0[[\frac{S_i}{\pi^{m}}]] \to \n{D}^0[[\frac{S_i}{\pi^m}]]/^{\bb{L}}(f_i(T)-S_i).
\end{equation}
Both  terms in  \eqref{eqMapAlgebras1}  are $I=( \frac{S_i}{\pi^m})$-complete, and their reduction modulo $I$ is an equivalence, this implies that \eqref{eqMapAlgebras1} is an equivalence. Taking generic fibers and colimits as $m\to \infty$ one gets the lemma.  
 \end{proof}

The following lemma is substracted from the proof of \cite[Corollary III.2.2]{ScholzeTorsion2015}.

\begin{lemma}[Quantitative Hensel's Lemma]
\label{LemmaHenselQuantitative}
Let $A \to B$ be a morphism of $\pi$-complete animated $R^+_{\sol}$-algebras such that  there is some $k\geq 1$ such that  multiplication by $\pi^k$ on $\bb{L}_{B/A}$ is homotopic to $0$. Let $C$ be a $\pi$-complete animated $A$-algebra, and suppose that we have a solid commutative diagram 
\[
\begin{tikzcd}
A \ar[r] \ar[d] & B \ar[d] \ar[ld, dashed] \\ 
C \ar[r] & C/^{\bb{L}} \pi^{2k+1}
\end{tikzcd}
\]
Then there is a dashed arrow as above making the diagram commute. 
\end{lemma}
\begin{proof}
It suffices to construct a sequence of compatible arrows $B \to \n{C}/^{\bb{L}}\pi^{n}$ for all $n\geq 0$. Suppose that we have the lift for $n\geq 2k+1$, we will construct a lift for $2(n-k)\geq n+1$ over $n-k$. Consider the algebra 
\[
C'= C/^{\bb{L}} \pi^{2(n-k)} \times_{C/^{\bb{L}} \pi^{n-k}  } C/^{\bb{L}} \pi^{n},
\]
then the fiber $C'\to  C/^{\bb{L}} \pi^n$  is equivalent to $C/^{\bb{L}} \pi^{n-k}$ under the map $\iota=(\pi^{n-k},0):C/^{\bb{L}} \pi^{n-k} \to C'$. Moreover, we have a commutative diagram 
\[
\begin{tikzcd}
C/^{\bb{L}} \pi^{2n-k}  \ar[r, "\Delta"]& C' \\ 
C/^{\bb{L}} \pi^{n-k} \ar[r, "\pi^k"] \ar[u, "\pi^{n}"] & C/^{\bb{L}} \pi^{n-k} \ar[u, "\iota"]
\end{tikzcd}
\]
where $\Delta$ is the diagonal map.  Let us write $J= C/^{\bb{L}} \pi^{n-k}$ for the square zero ideals of  the algebras  $C'$ and $C/^{\bb{L}} \pi^{2n-k}$ over $C/^{\bb{L}} \pi^n$. Then, the map $\Delta$ induces a morphism
\[
\iHom_{B}(\bb{L}_{B/A}, J ) \xrightarrow{\pi^k} \iHom_{B}(\bb{L}_{B/A}, J),
\]
which is homotopic to zero as the multiplication by $\pi^k$ is homotopic to zero on $\bb{L}_{B/A}$ by hypothesis. By deformation theory, we deduce that the obstruction to lift $B$  from $C/^{\bb{L}} \pi^n$ to $C'$ vanishes, but lifting from $C/^{\bb{L}}\pi^n$ to $C'$ is equivalent to lifting from $C/^{\bb{L}} \pi^{n-k}$ to $C/^{\bb{L}} \pi^{2(n-k)}$, which proves the lemma. 
\end{proof}


\section{Cartier duality for vector bundles}
\label{SectionCartierDualityVectorBundles}
After all the preliminaries in derived algebraic geometry, we are finally in shape of applying the theory to more interesting objects. In this section we study Cartier duality for different incarnations of vector bundles, following the spirit of \cite{laumon1996transformation}, but using the language of six-functor formalisms and the Lu-Zheng category. In the next sections we shall apply these results to study different incarnations of the de Rham stack.

\subsection{Vector bundles and torsors}
\label{SubsectionVectorBundles}

  First, let us briefly introduce the category of vector bundles on solid $\s{D}$-stacks. 

\begin{definition}
\label{DefinitionBun}
Let $\n{C}=\Sh_{\s{D}}(\Aff_{\bb{Z}_{\sol}})$ be the category of solid $\s{D}$-stacks.

\begin{enumerate}
\item   Let $X\in \n{C}$ be a solid $\s{D}$-stack, a vector bundle of rank $d$ over $X$ is a quasi-coherent sheaf $\s{F}\in \Mod_{\sol}(X)$ that is free of rank $d$ locally in the $\s{D}$-topology of $X$. 

\item  Let $\n{D} \to \n{C}$ be the co-cartesian fibration associated to the functor $\Mod_{\sol}: \n{C}^{\op} \to \CAlg(\n{P}r^{L,\ex})$. We let $\BUN_{d,\n{C}}\subset \n{D}$ be the subcategory whose objects are pairs $(X,\s{F})$ with $\s{F}\in \Mod_{\sol}(X)$ a vector bundle of rank $d$, and morphisms $(Y,\s{G})\to (X,\s{F})$ given by the space of connected components $(f:X\to Y, f^* \n{G} \to \n{F})$ such that  $f^* \n{G} \to \n{F}$ is an equivalence. 
\end{enumerate}
\end{definition}

Let $\GL_d$ be the linear algebraic group of $d\times d$ invertible matrices over $\bb{Z}$, and let $*/\GL_d$ be its classifying stack.  By definition, $*/\GL_d$ is the object representing the moduli problem of $\GL_d$-torsors on  $\s{D}$-stacks.  Over $*/\GL_d$ we have a vector bundle $\St$ associated to the standard left representation of $\GL_d$ on $\bb{Z}^{d}$. Let $\n{C}_{/[*/\GL_d]}$ be the slice category of $\s{D}$-stacks over $*/\GL_d$, the vector bundle $\St$  induces a functor 
\[
F: (\n{C}_{/[*/\GL_d]})^{\op} \to \BUN_{d,\n{C}}.
\]
by taking pullbacks.
\begin{prop}
\label{PropositionEquivalenceVBTorsors}
The functor $F$ is an equivalence of categories over $\n{C}^{\op}$. 
\end{prop}
\begin{proof}
 Both $(\n{C}_{/[*/\GL_d]})^{\op}$ and $\BUN_{d,\n{C}}$ are left fibrations over $\n{C}^{\op}$, by \cite[Proposition 3.3.1.5]{HigherTopos} it suffices to prove that the fibers over $\n{C}^{\op}$ are equivalent. Let $X\in \n{C}$, we want to show that  the natural functor 
 \begin{equation}
 \label{eqMapGeometricRealization}
 F_X: \Map_{X}(X, [X/\GL_d])^{\op} \to  \BUN_{d,X}
 \end{equation}
 from maps $f:X\to X/\GL_d$ to rank $d$-vector bundles over $X$ is an equivalence of anima. To see that $F_X$ is essentially surjective, note that for $\s{F}$ a vector bundle over $X$, the stack $\underline{\mathrm{Isom}}_X (\s{O}_X^d,\s{F})$ of isomorphisms in the $\s{D}$-topology is a $\GL_d|_{X}$-torsor over $X$, which is defined by some map $f:X\to [X/\GL_d]$ such that $f^*\St = \s{F}$. To show that $F$ is fully faithful, notice that both terms in \eqref{eqMapGeometricRealization} are Kan complexes, so it suffices to show that for a map $f:X \to X/\GL_d$, the anima of automorphisms of $f$ is equivalent to the anima of automorphisms of $\s{F}=f^*\St$. It suffices to show that the natural map of stacks 
 \begin{equation}
 \label{eqEquivalenceStacksAuto}
 \underline{\mathrm{Aut}}_{X/\GL_d}(X)\to \underline{\mathrm{Aut}}_X(\s{F})
 \end{equation}
 is an equivalence. Since $f:X\to [X/\GL_d]$ is an epimorphism of stacks, it satsifies universal $*$-descent and it suffices to show that \eqref{eqEquivalenceStacksAuto} is an equivalence after pullback along $f$. Let $Y= X\times_{[X/\GL_d]} X$, then $g:Y\to X$ is a $\GL_d$-torsor and we the multiplication map gives rise an equivalence  $\GL_d\times Y\xrightarrow{\sim} Y\times_X Y$, one deduces that 
 \[
 \underline{\mathrm{Aut}}_{X/\GL_d}(X)|_{Y}=\GL_d \times Y. 
 \]
 Similarly, the pullback of $\s{F}$ to $Y$ is naturally isomorphic to $\s{O}_Y^d$ and $\underline{\mathrm{Aut}}_X(\s{F})|_{Y}=\underline{\mathrm{Aut}}_Y(\s{O}_Y^d)=\GL_d\times Y$. It is clear that the restriction of \eqref{eqEquivalenceStacksAuto} to $Y$ is identified with the identity of $\GL_d\times Y$. 
\end{proof}

\begin{remark}
Definition \ref{DefinitionBun} and Proposition \ref{PropositionEquivalenceVBTorsors} are not special for the $\s{D}$-topology of the category of solid $\s{D}$-stacks. The same can be done for a general Grothendieck topology in a full subcategory of analytic rings stable under pullbacks. 
\end{remark}

\begin{definition}
For $\s{F}$ a vector bundle over $X$, we let $\bb{V}(\s{F})$ denote its geometrization. Explicitly, let $\underline{\St}$ be the analytic spectrum of $\Sym \St^{\vee}$, it is naturally endowed with the standard action of $\GL_d$ and defines a vector bundle $\bb{V}(\St)$ over $*/\GL_d$. Let $f:X\to */\GL_d$ be the map defining $\s{F}$ via Proposition \ref{PropositionEquivalenceVBTorsors}, then $\bb{V}(\s{F})=f^* \bb{V}(\St)$. Note that $\bb{V}(\s{F})$ is the relative analytic spectrum of $\Sym_{\s{O}_X} \s{F}^{\vee}$, which is an analytic ring locally in the $\s{D}$-topology of $X$. 
\end{definition}

\subsection{Algebraic Cartier duality for vector bundles}
\label{SubsectionCartierDualityAlgebra}

 Now that we have related the category of vector bundles of rank $d$ and the slice category of $*/\GL_d$, we can state our first Cartier duality that is nothing but the algebraic Cartier duality of \cite[Proposition 2.2.13]{BhattGauges}.  In order to simplify the theory, and since our main application will be for  rigid spaces over $\bb{Q}_p$, we will focus in characteristic  $0$, though some statements will be proven in general.

\begin{definition}  
\label{DefinitionConstructionVB}
 Let $X$ be a solid $\s{D}$-stack over $\bb{Z}$,  $\s{F}$ a vector bundle of rank $d$ over $X$ and  $\bb{V}(\s{F})$  its geometrization seen as an abelian group object over $X$.  For $n\in \bb{N}$, let $\bb{V}(\s{F})_n$ be the relative analytic  spectrum of $\Sym^{\leq n}_{X} \s{F}^{\vee} = \bigoplus_{k=0}^n \Sym^k_{X} \s{F}^{\vee}$.  The formal completion of $\bb{V}(\s{F})$ at zero is defined as the abelian group stack $\widehat{\bb{V}(\s{F})}= \varinjlim_n \bb{V}(\s{F})_n$, we let $\widehat{\Sym}_X(\s{F}^{\vee})= \varprojlim_n \Sym^{\leq n}_X(\s{F}^{\vee})$ denote the global sections of $\widehat{\bb{V}(\s{F})}$. 
\end{definition}

The following lemma will be useful to show cohomologically smoothness of classifying stacks. 
\begin{lemma}
\label{KeyLemmaSmoothRetraction}
Let $(\n{C},E)$ be a geometric set up and $\s{D}$ a six functor formalism on $(\n{C},E)$ taking values in stable $\infty$-categories.  Let $f:Y\to X$ be a  map in $E$ with $f^*$ conservative, and  let $g:X \to Y $ be a retraction of $f$ in $E$.

 \begin{enumerate}
 
 \item  Suppose we are given with the following data:
\begin{itemize}

\item[(i)] An object $\n{L}\in \s{D}(X)$.

\item[(ii)] A  map $s:f_! 1_{Y} \to \n{L}$. 

\item[(iii)] A retraction $ 1_{Y} \cong g_! f_! 1_{Y} \xrightarrow{g_! s} g_! \n{L} \xrightarrow{\eta} 1_{Y}$. 
\end{itemize} 
Then $1_{X}$ is $g$-smooth and there is a natural identification $g^! 1_{Y} \cong \n{L}$. 

\item   Suppose we are given with the following data:
\begin{itemize}

\item[(i)] An object $\n{L}\in \s{D}(X)$.

\item[(ii)] A  map $s:\n{L} \to f_! 1_{Y}$. 

\item[(iii)] A section $ 1_{Y} \xrightarrow{\mu} g_! \n{L} \xrightarrow{g_! s} g_! f_! 1_{Y} \cong 1_{Y}$. 
\end{itemize} 
Then $1_{X}$ is $g$-proper and there is a  natural identification $\n{P}_g( 1_{Y}) \cong \n{L}$.

\end{enumerate}

\end{lemma}
\begin{proof}
We only prove part (1), part (2) follows by taking the dual six functor formalism $\s{D}^{\op}$, see \cite[Remark 6.5]{SixFunctorsScholze}.

 Let $\Delta: X \to X\times_Y X$ be the diagonal map. We need to define a cycle morphism $\mu: \Delta_! 1_{X} \to  \pi_{2}^* \n{L}$ such that the following compositions are the identity
\begin{equation}
\label{eqTraceeMapLemma}
1_{X} \cong  \pi_{1,!} \Delta_{!} 1_{X} \xrightarrow{\pi_{1,!} \mu} \pi_{1,!} \pi_{2}^* \n{L} \cong  g^* g_! \n{L} \xrightarrow{g^* \eta} 1_{X}.
\end{equation}
\begin{equation}
\label{eqCycleMapLemma}
\n{L}\cong  \pi_{2,!}( \pi_{1}^*\n{L}\otimes \Delta_{!} 1_{X} ) \xrightarrow{\pi_{2,!}(\pi_{1}^*\n{L}\otimes \mu)}  \pi_{2,!}(\pi_{1}^*\n{L} \otimes \pi_2^* \n{L}) \cong  \pi_{2,!} \pi_1^* \n{L} \otimes \n{L} \cong g^*g_!\n{L} \otimes  \n{L} \xrightarrow{g^* \eta \otimes \n{L}} \n{L}.
\end{equation}
By Lemma \cite[Lemma 5.11]{SixFunctorsScholze}, after modifying $\eta$, it suffices that the composite are equivalences.  We have  the following  commutative diagram with cartesian squares
\begin{equation}
\label{eqCartesianSquares}
\begin{tikzcd}
Y  \ar[r,"f"] \ar[d,"f"]& X \ar[d, "(fg{,}\id)"] \ar[r, "g"] & Y  \ar[d, "f"]\\ 
X \ar[r, "\Delta"]  \ar[d, "g"] & X\times_{Y} X \ar[d, "\pi_2"]  \ar[r, "\pi_1"] &  X  \ar[d, "g"]\\ 
Y \ar[r, "f"] & X  \ar[r, "g"]& Y .
\end{tikzcd}
\end{equation}
In particular, we have that $\pi_{2}^*f_! \cong \Delta_!g^*$, we define $\mu:\Delta_! 1_{X} \cong \pi_{2}^*f_! 1_{Y} \to \pi_2^* \n{L}$ to be $\mu= \pi_2^* s$. 

\textbf{The composite \eqref{eqTraceeMapLemma} is the identity}.
Since $g^*g_!\cong \pi_{1,!} \pi_{2}^*$, the composite \eqref{eqTraceeMapLemma} is obtained by applying $g^*$ to the retraction 
\[
1_{Y}\cong g_! f_! 1_{Y} \xrightarrow{g_! s} g_! \n{L} \xrightarrow{\eta}1_{Y}.
\]

\textbf{The composite \eqref{eqCycleMapLemma} is the identity}. Since the pullback along $f:Y \to X$ is conservative, it suffices to show that \eqref{eqCycleMapLemma} is an equivalence  after taking $f^*$. By proper base change we have that $ f^*\pi_{2,!}\cong  g_! \Delta^*$. On the other hand, the equivalence $\n{L}\cong \pi_{2,!}(\pi_1^* \n{L}\otimes \Delta_! 1_{X})$ arises by the $\pi_{2,!}$ of the equivalence 
\[
\pi_{1}^*\n{L}\otimes \Delta_{!}1_{X}\cong \Delta_! \n{L} \cong \Delta_{!} \otimes \pi_{2}^* \n{L},
\]
taking pullbacks along $\Delta$ we get the natural equivalence 
\[
\n{L} \otimes \Delta^* \Delta_{!} 1_{X} \cong \Delta^* \Delta_{!} 1_{X} \otimes \n{L} 
\]
given by the braiding isomorphism.  Consider the (not necessarily commutative) diagram
\begin{equation}
\label{eqDiagramProof1}
\begin{tikzcd}
\pi_{1}^* \n{L}\otimes  \Delta_! 1_X \ar[r,"\sim"] \ar[rd, "\pi_1^*\n{L} \otimes  \pi_2^* s"'] &  \Delta_{!} \n{L} &  \Delta_{!} 1_{X} \otimes \pi_2^* \n{L}  \ar[l,"\sim"'] \ar[ld, "\pi_1^* s \otimes \pi_2^* \n{L}"]\\ 
& \pi_1^* \n{L} \otimes \pi_2^*\n{L} &   
\end{tikzcd}
\end{equation}
  Applying $\Delta^*$ to \eqref{eqDiagramProof1}, and using that $\Delta_!g^* \cong \pi_2^*f_!$, we get the commutative  diagram 
\[
\begin{tikzcd}
 \n{L}\otimes f_! 1_Y \ar[r,"\sim"] \ar[rd, "\n{L} \otimes   s"'] &  \Delta^* \Delta_{!} \n{L} & f_! 1_{Y} \otimes \n{L}  \ar[l,"\sim"'] \ar[ld, " s \otimes  \n{L}"]\\ 
&  \n{L} \otimes \n{L} & 
\end{tikzcd}
\]
where the composite of the horizontal maps is the braiding isomorphism. In a similar way, using that $\pi_{2,!}\pi_{1}^*\cong g^*g_!$ and $\pi_{1,!}\pi_{2}^* \cong g^*g_!$,    one deduces that the equivalence 
\[
f^*\n{L} \otimes g_! \n{L} \cong f^*(\n{L}\otimes  g^* g_! \n{L}) \cong f^* \pi_{1,!}( \pi_{1}^*\n{L}\otimes \pi_{2}^* \n{L}) \cong f^* \pi_{2,!}(\pi_{1}^*\n{L}\otimes \pi_{2}^* \n{L}) \cong f^*(g^*g_! \n{L} \otimes  \n{L}) \cong g_!\n{L} \otimes f^*\n{L}
\]
is also the braiding isomorphism. On the other hand, the $f$-pullback of the map $g^* \eta \otimes \n{L}$ is nothing but the map 
\[
g_! \n{L}\otimes f^*\n{L} \xrightarrow{\eta \otimes f^* \n{L}} f^* \n{L}.
\]
Putting all together, the $f$-pullback of the map \eqref{eqCycleMapLemma} becomes 
\begin{equation}
\label{eqCycleMapcomp2}
f^*\n{L} \xrightarrow{\sim} g_!( f_! 1_{Y} \otimes \n{L} ) \xrightarrow{g_!(s\otimes \n{L})} g_! (\n{L}\otimes \n{L}) \cong g_!(\n{L}\otimes g^*f^*\n{L})\cong  g_! \n{L} \otimes f^* \n{L} \xrightarrow{\eta \otimes f^* \n{L}} f^*\n{L},
\end{equation}
but the map $g_!(f_! 1_{Y} \otimes \n{L}) \xrightarrow{g_!(s\otimes \n{L})} g_! (\n{L}\otimes \n{L}) \cong g_!\n{L} \otimes f^* \n{L}$ is equal to the composite
\[
g_!(f_! 1_{Y}\otimes \n{L})\cong g_!f_! 1_{Y} \otimes f^*\n{L} \xrightarrow{g_! s \otimes f^*\n{L}} g_! \n{L}\otimes f^* \n{L},
\]
where the equivalence $g_!(f_! 1_{Y}\otimes \n{L})\cong g_!f_! 1_{Y} \otimes f^*\n{L}$ arises from the natural isomorphism $f^*\pi_{2,!}\cong g_! \Delta^*$ applied to  $ \pi_2^*f_! 1_{Y} \otimes \pi_{2}^*\n{L}$. This shows that the composite \eqref{eqCycleMapcomp2} is an equivalence, proving what we wanted. 
\end{proof}

Before we state the algebraic Cartier duality theorem we need to show some cohomological properties of vector bundles.  We start with a key lemma that is the core of the computations. 

\begin{lemma}
\label{LemmaDeRhamKoszulAlgebraic}
Let $X$ be a solid $\s{D}$-stack over $\bb{Q}$ and let $\s{F}$ be a vector bundle of rank $d$ over $X$. 

\begin{enumerate}
\item  There is a natural de Rham resolution of $\s{O}_X$ as $\Sym^{\bullet}_X (\s{F}^{\vee})$-comodule  given by a complete and decrasing filtration
\[
\s{O}_X \to \Sym^{\bullet}_X (\s{F}^{\vee}) \xrightarrow{d} \Sym^{\bullet}_X (\s{F}^{\vee})  \otimes  \s{F}^{\vee} \xrightarrow{d} \cdots \xrightarrow{d} \Sym^{\bullet}_X (\s{F})\otimes   \bigwedge^d \s{F}^{\vee}.
\]

\item There is a natural Koszul resolution of $\s{O}_X$ as $\Sym^{\bullet}_X (\s{F})$-module  given by a complete and increasing filtration 
\[
 \Sym^{\bullet}_X (\s{F}) \otimes \bigwedge^d \s{F} \cdots \to \Sym^{\bullet}_X (\s{F}) \otimes \s{F}\to\Sym^{\bullet}_X (\s{F}) \to \s{O}_X ,
\]
whose dual is the de Rham complex for $\widehat{\Sym}_X(\St^{\vee})$. 
\end{enumerate}
\end{lemma}
\begin{proof}
By base change, it suffices to deal with the universal case $X=*/ \GL_d$ and $\s{F}=\St$ the standard vector bundle. We have  $\GL_d$-equivariant de Rham and Koszul resolutions for both $\Sym^{\bullet}_X(\St^{\vee})$ and $\widehat{\Sym}_X(\St)$: 
\begin{equation}
\label{eqResolutionsSym}
\begin{aligned}
0 \to \bb{Q} \to \Sym^{\bullet}_X(\St^{\vee})\xrightarrow{d}  \Sym^{\bullet}_X(\St^{\vee}) \otimes \St^{\vee} \xrightarrow{d} \cdots \xrightarrow{d}  \Sym^{\bullet}_X(\St^{\vee}) \otimes  \bigwedge^d \St^{\vee}  \to 0 \\ 
0 \to \Sym^{\bullet}_X(\St^{\vee}) \otimes \bigwedge^d \St^{\vee} \to \cdots \to \cdots \to \Sym^{\bullet}_X (\St^{\vee}) \otimes \St^{\vee} \to \Sym^{\bullet}_X (\St^{\vee}) \to \bb{Q} \to 0
\end{aligned}
\end{equation}
and 
\begin{equation}
\label{eqResolutionshatSym}
\begin{aligned}
0 \to \widehat{\Sym}_X(\St) \otimes \bigwedge^d \St \to \cdots \to \cdots \to\widehat{\Sym}_X(\St) \otimes \St \to \widehat{\Sym}_X(\St) \to \bb{Q} \to 0\\
0 \to \bb{Q} \to \widehat{\Sym}_X(\St) \xrightarrow{d}  \widehat{\Sym}_X(\St)  \otimes \St \xrightarrow{d} \cdots \xrightarrow{d}  \widehat{\Sym}_X(\St) \otimes  \bigwedge^d \St   \to 0 
\end{aligned}
\end{equation}
where we have identified the differentials $de\cong e$ for $e\in \St^{\vee}$ (resp. for $\St$). The resolutions of \eqref{eqResolutionsSym} are the duals of those in \eqref{eqResolutionshatSym}, this prove the lemma since the de Rham resolution of a vector bundle is a complex of comodules while the de Koszul resolution is a complex of modules. 
\end{proof}

\begin{prop}
\label{PropAlgCartier1}
Let $X$ be a solid $\s{D}$-stack  over $\bb{Z}_{\sol}$ and let $\s{F}$ be a vector bundle of rank $d$ over $X$. 

\begin{enumerate}

\item The map $\bb{V}(\s{F}) \to X$  is weakly cohomologically proper. 

\item  The map $f:\widehat{\bb{V}(\s{F})} \to X$ is cohomologically smooth and  there is a natural isomorphism $f^! 1_{X}= \Omega^{d}_{\widehat{\bb{V}(\s{F})}/X}[d]= f^*  \bigwedge^d \s{F}^{\vee}[d]$. If in addition $X$ is defined over $\bb{Q}$  then  $f_! 1_{\widehat{\bb{V}(\s{F})}} =  \bigwedge^d \s{F}\otimes (\Sym^{\bullet}_X \s{F})[-d]$. In particular,   $f_!f^! 1_{X}= \Sym^{\bullet}_X \s{F}$.   
\end{enumerate}

\end{prop}
\begin{proof}
Part (1) is clear since $\bb{V}(\s{F})$ is just the relative analytic spectrum of the algebra $\Sym^{\bullet}_X \s{F}$   (locally in the $\s{D}$-topology) with the induced analytic structure.

 For part (2), without loss of generality we can reduce to the universal case $X= */ \GL_d$ and $\s{F}= \St$ the standard representation. Then, to see that $\widehat{\bb{V}(\St)} \to */ \GL_d$ is cohomologically smooth, we can take the pullback along $*\to */ \GL_d$. In this case, $\St= \bb{Z}^d$ is a free $\bb{Z}$-module of rank $n$, and by induction it suffices to treat the case $d=1$. We have a map 
\[
j:\widehat{\bb{G}}_a \subset \bb{G}_a=\AnSpec \bb{Z}[T] \subset  \bb{P}^1_{\bb{Z}}.
\]
Note  that  $\bb{P}^1_{\bb{Z}}$ is cohomologically smooth over $\bb{Z}_{\sol}$, namely, it has an open cover in the sense of locale by the solid affine spaces $\AnSpec \bb{Z}[T]_{\sol}$ and $\AnSpec \bb{Z}[T^{-1}]_{\sol}$.   Then, it suffices to show that $j$ is an open immersion, this follows since it is the complement  of the idempotent algebra over $\bb{P}^1_{\bb{Z}}$ given by $\bb{Z}[T^{-1}]$. 

Next, we  show that $f^! 1_{X}= f^* \bigwedge^d \s{F}^{\vee}[d]$ and $f_{!} 1_{\widehat{\bb{V}(\s{F})}}= \bigwedge^{d} \s{F} (\Sym^{\bullet}_X \s{F})[-d]$. For this, we can reduce to the universal case of $X=*/\GL_d$ and $\s{F}= \St$.

Let $\{v_1,\ldots, v_d\}$ be the standard basis of $\St$ and $T_1,\ldots, T_d$ its dual basis. Then 
\[
\St = \AnSpec_X (\bb{Z}[\underline{T}]  ). 
\]
Consider the idempotent $(\bb{Z}[\underline{T}],\bb{Z})_{\sol}$-algebras
\[
D_i = \bb{Z}[\underline{T}][T_{i}^{-1}],
\]
and let $C$ be the union of the algebras in the sense of locale. More precisely, for $I\subset \{1,\ldots, d\}=:[d]$ let $D_I=  \otimes_{i\in I, \bb{Z}[\underline{T}]} D_i$ and let $C$ be the idempotent $dg$-$\bb{Z}[\underline{T}]$-algebra
\[
C= [D_{\empty} \to \bigoplus_{i\in I} D_i \to \cdots \to \bigoplus_{|I|=k}  D_I \to  \cdots \to D_{[d]}]. 
\]
 Then $j:\widehat{\bb{V}(\St)}\subset \bb{V}(\St)$ is the open subspace complement to $C$, and we can compute 
\[
f_! 1_{\widehat{\bb{V}(\St)}} =  [\bb{Z}[\underline{T}] \to C].
\]
Unravelling the construction of $C$, and identifying $T_i^{-1}=v_i$, one finds that 
\begin{equation}
\label{eqLoweShrierk1}
f_! 1_{\widehat{\bb{V}(\St)}}= (T_1T_2\cdots T_d)^{-1} \bb{Z}[T_i^{-1}:i=1,\ldots, d] [-d].
\end{equation}
Note that shifting the variables $T_i$ and $T_j$ alters the formulas by a  $-1$ factor, namely, in the union of idempotent algebras $D_i$ and $D_j$ one fixes the map 
\[
C \xrightarrow{(1,-1)} D_i \bigoplus D_j,
\]
which differs from the map 
\[
C \xrightarrow{(1,-1)} D_j \bigoplus D_i
\]
under the natural isomorphism $D_i \bigoplus D_j=D_j \bigoplus D_i$ by multiplication by a $-1$. If $X$ is defined over $\bb{Q}$, one deduces that $f_{!} 1_{\widehat{\bb{V}(\St)}} = \bigwedge^{d}\St \otimes \Sym^{\bullet}_X \St[-d]$.    In general,  using that 
\[
\iHom_{\bb{Z}}(f_! \bb{Z}[\underline{T}], \bb{Z}) = f_* f^! \bb{Z}, 
\]
we find that $f^{!} \bb{Z}$ is the localization at $\widehat{\bb{V}}(\St)$  of the object
\[
f^! \bb{Z} = T_1\cdots T_d \otimes \bb{Z}[[\underline{T}]][d]
\]
as $\GL_n$-equivariant  $\bb{Z}[\underline{T}]$-module, which is nothing but $f^{*} \bigwedge^{d} \St^{\vee} [d]$. 
\end{proof}

Next, we study cohomological properties of quotient stacks associated to vector bundles in characteristic $0$. Let us first describe the categories of sheaves on $X/\bb{V}(\s{F})$ and $X/\widehat{\bb{V}(\s{F})}$ via the monadicity theorem.

\begin{prop}
\label{PropAlgebraicCartier2}
Let $X$ be a solid $\s{D}$-stack over $\bb{Q}$ and let $\s{F}$ be a vector bundle of rank $d$ over $X$. 

\begin{enumerate}

\item  There are natural equivalences 
\[
\Mod(X/\bb{V}(\s{F}))=\mathrm{CoMod}_{\Sym_{X}^{\bullet} \s{F}^{\vee}}(\Mod(X))
\]
and 
\[
\Mod(X/\widehat{\bb{V}(\s{F})})=\Mod_{\Sym_X^{\bullet} \s{F}}(\Mod_X). 
\]

\item Consider the maps $X\xrightarrow{f} X/\bb{V}(\s{F}) \xrightarrow{g} X$. Then  $f$ is a descendable $\s{D}$-cover and $g$ is both weakly cohomologically proper and cohomologically smooth. Moreover, there is a natural equivalence $g^! 1_{X} \cong g^* \bigwedge^d \s{F}[d]$.

\item Consider the maps $X\xrightarrow{f} X/ \widehat{\bb{V}(\s{F})} \xrightarrow{g} X$.  Then  $f$ is a smooth $\s{D}$-cover,   $g$ is cohomologically smooth, and there is a natural equivalence $g^! 1_{X}\cong \bigwedge^d \s{F}[-d]$.   Moreover, $g$ is co-smooth with proper dual $\n{P}_g(1_{X/\widehat{\bb{V}(\s{F})}})\cong 1_{X/\widehat{\bb{V}(\s{F})}}[-2d]$.

\end{enumerate}
\end{prop}

\begin{proof}
By base change, we can reduce all the assertions to the universal case $X=*/ \GL_d$ and $\s{F}= \St$. We first use Proposition \ref{PropCoverModuleComoduleDescription} to deduce part (1), and then we apply Lemma \ref{LemmaDeRhamKoszulAlgebraic} to construct the data required in Lemma \ref{KeyLemmaSmoothRetraction} and show (2) and (3).

\begin{enumerate}

\item By Proposition \ref{PropAlgCartier1} the maps $f:X\to X/\bb{V}(\St)$  and $f':X\to X/\widehat{\bb{V}(\St)}$ are weakly cohomologically proper and cohomologically smooth respectively. Then, by Proposition \ref{PropCoverModuleComoduleDescription} and Remark \ref{RemarkComoduleDescription} we have that 
\[
\Mod(X/\bb{V}(\St))=\mathrm{CoMod}_{f^*f_* 1_X}(\Mod(X))
\]
and 
\[
\Mod(X/\widehat{\bb{V}(\St)})=\Mod_{f^!f_! 1_X}(\Mod(X)).
\]
It is left to compute the monad and comonad, for this one uses the fact that the \v{C}ech nerves of $X\to X/\bb{V}(\St)$ and $X\to \widehat{\bb{V}(\St)}$ are given by the simplicial stack $(\bb{V}(\St)^{n/X})_{[n]\in \Delta^{\op}}$ (resp. $(\widehat{\bb{V}(\s{F})}^{n/X})_{[n]\in \Delta^{\op}}$) encoding the commutative group structure of $\bb{V}(\St)$ and $\bb{V}(\St)$ respectively, which arise from the $\GL_d$-equivariant  Hopf-algebra structure of $\Sym_X^{\bullet} \St^{\vee}$ and $\widehat{\Sym}_X^{\bullet} \St^{\vee}$  respectively (see \cite[Theorem 4.7.5.2 (3)]{HigherAlgebra} and Proposition \ref{PropAlgCartier1} (2)), we left the details to the reader.

\item By part (1), the category $\Mod_{\sol}(X/\bb{V}(\St))$ is the category of $\GL_d$-equivariant (left) comodules over $\Sym^{\bullet}_X(\St^{\vee})$. The map $X \to X/ \bb{V}(\St)$ is the vector bundle associated to the algebra  $\Sym^{\bullet}_X(\St^{\vee})$ over $X/ \bb{V}(\St)$, endowed with the natural  comodule action given by co-multiplication.

 The map $f$ is weakly cohomologically proper since $\bb{V}(\St)$ has the induced analytic structure from $X$. To show that $g$ has $!$-functors, it suffices to prove that $f$ is a descendable $\s{D}$-cover, which amounts to show that $f_* 1_{X}$ is  descendable over $1_{X/\bb{V}(\St)}$. For this, we can use the de Rham complex  of Lemma \ref{LemmaDeRhamKoszulAlgebraic}
\begin{equation}
\label{equationDeRhamPolyAlgebras}
0 \to \bb{Q} \to \Sym^{\bullet}_X \St^{\vee} \xrightarrow{d} \Sym^{\bullet}_X \St^{\vee} \otimes \St^{\vee}  \xrightarrow{d} \cdots  \xrightarrow{d} \Sym^{\bullet}_X \St^{\vee} \otimes \bigwedge^d \St^{\vee} \to 0 
\end{equation}
which is a $\GL_d$-equivariant complex of $\Sym^{\bullet}_X \St^{\vee}$-comodules, where $\St^{\vee}$ has the trivial comodule action (we higlight for future reference that this is equal to the adjoint comodule action since $\bb{V}(\St)$ is abelian).  Note that if $\{e_i\}_{i=1}^d$ is a basis of $\St^{\vee}$, we have identified $\bigoplus_{i=1}^d \bb{Q} d e_i\cong \St^{\vee}$.    One formally deduces that $g$ is weakly cohomologically proper  as $f$ is so, see Corollary \ref{CorollaryDescentSmoothProperCovers}. Finally, it is left to show that $g$ is cohomologically smooth with $g^! 1_{X}=\bigwedge^d \St [d]$. Let $\bb{Q}\in \Mod_{\sol}(X/ \bb{V}(\St))$ be the trivial representation seen as a morphism in  $\LZ_{\s{D},X}(X/ \bb{V}(\St), X)$.   Let $\n{L}:=(\bigwedge^d \St)[d]\in \Mod_{\sol}(X/ \bb{V}(\St))$ endowed with the trivial comodule action. Twisting \eqref{equationDeRhamPolyAlgebras} by $\bigwedge^d \St$ we get a map $f_! 1_{X}= \Sym^{\bullet}_X \St^{\vee} \to \n{L}$. Moreover, \eqref{equationDeRhamPolyAlgebras} shows that 
\[
g_! \n{L}= \bigoplus_{i=0}^d \bigwedge^i \St[i],
\] 
this gives the retraction $g_! \n{L} \to 1_{X}$. We conclude by applying Lemma \ref{KeyLemmaSmoothRetraction}.

\item 

 The map $f$ is cohomologically smooth  by Proposition \ref{PropAlgCartier1}. Since the pullback $f^*$ is conservative, $f$ is in fact a smooth $\s{D}$-cover, in particular of universal $!$-descent, and thus $g$ admits $!$-functors. On the other hand, since being cohomologically smooth is local on the source (Corollary \ref{CorollaryDescentSmoothProperCovers}), we see that $g$ is also cohomologically smooth. It is left to   compute   $g^! 1_{X}$. By functoriality of the Lu-Zheng category, it suffices to treat the universal case $X=*/\GL_n$ and $\s{F}=\St$.  By Proposition \ref{PropAlgebraicCartier2}, one has that $1_X=f^!g^! 1_X=( f^! 1)\otimes f^*(g^! 1_X)= f^* \bigwedge^d\St^{\vee}[d] \otimes f^*(g^! 1_X)$. This gives $f^*g^! 1_{X}=f^*\bigwedge^d \St[-d]$, and we only need to  identify $\bigwedge^d \St$ as object in $\Mod_{\sol}(X/\widehat{\bb{V}(\St)})$. Consider the diagram  
 \[
 \begin{tikzcd}
 X/\widehat{\bb{V}(\St)}\ar[r, "\Delta"] & X/\widehat{\bb{V}(\St_1\oplus \St_2)} \ar[r, "\pi_2"] \ar[d, "\pi_1"] & X/ \widehat{\bb{V}(\St)} \ar[d, "g"] \\ 
 & X/ \widehat{\bb{V}}(\St)  \ar[r, "g"]& X, 
 \end{tikzcd}
 \]
 by smooth base change, we have that $g^! 1_{X} = \Delta^* \pi_{2}^! (1_{X/ \widehat{\bb{V}(\St)}})$.  On the other hand, let $Q= \St_1\oplus \St_2/ \Delta(\St)$, endowed with the left regular action of $\St_1\oplus \St_2$.  The map $\Delta$ is equivalent to the map 
 \begin{equation}
 \label{eqDiagonalMapFormal}
 \widehat{\bb{V}(Q)}/ \widehat{\bb{V}(\St_1\oplus \St_2)} \to X/ \bb{V}(\St_1\oplus \St_2).  
 \end{equation}
 This shows that $\Delta^! 1_{X/ \bb{V}(\St_1\oplus \St_2)}=  \Delta^* \bigwedge^d Q^{\vee} [d]$.  One gets that 
 \[
 \begin{aligned}
 g^! 1_{X} & = \Delta^* \pi_{2}^!(1_{X/ \widehat{\bb{V}}(\St)}) \\ 
 & = \Delta^! \pi_{2}^!(1_{X/ \widehat{\bb{V}}(\St)}) \otimes \Delta^* \bigwedge^d Q [-d] \\ 
 & = \Delta^* \bigwedge^d Q [-d] \\ 
 & = \bigwedge^d \St[-d]. 
 \end{aligned}
 \]

 It is left to show that $g$ is co-smooth with proper dualizing sheaf $\n{L}=1_{X/\widehat{\bb{V}(\St)}}[-2d]$; we use Lemma \ref{KeyLemmaSmoothRetraction}. By part (1), $\Mod(X/\widehat{\bb{V}(\St)})$ is equivalent to  the category of $\Sym_X^{\bullet} \St$-modules in $\Mod(X)$.   The Koszul sequence
\begin{equation} 
\label{eqKoszulResolution2}
  0 \to \Sym^{\bullet}_X \St \otimes \bigwedge^d \St \to \cdots \to \Sym^{\bullet}_X \St \to \bb{Q} \to 0
\end{equation} 
gives rise a map 
 \[
 1_{X/\widehat{\bb{V}(\St)}}[-2d] \to f_! 1_{X}= \Sym^{\bullet}_X \St \otimes \bigwedge^d \St[-d]. 
 \]
 To construct a retraction $g_! \n{L} \to 1_{X}$ it suffices to show that the Koszul complex induces a splitting  $g_{!} \n{L}= \bigoplus_{i=0}^d  \bigwedge^d \St^{\vee}[-i]$. We know by \eqref{eqKoszulResolution2} that $g_! \n{L}$ is a perfect complex of $\GL_d$-equivariant  $\bb{Q}$-vector spaces. Then, to show that it is split it suffices to do it for its dual. We have that 
 \[
 \iHom_{X}( g_! \n{L}, \bb{Q})= g_* \n{L}[2d]= g_* \bigwedge^d \St [d]. 
 \]
 But we have that $g_* \bigwedge^d \St [d]= (g_* 1_{X/\widehat{\bb{V}(\St)}}) \otimes \bigwedge^d \St [d]$, thus, it suffices to show that the cohomology $g_* 1_{X/\widehat{\bb{V}(\St)}}$ splits, this follows by taking the de Rham resolution for $\widehat{\Sym}_X(\St^{\vee})$, since $g_* 1_{X}$ is the ``complete'' comodule associated to  $\widehat{\Sym}_X(\St^{\vee})$.
 \end{enumerate}
\end{proof}

\begin{remark}
\label{RemarkAdjointAction}
In the previous proposition we always keep track of the adjoint action of $\St$, even if it is trivial. The reason is that similar computations will hold for the case of classifying stacks of more general groups, see Proposition \ref{PropDualizingSheafDaggerGRoupoid}. 
\end{remark}

\begin{theorem}[Algebraic Cartier duality]
\label{TheoAlgebraicCartierDuality}

Let $X$ be a solid $\s{D}$-stack over $\bb{Q}$ and $\s{F}$ a vector bundle over $X$. 
 
\begin{enumerate}

\item  There is a natural bi-linear map
\[
F:\bb{V}(\s{F}) \times_X X/\widehat{\bb{V}(\s{F}^{\vee})} \to */ \bb{G}_m
\]
functorial in the category $\BUN_{d,\n{C}}$,  such that $F^*(\s{O}(1))$ is an isomorphism in the Lu-Zheng category, considered as a map $\LZ_{X}(\bb{V}(\s{F}) , X/\widehat{\bb{V}(\s{F}^{\vee})} )$.  Moreover, the inverse of $F^*(\s{O}(1))$ is naturally isomorphic to $F^*(\s{O}(-1)) \otimes_{\s{O}_X} \bigwedge^d\s{F}^{\vee}[-d]$.

\item There is a  natural bi-linear map 
\[
G:\widehat{\bb{V}(\s{F}^{\vee})} \times_X X/ \bb{V}(\s{F})\to */ \bb{G}_m 
\]
functorial in the category $\BUN_{d,\n{C}}$,  such that  $G^*(\s{O}(1))$ is an isomorphism in the Lu-Zheng category considered a a map in $\LZ_{X}(\widehat{\bb{V}(\s{F}^{\vee})},X/ \bb{V}(\s{F}))$. Moreover, the inverse  of $G^*(\s{O}(1))$ is  naturally isomorphic to $G^*(\s{O}(-1))\otimes \bigwedge^d \s{F}[d]$. 

\end{enumerate} 

In particular, we have Cartier duality isomorphisms induced by a Fourier-Moukai transform 
\begin{equation}
\label{eqFM1}
FM_1:\Mod_{\sol}(\bb{V}(\s{F})) \xrightarrow{\sim} \Mod_{\sol}(X/\widehat{\bb{V}(\s{F}^{\vee})})
\end{equation}
\begin{equation}
\label{eqFM2}
FM_2:\Mod_{\sol}(\widehat{\bb{V}(\s{F}^{\vee})}) \xrightarrow{\sim} \Mod_{\sol}(X/ \bb{V}(\s{F})), 
\end{equation}
defined by the convolution
\[
FM_1(M)= F^*(\s{O}(1)) \star  M \mbox{ and } FM_2(M)= G^*(\s{O}(1)) \star M. 
\]
\end{theorem}

\begin{proof}
By functoriality of the Lu-Zheng category, we can assume without loss of generality that $X=*/\GL_d$ and that $\s{F}=\St$ is the standard vector bundle. On the other hand, by \cite[Lemma 5.11]{SixFunctorsScholze} it suffices to construct a unit and co-unit  for the adjunction and prove that they are equivalences.

\textbf{Step 1. Construction of $F$ and $G$.} We first construct the maps $F$ and $G$ of parts (1) and (2). Let us start with part (1). Consider the $\GL_n$-equivariant vector bundle $\bb{V}(\St)= \AnSpec_X (\Sym^{\bullet}_X \St^{\vee})$, as well as its formal dual  $\widehat{\bb{V}(\St^{\vee})}$.  By Proposition \ref{PropAlgebraicCartier2} (1), a line bundle over $\bb{V}(\St)\times_{X} X/ \widehat{\bb{V}(\St)}$ is the same as a $\GL_n$-equivariant line bundle over $\bb{V}(\St)$ endowed with a module action of $\Sym^{\bullet}_X \St^{\vee}$ commuting with the  $\s{O}(\bb{V}(\St))$-linear structure. We simply take  $\n{L}=\Sym^{\bullet}_{X} \St^{\vee}$ as a right module over $\bb{V}(\St)$ endowed with the natural left multiplication of $\Sym^{\bullet}_{X} \St^{\vee}$, we say that $\n{L}$ has the \textbf{left regular action}. This defines the map $F$.

For the map $G$ in (2), we argue in a similar way. A line bundle  in $\widehat{\bb{V}(\St^{\vee})}\times_X X/\bb{V}(\St)$ is the same as a compatible system of $\GL_n$-equivariant line bundles $(\n{L}_n)$ on $\bb{V}(\St^{\vee})_n$ for all $n$, together with a compatible system of $\GL_d$-equivariant comodule structures
\[
(\n{L}_n \to  \Sym^{\bullet}_X \St^{\vee} \otimes  \n{L}_n ).
\]
To define such an object, we  take $\n{L}_n= \Sym^{\leq n}_X \St$ endowed with right module structure, and construct the comodule structure of $\Sym^{\bullet}_X \St^{\vee}$ by taking adjoints of the multiplication map $\Sym^{\leq n}_X \St \otimes \n{L}_n \to \n{L}_n$, namely, 
\[
\n{L}_n \to \Sym^{\leq n}_X \St^{\vee} \otimes \n{L}_n \to \Sym^{\bullet}_X \St^{\vee}  \otimes  \n{L}_n.
\]
This defines a  line bundle on $\widehat{\bb{V}(\St^{\vee})} \times_X X/\bb{V}(\St)$, and so the map $G$. We say that $\n{L}$ has the \textbf{left regular action}.

\textbf{Step 2. Identification of $F^*(\s{O}(-1))$ and $G^*(\s{O}(-1))$}. By construction $F^*(\s{O}(1))$ is the line bundle $\Sym^{\bullet}_X \St^{\vee}$ over $\bb{V}(\St)$ endowed with the left multiplication by $\Sym^{\bullet}_X \St^{\vee}$. Then, $F^*(\s{O}(-1))=F^*(\s{O}(1))^{-1}$ is the line bundle $\Sym ^{\bullet}_X \St^{\vee}$ over $\bb{V}(\St)$ endowed with the $\Sym^{\bullet}_X \St^{\vee}$-multiplication arising from  the composite
\[
\Sym^{\bullet}_X \St^{\vee} \otimes \Sym^{\bullet}_X \St^{\vee} \xrightarrow{s\otimes \id}\Sym^{\bullet}_X \St^{\vee} \otimes \Sym^{\bullet}_X \St^{\vee} \xrightarrow{m} \Sym^{\bullet}_X \St^{\vee}
\]
where  $s: \Sym^{\bullet}_X \St^{\vee} \to \Sym^{\bullet}_X \St^{\vee}$ is the antipode map sending $v\mapsto -v$ for $v\in \St^{\vee}$. We say that $F^{*}(\s{O}(-1))$ has the \textbf{right regular action}. 

Similarly, $G^*(\s{O}(1))$ is the line bundle $(\Sym^{\leq n}_X \St)_{n}$ over $\widehat{\bb{V}(\St^{\vee})}$ endowed with the left regular $\Sym^{\bullet}_X \St^{\vee}$-comodule structure constructed as the adjoint of the multiplication map. Then, $G^*(\s{O}(-1))$ is the line bundle $(\Sym^{\leq n}_X  \St)_{n}$ endowed with the right regular comodule structure of $\widehat{\Sym}_X \St^{\vee}$ obtained by composing the multiplication map with the antipode.  We say that $G^*(\s{O}(-1))$ has the \textbf{right regular action}. 

\textbf{Step 3. Unit and co-unit for $F$.} Next, we construct the unit and co-unit maps of the convolutions and see that they are equivalences. We first deal with (1). Let us write $X= */\GL_d$, $Y=\bb{V}(\St)$ and $Z= X/\widehat{\bb{V}(\St^{\vee})}$. We also let $\n{L}= F^*(\s{O}(1))$ and $\n{G}:= F^*(\s{O}(-1))\otimes \bigwedge^d\St^{\vee}[-d]$. Recall that we consider $\n{L}\in \LZ_{X}(Y,Z)$ and $\n{G} \in \LZ_{X}(Z,Y)$, so that we have the convolution
\[
\n{G} \star \n{L}=  \pi_{1,3,!}( \pi_{1,2}^* \n{L}\otimes \pi_{2,3}^* \n{G}  )
\]
for the fiber product $ Y \times_{X} Z \times_X Y$, and the  convolution 
\[
\n{L} \star \n{G} = \pi_{1,3,!}( \pi_{1,2}^* \n{G}\otimes \pi_{2,3}^* \n{L} )
\]
for the fiber product $Z \times_{X} Y \times_X Z$. Thus, we want to construct equivalences
\begin{equation}
\label{eqUnitF}
\Delta_{Y,!} 1_{Y}  \xrightarrow{\sim} \n{G} \star \n{L}
\end{equation}
and
\begin{equation}
\label{eqCounitF}
 \n{L} \star \n{G}\xrightarrow{\sim} \Delta_{Z,!} 1_{Z}.
\end{equation}
Let us first compute $\n{G}\star \n{L}$. The tensor $\pi_{2,3}^* \n{G} \otimes  \pi_{1,2}^* \n{L}$ lies in $Y\times_{X} Z \times_X Y = \bb{V}(\St_1\oplus \St_2)/ \widehat{\bb{V}(\St^{\vee})}$ where the quotient is with respect to the trivial action. By step 2 one deduces that $ \pi_{1,2}^* \n{L} \otimes \pi_{2,3}^* \n{G} $ is nothing but the line bundle $\bigwedge^d\St_{2}^{\vee}\otimes\Sym^{\bullet}_X(\St^{\vee}_1 \oplus \St_2^{\vee})[-d]$ endowed with the $\Sym^{\bullet}_X \St^{\vee}$-module action  which is left regular on $\St^{\vee}_1$ and right regular on $\St^{\vee}_2$. Thus, by taking the anti-diagonal embedding $\Delta^{ant}=(\id,-\id):\St^{\vee} \to\St_1^{\vee} \oplus \St^{\vee}_2$, we can write 
\[
\begin{aligned}
\bigwedge^d\St_{2}^{\vee}\otimes\Sym^{\bullet}_X(\St^{\vee}_1 \oplus \St_2^{\vee})[-d] & = \bigwedge^d \St_{2}^{\vee} \otimes  \Sym^{\bullet}_X(\Delta^{ant}(\St^{\vee})) \otimes \Sym^{\bullet}_X((\St_1^{\vee} \oplus \St^{\vee}_2)/\Delta^{ant}(\St^{\vee}))[-d]\\
& \cong \bigwedge^{d} (\Delta^{ant} \St)\otimes  \Sym^{\bullet}_X(\Delta^{ant}(\St^{\vee})) \otimes \Sym^{\bullet}_X((\St_1^{\vee} \oplus \St^{\vee}_2)/\Delta^{ant}(\St^{\vee}))[-d]
\end{aligned}
\]
where we use the composite $\St^{\vee}\xrightarrow{\Delta^{ant}} \St_1\oplus \St_2 \to \St_2$ to identify $\bigwedge^{d} \Delta^{ant}(\St) \cong \bigwedge^{d} \St_2$,  and the module action of $\Sym_X^{\bullet} (\Delta^{ant} (\St^{\vee}))$  on $\Sym^{\bullet}_X((\St_1^{\vee} \oplus \St^{\vee}_2)/\Delta^{ant}(\St^{\vee}))$ factors through the counit.  Consider the composite 
\[
Y\times_X Y \xrightarrow{\tilde{f}} Y\times_X Z\times_X Y  \xrightarrow{\tilde{g}} Y\times_X Y
\]
arising from the maps $X\xrightarrow{f} Z \xrightarrow{g} X$.  By Proposition \ref{PropAlgebraicCartier2} we can write
\[
\widetilde{f}_!1_{Y\times_X Y} = \bigwedge^{d} (\Delta^{ant} \St) \otimes \Sym^{\bullet}_X (\Delta^{ant} \St) \otimes_{1_{\widehat{\bb{V}(\St)}}} 1_{Y\times_{X} Z \times_X Y}[-d].
\]
Since $Y\to X$ is just a vector bundle, it is clear that $\Delta_{Y,!} 1_Y = \Sym^{\bullet}_X((\St_1\oplus \St_2)/(\Delta^{ant}(\St)))$. Thus, we find that 
\[
\pi_{1,2}^{*}\n{L} \otimes \pi_{2,3}^{*} \n{G} \cong \widetilde{f}_! 1_{Y\times_X Y} \otimes \widetilde{g}^{*} (\Delta_{Y,!} 1_Y). 
\]
Applying $\pi_{1,3,!}=\widetilde{g}_!$, we get that $\n{G}\star \n{L}\cong \Delta_{Y,!} 1_Y$, which gives the unit map  \eqref{eqUnitF} that in addition an equivalence.

Now, let us construct the co-unit map for $\n{L}\star \n{G}$. The object $\pi_{1,2}^* \n{G} \otimes \pi_{2,3}^* \n{L}$ lies over $Z\times_{X} Y \times_X Z = \bb{V}(\St)/\widehat{\bb{V}(\St_1^{\vee} \oplus \St_2^{\vee})}$. By step (2) it is the line bundle $\bigwedge^d \St^{\vee}_1 \otimes( \Sym^{\bullet}_X \St^{\vee})[-d]$, where $\Sym^{\bullet}_X \St^{\vee}$  is endowed with the $\widehat{\Sym}_X(\St_1^{\vee}\oplus \St_2^{\vee})$-module structure   which is right regular for $\St_1^{\vee}$ and left regular for $\St_2^{\vee}$. Equivalently, let $Q^{\vee}=(\St_1^{\vee}\oplus \St_2^{\vee})/(\Delta(\St^{\vee}))$, then 
\[
\pi_{2,3}^* \n{L} \otimes \pi_{1,2}^* \n{G} = \bigwedge^d \St_1^{\vee}\otimes \Sym^{\bullet}_X Q^{\vee} [-d]
\]
endowed with its natural $\Sym^{\bullet}_X(\St_1^{\vee}\oplus \St_2^{\vee})$-module structure given  by left multiplication. Hence, 
\[
\n{L}\star \n{G} =\bigwedge^d \St_1^{\vee}\otimes (\Sym^{\bullet}_X Q^{\vee}) [-d].
\]
Now, the diagonal map is equivalent to \eqref{eqDiagonalMapFormal}, and Proposition \ref{PropAlgCartier1} (2) provides the isomorphism $\Delta_! 1_{X/\widehat{\bb{V}(\St^{\vee})}} \xrightarrow{\sim} \n{L}\star \n{G}$ as wanted.

\textbf{Step 4. Unit and co-unit for $G$. } Now we move to (2). We set $X=*/\GL_d$, $Y=\widehat{\bb{V}(\St^{\vee})}$ and $Z=X/ \bb{V}(\St)$. We also write $\n{L}= G^*(\s{O}(1))$ and $\n{G}= G^*(\s{O}(-1))\otimes \bigwedge^d \St [d]$. We want to construct equivalences  \eqref{eqUnitF} and \eqref{eqCounitF}. In the first case, the tensor product $\pi_{1,2}^* \n{L} \otimes \pi_{2,3}^*\n{G}$ lies over $ \widehat{\bb{V}(\St_1^{\vee} \oplus \St_{2}^{\vee})}/\bb{V}(\St)$ with quotient given by the trivial action. By step (2), it is described as the line bundle $(\Sym^{\leq n}_X(\St_1\oplus \St_2)\otimes \bigwedge^d \St_2 [d])_n$ endowed with the comodule structure over $\Sym^{\bullet}_X \St^{\vee}$ defined by the composite of the anti-diagonal embedding $\Delta^{ant}=(\id,-\id):\St \to \St_1\oplus \St_2$ and the left regular action of $\St_1\oplus \St_2$. Then, we can write 
\[
(\Sym^{\leq n}_X (\St_1\oplus \St_2)\otimes \bigwedge^d \St_2 [d])_n = ( \Sym^{\leq n}_X (\Delta^{ant}(\St))\otimes \bigwedge^d \St_2[d] \otimes \Sym^{\leq n}_X (( \St_1\oplus \St_2)/\Delta^{ant}(\St)) )_{n},
\]
so that the comodule associated to the limit of $(\Sym^{\leq n}_X \Delta^{ant}(\St) )_{n}$ is the dual to the module structure of $\Sym^{\bullet}_X (\Delta^{ant}(\St^{\vee}))$.
 Let $g: X/\bb{V}(\St)\to X$,  Proposition \ref{PropAlgebraicCartier2} (1) implies 
 \[
 \begin{aligned}
 g_*( (\Sym^{\bullet}_X \St^{\vee})^{\vee} \otimes \bigwedge^d\St[d]) & = g_*\iHom_{X/\bb{V}(\St)}(\Sym^{\bullet}_X \St^{\vee}, g^! \bb{Q}) \\ 
 & = \iHom_{X}(g_!(\Sym^{\bullet}_X \St^{\vee}), \bb{Q}) \\ 
 & = \bb{Q}.
 \end{aligned}
 \]
 We deduce that 
\[
\n{G} \star \n{L} = \Sym^{\leq n}_X((\St_1 \oplus \St_2)/ \Delta^{ant}(\St)),
\]
which produces the unit map \eqref{eqUnitF} that is clearly an equivalence.  Next, we construct the co-unit \eqref{eqCounitF}, consider the tensor product $ \pi_{1,2}^* \n{G} \otimes \pi_{2,3}^* \n{L}$ over $\widehat{\bb{V}(\St^{\vee})}/\bb{V}(\St_1\oplus \St_2)$. By step (2) it consists on the line bundle $(\Sym^{\leq n}_X (Q)\otimes \bigwedge^d Q[d])_n$ where $Q= (\St_1\oplus \St_2)/\Delta(\St)$ is endowed with its natural comodule structure of $\Sym^{\bullet}_X (\St_1^{\vee} \oplus \St_2^{\vee})$ given by left regular action. Thus, by Proposition \ref{PropAlgCartier1} (2) we get that 
\[
\n{L}\star \n{G} =  \Sym^{\bullet}_X (Q^{\vee}).
\]
On the other hand, the diagonal map $X/\bb{V}(\St) \to X/\bb{V}(\St_1\oplus \St_2)$ is isomorphic to $\bb{V}(Q)/ \bb{V}(\St_1\oplus \St_2) \to X/ \bb{V}(\St_1\oplus \St_2)$, this shows that $\Delta_! 1_{X/ \bb{V}(\St )} \cong \Sym^{\bullet}_X  Q^{\vee} \cong \n{L}\star \n{G}$ as wanted. 
\end{proof}

\begin{remark}
The statement and proof of Theorem \ref{TheoAlgebraicCartierDuality} also apply for fpqc-stacks in classical derived algebraic geometry. Indeed, the functors and objects involved in the universal case arise from  stacks on schemes endowed with the theory of classical quasi-coherent modules. On the other hand, a careful bookkeeping of the construction of the units and co-units  should  prove  that  the composites \eqref{eqUnitF} and \eqref{eqCounitF} are actually the identity and that  \cite[Lemma 5.11]{SixFunctorsScholze}  would not be necessary, we left this computation to the curious reader. 
\end{remark}

We finish this section with some classical properties of the Fourier-Moukai transform of Theorem \ref{TheoAlgebraicCartierDuality}.

\begin{prop}
\label{PropConvolutionFMTransforms}
Let $X$ be a solid $\s{D}$-stack over $\bb{Q}$ and $\s{F}$ a vector bundle of rank $d$.  Consider the Fourier-Moukai transforms $FM_1$ and $FM_2$ of Theorem \ref{TheoAlgebraicCartierDuality}. The following hold
\begin{enumerate}

\item Let us write $Y= \bb{V(\s{F})}$ and $Z=X/\widehat{\bb{V}(\s{F}^{\vee})}$. Denote $\iota: X \to Y$ the zero section, and $p:Y \to X$, $f:X \to Z$, $g:Z \to X$ the natural maps. We have the following natural identities of convolutions in the Lu-Zheng category over $X$:

\begin{itemize}
\item[(i)] $F^*(\s{O}(1))\star \iota_{!} 1_X= 1_{Z}$

\item[(ii)]  $\iota_! 1_X \star    (F^*(\s{O}(-1))\otimes \bigwedge^d \s{F}^{\vee}[-d])=\bigwedge^d \s{F}^{\vee}[-d]=g^! 1_{X}$. 

\item[(iii)]  $(F^*(\s{O}(-1))\otimes \bigwedge^d \s{F}^{\vee}[-d] ) \star f_! f^!1_Z= 1_{Y}$. 

\item[(iv)]  $f_! 1_{X} \star F^*(\s{O}(1)) = 1_Y$.

\end{itemize}
\item Let us write $Y= \widehat{\bb{V}(\s{F}^{\vee})}$ and $Z=X/\bb{V}(\s{F})$. Denote $\iota: X \to Y$ the zero section, and $p:Y \to X$, $f:X \to Z$, $g:Z \to X$ the natural maps. We have the following natural identities of convolutions in the Lu-Zheng category over $X$:

\begin{itemize}

\item[(i)] $G^*(\s{O}(1))\star \iota_! 1_X= 1_{Z}$

\item[(ii)] $\iota_! 1_X \star   (G^*(\s{O}(-1)) \otimes \bigwedge^d \s{F}[d])= \bigwedge^d \s{F}[d] = g^! 1_{X}$.

\item[(iii)]   $(G^*(\s{O}(-1)) \otimes \bigwedge^d \s{F}[d]) \star  f_! 1_X= 1_{Y}$. 

\item[(iv)] $ f_! 1_{X} \star  G^*(\s{O}(1)) = p^! 1_{X}$.

\end{itemize}
\end{enumerate} 
\end{prop} 
\begin{proof}

We only show part (1), part (2) is done in a similar way.   The object $\iota_! 1_X$ is just $\iota_* 1_X \in \Mod_{\sol}(Y)$, let us   considered it as a morphism in $\LZ_{X}(X,Y)$. Since $\iota$ is both cohomologically smooth (by Corollary \ref{CoroSmoothImmersion}), and weakly cohomologically proper (being affinoid with induced analytic structure),   it has right and left adjoints given by $\iota_*\iota^! 1_{Y}= \iota_*\bigwedge^d \s{F} [-d]$ and $\iota_* 1_{Y}$ respectively. 

\begin{itemize}

\item[(i)] Let us compute $F^*(\s{O}(1)) \star \iota_!  1_{X}$, the term $F^*(\s{O}(1))$ is seen as a map $Y \to Z$.  Consider the fiber product $X\times_X Y \times_X Z $, by definition $ F^*(\s{O}(1)) \star \iota_! 1_{X} = \pi_{1,3,!}(\pi_{1,2}^*  \iota_! 1_{X} \otimes \pi_{2,3}^* F^*(\s{O}(1)) )$. We obtain that 
\[
\pi_{1,2}^*  \iota_! 1_{X} \otimes \pi_{2,3}^* F^*(\s{O}(1))  = \Sym^{\bullet}_X(\s{F})\otimes_{\Sym^{\bullet}_X (\s{F})} \s{O}_X = \s{O}_X
\] 
endowed with the trivial comodule structure. Since $\pi_{1,3}: Y\times_X Z \to Z$ is the base change of $Y \to X$, one deduces that $F^*(\s{O}(1)) \star \iota_! 1_{X} = 1_{Z}$ as wanted. 

 \item[(ii)] Next, let us  take left adjoints to the expression $F^*(\s{O}(1)) \star \iota_! 1_{X} = 1_{Z}$, recall that we see  $\iota_! 1_{X}$ and  $F^*(\s{O}(1)) $ as maps $X\to Y$ and $X\to Z$ respectively, so that $1_{Z}$ is seen as a map $X \to Z$. Since $g: Z \to X$ is cohomologically smooth, $1_{Z}$ is a right adjoint as a map $X \to Z$ with left adjoint given by $ g^! 1_{X} = \bigwedge^d \s{F}^{\vee} [-d]$. One deduces that 
\[
\iota_! 1_{X}  \star( F^*(\s{O}(-1)) \otimes \bigwedge^d \s{F}^{\vee}[-d] ) = \bigwedge^d \s{F}^{\vee}[-d].
\]

\item[(iii)] Let us consider $ f_! 1_X \in \Mod_{\sol}(Z)$ as a morphism $X\to Z$ in the Lu-Zheng category over $X$. Since $f$ is cohomologically smooth, it is a left adjoint and has by right adjoint $f_!  f^!1_{X}= f_! \bigwedge^d \s{F}[d]$. Consider the fiber product $ X \times_X  Z \times_X Y= Z \times_X Y$, with projections $\pi_Y$ and $\pi_Z$, then 
\[
(F^*(\s{O}(-1))\otimes \bigwedge^d \s{F}^{\vee}[-d] ) \star f_! f^!1_{Z} =  \pi_{Y,!}( (F^*(\s{O}(-1))\otimes \bigwedge^d \s{F}^{\vee}[-d] ) \otimes \pi_{Z}^* (f_! f^! 1_{Z})).
\]
We have a cartesian square
\[
\begin{tikzcd}
 \widehat{\bb{V}(\s{F}^{\vee})} \times_X Y  \ar[d] \ar[r, "f'"] & Z\times_X Y \ar[d, "\pi_{Z}"] \\ 
X \ar[r, "f"]& Z.
\end{tikzcd}
\]
By proper and smooth base change we have a natural equivalence $ \pi_{Z}^* f_! f^! 1_{Z} =  f'_!  f^{'!}1_{Y\times_X Z}$, which by  Proposition \ref{PropAlgCartier1} (2) yields that 
\[
(F^*(\s{O}(-1))\otimes \bigwedge^d \s{F}^{\vee}[-d] ) \otimes \pi_{Z}^* f_! f^! 1_{Z}  =  (F^*(\s{O}(-1))\otimes \bigwedge^d \s{F}^{\vee}[-d] ) \otimes  \Sym (\s{F}^{\vee}). 
\]
where $\Sym (\s{F}^{\vee})$ is endowed with the \textbf{left} regular comodule structure. Equivalently, consider the map 
\[
Y \times_{X} Z \times_X Y \xrightarrow{\pi_{2,3}} Z\times_X Y, 
\]
then 
\[
\begin{aligned}
(F^*(\s{O}(-1))\otimes \bigwedge^d \s{F}^{\vee}[-d] ) \otimes \pi_{Z}^* f_! f^! 1_{Z} &   = (F^*(\s{O}(-1))\otimes \bigwedge^d \s{F}^{\vee}[-d] ) \otimes  \pi_{2,3,!}\pi_{1,2}^* F^*(\s{O}(1)) \\  
 &=  \pi_{2,3,!}(  \pi_{2,3}^*(F^*(\s{O}(-1))\otimes \bigwedge^d \s{F}^{\vee}[-d] ) \otimes  \pi_{1,2}^*F^*(\s{O}(1))).
\end{aligned}
\]
One gets that 
\[
\begin{aligned}
(F^*(\s{O}(-1))\otimes \bigwedge^d \s{F}^{\vee}[-d] ) \star f_!f^!1_{Z} & = \pi_{Y,!}\pi_{2,3,!}(  \pi_{2,3}^*(F^*(\s{O}(-1))\otimes \bigwedge^d \s{F}^{\vee}[-d] ) \otimes  \pi_{1,2}^*F^*(\s{O}(1))) \\ 
& = \pi_{Y,!} \pi_{1,3,!}(  \pi_{2,3}^*(F^*(\s{O}(-1))\otimes \bigwedge^d \s{F}^{\vee}[-d] ) \otimes  \pi_{1,2}^*F^*(\s{O}(1))) \\ 
& = \pi_{Y,!} (F^*(\s{O}(-1))\otimes \bigwedge^d \s{F}^{\vee}[-d] ) \star F^*(\s{O}(1)) \\ 
& = \pi_{Y,!} \Delta_{Y,!} 1_{Y} \\ 
& = 1_{Y},
\end{aligned}
\]
proving what we wanted.

\item[(iv)] For the last identity, we take right adjoints to the identity $(F^*(\s{O}(-1))\otimes \bigwedge^d \s{F}^{\vee}[-d] ) \star f_! f^!1_{Z} = 1_{Y}$. Indeed, $1_{Y}$ is seen as a map $X \to Y$ in the Lu-Zheng category over $X$, and since $p:Y \to X$ is weakly cohomologically proper, $1_{Y}$ is a left adjoint with right adjoint given by itself. Similarly,  since $f:X \to Z$  is cohomologically smooth,  $f_! f^! 1_{Z}$ is a left adjoint when seen as a map $X \to Z$, with right adjoint given by $f_! 1_{Z}$. One obtains the identity 
\[
f_! 1_{Z} \star F^*(\s{O}(1)) = 1_{Y}.
\]

\end{itemize}
\end{proof}

\begin{corollary}
\label{CoroFourierMoukaiIdentities}
Keep the notation of Proposition \ref{PropConvolutionFMTransforms}. 

\begin{enumerate}
\item  In the conventions of part (1) there are natural equivalences of functors 
\begin{itemize}

\item[(i)]  $FM_1\circ  \iota_* = g^*$.

\item[(ii)] $ \iota^* \circ FM_1^{-1}=g_!(-\otimes g^!1_{X})$.

\item[(iii)]  $FM_1^{-1} \circ  f_!(- \otimes f^! 1_{Z})= p^*$. 

\item[(iv)] $f^* \circ FM_1= p_!$.

\end{itemize}

\item In the conventions of part (2) there are natural equivalences of functors
\begin{itemize}

\item[(i)] $ FM_2 \circ \iota^* =f_!$.
 
\item[(ii)] $ \iota_* \circ FM_2^{-1}= g_!(-\otimes g^! 1_{X})$.

\item[(iii)] $ FM_2^{-1} \circ f_!=p^*$.

\item[(iv)] $ f^* \circ FM_2= p_!(-\otimes p^! 1_{X})$.

\end{itemize}

\end{enumerate}

\end{corollary}
\begin{proof}
This follows by translating the kernels in the Lu-Zheng category over $X$ to their associated functors by convolution. 
\end{proof}

\begin{example}
\label{ExampleBeilinsontstructure}
We now explain the relation of Cartier duality and the Beilinson $t$-structure.  Consider the action of $\bb{G}_m$ on $\bb{G}_a$ by multiplication, the map $\bb{G}_a/ \bb{G}_m \to B\bb{G}_m$ is the standard line bundle over $B\bb{G}_m$ and its Cartier dual over $B\bb{G}_m$ is the quotient stack $B(\bb{G}_m \ltimes \widehat{\bb{G}}_a)$ where $\bb{G}_m$ acts on $\widehat{\bb{G}}_a$ by multiplication. Similarly, the Cartier dual of $\widehat{\bb{G}}_a/ \bb{G}_m$ is   $B(\bb{G}_m \ltimes \bb{G}_a)$.  The category $\Fil(\Mod(\bb{Q}))$ has two different $t$-structures, the standard and the Beilinson $t$-structure, it turns out that they are actually the natural $t$-structures of the modules over the stacks $\bb{G}_a/ \bb{G}_m$ and $B(\bb{G}_m \ltimes \widehat{\bb{G}_a})$ under the Cartier duality isomorphism respectively.
\end{example}

\subsubsection{Solid vector bundles}

We finish the section with some short discussion about a variant of vector bundles for solid $\s{D}$-stacks, namely, solid vector bundles:

\begin{definition}
We let $\GL_{d,\sol}$ be the analytic spectrum of the ring $\bb{Z}[X_{i,j},Y: 1\leq i,j\leq d]_{\sol}/(Y\det(X_{i,j})-1)$. The category of \textit{solid} vector bundles of rank $d$ on solid $\s{D}$-stacks is the slice category $\Sh_{\s{D}}(\Aff_{\bb{Z}_{\sol}})_{/ [*/\GL_{d,\sol}]}$.  A solid vector bundle is denote by $\s{F}_{\sol}$, where $\s{F}$ is the underlying vector bundle associated to the composite $X\to */\GL_{d,\sol}\to */ \GL_d$.  Let $\St$ be the standard representation of $\GL_d$, then $(\Sym^{\bullet}_{\bb{Z}} \St^{\vee})_{\sol}$ has a natural action of $\GL_{d,\sol}$, which defines an analytic space $\bb{V}(\St)_{\sol}\to */\GL_d$. For a solid vector bundle $\s{F}_{\sol}$ over a stack $X$, we let $(\Sym^{\bullet}_{X} \s{F}^{\vee})_{\sol}$ and $\bb{V}(\s{F})_{\sol}$ be the pullback of   $(\Sym^{\bullet}_{\bb{Z}} \St^{\vee})_{\sol}$ and $\bb{V}(\St)_{\sol}$ along $X\to */\GL_d$ respectively. 
\end{definition}

We have a partial analogue of Proposition \ref{PropAlgCartier1}

\begin{prop}
\label{LemmaSerreDualityGLn}
Let $X$ be a solid stack over $\bb{Z}_{\sol}$ and $\s{F}_{\sol}$ a solid vector bundle of rank $d$. Let $f: \bb{V}(\s{F})_{\sol} \to X$. Then $f$ is cohomologically smooth and there are  natural equivalences $f^{!} 1_X \cong f^{*} \bigwedge^{d} \s{F}^{\vee} [d]$. 
\end{prop}

\begin{proof}
This follows the same proof of Proposition \ref{PropAlgCartier1} after taking some modifications on the idempotent algebras $D_i$. Indeed, using the same notation as \textit{loc. cit.} consider the idempotent $\bb{Z}[\underline{T}]$-algebras $D_i= \bb{Z}[[T_{i}^{-1}]][\underline{T}]$, and let $C$ be its union in the sense of locale. Then,  $\bb{V}(\St)_{\sol}$ is the open complement of $C$ in $\bb{V}(\St)$, and we can compute 
\[
f_! 1_{\bb{V}(\St)_{\sol}} = [\bb{Z}[\underline{T}]\to C].
\]
An explicit calculation gives that 
\[
f_{!} 1_{\bb{V}(\St)_{\sol}} = (T_1\cdots T_d)^{-1} \bb{Z}[[T_{1}^{-1},\cdots, T_{d}^{-1}]][-d]. 
\]
Taking duals  one finds that 
\[
f_* f^{1} 1_{X}= (T_1\cdots T_d) \otimes \bb{Z}[\underline{T}] [d],
\]
and that $f^{!} 1_X = f^{*} \bigwedge^{d} \St^{\vee}[d]$. 
\end{proof}

\subsection{Analytic Cartier duality for vector bundles}
\label{SubsectionAnalyticCartier}

We have proven an algebraic Cartier duality for vector bundles, in this section we shall study three additional incarnations of this phenomena in rigid geometry for Tate stacks  over $\bb{Q}_p$, for some fixed prime $p$. Nevertheless, some of the constructions and statements still make sense for Tate stacks over $(R,R^{+})=(\bb{Z}((\pi)), \bb{Z}[[\pi]])$, we will make explicit this distinction when necessary.

\subsubsection{Cartier duality for unitary overconvergent vector bundles}

In Proposition \ref{PropositionEquivalenceVBTorsors} we saw that the category of vector bundles of rank $d$ on solid $\s{D}$-stacks is equivalent to the category of $\s{D}$-stacks over $*/\GL_d$. Therefore, in order to construct different incarnations of vector bundles it suffices to construct different incarnations of the group $\GL_d$.

\begin{definition}
Let $R \langle  T \rangle^{\dagger}= \varinjlim_{\epsilon\to 0^{+}} R \langle \pi^{\epsilon} T\rangle $ be the overconvergent algebra defining the closed disc of radius $1$.

\begin{enumerate}
\item  We define the overconvergent linear group $\GL_d^{\dagger}$ to be the analytic spectrum of the algebra 
\[
R \langle X_{i,j},T: 1 \leq i,j\leq d \rangle^{\dagger} / (\det(X_{ij}) T-1)
\]
representing invertible matrices $A$ such that $|A| \leq |\pi^{-\epsilon}|$  and $|A^{-1}| \leq |\pi^{-\epsilon}|$ for all $\epsilon>0$. 

\item We define the category of unitary overconvergent vector bundles of rank $d$ on analytic $\s{D}$-stacks over $R_{\sol}$ to be the slice category $\Sh_{\s{D}}(\Aff^{b}_{R_{\sol}})_{/ [*/ \GL_d^{\dagger}]}$.  

\item Given an analytic $\s{D}$ stack $X$ over $R_{\sol}$, and a vector bundle $\s{F}$ of rank $d$ defined by a map $f: X \to */ \GL_d$, a \textit{lattice} $\s{F}^+$ of $\s{F}$ is a factorization 
\[
X \to */ \GL_d^{\dagger} \to */ \GL_d.
\]
We also say that $\s{F}^+$ is an \textit{unitary  overconvergent vector bundle} over $X$.
\end{enumerate}
\end{definition}

\begin{construction}
Let $X$ be an Tate stack over $R_{\sol}$. Let $f:X \to */ \GL_d^{\dagger}$ be a unitary overconvergent vector bundle of rank $d$, and let $\s{F}$ denote the vector bundle associated to the composite $X \to */ \GL_d^{\dagger} \to */ \GL_d$.  Similarly as for algebraic vector bundles, we can construct two different geometric incarnations that are analogue to $\bb{V}(\s{F})$ and $\widehat{\bb{V}(\s{F})}$.  Let $\St$ be the standard representation of $\GL_d^{\dagger}$ with canonical basis $e_1,\ldots, e_d$. Let $\overline{\bb{V}(\St^+)}\subset \bb{V}(\St)$ be the closed subspace given by the analytic spectrum of $R \langle e_1^{\vee}, \ldots, e_d^{\vee} \rangle^{\dagger}$ where $e_{i}^{\vee} \in \St^{\vee}$ is the dual basis. By construction, $ \overline{\bb{V}(\St^+)} \subset \bb{V}(\St)$ admits a descent datum for the action of $\GL_d^{\dagger}$, and thus it defines an analytic space over $*/ \GL_d^{\dagger}$. We define $\overline{\bb{V}(\s{F}^+)}:= f^* \overline{\bb{V}(\St^+)}$ and call it the \textit{closed  overconvergent ball of radius $1$} in $\bb{V}(\s{F})$. Dually, let $\mathring{\bb{V}}(\s{F}^+):= \bigcup_{\epsilon>0} \pi^{\epsilon} \overline{\bb{V}(\s{F}^+)}$  be the unit open ball in $\bb{V}(\s{F})$. We have a series of inclusions 
\[
\widehat{\bb{V}(\s{F})} \subset  \mathring{\bb{V}}(\s{F}^+) \subset \overline{\bb{V}(\s{F}^+)} \subset \bb{V}(\s{F}). 
\] 
\end{construction}

\begin{remark}
The map $\mathring{\bb{V}}(\s{F}^+) \to \bb{V}(\s{F})$ is an open analytic inclusion, namely, locally in the $\s{D}$-topology it is nothing but the inclusion of the open unit polydisc in the algebraic  affine space. In particular, we can always localize modules over $\Sym^{\bullet}_X(\s{F}^{\vee})$ to solid sheaves over $\mathring{\bb{V}}(\s{F}^+)$. 
\end{remark}

\begin{definition}
Let $\s{F}^+$ be an overconvergent vector bundle over $X$. We let $\Sym^{\dagger}_X(\s{F}^{\vee,+})$ denote the algebra of functions of $\overline{\bb{V}(\s{F}^+)}$ seen as an object in $\Mod_{\sol}(X)$.  We also let $\mathring{\Sym}_X(\s{F}^{\vee,+})$ denote the global sections over $X$ of $\mathring{\bb{V}}(\s{F}^{+})$. 
\end{definition}

We now prove the analogue of Proposition \ref{PropAlgCartier1} 

\begin{prop}
\label{PropositionDAggerCartier1}
Let $X$ be an analytic $\s{D}$-stack over $R$ and let $\s{F}^+$ be a unitary overconvergent vector bundle over $X$. 
\begin{enumerate}

\item  The map $\overline{\bb{V}(\s{F}^+)}\to X$ is weakly cohomologically proper. 

\item  The map $f: \mathring{\bb{V}}(\s{F}^+) \to X$ is cohomologically smooth and there is a natural isomorphism $f^! 1_{X}= f^* \bigwedge^d \s{F}^{\vee}[d]$. If in addition $X$ is defined over $\bb{Q}_p$ then   $f_! 1_{\mathring{\bb{V}}(\s{F}^+)} = \bigwedge^d \s{F} \otimes \Sym^{\dagger}_X(\s{F}^+) [-d]$. In particular, $f_! f^! 1_{X}= \Sym^{\dagger}_X(\s{F}^+)$.

\end{enumerate} 
\end{prop}
\begin{proof}
Part (1) is clear since $\overline{\bb{V}(\s{F}^+)}$ is the relative analytic spectrum of the $\s{O}_X$-algebra $\Sym^{\dagger}_X(\s{F}^+)$ endowed with the induced analytic structure. 

For part (2), smoothness of $f$ follows from Proposition \ref{PropGeoSmoothIsSmooth}.  The computation of $f^{!} 1_X$ and  $f_! 1_{\mathring{\bb{V}}(\s{F}^+)}$ follows the same lines of the proof of Propositions \ref{PropAlgCartier1} and \ref{LemmaSerreDualityGLn}: we first reduce to the univesal case $X=*/\GL_d^{\dagger}$ and $\s{F}^+= \St^+$, we let $e_1,\ldots, e_d$ be the standard basis of $\St^+$ with dual basis $T_1,\ldots, T_d$. We then define the idempotent $R[T_1,\ldots, T_d]$-algebras $D_i= R\langle T_i^{-1} \rangle^{\dagger}[T_1,\ldots, T_d]$. The same computations will show that 
\[
f_! 1_{\mathring{\bb{V}}(\St^+)}= (T_1\cdots T_d)^{-1} R \langle T_1^{-1},\ldots, T_d^{-1} \rangle^{\dagger}[-d] = \bigwedge^{d} \St \otimes (\mathring{\Sym}_X(\St^{\vee}))^{\vee}[-d].
\]
If $X$ is defined over $\bb{Q}_p$ this is precisely  $\bigwedge^d \St \otimes \Sym^{\dagger}_X(\St^+)[-d]$. In general,     taking duals one gets that 
\[
f_* f^! 1_X = \bigwedge^{d} \St^{\vee}\otimes  \mathring{\Sym}_X(\St^{\vee})[d],
\]
localizing at $\mathring{\bb{V}}(\s{F}^{+})$ one gets that $f^{!} 1_X = f^{*} \bigwedge^d \St^{\vee} [d]$ as wanted. 
\end{proof}

Before giving a proof of the analogue of Proposition  \ref{PropAlgebraicCartier2},  we  need to find suitable de Rham and Koszul resolutions as in  Lemma \ref{LemmaDeRhamKoszulAlgebraic}.

\begin{lemma}
\label{LemmaDeRhamKoszulDagger}
 Let $X$ be an Tate stack over $\bb{Q}_p$ and let $\s{F}^+$ be a unitary overconvergent vector bundle of rank $d$ over $X$.

\begin{enumerate}

\item   There is a natural de Rham resolution of $\s{O}_X$ as $\Sym^{\dagger}_X(\s{F}^+)$-comodule  given by  the complete decreasing filtration
\[
 \s{O}_X \to \Sym^{\dagger}_X(\s{F}^{\vee,+}) \xrightarrow{d} \Sym^{\dagger}_X(\s{F}^{\vee,+}) \otimes \s{F}^{\vee} \xrightarrow{d} \cdots \xrightarrow{d} \Sym^{\dagger}_X(\s{F}^{\vee,+}) \otimes \bigwedge^d \s{F}^{\vee}. 
\]

\item There is a natural Koszul resolution of $\s{O}_X$ as $\Sym^{\dagger}_X(\s{F}^+)$-module  given by the complete increasing filtration
\[
\Sym^{\dagger}_X(\s{F}^{+})\otimes \bigwedge^d \s{F} \to \cdots \to \Sym^{\dagger}_X(\s{F}^+) \otimes \s{F} \to \Sym^{\dagger}_X(\s{F}) \to \s{O}_X,
\]
whose dual is the de Rham complex for $\mathring{\Sym}_X(\s{F}^{\vee})$. 
\end{enumerate}

\end{lemma}
\begin{proof}
By base change it suffices to treat the universal case $X=*/\GL_d^{\dagger}$ and $\s{F}^+= \St^+$. We have $\GL_d^{\dagger}$-equivariant de Rham  and Koszul complexes 
\begin{equation}
\label{eqdeRhamKozClosedDisc}
\begin{aligned}
0 \to \bb{Q}_p \to \Sym^{\dagger}_X( \St^{+,\vee}) \xrightarrow{d} \Sym^{\dagger}_X( \St^{+,\vee}) \otimes \St^{\vee} \xrightarrow{d} \cdots  \xrightarrow{d} \Sym^{\dagger}_X(\St^{+,\vee}) \otimes \bigwedge^d \St^{\vee} \to 0 \\
0\to \Sym^{\dagger}_X(\St^{\vee,+})\otimes \bigwedge^d \St^{\vee} \to \cdots \to \Sym^{\dagger}_X(\St^{\vee,+})\otimes \St^{\vee} \to \Sym^{\dagger}_X(\St^{\vee,+}) \to \bb{Q}_p \to 0
\end{aligned}
\end{equation}
and 
\begin{equation}
\label{eqdeRhamKozOpenDisc}
\begin{aligned}
0 \to \bb{Q}_p \to \mathring{\Sym}_X( \St^{+}) \xrightarrow{d} \mathring{\Sym}_X( \St^{+}) \otimes \St \xrightarrow{d} \cdots  \xrightarrow{d} \mathring{\Sym}_X( \St^{+})  \otimes \bigwedge^d\St \to 0 \\
0\to \mathring{\Sym}_X( \St^{+}) \otimes \bigwedge^d \St  \to \cdots \to \mathring{\Sym}_X( \St^{+}) \otimes \St \to \mathring{\Sym}_X( \St^{+}) \to \bb{Q}_p \to 0.
\end{aligned}
\end{equation}
By the Poincar\'e lemma for open unit polydiscs \cite[Lemma 26]{Tamme}, the de Rham sequences in both \eqref{eqdeRhamKozClosedDisc} and \eqref{eqdeRhamKozOpenDisc} are exact (one can write $\Sym^{\dagger}_X(\St^{\vee,+})$ as filtered colimit of functions in open unit polydiscs).  Moreover, the Koszul resolutions in both equations are duals to the de Rham resolutions via the naive duality between nuclear Fr\'echet and $LB$ spaces of compact type, cf. \cite[Theorem 3.40]{RRLocallyAnalytic}, one deduces that the Koszul resolutions are also exact. Then, the de Rham complex of \eqref{eqdeRhamKozClosedDisc} is a complex of $\Sym^{\dagger}_X(\St^{\vee,+})$-comodules  proving (1). Similarly, the Koszul resolution of \eqref{eqdeRhamKozOpenDisc} is also a complex of $\Sym^{\dagger}_X(\St^{+})$-modules obtaining (2). 
\end{proof}

\begin{prop}
Let $X$ be an analytic $\s{D}$-stack over $\bb{Q}_p$ and let $\s{F}^+$ be a unitary overconvergent  vector bundle of rank $d$ over $X$. 

\begin{enumerate}

\item There are natural equivalences 
\[
\Mod(X/\overline{\bb{V}(\s{F})})=\mathrm{CoMod}_{\Sym^{\dagger}_X(\s{F}^{\vee,+})}(\Mod(X))
\]
and
\[
\Mod(X/\bb{V}(\s{F}))=\Mod_{\Sym^{\dagger}_X(\s{F}^+)}(\Mod(X)).
\]

\item  Consider the maps $X \xrightarrow{f} X/\overline{\bb{V}(\s{F}^+)} \xrightarrow{g} X$. Then $f$ is a descendable $\s{D}$-cover and $g$ is both weakly cohomologically proper and cohomologically smooth. Moreover, there is a natural equivalence $g^! 1_{X}= \bigwedge^d \s{F}[d]$.

\item  Consider the maps $X\xrightarrow{f} X/ \mathring{\bb{V}}(\s{F}^+) \xrightarrow{g} X$. Then $f$ is a smooth $\s{D}$-cover,  $g$ is cohomologically smooth, and there is a natural equivalence $g^! 1_{X}= \bigwedge^d \s{F}[-d]$. Moreover, $g$ is co-smooth with proper dual $\n{P}_{g}(1_{X/\mathring{\bb{V}}(\s{F}^+)})=1_{X/\mathring{\bb{V}}(\s{F}^+)}[-2d]$.

\end{enumerate}
\end{prop}
\begin{proof}
The proof is exactly the same of Proposition \ref{PropAlgebraicCartier2} where Proposition \ref{PropAlgCartier1} is replaced by Proposition \ref{PropositionDAggerCartier1}, and the de Rham and Koszul resolutions are those of Lemma \ref{LemmaDeRhamKoszulDagger}. 
\end{proof}

\begin{theorem}[Cartier duality for open and closed discs]
\label{TheoremCartierDualityAnalyticI}
Let $X$ be an analytic $\s{D}$ over $\bb{Q}_p$ and $\s{F}$ a vector bundle of rank $d$ over $X$. 

\begin{enumerate}

\item  There is a natural bi-linear map
\[
F: \overline{\bb{V}(\s{F}^+)} \times_X X/ \mathring{\bb{V}}(\s{F}^{\vee,+}) \to */ \bb{G}_m
\]
such that $F^*(\s{O}(1))$ is an isomorphism in the Lu-Zheng category, considered in $\LZ_{X}(\overline{\bb{V}(\s{F}^+)}, X/ \mathring{\bb{V}}(\s{F}^{\vee,+}))$.  Furthermore, the inverse of $F^*(\s{O}(1))$ is naturally isomorphic to $F^*(\s{O}(-1)) \otimes_{\s{O}_X} \bigwedge^d \s{F}^{\vee}[-d]$.

\item There is a natural bi-linear map 
\[
G: \mathring{\bb{V}}(\s{F}^{\vee,+}) \times_X X/\overline{\bb{V}(\s{F}^+)} \to */ \bb{G}_m,
\]
such that $G^*(\s{O}(1))$ is an isomorphism in the Lu-Zheng category considered in $\LZ_X(\mathring{\bb{V}}(\s{F}^{\vee,+}), X/\overline{\bb{V}(\s{F}^+)})$. Furthermore, the inverse of $G^*(\s{O}(1))$ is naturally isomorphic to $G^*(\s{O}(-1)) \otimes \bigwedge^d \s{F}[d]$.
\end{enumerate}
In particular, we have  analogue Fourier-Moukai isomorphisms as in \eqref{eqFM1} and \eqref{eqFM2}. Moreover, the analogues of Proposition \ref{PropConvolutionFMTransforms} and Corollary \ref{CoroFourierMoukaiIdentities} hold.  
\end{theorem}
\begin{proof}
The proof is totally analogue to the proof of Theorem \ref{TheoAlgebraicCartierDuality}, we only explain the construction of the vector bundles $F^*(\s{O}(1))$ and $G^*(\s{O}(1))$. By functoriality we can always reduce to the universal case $X=*/\GL_d^{\dagger}$ and $\s{F}^+= \St^+$.  For $F$, we need to construct a line bundle on $\overline{\bb{V}(\s{F}^+)}/ \mathring{\bb{V}}(\s{F}^{\vee,+})$, where the quotient is for the trivial action.  We take $F^*(\s{O}(1))$ as the line bundle $\Sym^{\dagger}_X(\St^{+,\vee})$ endowed with the (left) multiplication map 
\[
\Sym^{\dagger}_X(\St^{+,\vee}) \otimes \Sym^{\dagger}_X(\St^{+,\vee}) \to \Sym^{\dagger}_X(\St^{+,\vee}). 
\]
Similarly, $G^*(\s{O}(1))$ is the line bundle on $\mathring{\bb{V}}(\St^{\vee,+})/ \overline{\bb{V}(\St)}$ induced by the localization of $\s{O}(\mathring{\bb{V}}(\St^{\vee,+}))$  on $\mathring{\bb{V}}(\St^{\vee,+})$ endowed with the $\Sym^{\dagger}_X(\St^{+,\vee})$-comodule structure given by the adjoint of the (left) multiplication map
\[
\s{O}(\mathring{\bb{V}}(\St^{\vee,+})) \otimes_{\s{O}_X} \s{O}(\mathring{\bb{V}}(\St^{\vee,+})) \to \s{O}(\mathring{\bb{V}}(\St^{\vee,+})). 
\]
Finally, the last statement follows by the analogue computations of Proposition \ref{PropConvolutionFMTransforms} and Corollary \ref{CoroFourierMoukaiIdentities}. 
\end{proof}

\subsubsection{Cartier duality for analytic vector bundles}

We now state a Cartier duality for the analytification of vector bundles. The restriction of the algebraic group $\GL_d$ to Tate stacks over $R$ is represented by its analytification $\GL_d^{\an}$ as an adic space, see Definition \ref{DefinitionAnalytificationFunctor}. We call $\Sh_{\s{D}}(\Aff^{b}_{R_{\sol}})_{/[*/\GL_d^{\an}]}$ the category of \textit{analytic vector bundles of rank $d$}, note however that the data of an analytic vector bundle for a Tate stack is the same as the data of an algebraic vector bundle (this is not true for  general solid stacks).  

\begin{construction}
\label{ConstructionAnalytificationVB}
Let $X$ be an Tate stack over $X$,  and let $f:X\to */ \GL_d^{\an}$ be an analytic vector bundle of rank $d$. Let $\St$ be the standard vector bundle over $*/ \GL_d^{\an}$, then the analytification of the algebra $\Sym^{\bullet}_X (\St^{\vee})$ admits a group action of $\GL_d^{\an}$ and defines an analytic space over $*/\GL_d^{\an}$ that we denote by $\bb{V}(\St)^{\an}$. We let $\bb{V}(\s{F})^{\an}:= f^* \bb{V}(\St)^{\an}$  be the \textit{analytification} of $\bb{V}(\s{F})$. Dually, let $\iota:X \to \bb{V}(\s{F})$ be the zero section, we let $\bb{V}(\s{F})^{\dagger}$ be the overconvergent neighbourhood of $\iota$, equivalently, we let 
\[
\bb{V}(\s{F})^{\dagger} = \bigcap_{\epsilon \to \infty} \pi^{\epsilon} \overline{\bb{V}(\s{F}^+)}
\]
for any lattice  $\s{F}^+ \subset \s{F}$ that exists locally in the $\s{D}$-topology. 
\end{construction}

\begin{definition}
Let $X$ be a Tate  stack over $R_{\sol}$ and  let $\s{F}$  be a vector bundle of rank $d$ over $X$. We let $\Sym^{\dagger}_X(\s{F}^{\vee})$ be the  algebra over $\s{O}_X$  defining the analytic space $\bb{V}(\s{F})^{\dagger}$. Similarly, we let $\Sym^{\an}_X(\s{F}^{\vee})$ denote the global  sections  over $X$ of $\bb{V}(\s{F})^{\an}$.
\end{definition}

\begin{prop}
\label{PropositionDAggerCartier2}
Let $X$ be an analytic $\s{D}$-stack over $R_{\sol}$ and let $\s{F}$ be an analytic vector bundle over $X$. 
\begin{enumerate}

\item  The map $\bb{V}(\s{F})^{\dagger} \to X$ is weakly cohomologically proper. 

\item  The map $f: \bb{V}(\s{F})^{\an} \to X$ is cohomologically smooth and there is a natural isomorphism $f^! 1_{X}= f^* \bigwedge^d \s{F}^{\vee}[d]$. If in addition $X$ is defined over $\bb{Q}_p$, then there is a natural isomorphism  $f_! 1_{\mathring{\bb{V}}(\s{F}^+)} = \bigwedge^d \s{F} \otimes \Sym^{\dagger}_X(\s{F}) [-d]$. In particular, $f_! f^! 1_{X}= \Sym^{\dagger}_X(\s{F})$ and the localization of $\Sym^{\dagger}_X(\s{F})$ in $\bb{V}(\s{F})^{\an}$ is $f^! 1_{X}$.

\end{enumerate} 
\end{prop}
\begin{proof}
Part (1) is clear since $\bb{V}(\s{F})^{\dagger}$ is the relative analytic spectrum of the $\s{O}_X$-algebra $\Sym^{\dagger}_X(\s{F})$ endowed with the induced analytic structure. 

For part (2), smoothness of $f$  smoothness follows from Proposition \ref{PropGeoSmoothIsSmooth}.  The computation of $f^{!} 1_X$ and  $f_! 1_{\mathring{\bb{V}}(\s{F}^+)}$ follows the same lines of the proof of Propositions \ref{PropAlgCartier1} and  \ref{LemmaSerreDualityGLn} where we use the idempotent $R[T]$-algebra $R\{T^{-1}\}^{\dagger}[T]$ instead, namely, the analytification $\bb{A}^{1,\an}_{R}\subset \bb{A}^{1,alg}_R$ is the complement of the idempotent algebra $R\{T^{-1}\}^{\dagger}[T]$ .  We leave the details to the reader.
\end{proof}

As we saw before, a key point in the construction of the Cartier duality is having available the de Rham and Koszul resolutions:

\begin{lemma}
\label{LemmaDeRhamKoszulDagger2}
Let $X$ be an Tate stack over $\bb{Q}_p$ and let $\s{F}$ be a vector bundle of rank $d$ over $X$.
\begin{enumerate}
\item  There is a natural de Rham resolution as $\Sym^{\dagger}_X(\s{F}^{\vee})$-comodule  given by the complete decreasing filtration
\[
 \s{O}_X \to \Sym^{\dagger}_X(\s{F}^{\vee}) \xrightarrow{d} \Sym^{\dagger}_X(\s{F}^{\vee}) \otimes \s{F}^{\vee} \xrightarrow{d} \cdots \xrightarrow{d} \Sym^{\dagger}_X(\s{F}^{\vee}) \otimes \bigwedge^d \s{F}^{\vee}. 
\]

\item There is a natural Koszul resolution as $\Sym^{\dagger}_X(\s{F}^{\vee})$-module  given by the complete increasing filtration
\[
\Sym^{\dagger}_X(\s{F})\otimes \bigwedge^d \s{F} \to \cdots \to \Sym^{\dagger}_X(\s{F}) \otimes \s{F} \to \Sym^{\dagger}_X(\s{F}) \to \s{O}_X,
\]
 whose dual is the  Rham complex for $\Sym^{\an}_X(\s{F}^{\vee})$.
\end{enumerate}
\end{lemma}
\begin{proof}
The same proof of Lemma \ref{LemmaDeRhamKoszulDagger} applies; note that the restriction to Tate stacks over $\bb{Q}_p$ is for the Poincar\'e lemma to hold. 
\end{proof}
\begin{theorem}[Analytic Cartier duality for vector bundles]
\label{TheoremAnalyticCartierII}
Let $X$ be a Tate stack over $\bb{Q}_p$ and  let $\s{F}$  be a vector bundle of rank $d$ over $X$. Then the analogue of Propositions  \ref{PropAlgebraicCartier2} and  \ref{PropConvolutionFMTransforms}, Theorem \ref{TheoAlgebraicCartierDuality} and Corollary \ref{CoroFourierMoukaiIdentities} hold by replacing the following objects:
\begin{itemize}
\item $\bb{V}(\s{F})$ for $\bb{V}(\s{F})^{\dagger}$,

\item $\widehat{\bb{V}(\s{F}^{\vee})}$ for $\bb{V}(\s{F}^{\vee})^{\an}$,

\item $\Sym^{\bullet}_X \St^{\vee}$ for $\Sym^{\dagger}_X \St^{\vee}$,

\item $\widehat{\Sym}_X(\St)$ by $\Sym^{\an}_X(\St)$. 
\end{itemize}
\end{theorem}
\begin{proof}
The proof follows the same lines of the cited references after some minor adaptations, we left the details to the reader.  
\end{proof}

\subsubsection{Cartier duality for locally analytic $\bb{Z}_p$-vector bundles}

We finish this section with a new Cartier duality that is closely related with the theory of solid locally analytic representations of \cite{RJRCSolidLocAn2}. Let $\bb{Z}_p$ be the ring of $p$-adic integers seen as a $p$-adic Lie group, we let $C^{la}(\bb{Z}_p, \bb{Q}_p)$ denote the space of locally analytic functions of $\bb{Z}_p$, and denote by $\bb{Z}_p^{la}$ its analytic spectrum.   Let $\n{D}^{la}(\bb{Z}_p,\bb{Q}_p):= \iHom_{\bb{Q}_p}(C^{la}(\bb{Z}_p,\bb{Q}_p),\bb{Q}_p)$ be the locally analytic distribution algebra of $\bb{Z}_p$. By a theorem of Amice,  the algebra $\n{D}^{la}(\bb{Z}_p,\bb{Q}_p)$ is isomorphic to the global sections of the open unit disc centered at $1$, namely  $\widehat{\bb{G}}_{m,\eta}:=1+\mathring{\bb{G}}_a\subset \bb{G}_m^{\an}$. The algebra $\n{D}^{la}(\bb{Z}_p,\bb{Q}_p)$ can be written as a limit of analytic distribution algebras $\n{D}^{h}(\bb{Z}_p,\bb{Q}_p)$,  which are the dual of functions of the rigid group $\bb{Z}_p+ p^{h}\mathring{\mathbb{G}}_{a,\bb{Q}_p}\subset \bb{G}^{\an}_{a,\bb{Q}_p}$. The rings $\n{D}^{h}(\bb{Z}_p,\bb{Q}_p)$ correspond to suitable closed overconvergent discs in $\widehat{\bb{G}}_{m,\eta}$ of radius $p^{-b(h)}$ with $b(h) \to 0^+$ as $h \to \infty$.   By \cite[Theorem 4.1.7]{RJRCSolidLocAn2}, there is a natural equivalence between the category of solid locally analytic representations of $\bb{Z}_p$ and that of quasi-coherent shaves on $\widehat{\bb{G}}_{m,\eta}$. Moreover, under this equivalence $\bb{Z}_p+p^{h}\mathring{\mathbb{G}}_{a,\bb{Q}_p}$-analytic representations correspond to modules over $\n{D}^{h}(\bb{Z}_p, \bb{Q}_p)$, which is an idempotent algebra on $\widehat{\bb{G}}_{m,\eta}$. Our next goal is to improve this statement to a Cartier duality theorem in a relative setting for a suitable notion of locally analytic $\bb{Z}_p$-vector bundle. To make this concrete we first need a construction.  

\begin{lemma}
There is a natural action $\bb{Z}_p^{la} \times \widehat{\bb{G}}_{m,\eta} \to \widehat{\bb{G}}_{m,\eta}$ associated to the adjoint of the multiplication map $\n{D}^{la}(\bb{Z}_p,\bb{Q}_p)\otimes_{\bb{Q}_{p,\sol}}\n{D}^{la}(\bb{Z}_p,\bb{Q}_p) \to \n{D}^{la}(\bb{Z}_p,\bb{Q}_p)$ making $\widehat{\bb{G}}_{m,\eta} $ an $\bb{Z}_p^{la}$-module.  At the level of points this action corresponds to $(a,\chi)\mapsto \chi^{a}$. 
\end{lemma}
\begin{proof}
Let $h>0$ and let $\n{D}^{h}(\bb{Z}_p,\bb{Q}_p)$ be the $h$-analytic distribution algebra.  The multiplication map
\[
\n{D}^{h}(\bb{Z}_p,\bb{Q}_p) \otimes_{\bb{Q}_{p,\sol}}\n{D}^{h}(\bb{Z}_p,\bb{Q}_p) \to  \n{D}^{h}(\bb{Z}_p,\bb{Q}_p) 
\]
has by adjoint a map 
\[
\n{D}^{h}(\bb{Z}_p,\bb{Q}_p) \to \iHom_{\bb{Q}_p}(\n{D}^{h}(\bb{Z}_p,\bb{Q}_p) ,\n{D}^{h}(\bb{Z}_p,\bb{Q}_p)).
\]
Since the map $\n{D}^{h'}(\bb{Z}_p,\bb{Q}_p)  \to \n{D}^{h}(\bb{Z}_p,\bb{Q}_p) $ is of trace class for $h'>h$, the composite 
\[
\n{D}^{h'}(\bb{Z}_p, \bb{Q}_p)\to \iHom_{\bb{Q}_p}(\n{D}^{h'}(\bb{Z}_p,\bb{Q}_p) ,\n{D}^{h'}(\bb{Z}_p,\bb{Q}_p) ) \to \iHom_{\bb{Q}_p}(\n{D}^{h'}(\bb{Z}_p,\bb{Q}_p) ,\n{D}^{h}(\bb{Z}_p,\bb{Q}_p) )
\]
factors through $C^{la}(\bb{Z}_p,\bb{Q}_p)\otimes_{\bb{Q}_{p,\sol}} \n{D}^{h}(\bb{Z}_p,\bb{Q}_p)  $. Taking colimits as $h\to h^{',-}$, one gets a map 
\[
\n{D}^{h'}(\bb{Z}_p,\bb{Q}_p)   \to C^{la}(\bb{Z}_p,\bb{Q}_p)\otimes_{\bb{Q}_{p,\sol}} \n{D}^{h'}(\bb{Z}_p,\bb{Q}_p) . 
\]
One easily checks that this is a morphism of algebras that endows $\n{D}^{h'}(\bb{Z}_p,\bb{Q}_p) $ with a $C^{la}(\bb{Z}_p,\bb{Q}_p)$-comodule structure, namely, it is nothing but the orbit map of $\n{D}^{h'}(\bb{Z}_p,\bb{Q}_p) $ as locally analytic representation. One checks that these maps are compatible for $h'>0$, defining the $\bb{Z}_p^{la}$-module structure $\bb{Z}_p^{la} \times \widehat{\bb{G}}_{m,\eta} \to  \widehat{\bb{G}}_{m,\eta}$ as wanted. 
\end{proof}

\begin{construction}
We define the category of locally analytic  $\bb{Z}_p$-vector bundles of rank $d$ to be the slice category $\Sh_{\s{D}}(\Aff^{b}_{\bb{Q}_p})_{[*/ \GL_d(\bb{Z}_p)^{la}]}$ where $\GL_d(\bb{Z}_p)^{la}$ is the analytic group space associated to the Hopf algebra of  locally analytic functions of $\GL_d(\bb{Z}_p)$.  Let $\St^+= \bb{Z}_p^{d}$ be the  standard representation over $*/ \GL_{d}(\bb{Z}_p)^{la}$, then the locally analytic Lie group $\St^{+,la}$ has a natural action of  $*/ \GL_{d}(\bb{Z}_p)^{la}$ that defines an analytic space $\bb{V}(\St^{+,la})$.  For $X \to */ \GL_d(\bb{Z}_p)^{la}$, with associated vector bundle $\s{F}$ induced by $f:X \to */ \GL_d(\bb{Z}_p)^{la} \to */ \GL_d$, we denote by $\s{F}^{+,la}$ the $\bb{Z}_p$-locally analytic vector bundle. 
We let $\bb{V}(\s{F}^{+,la}):= f^* \bb{V}(\St^{+,la})$ be the  geometric incarnation of $\s{F}^{+,la}$.   Let $\widehat{\bb{G}}_{m,\eta}$ be the $p$-adic generic fiber of the  formal multiplicative group at $1$ endowed with its  $\bb{Z}_p^{la}$-module structure.  We define the dual space of  $\bb{V}(\s{F}^{+,la})$ to be $\widehat{\bb{G}}_{m,\eta}(\s{F}^{\vee,+,la}):= \bb{V}(\s{F}^{\vee,+,la}) \otimes_{\bb{Z}_p^{la}} \widehat{\bb{G}}_{m,\eta}$. 
\end{construction}

\begin{definition}
Let $\s{F}^{+,la}$ be a locally analytic  $\bb{Z}_p$-vector bundle of rank $d$ associated to a map $X \to */ \GL_d(\bb{Z}_p)^{la}$. We let $\Sym^{la}(\s{F}^{\vee,+})$ denote the $\s{O}_X$-algebra of functions of the space $\bb{V}(\s{F}^{+,la})$. We call $\Sym^{la}_X(\s{F}^{\vee,+})$ the algebra of locally analytic functions of $\s{F}^{+,la}$.   We denote by $\Sym^{\n{D}}_X(\s{F}^{\vee,+})$ the global sections over $X$ of the dual space  $\widehat{\bb{G}}_{m,\eta}(\s{F}^+)$, we call this object the algebra of locally analytic distributions of $\s{F}^{\vee,+,la}$. 
\end{definition}

\begin{prop}
\label{PropositionDAggerCartier3}
Let $X$ be an analytic $\s{D}$-stack over $\bb{Q}_p$ and let $\s{F}^{+,la}$ be a locally analytic $\bb{Z}_p$-vector bundle over $X$. 
\begin{enumerate}

\item  The map $\bb{V}(\s{F}^{+,la}) \to X$ is weakly cohomologically proper. 

\item  The map $f: \widehat{\bb{G}}_{m,\eta}(\s{F}^{+,la}) \to X$ is cohomologically smooth and there are natural isomorphisms $f^! 1_{X}= f^* \bigwedge^d \s{F}^{\vee}[d]$ and $f_! 1_{\mathring{\bb{V}}(\s{F}^+)} = \bigwedge^d \s{F} \otimes \Sym^{la}_X(\s{F}^+) [-d]$. In particular, $f_! f^! 1_{X}= \Sym^{la}_X(\s{F}^+)$.

\end{enumerate} 
\end{prop}
\begin{proof}
Part (1) is clear since $\bb{V}(\s{F}^{+,la})$ is the relative analytic spectrum of the $\s{O}_X$-algebra $\Sym^{la}_X(\s{F}^+)$ endowed with the induced analytic structure. 

For part (2), smoothness of $f$ follows from Proposition \ref{PropGeoSmoothIsSmooth}.  The computation of $f^{!} 1_X$ and  $f_! 1_{\mathring{\bb{V}}(\s{F}^+)}$ follows the same lines of the proof of Propositions \ref{PropAlgCartier1} and \ref{LemmaSerreDualityGLn} after we  modify the idempotent algebras. For this, let $e_1,\ldots, e_d$ be the canonical basis of $\St$. For $i=1,\ldots,d$ let $D_i:= \n{D}^{la}(\St,\bb{Q}_p)\langle 1/e_i \rangle^{\dagger}= \varinjlim_{\epsilon\to 0^+} \n{D}^{la}(\St,\bb{Q}_p) \langle\frac{p^{\epsilon}}{e_i} \rangle$. Then $D_i$ is an idempotent $\n{D}^{la}(\St, \bb{Q}_p)$-algebra, we let  $C$ be the $dg$ algebra obtained by taking  the ``union'' of $D_1,\ldots, D_d$. Then, the space $\widehat{\bb{G}}_{m,\eta}(\St)$ is the open subspace of $\AnSpec \n{D}^{la}(\St,\bb{Q}_p)$ obtained as the complement of the idempotent algebra  $C$. Therefore, we can write 
\[
f_!1_{\widehat{\bb{G}}_{m,\eta}(\St)} = [\n{D}^{la}(\St,\bb{Q}_p) \to C].
\]
An explicit power series computation shows that $f_{!} 1_{\widehat{\bb{G}}_{m,\eta}(\St)}$ is $\GL_d$-equivariantly isomorphic to   $C^{la}(\St,\bb{Q}_p )\otimes \bigwedge^d \St^{\vee}[-d]$. Let us see that it has the natural comodule action of $C^{la}(\St, \bb{Q}_p)$. For this, it suffices to compute its dual since it is a $LB$ space of compact type. But the dual is given by 
\[
f_*  f^{!} 1_X= \bigwedge^{d} \St^{\vee} \otimes \n{D}^{la}(\St,\bb{Q}_p) [d], 
\]
this shows that  $f^{!} 1_X = f^{*} \bigwedge^{d} \St^{\vee}[d]$, and so
\[
f_* f^{!} 1_X = (f_{*} 1_{\widehat{\bb{G}}_{m,\eta}(\St)} )\otimes \bigwedge^{d}\St^{\vee}[d],
\]
proving that the algebra structure of $\n{D}^{la}(\St, \bb{Q}_p)$ is the one arising from $\widehat{\bb{G}}_{m,\eta}(\St)$ as wanted.
\end{proof}

In order to obtain an analogue of Theorem \ref{TheoAlgebraicCartierDuality},  we need to have access to an analogue of the de Rham and Koszul resolutions of $\s{O}_X$ of Lemma \ref{LemmaDeRhamKoszulAlgebraic}:
 
\begin{lemma}
\label{LemmaDeRhaKoszulZp}
Let $X$ be an analytic $\s{D}$-stack over $\bb{Q}_p$ and $\s{F}^{+,la}$ a locally analytic $\bb{Z}_p$-vector bundle over $X$ of rank $d$. 
\begin{enumerate}

\item  We have a natural  resolution as $\Sym^{la}_X(\s{F}^{\vee,+})$-comodule  given by a decreasing complete filtration
\begin{equation}
\label{eqdeRhamZpla}
\s{O}_X \to \Sym^{la}_X(\s{F}^{\vee,+}) \xrightarrow{d} \Sym^{la}_X(\s{F}^{\vee,+})\otimes \s{F}^{\vee} \xrightarrow{d} \cdots  \xrightarrow{d} \Sym^{la}_X(\s{F}^{\vee,+}) \otimes \bigwedge^d \s{F}^{\vee},   
\end{equation}
whose dual is the Koszul resolution 
\[
 \Sym^{\n{D}}_X(\s{F}) \otimes \bigwedge^d \s{F} \to  \cdots  \to \Sym^{\n{D}}_X(\s{F}) \otimes \s{F}  \to \Sym^{\n{D}}_X(\s{F}) \to \s{O}_X.
\]
\item We have a natural  resolution as $\Sym^{la}_X(\s{F}^{\vee,+})$-module  given by an increasing complete filtration
\begin{equation}
\label{eqKoszulZp}
 \Sym^{la}_X(\s{F^+})\otimes \bigwedge^d \s{F} \to \cdots \to \Sym^{la}_X(\s{F}^+) \otimes \s{F} \to \Sym^{la}_X(\s{F}) \to \s{O}_X,
\end{equation}
whose dual is the  Rham complex \[
\s{O}_X \to \Sym^{\n{D}}_X(\s{F}^{\vee}) \otimes \s{F}^{\vee} \xrightarrow{d} \cdots \xrightarrow{d} \Sym^{\n{D}}_X(\s{F}^{\vee}) \otimes  \bigwedge^d \s{F}^{\vee}.
\]
\end{enumerate}
\end{lemma}

\begin{proof}
It suffices to prove the universal case with $X=*/ \GL_d(\bb{Z}_p)^{la}$ and $\s{F}^{+,la}= \St^{la}$. Recall that we have an equivalence of Hopf algebras  $\bb{Z}_{p,\sol}[\bb{Z}_p]=\bb{Z}_{p}[[1-X]]$ sending the unit $[1]\in \bb{Z}_p$ to $X\in \bb{Z}_{p}[[1-X]]$. Under this equivalence $\n{D}^{la}(\bb{Z}_p,\bb{Q}_p)$  becomes isomorphic to $\varprojlim_{h\to \infty} \bb{Q}_p \langle \frac{1-X}{p^{1/h}}\rangle$. One obtains a $\GL_d$-equivariant Koszul resolution  induced by left multiplication of $\St$:
\[
0  \to \n{D}^{la}(\St,\bb{Q}_p) \otimes \bigwedge^d \St \to \cdots \to \n{D}^{la}(\St, \bb{Q}_p) \otimes \St \to  \n{D}^{la}(\St, \bb{Q}_p)  \to \bb{Q}_p \to 0.
\]
taking duals one gets the resolution \eqref{eqdeRhamZpla}. It is clear that this is a resolution as $C^{la}(\St, \bb{Q}_p)$-comodules and $\n{D}^{la}(\St, \bb{Q}_p)$-modules. 

For part (2) we argue in the similar way, this time knowing that the de Rham complex 
\[
0\to \bb{Q}_p \to \n{D}^{la}(\St, \bb{Q}_p) \xrightarrow{d} \n{D}^{la}(\St, \bb{Q}_p) \otimes \St^{\vee} \xrightarrow{d} \cdots \xrightarrow{d}  \n{D}^{la}(\St, \bb{Q}_p) \otimes  \bigwedge^d \St^{\vee} \to 0
\]
is exact by the Poincar\'e Lemma, $\GL_d$-equivariant, and a complex as $\n{D}^{la}(\St,\bb{Q}_p)$-comodule and $C^{la}(\St,\bb{Q}_p)$-module. Note that, after fixing a basis  $e_1,\ldots, e_d$ of $\St$, the resolution \eqref{eqKoszulZp} is the Koszul resolution of the regular sequence $(X_{e_1}-1,\ldots, X_{e_d}-1)$, where $X_{e_i}-1$ is the projection to the $i$-th component. 
\end{proof}

\begin{remark}
The resolution of Lemma \ref{LemmaDeRhaKoszulZp} (1) is not the de Rham complex of $\Sym^{la}_X(\n{F}^{\vee,+})$. It is actually obtained by the Lazard-Serre resolution of the Iwasawa algebra, see \cite[Theorem 4.4]{Kohlhaase}. 
\end{remark}

\begin{theorem}[Cartier duality for locally analytic $\bb{Z}_p$-lattices]
\label{TheoremCartierDualityAnalyticIII}
Let $X$ be an analytic stack over $\bb{Q}_p$ and let $\s{F}^{+,la}$ be a locally analytic $\bb{Z}_p$-vector bundle over $X$. Then the analogue of Propositions \ref{PropAlgebraicCartier2} and \ref{PropConvolutionFMTransforms}, Theorem \ref{TheoAlgebraicCartierDuality} and Corollary \ref{CoroFourierMoukaiIdentities} hold after replacing the following objects:
\begin{itemize}
\item $\bb{V}(\s{F})$ for $\bb{V}(\s{F}^{+,la})$,

\item $\widehat{\bb{V}(\s{F}^{\vee})}$ for $\widehat{\bb{G}}_{m,\eta}(\s{F}^{\vee,+,la})$,

\item $\Sym^{\bullet}_X \s{F}^{\vee}$ for $\Sym^{la}_X (\s{F}^{\vee,+})$,

\item $\widehat{\Sym}_X(\s{F})$ by $\Sym^{\n{D}}_X(\s{F}^+)$. 
\end{itemize}
\end{theorem}
\begin{proof}
The same proof of the references applies in this context after making two modifications: the first one is replacing  Lemma \ref{LemmaDeRhamKoszulAlgebraic} for Lemma \ref{LemmaDeRhaKoszulZp}. The second is to provide the suitable compactifications of $f:\widehat{\bb{G}}_{m,\eta}(\St^+)\to *$ that are used to compute  $f_! f^! \bb{Q}_p$ as in Proposition \ref{PropAlgCartier1}. 
\end{proof}

\begin{remark}
There are at least two different generalizations of the Cartier duality for locally analytic $\bb{Z}_p$-vector bundles. The first consists in taking a finite extension $L/\bb{Q}_p$ and considering instead $L$-locally analytic $\n{O}_L$-vector bundles; this theory should be a consequence of  what we have done previously since we have a fiber sequence   of group objects over $L$:
\[
(\Lie \f{k})^{\dagger} \to \n{O}_L^{\bb{Q}_p-la} \to \n{O}_{L}^{la}, 
\]
where $\n{O}_L^{\bb{Q}_p-la}$ is the group $\n{O}_L$ seen as a $p$-adic Lie group over $\bb{Q}_p$, and $\f{k}=\ker(\Lie_{\bb{Q}_p} \n{O}_L \otimes \bb{Q}_p L \to \Lie_L \n{O}_L)$.   A different and more interesting generalization is the passage from locally analytic $\bb{Z}_p$-vector bundles to locally analytic $\bb{Q}_p$-vector bundles. In this case, the Cartier dual of a locally analytic vector bundle $\s{F}^{la}$ should be given by the ``\textit{universal cover}'' $\widetilde{\widehat{\bb{G}}}_{m,\eta}(\s{F}^{\vee})$, where $\widetilde{\widehat{\bb{G}}}_{m,\eta}(\s{F}^{\vee})= \varprojlim_{p} \widehat{\bb{G}}_{m,\eta}(\s{F}^{\vee,+}) $ is the limit as analytic spaces of multiplication by $p$ of the dual of a $\bb{Z}_p$-lattice $\s{F}^{+,la} \subset \s{F}^{la}$ (such a lattice exists locally in the $\s{D}$-topology). 
\end{remark}

\section{Algebraic and analytic de Rham stacks}
\label{SectionAnalyticDeRham}

In this section we construct the algebraic and analytic (filtered) de Rham stacks for derived Tate adic spaces over $\bb{Q}_p$. Following \cite{BhattGauges}, we will obtain the Hodge-filtration of the de Rham cohomology by reading the geometry of these stacks. We shall prove that both the algebraic and analytic de Rham stacks have a nice theory of six functors for morphisms locally of solid finite presentation of derived Tate adic spaces. Finally, we compute the dualizing sheaves of both filtered de Rham stacks for solid smooth morphisms by applying a deformation to the normal cone argument as in \cite{MannSix} and \cite{CondensedComplex}.

\subsection{The algebraic de Rham stack}
\label{SubsectionAlgebraicdeRhamStack}

The (algebraic) de Rham stack was introduced by Simpson \cite{SimpsonDeRham,SimpsonTelemandeRham}, and plays a fundamental role in the geometric Langlands correspondence,   cf. \cite{GRdeRhamStack}. In the following we will define algebraic de Rham prestacks in the realm of analytic geometry over $\bb{Q}$. We then specialize to solid prestacks and show that the theory of $D$-modules obtained in this way admits a good behaved six-functor formalism.

\subsubsection{General definition}
\begin{definition}
\label{DefinitionAlgdeRhamStack}
Let $\PSh(\AnRing_{\bb{Q}}^{\op})$ be the category of prestacks on analytic rings over $\bb{Q}$. We define the following objects 

 \begin{enumerate}

\item Let $X\in \PSh(\AnRing_{\bb{Q}}^{\op})$. The \textit{absolute  filtered algebraic de Rham prestack of $X$} is the prestack of anima over $\bb{A}^1/ \bb{G}_m$ given by
\[
X_{dR^+}^{\alg}(\s{O}(1) \to \n{A}) = \varinjlim_{I \to \n{A}}X(\mathrm{cone}(I \otimes \s{O}(-1) \to \n{A})),
\] 
where $I$ runs over all the uniformly nilpotent ideals of $\n{A}$, see Definition \ref{DefUnifNilpotent}. We define the \textit{absolute de Rham (resp. Hodge) prestack} $X_{dR}^{\alg}$ (resp, $X_{Hodge}^{\alg}$) to be the pullback along $*=\bb{G}_m/ \bb{G}_m \to \bb{A}^1_{\bb{G}_m}$ (resp. the pullback along $*/ \bb{G}_m \to \bb{A}^1/ \bb{G}_m$). When restricted to solid rings,  the \textit{filtered  de Rham stack} is  the $\s{D}$-sheafification of the filtered de Rham prestack (resp. for the de Rham  and the Hodge stacks).

 \item For a morphism $X \to Y$ of prestacks on analytic rings  over $\bb{Q}$ we let $X_{dR^+,Y}^{\alg}$ be the pullback
\[
\begin{tikzcd}
X_{dR^+,Y}^{\alg}  \ar[d] \ar[r] &  Y\times \bb{A}^1/ \bb{G}_m \ar[d] \\ 
X_{dR^+}^{\alg} \ar[r] & Y_{dR^+}^{\alg}.
\end{tikzcd}
\]
Similarly, we define $X_{dR,Y}^{\alg}$ and $X_{Hodge,Y}^{\alg}$ to be the pullback of $X_{dR^+}^{\alg}$ to $Y$ and $Y\times */ \bb{G}_m$ respectively. 
\end{enumerate}
\end{definition}

The following proposition describes some basic properties of the algebraic de Rham stack. 

\begin{prop}
\label{PropAlgebraicDeRham}
Let $f:X \to Y$ be a morphism of prestacks on $\AnRing_{\bb{Q}}$.
 
\begin{enumerate}

\item  Suppose that $f$ is formally \'etale, then the natural map $X\times \bb{A}^1/ \bb{G}_m \to X_{dR^+,Y}^{\alg}$ is an equivalence.

 \item  Suppose that $f$ is formally smooth and let $\n{T}_{X/Y} = \AnSpec_X \Sym_{X}^{\bullet} \bb{L}_{X/Y}$. The following hold

\begin{itemize}

\item[(a)]  The map $X\times \bb{A}^1/ \bb{G}_m \to X_{dR^+,Y}^{\alg}$ is an epimorphism. 

\item[(b)]   There is a natural equivalence $X_{Hodge,Y}^{\alg}= (X\times */ \bb{G}_m)/\widehat{\n{T}_{X/Y}}(-1)$.  

\end{itemize}

\item The formation of $X \mapsto X_{dR^+}^{\alg}$ commutes with small colimits and finite limits of prestacks.

\end{enumerate}
\end{prop}
\begin{proof}
\begin{enumerate}

\item Let $\n{A} \in \AnRing_{\bb{Q}}$, and let $\s{O}(-1) \to \n{A}$ be a generalized Cartier divisor and  $I \to \n{A}$  an uniformly nilpotent ideal (cf. Definition \ref{DefUnifNilpotent}). Let $\n{B}= \mathrm{cone}(I\otimes \s{O}(-1) \to \n{A})$, then $\n{A} \to \n{B}$ is a nilpotent thickening and there is a natural  equivalence 
\begin{equation}
\label{eqMapdeRhamplus}
(X\times \bb{A}^1/\bb{G}_m) (\n{A}) \xrightarrow{\sim} X(\n{B}) \times_{Y(\n{B})} (Y \times \bb{A}^1/ \bb{G}_m )(\n{A})= (X_{dR^+}^{\alg} \times_{Y_{dR^+}^{\alg}} (Y\times \bb{A}^1/ \bb{G}_m)) (\n{A})
\end{equation}
proving what we wanted.

\item It is clear that if $f$ is formally smooth then $X\times \bb{A}^1/ \bb{G}_m \to X_{dR^+}^{\alg}$ is an epimorphism, namely, the map \eqref{eqMapdeRhamplus} is surjective by definition of formally smoothness; this shows (a). To prove part (b), note that $\n{A}$-points of  $X^{\alg}_{Hodge,Y}$ are given by 
\[
X^{\alg}_{Hodge,Y}(\n{A}) = \varinjlim_{I \subset \n{A}} X(I\otimes\s{O}(-1)[1] \oplus \n{A}) \times_{Y(I\otimes\s{O}(-1)[1] \oplus \n{A})} (Y \times */\bb{G}_m)(\n{A}).
\]
Thus, given $\eta \in Y(\n{A})$, the fiber of $X^{\alg}_{Hodge,Y} \to Y\times */ \bb{G}_m$ at $\eta \times */ \bb{G}_m$ is given by 
\[
\varinjlim_{I\subset \n{A}} \Map_{\Mod_{\geq 0}(\n{A})}(\eta^* \bb{L}_{X/Y}, I \otimes \s{O}(-1)[1])   =  (\eta^*\bb{L}^{\vee}_{X/Y} \otimes \Nil(\n{A})\otimes \s{O}(-1) )[1])(*),
\]
which shows that 
\[
X_{Hodge,Y}^{\alg} = (X\times */\bb{G}_m)/ \widehat{\n{T}_{X/Y}}(-1). 
\]

\item Finally, by definition at the level of points,  $X \mapsto X_{dR^+}^{\alg}$ commutes with small colimits.  The commutation with finite limits follows from the definition  and the fact that the system of uniformly bounded ideals of a ring $\n{A}$ is filtered. 
\end{enumerate}
\end{proof}

\begin{remark}
\label{RemarkAlternativedeRham}
There is a different definition for the de algebraic (filtered) de Rham prestack, namely, the prestack given by 
\[
X^{alg'}_{dR^+}(\s{O}(1) \to \n{A}) = X(\mathrm{cone}(\Nil(\n{A}) \otimes \s{O}(-1) \to \n{A})),
\]
see \S \ref{SubsectionCondensedNil-radical} for the definition of the condensed nil-radical.  The apparent advantage of this definition is that the formation of $X\mapsto X_{dR^+}^{alg'}$ commutes with small limits and colimits of prestacks, however, after taking sheafifications with respect to some Grothendieck topology the formation of the de Rham stack will only commutes with small colimits and finite limits. The disadvantage of this definition is that it is not clear whether formally smooth maps produce epimorphisms. Nevertheless, for all the spaces in practice both constructions are the same after a suitable sheafification, eg. smooth morphisms  of classical derived schemes (this follows from the fact that the nilpotent radical of a finitely generated algebra is nilpotent), and solid smooth  morphisms of derived Tate adic spaces (this follows from Proposition \ref{PropFormallyInftSmoothEtale} as $\Nil(\n{A}) \subset \Nil^{\dagger}(\n{A})$ for a bounded affinoid ring). 
\end{remark}

Another feature of the de Rham stack is its relation with formal completions of Zariski closed immersions. 

\begin{prop}
\label{PropZariskiAlgDeRham}
Let $X= \AnSpec \n{B} \to Y=\AnSpec \n{A}$ be a morphism of analytic affine spaces over $\bb{Q}$ surjective on $\pi_0$, and such that $\n{B}$ has the induced analytic structure from $\n{A}$ (i.e. a Zariski closed immersion). Let $I=[\n{A} \to \n{B}]$ and  suppose that there is an extremally disconnected set $S$ and a map $S\to I$ such that $\n{A}[S] \to I$ is surjective on $\pi_0$.  Then the algebraic  de  Rham stack $X_{dR,Y}^{\alg}$ is the full substack of $Y$ given by 
\[
X_{dR,Y}^{\alg} = \varinjlim_{n} Y \times_{\AnSpec \bb{Z}[\bb{N}[S]]} \AnSpec \bb{Z}[\bb{N}[S]]_n^{\bb{L}}.
\]
 We call $X_{dR,Y}^{\alg}$ the formal completion of $Y$ along $X$ and denote it by $\widehat{Y}^{X}$. 
\end{prop}
\begin{proof}
Let us write $\widehat{Y}^{X}:=\varinjlim_{n} Y \times_{\AnSpec \bb{Z}[\bb{N}[S]]} \AnSpec \bb{Z}[\bb{N}[S]]_n^{\bb{L}}$, by Proposition \ref{PropRepresentedNilRadical} the map $\widehat{Y}^{X} \to X$ is an immersion. Since for any analytic ring $\n{C}$ the ideal of the map $\n{C} \to \pi_0( \n{C})$ is uniformly nilpotent, the algebraic de Rham prestack is the right Kan extension of its restriction to static analytic rings. In particular,  $X_{dR}$ and $Y_{dR}$ are just  presheaves on sets and $Y_{dR} \to X_{dR}$ is an immersion. Thus, to show that $X_{dR,Y}^{\alg}= \widehat{Y}^{X}$, it suffices to check at the level of points in static analytic rings. Let $J\to \n{C}$ be an uniformly nilpotent ideal of an static analytic ring $\n{C}$, suppose we have a commutative diagram 
\[
\begin{tikzcd}
\n{A}\ar[r] \ar[d] & \n{C} \ar[d] \\ 
\n{B}\ar[r] & \n{C}/J.
\end{tikzcd}
\]
Let $S \to I$ be such that $\n{A}[S] \to I$ is surjective, then the image of $S$ in $\n{C}$ belongs to $J$, and since $J$ is uniformly nilpotent there is $n$ such that $\n{A} \to \n{C}$ factors though 
\[
\n{A}\otimes_{\bb{Z}[\bb{N}[S]]} \bb{Z}[\bb{N}[S]]_n^{\bb{L}} \to \n{C},
\]
proving that the map $\AnSpec \n{C} \to Y$ factors through $\widehat{Y}^{X}$ as wanted. 
\end{proof}

\begin{remark}
The hypothesis that $\n{A}[S] \to I$ is surjective on $\pi_0$ can be thought as a finitely generated assumption for the Zariski closed immersion.   With the alternative definition  of the de Rham stack of Remark \ref{RemarkAlternativedeRham} one can extend Proposition \ref{PropZariskiAlgDeRham} to an arbitrary Zariski closed immersion.
\end{remark}

\subsubsection{Six functors for algebraic $D$-modules}

Let us now restrict ourselves to the category of solid prestacks over $\bb{Q}$, namely, the category $\PSh(\Aff_{\bb{Q}})$ of presheaves on anima of solid affinoid rings over $\bb{Q}$. We can transmute the six functor formalism from quasi-coherent sheaves to $D$-modules as follows:

\begin{definition}
Let $S \in \PSh(\Aff_{\bb{Q}})$ be a fixed prestack. We define the six-functor formalisms $\s{D}_{dR^+,S}^{\alg}$, $\s{D}_{dR,S}^{\alg}$ and $\s{D}_{Hodge,S}^{\alg}$ from $\PSh(\Aff_{\bb{Q}})_{/S}$ to be that obtained via Lemma \ref{LemmaFunctoriality6Functors} and the functors $X \mapsto X_{dR^+,S}^{\alg}$, $X_{dR,S}^{\alg}$ and $X_{Hodge,S}$, landing on $\s{D}$-stacks over $S\times \bb{A}^1/ \bb{G}_m$, $S$ and $S \times */ \bb{G}_m$ respectively.  We call $\s{D}_{dR^+,S}^{\alg}$ (resp. $\s{D}_{dR,S}^{\alg}$, resp. $\s{D}_{Hodge,S}^{\alg}$) the six functor formalisms of   filtered  algebraic $D$-modules (resp.  algebraic $D$-modules, resp. algebraic Hodge modules) over $S$. 
\end{definition}

By Theorem \ref{TheoSixFunctorsScholze} we have a six functor formalism on prestacks for the different $\s{D}$-topologies defined by the functors $\s{D}_{dR^+,S}^{\alg}$, $\s{D}_{dR,S}^{\alg}$ and $\s{D}_{Hodge,S}^{\alg}$. There is no reason for these topologies to be the same, and not even comparable with the $\s{D}$-topology for solid quasi-coherent sheaves. Nevertheless there are some particular cases where one can relate covers in different topologies.

\begin{lemma}
\label{LemmaTransmuteDcovers}
Let $S$ be a base prestack and  $X\to Y$  a morphism of $\s{D}^{\alg}_{dR^+,S}$-stacks which is a canonical cover and such that $X_{dR^+}^{\alg} \to Y_{dR^+,S}^{\alg}$ is a $\s{D}$-cover. Then $X \to Y$ is a  $\s{D}^{\alg}_{dR^+,S}$, $\s{D}^{\alg}_{dR,S}$ and $\s{D}_{Hodge,S}^{\alg}$-cover. 
\end{lemma}
\begin{proof}
Being a $\s{D}^{\alg}_{dR^+,S}$-cover means that it is a canonical cover and of universal $*$ and $!$-descent for the six functor formalism $\s{D}^{\alg}_{dR^+,S}(X)= \Mod_{\sol}(X^{\alg}_{dR^+,S})$. The map $X\to Y$ is a canonical cover by definition, and by hypothesis $X_{dR^+,S}^{\alg} \to Y_{dR^+,S}^{\alg}$ is a $\s{D}$-cover, in particular of universal $*$ and $!$-descent. The lemma follows since the formation of the filtered de Rham stack commutes with finite limits (see Proposition \ref{PropAlgebraicDeRham} (c) and \cite[Remark 6.2.2.11]{HigherTopos}), and the fact that $X_{dR,S}^{\alg}$ and $X_{Hodge,S}^{\alg}$ are the fibers over $\bb{G}_m/ \bb{G}_m$ and $*/ \bb{G}_m$ respectively. 
\end{proof}

 Let $R= \bb{Z}((\pi))$, our next task is to show that morphisms locally of solid finite presentation of derived Tate adic spaces admit $!$-functors for the six functor formalism of algebraic filtered $D$-modules. Moreover, we prove an analogue existence result for  morphisms of locally finite presentation in Berkovich geometry and $\dagger$-geometry.   Finally, we show  that solid smooth (resp. \'etale) maps are sent to cohomologically smooth (resp. \'etale) maps under $\s{D}^{\alg}_{dR}$ in characteristic $0$.

 \begin{definition}
\begin{enumerate}
\item We let $R\langle T \rangle$ denote the Tate algebra of $R$ endowed with the induced analytic structure from $R$. Given $\n{A}$ a bounded affinoid ring over $R$, we let $\n{A} \langle T \rangle:= \n{A} \otimes_{R_{\sol}} R \langle T \rangle$. 

\item  We let $R \langle T\rangle^{\dagger}= \varinjlim_{\epsilon \to 0^+} R\langle \pi^{\epsilon} T \rangle$ denote the algebra of overconvergent functions of a closed unit disc. Given $\n{A}$ a bounded affinoid ring we let $\n{A} \langle T \rangle^{\dagger}:= \n{A}\otimes_{R_{\sol}} R\langle T \rangle^{\dagger}$.

\end{enumerate}
\end{definition}

Recall the notion of a coordinate theory of Definition \ref{DefinitionCoordinateTheory}. The natural maps $R[T] \to R\langle T \rangle$ and $R[T] \to R\langle T \rangle^{\dagger}$ define two coordinate theories over $R$, this allows us to talk about morphisms locally of finite presentation for $R\langle T \rangle$ and $R\langle T \rangle^{\dagger}$, see Definition \ref{DefinitionGeneralFinitePresentation}.

\begin{definition}
\label{DefinitionBerkovichDaggerSpaces}
A Berkovich adic space (resp. a $\dagger$-space) is a $R\langle T\rangle$-adic space (resp. a $R\langle T \rangle^{\dagger}$-adic space) as in Definition \ref{DefinitionGeneralFinitePresentation} (3).  A morphism of Berkovich adic spaces (resp. $\dagger$-adic spaces) is locally of finite presentation if it is locally of $R\langle T \rangle$-finite presentation (resp. of $R\langle T \rangle^{\dagger}$-finite presentation). 
\end{definition}

\begin{lemma}
\label{LemmaAlgdeRhamKeyCases}
Let $\bb{G}_{a,\star}$ denote the analytic spectrum of $\bb{Q}[T]$,   $\bb{Z}[T]_{\sol} \otimes \bb{Q}$, $R\langle T \rangle \otimes \bb{Q}$ or $R\langle T \rangle^{\dagger} \otimes \bb{Q}$.  Let $\widehat{\bb{G}}_a(-1) $ be the formal completion at $0$ of the  twisted line bundle $\s{O}(-1)$ over $\bb{A}^1/\bb{G}_m$, and consider the natural morphism of abelian group objects $ \widehat{\bb{G}}_{a}(-1) \to \bb{G}_{a,\star} \times \bb{A}^{1}/ \bb{G}_m$. The following hold: 

\begin{enumerate}

\item  $\bb{G}_{a,\star,dR^+}^{\alg} = ( \bb{G}_{a,\star}\times \bb{A}^1/ \bb{G}_m )/ \widehat{\bb{G}}_{a}(-1)$.  In particular the maps 
\[
f: *_{dR^+}^{\alg}= \bb{A}^1/\bb{G}_m \to \bb{G}_{a,\star,dR^+}^{\alg}
\]
and 
\[
g: \bb{G}_{a,\star,dR^+}^{\alg} \to *_{dR^+}
\]
admit $!$-functors. Furthermore, the following properties are satisfied
\begin{itemize}

\item[(a)] The pullback of $f$ to the algebraic de Rham stack is $(-1)$-truncated and  cohomologically \'etale. 

\item[(b)] If $\bb{G}_{a,\star}= \bb{G}_{a,\sol}= \AnSpec (\bb{Z}[T]_{\sol} \otimes \bb{Q})$,  then $g$ is cohomologically smooth. 

\item[(c)] If $\bb{G}_{a,\star}= \AnSpec \bb{Q}[T]$, $\AnSpec R\langle T\rangle\otimes \bb{Q}$ or $\AnSpec R\langle T\rangle^{\dagger} \otimes \bb{Q}$, then $g$ is cohomologically  co-smooth. 

\end{itemize}

\item  Let $n\geq 1$ and denote $X_n:= \AnSpec( \Sym^{\bullet}_{ \bb{Q}} ( \bb{Q}[n]))$. Then 
\[
X_{n,dR^+}^{\alg}= (\widehat{\bb{G}}_{a}(-1)^{\times n+1/\bb{G}_{a}})/ \widehat{\bb{G}}_{a}(-1),
\]
where $\widehat{\bb{G}}_{a}(-1)$ acts diagonally on the fiber product. In particular, the maps $X_{n,dR^+}^{\alg} \to *_{dR^+}^{\alg}$ and $*_{dR^+}^{\alg} \to X_{n,dR^+}^{\alg}$ admit $!$-functors.  Furthermore, their restriction to the de Rham stack is an equivalence.

\end{enumerate}
\end{lemma} 
\begin{proof}
\begin{enumerate}
\item By definition, $\bb{G}_{a,dR^+}^{\alg}$ represents the functor 
\[
(\s{O}(-1) \to \n{A})\mapsto \mathrm{cone}( \Nil(\n{A})\otimes \s{O}(-1)  \to \n{A})(*),
\]
this implies that $\bb{G}_{a,dR^+}^{\alg} = (\bb{G}_{a}\times \bb{A}^1/ \bb{G}_m)/ \widehat{\bb{G}}_a(-1)$. On the other hand, for any other group $\bb{G}_{a,\star}$, the map $\bb{G}_{a,\star} \to \bb{G}_{a}$ (with $R$-extension of scalars if necessary) is formally \'etale. Since $\bb{G}_{a}$ and $\bb{G}_{a,\star}$ are formally smooth, by Proposition \ref{PropAlgebraicDeRham} (2) we have a cartesian diagram 
\[
\begin{tikzcd}
\bb{G}_{a,\star}\times \bb{A}^1/ \bb{G}_m  \ar[d] \ar[r] & \bb{G}_a \times \bb{A}^1 / \bb{G}_m \ar[d]  \\ 
\bb{G}_{a,\star,dR^+}^{\alg} \ar[r] &   \bb{G}_{a,dR^+}^{\alg}
\end{tikzcd}
\]
where the vertical arrows are epimorphisms. One deduces that $\bb{G}_{a,\star,dR^+}=(\bb{G}_{a,\star}\times \bb{A}^1/ \bb{G}_m) / \widehat{\bb{G}}_a(-1)$ as wanted.

\begin{itemize}
\item[(a)] The map $f:*_{dR^+}^{\alg}\to \bb{G}_{a,\star,dR^+}^{\alg}$ is equivalent to the map 
\[
\widehat{\bb{G}}_a(-1)/ \widehat{\bb{G}}_{a}(-1) \to (\bb{G}_{a,\star}\times \bb{A}^1/ \bb{G}_m)/\widehat{\bb{G}}_{a}(-1).
\] 
since $\widehat{\bb{G}}_a(-1) \to \bb{G}_{a,\star} \times \bb{A}^1/ \bb{G}_m$ has $!$-functors, one deduces that $f$ does so. The restriction to the de Rham stack yields the functor 
\[
\widehat{\bb{G}}_a/ \widehat{\bb{G}}_a \to  \bb{G}_{a,\star}/ \widehat{\bb{G}}_{a},
\]
 which defines an open immersion of locales, so it is  cohomologically \'etale. 

 \item[(b) and (c)] The map $\bb{G}_{a,\star,dR^+}^{\alg} \to *_{dR^+}^{\alg}$ factors as the composite 
\[
(\bb{G}_{a,\star} \times \bb{A}^1/ \bb{G}_m)/ \widehat{\bb{G}}_a(-1) \to (\bb{A}^1/\bb{G}_m)/ \widehat{\bb{G}}_{a}(-1) \to \bb{A}^1/ \bb{G}_m.
\]
 When $\bb{G}_{a,\star}= \bb{G}_{a,\sol}$, the first map is   smooth by Lemma \ref{LemmaCohoSmoothAffineLine} and the second is smooth by Proposition \ref{PropAlgebraicCartier2}, then so is the composition.   If $\bb{G}_{a,\star}$ is any other coordinate, the first map is weakly cohomologically proper being a quotient of a map of analytic rings with the induced analytic structure, and the second is co-smooth by Proposition \ref{PropAlgebraicCartier2} again, then so is the composition. 

\end{itemize}

\item  Since the filtered algebraic  de Rham functor commutes with finite limits, and   $\Sym^{\bullet}_{\bb{Q}}(\bb{Q}[n+1])= \bb{Q} \otimes_{\Sym^{\bullet}_{\bb{Q}}( \bb{Q}[n])} \bb{Q}$ for $n \geq 0$, we have the desired description of $X_{n,dR^+}^{\alg}$. It is also clear that the maps $X_{dR^+}^{\alg} \to *_{dR^+}^{\alg}$ and $*_{dR^+}^{\alg} \to X_{dR^+}^{\alg}$ admit $!$-functors from this description as stacks. 

\end{enumerate}
\end{proof}

\begin{lemma}
\label{LemmaDescentSmoothAlgdeRhamStack}
Let $f:X= \AnSpec \n{A} \to Y= \AnSpec \n{B}$ be a morphism of bounded affinoid rings over $(R\otimes \bb{Q})^{\an}$. 

\begin{enumerate}

\item  If $f$ is  standard solid smooth (resp. \'etale)  then $X_{dR^+}^{\alg} \to Y_{dR^{+}}^{\alg}$ is cohomologically smooth (resp.\'etale).

\item Suppose that  $X= \bigsqcup_{i=1}^d X_i$ with $X_i \to Y$  standard solid smooth. If $f$ is a smooth $\s{D}$-cover then $X_{dR^+}^{\alg} \to Y_{dR^+}^{\alg}$ is a smooth $\s{D}$-cover. In particular, solid rational covers of $Y$ give rise $\s{D}$-covers of $Y_{dR^+}^{\alg}$. 

\item Suppose that $f:X \to Y$ is a rational cover for the coordinate theories $R\langle T \rangle$ and $R\langle T \rangle^{\dagger}$. Then $X_{dR^+}^{\alg} \to Y_{dR^+}^{\alg}$ is a descendable $\s{D}$-cover. 
\end{enumerate}

\end{lemma}
\begin{proof} 
\begin{enumerate}

\item We first assume that $\bb{L}_{Y}$ is a projective $\n{B}$-module. By deformation theory the map $Y\times \bb{A}^1/\bb{G}_m \to Y_{dR^+}^{\alg}$ is surjective, in particular of universal $*$-descent. Then, to show that $X_{dR^+}^{\alg} \to Y_{dR^+}^{\alg}$ is cohomologically smooth it suffices to show that $X_{dR^+,Y}^{\alg} \to Y$ is cohomologically smooth. We can factor $X \to Y \times \bb{G}_{a,\sol}^n \xrightarrow{\pr_Y} Y$ with the first map being standard solid \'etale. By Lemma \ref{LemmaAlgdeRhamKeyCases} the projection $(Y \times \bb{G}_{a,\sol}^n)_{dR^+,Y}^{\alg} \to Y\times \bb{A}^1/ \bb{G}_m$ is cohomologically smooth, so we only need to show that the arrow  $X_{dR^+,Y}^{\alg} \to Y \times \bb{G}_{a,\sol,dR^+}^n$ is cohomologically \'etale. But again, the cotangent complex of $Y \times \bb{G}_{a,\sol}^n$ is a projective $\n{B}\langle T_1,\ldots, T_n \rangle_{\sol}$-module, and by the same argument it suffices to show that if $X \to Y$ is standard solid \'etale then $X_{dR^+,Y}^{\alg} \to Y\times \bb{A}^1/ \bb{G}_m$ is cohomolgically \'etale. By Proposition \ref{PropAlgebraicDeRham} (1) we have that $X_{dR^+,Y}^{\alg} = X \times \bb{A}^1/ \bb{G}_m$ proving   what we wanted. Now we show the general case, since the formation of the filtered de Rham stack commutes with finite limits, it suffices to construct a Cartesian diagram 
\begin{equation}
\label{eqCartesianSmoothdeRhamalg}
\begin{tikzcd}
X  \ar[d] \ar[r]& X'\ar[d] \\ 
Y \ar[r] & Y'
\end{tikzcd}
\end{equation}
where $X'=\AnSpec\n{A}' \to Y'=\AnSpec \n{B}'$ is a standard solid smooth (resp. \'etale) morphism of bounded affinoid rings (not necessarily over $\bb{Q}$!), and $\bb{L}_{Y'}$ a projective $\n{B}'$-module. Since $\n{B}$ is bounded affinoid, we can write $\pi_{0}(\n{B})$ as a filtered colimit of quotients of rings of the form $R\langle \underline{X} \rangle_{\sol}\langle \bb{N}[K] $ where $\underline{X}$ is a finite set of variables and $K$ is a profinite set. Writing $\n{A}=\n{B} \langle T_1,\ldots, T_d  \rangle_{\sol}/^{\bb{L}}(f_1,\ldots, f_c)$ as a standard solid smooth map, we can find a map $\n{B}':= R\langle \underline{X} \rangle_{\sol}\langle \bb{N}[K] \rangle \to \n{A}$ and a lift $f_i'$ of $f_i$ to  $\n{B}'\langle T_1,\ldots, T_d \rangle_{\sol}$ such that $\n{A}'= \n{B}'\langle T_1,\ldots, T_d \rangle_{\sol}/^{\bb{L}}(f'_1,\ldots, f'_c)$ is standard solid smooth.   Taking $Y'=\AnSpec \n{B}'$ and $X'=\AnSpec  \n{A}' $ we get the desired cartesian diagram, namely, the cotangent complex of $\n{B}'$ is isomorphic to $\bigoplus_{i=1}^k\n{B}'[K] $ with $\underline{X}=(X_1,\ldots, X_k)$.

\item By Proposition \ref{PropositionSmoothDescent}, it suffices to show that the pullback along the map $X_{dR^+}^{\alg} \to Y_{dR^+}^{\alg}$ is conservative, for this it suffices to see that it is a surjection as $\s{D}$-stacks. Suppose first that $Y\times \bb{A}^1/ \bb{G}_m \to Y_{dR^+}^{\alg}$ is surjective as $\s{D}$-stacks, in particular of universal $*$-descent. Then, it suffices to show that the map $X_{dR^+,Y}^{\alg} \to Y \times \bb{A}^1/ \bb{G}_m$ is surjective, but we have a factorization  $X\times \bb{A}^1/\bb{G}_m \to  X_{dR^+,Y}^{\alg} \to Y \times \bb{A}^1/ \bb{G}_m$.  Since $X \to Y$ is a smooth $\s{D}$-cover we deduce that $X_{dR^+,Y}^{\alg} \to Y \times \bb{A}^1/ \bb{G}_m$ is surjective as wanted.  In general, consider the topology $\s{T}$ on $\Aff^{b}_{R\otimes \bb{Q}}$ with covers given by solid smooth morphisms that are  $\s{D}$-covers. We proved that the formation of the  filtered de Rham stack satisfies $\s{T}$-descent  for covers $U \to W$ such that $W \to W_{dR^+}^{\alg}$ is surjective. To show descent for a general morphism it suffices to construct a Cartesian diagram as in \eqref{eqCartesianSmoothdeRhamalg} such that $X' \to Y'$ and $Y'\times \bb{A}^1/ \bb{G}_m \to Y_{dR^+}^{',\alg}$ are surjective as morphisms of $\s{D}$-stacks. By part (1) we can find a cartesian diagram 
\begin{equation}
\label{eqCartesian2deRhamRational}
\begin{tikzcd}
X \ar[r] \ar[d] & X'\ar[d] \\ 
Y \ar[r]& Y''
\end{tikzcd}
\end{equation}
of bounded affinoid rings where $\bb{L}_{Y''}$ is a projective module in $\Mod_{\sol}(Y'')$ (even compact projective), in particular $Y'' \to Y_{dR^+}^{'',\alg}$ is surjective. However, $X' \to Y''$ might not be a surjection. Let $Y' \subset Y''$ be the full $\s{T}$-substack consisting in the essential image of $X'$ in $Y''$, concretely, it is the geometric realization of the simplicial $\s{T}$-stack $(X^{' n+1/Y''})_{[n]\in \Delta^{\op}}$.  Then, since $X \to Y$ is a $\s{T}$-cover, we still have a Cartesian diagram 
\[
\begin{tikzcd}
X \ar[r] \ar[d] & X' \ar[d] \\ 
Y \ar[r] & Y'
\end{tikzcd}.
\]
On the other hand, since $Y' \subset Y''$ is a full $\s{T}$-substack, the morphism $Y'\times \bb{A}^1/ \bb{G}_m \to Y_{dR^+}^{',\alg}$ is surjective as $\s{D}$-stacks, namely, it is the geometric realization of the surjective morphism of simplicial $\s{D}$-stacks 
\[
(X^{'n+1/Y''}\times \bb{A}^1/ \bb{G}_m)_{[n]\in \Delta^{\op}} \to (X^{'n+1/Y'',\alg}_{dR^+})_{[n]\in \Delta^{\op}}.
\] 
The lemma follows. 

\item For part (3), denote $R(T)$ for $R\langle T \rangle$ or $R\langle T \rangle^{\dagger}$, and  let us write $f:X= \bigsqcup X_i \to Y$ with $X_i \to Y$ a  $R(T)$-rational localization. In particular, each map $X_i \to Y$ is defined by an idempotent $\s{O}_Y$-algebra and defines a closed subset of the locale of $Y$.  Suppose first that $Y\to Y_{dR^+}^{\alg}$ is surjective, then since $X_i \to Y$ is formally \'etale (cf. Example \ref{ExampleCotangentComplexes}), we have a cartesian square
\[
\begin{tikzcd}
 X_i\times \bb{A}^1/ \bb{G}_m \ar[r] \ar[d] & Y \times \bb{A}^1/ \bb{G}_m  \ar[d] \\ 
 X_{i,dR^+}^{\alg}  \ar[r] & Y_{dR^+}^{\alg} 
\end{tikzcd}
\]
where the vertical arrows are epimorphisms. Then, the inclusion $X_{i,dR^+}^{\alg} \to Y_{dR^+}^{\alg}$ is defined by an idempotent algebra in $\Mod_{\sol}(Y_{dR^+}^{\alg})$, and defines a closed subspace of the locale of $Y_{dR^+}^{\alg}$. Since $X \to Y$ is a cover, we have $\bigcup_i X_i = Y$ as closed subspaces of the locale, which implies that $\bigcup_{i} X_{i,dR^+}^{\alg} = Y_{dR^+}^{\alg}$. Since there are only finitely many $i$'s, one deduces that the map $X_{dR^+}^{\alg} \to Y_{dR^+}^{\alg}$ is descendable, and a $\s{D}$-cover by Proposition \ref{PropositionProperDescent}. 

Let us now deal with the general case. As in part (1), we can assume that there is a bounded affinoid space $Y''= \AnSpec \n{B}'$ such that $Y'' \times \bb{A}^1/ \bb{G}_m \to Y_{dR^+}^{'',\alg}$ is surjective,  a  morphism $X'= \bigsqcup_i X_i' \to Y''$ consisting in finitely many $R(T)$-rational localizations (that might not cover $Y''$),  and a cartesian square as in \eqref{eqCartesian2deRhamRational}. Then, taking $Y'\subset Y$ to be the union of the $X_i'$ in $Y''$, we have a Cartesian diagram as in \eqref{eqCartesianSmoothdeRhamalg} where $X' \to Y'$  and  a descendable  morphism. By the first case treated we get a descendable morphism $X_{dR^+}^{',\alg} \to Y_{dR^+}^{',\alg}$ whose pullback to $Y_{dR^+}^{\alg}$ gives rise a descendable morphism $X_{dR^+}^{\alg} \to Y_{dR^+}^{\alg}$, in particular a $\s{D}$-cover as wanted. 
\end{enumerate}
\end{proof}

\begin{theorem}[Six functors for algebraic $D$-modules]
\label{TheoSixFunctorsAlgebraicDMod}
Let us write $R(T)$ for $R\langle T\rangle_{\sol}$, $R\langle T \rangle$ and $R\langle T \rangle^{\dagger}$.  Let $f: X \to Y$ be a morphism of derived Tate adic spaces over $R\otimes \bb{Q}$ locally of $R(T)$-finite presentation, and let $f_{dR^+}^{\alg}: X_{dR^+}^{\alg} \to Y_{dR^+}^{\alg}$ be the associated morphism of algebraic filtered de Rham stacks. Then $f_{dR^+}^{\alg}$ admits $!$-functors. Furthermore, if $f$ is solid smooth (resp. solid \'etale), then $f_{dR^+}^{\alg}$ is cohomologically smooth (resp. cohomologically \'etale).   Moreover, the formation of the filtered de Rham stack satisfies descent for solid smooth covers, namely, solid smooth maps $f$ such that $f^*$ is conservative (cf. Proposition \ref{PropositionSmoothDescent}).  
\end{theorem}
\begin{proof}
By Lemma \ref{LemmaDescentSmoothAlgdeRhamStack} (2) and (3), the formation of the de Rham stack satisfies analytic  $R(T)$-descent for $R(T)= R\langle T \rangle_{\sol}$, $R\langle T \rangle$ and $R\langle T \rangle^{\dagger}$. Then the existence of $!$-functors for morphisms locally of  $R(T)$-finite presentation follows from Lemma \ref{LemmaAlgdeRhamKeyCases} and Proposition \ref{PropositionDevisageSixFiniteType}. The fact that  solid smooth maps, resp. \'etale maps, resp. analytic open subspaces are sent to cohomologically smooth maps, resp. cohomologically \'etale maps, resp. open immersions of locales, follows from Lemma \ref{LemmaDescentSmoothAlgdeRhamStack} (1) and the fact that the formation of the filtered de Rham stack commutes with finite limits.  Finally, descent for solid smooth covers follows from Lemma \ref{LemmaDescentSmoothAlgdeRhamStack} (2).  
\end{proof}

\subsubsection{Hodge filtration of the de Rham cohomology}

We finish this section with the construction of the Hodge filtration of the compactly supported de Rham cohomology  of  a solid smooth morphism of derived Tate adic spaces over $R\otimes \bb{Q}$, this discussion follows closely \cite[\S 2]{BhattGauges}. 

\begin{theorem}
\label{TheoHodgeFiltration}
Let $f:X \to Y$ be a solid smooth morphism of derived Tate adic spaces over $R\otimes \bb{Q}$ of relative dimension $d$, and let $f_{dR^+}: X_{dR^+}^{\alg} \to Y_{dR^+}^{\alg}$ be the associated map of filtered de Rham stacks. Then the compactly supported de Rham cohomology 
\[
DR_{c}(X/Y):= f_{dR^+,!}^{alg} f_{dR^+}^! 1_{Y_{dR^+}^{\alg}}
\]
is complete with respect to its natural filtration. Moreover, we have a Hodge filtration 
\[
\gr^{-i}(DR_{c}(X/Y))\cong f_{!}\Omega^{i,\vee}_{X/Y}(-i)[i].
\]
when pullbacked to an object in $Y$.

\end{theorem}
\begin{proof}
In order to see that $DR_{c}(X/Y)$ is complete, we can work locally in the analytic topology of $X$ and $Y$, and assume that both are affinoid and $f$ is standard solid smooth. Furthermore,  by taking a cartesian square \eqref{eqCartesianSmoothdeRhamalg} as in  the proof of Lemma \ref{LemmaDescentSmoothAlgdeRhamStack}, and proper base change,  we can assume that $Y \times \bb{A}^1/ \bb{G}_m \to Y_{dR^+}^{\alg}$ is surjective. Thus, to prove that $DR_{c}(X/Y)$ is complete, it suffices to consider its pullback to $Y\times \bb{A}^1/ \bb{G}_m$, or equivalently, take the cohomology with compact support of the map $X_{dR^+,Y}^{\alg} \to Y\times \bb{A}^1/ \bb{G}_m$. On the other hand, we have a factorization 
\[
X\to Y\times \bb{G}_{a,\sol}^d \to Y
\]
which gives rise to a factorization 
\[
X_{dR^{+},Y}^{\alg} \xrightarrow{h} Y \times \bb{G}_{a,\sol,dR^+}^{alg,d} \xrightarrow{g} Y \times B \widehat{\bb{G}}_{a}(-1)^{d} \xrightarrow{k} Y\times \bb{A}^1/ \bb{G}_m. 
\]
 Note that we have a diagram with cartesian squares 
\[
\begin{tikzcd} 
 X\times \bb{A}^1/\bb{G}_m \ar[r] \ar[d] & Y\times \bb{G}_{a,\sol}^d  \times \bb{A}^1/\bb{G}_m  \ar[r] \ar[d]& Y\times \bb{A}^1/\bb{G}_m \ar[d] \\
X^{\alg}_{dR^+,Y}\ar[r] & Y \times \bb{G}_{a,\sol,dR^+}^{alg,d} \ar[r] & Y \times B\widehat{\bb{G}}_a(-1)^d.
\end{tikzcd}
\]
By Theorem \ref{TheoSerreDuality} and  Proposition  \ref{PropAlgebraicCartier2}  we have an (a priori non-natural) equivalence 
\[
f_{dR^+,Y}^! 1_{Y\times \bb{A}^1/\bb{G}_m} \simeq h^* g^! k^! 1_{Y \times \bb{A}^1/\bb{G}_m} \simeq 1_{X_{dR^+,Y}^{\alg}}(-d),
\]
in particular it is filtered complete, and then so is   $h_!f_{dR^+,Y}^! 1_{Y\times \bb{A}^1/\bb{G}_m}$.  One deduces  completeness for $DR_{c}(X/Y)$ from the following lemma: 
\begin{lemma}
Let $g:Y\times  B\widehat{\bb{G}}_{a}(-1)^d \to  Y\times  \bb{A}^1/ \bb{G}_m$. For any  filtered complete module $\n{F}\in \Mod_{\sol}(Y\times  B\widehat{\bb{G}}_{a}(-1)^d)$, $g_! \n{F} \in \Mod_{\sol}(Y\times \bb{A}^1/ \bb{G}_m)$ is also filtered complete. 
\end{lemma}
\begin{proof}
By Cartier duality Theorem \ref{TheoAlgebraicCartierDuality}, we have a natural $\Mod_{\sol}(Y\times \bb{A}^1/\bb{G}_m)$-linear equivalence of categories of solid quasi-coherent sheaves 
\[
FM^{-1}:\Mod_{\sol}(Y\times B\widehat{\bb{G}}_{a}(-1)^d) \cong \Mod_{\sol}(Y\times \bb{G}_{a}(1)^d).
\]
Let $\iota: Y\times \bb{A}^1/\bb{G}_m\to  Y\times \bb{G}_a(1)^{d}$ be the zero section, then Corollary \ref{CoroFourierMoukaiIdentities} (ii) implies that 
\[
g_! \n{F} =   \iota^* FM^{-1}(\n{F} \otimes  (g^! 1_{Y\times \bb{A}^1/\bb{G}_m} )^{-1}) = \iota^* FM^{-1}(\n{F} \otimes \s{O}(d)[d]).
\]
But the object $\iota^* FM^{-1}(\n{F} \otimes \s{O}(d)[d])$ has a finite Koszul filtration with graded pieces given by twists of finite direct sums of $FM^{-1}(\n{F} \otimes \s{O}(d)[d])$, proving that it is still complete. 
\end{proof}

It is left to compute the graded pieces of $DR_{c}(X/Y)$, this follows from Proposition \ref{PropAlgebraicCartier2} and that $X_{Hodge,Y}^{\alg}= (X\times B\bb{G}_m)/\widehat{\n{T}}_{X/Y}(-1)$ by Proposition \ref{PropAlgebraicDeRham} (2.b).
\end{proof}

\begin{remark}
\label{remarkDeRhamcohomology}
Theorem \ref{TheoHodgeFiltration} was stated for de Rham cohomology with compact supports due to its well behaviour with respect to the six functors. One recovers the completeness and Hodge filtration of the usual de Rham cohomology by taking duals. Moreover, the fact that the cohomology of the de Rham stack coincides with the hypercohomology of the de Rham complex for classical smooth morphisms of rigid spaces follows by the same argument as in \cite[Theorem 2.3.6]{BhattGauges} by reduction to the case of the unit disc, we left the details to the reader. 
\end{remark}

\subsection{The analytic de Rham stack}
\label{SubsectionAnalyticdeRhamStack}

In Section \ref{SubsectionAlgebraicdeRhamStack} we introduced the algebraic filtered de Rham stack and proved that it has a reasonable theory of six functors for derived Tate adic spaces. In the next section we will introduce a variant of this construction that for our convenience we  specialize to Tate stacks over $\bb{Q}_p$.  The new  theory of $D$-modules  obtained from this stack is an enhancement of the theory of $\wideparen{\n{D}}$-modules of Ardakov and Wadsley \cite{ArdakovWadsleyDI,ArdakovWadsleyII}, that we call analytic $D$-modules.  A more concrete comparison between analytic $D$-modules and Ardakov and Wadsley's $\wideparen{\n{D}}$-modules is left to a future work. For example, for a smooth rigid space $X$,  we expect coadmissible $\wideparen{\n{D}}$-modules to be precisely the smooth objects on $X_{dR}$.  Instead, we  shall construct  a six functor formalism for analytic $D$-modules, and prove good cohomological properties for morphisms of solid finite presentation. Once  the relation between analytic $D$-modules and $\wideparen{\n{D}}$-modules  is made, the six functors constructed hereby will give a very large extension of the six functors of Bode in \cite{bode2021operations}.  

\subsubsection{Construction of the stacks}

Let $\Aff^{b}_{\bb{Q}_p}$ be the category of bounded affinoid analytic spaces over $\bb{Q}_p$, $\PSh(\Aff^{b}_{\bb{Q}_p})$ the category of prestacks on $\Aff^{b}_{\bb{Q}_p}$ and $\Sh_{\s{D}}(\Aff^{b}_{\bb{Q}_p})$ the category of Tate stacks over $\bb{Q}_p$. 

\begin{remark}
The functor $\Aff^{b,\op}_{\bb{Q}_p} \to \Ani$ sending $\n{A} \mapsto \n{A}(*)$ is not longer represented by the algebraic affine line $\bb{A}^1$, instead, it is represented by its analytification $\bb{A}^{1,\an}$ as a rigid space over $\bb{Q}_p$. Similarly, the functor $\n{A}\mapsto \n{A}(*)^{\times}$ of units is represented by the analytification of the multiplicative group $\bb{G}_m^{\an}$. 
\end{remark}

\begin{definition}
\label{DefinitionAnalyticDeRhamStack}
We define the following objects 
\begin{enumerate}

\item Let $X\in \PSh(\Aff^{b}_{\bb{Q}_p})$. The \textit{absolute filtered analytic  de Rham prestack of $X$} is the prestack  over  $\bb{A}^{1,\an}/ \bb{G}_m^{\an}$ defined as 
\[
X_{dR^+}(\s{O}(1)\to \n{A}) = X(\cone[\Nil^{\dagger}(\n{A}) \otimes \s{O}(-1) \to \n{A}]). 
\]
The \textit{absolute analytic de Rham prestack $X_{dR}$ (resp. the absolute analytic  Hodge prestack $X_{Hodge}$)}  is the pullback of $X_{dR^+}$ to $*= \bb{G}_m^{\an}/ \bb{G}_m^{\an}$ (resp. the pullback to $*/\bb{G}_{m}^{\an}$).  The $\s{D}$-sheafification of $X_{dR^+}$ is called the filtered analytic de Rham stack and denoted in the same way (similarly for $X_{dR}$ and $X_{Hodge}$). 

\item Let $S \in \PSh(\Aff^{b}_{\bb{Q}_p})$  and let $X$ be prestack over $S$. The \textit{relative filtered analytic de Rham prestack of $X$ over $S$} is the pullback 
\[
\begin{tikzcd}
X_{dR^+,S}\ar[r] \ar[d]& S \times \bb{A}^{1,\an}/ \bb{G}_m^{\an} \ar[d]\\ 
X_{dR^+} \ar[r]& S_{dR^+}.
\end{tikzcd}
\]
 The \textit{relative analytic de Rham and Hodge prestacks} are defined as the pullback of $X_{S,dR^+}$ to  $*= \bb{G}_m^{\an}/ \bb{G}_m^{\an}$ and $*/\bb{G}_m^{\an}$ respectively. If $S$ is a Tate stack the relative filtered analytic de Rham stack of $X$ over $S$ is the $\s{D}$-sheafification of $X_{dR^+,S}$ that  denote it in the same way. We define in the obvious way the relative analytic de Rham  and Hodge stacks. 

\end{enumerate}
\end{definition}

Next, we prove some formal properties of the analytic de Rham stack that are deduced from the definition, cf. Proposition \ref{PropAlgebraicDeRham}.

\begin{prop}
\label{PropAnDeRhamStack}
Let $f:X \to Y$ be a morphism of prestacks on $\Aff^{b}_{\bb{Q}_p}$. 

\begin{enumerate}

\item  Suppose that $f$ is $\dagger$-formally  \'etale, then the natural map $X \times \bb{A}^{1,\an}/ \bb{G}_{m}^{\an} \to X_{dR^+,Y}$ is an equivalence.

\item Suppose that $f$ is $\dagger$-formally smooth and let $\n{T}_{X/Y}^{\an}= (\AnSpec_{X} \Sym^{\bullet}_{X} \bb{L}_X)^{\an}$. The following hold

\begin{enumerate}
\item  The map $X \times \bb{A}^{1,\an}/ \bb{G}_m^{\an} \to X_{dR^+,Y}$ is an epimorphism. 

\item There is a natural equivalence $X_{Hodge,Y}= (X\times */ \bb{G}_m^{\an})/ \n{T}_{X/Y}^{\dagger}(-1)$.

\end{enumerate}

\item The formation of $X\mapsto X_{dR^+}$ commutes with small colimits and  limits of prestacks.

\end{enumerate}
\end{prop}
\begin{proof}
\begin{enumerate}

\item This follows from the notion of $\dagger$-formally \'etaleness, cf. Definition \ref{DefinitionDaggerFormalSmoothEtale}.

\item Part (a) follows from the notion of $\dagger$-formally smoothness. For part (b), let $\n{A}\in \Aff^{b}_{\bb{Q}_p}$, let $\s{O}(-1) \to \n{A}$ be a generalized Cartier divisor and let $\eta\in Y(\n{A})$.   Then the fiber of $X_{Hodge,Y}$ over $\eta$ is given by the fiber product of  
\[
\begin{tikzcd}
& (\eta ,\s{O}(1))\ar[d]  \\
& Y(\n{A})\times */ \bb{G}_m^{\an}  \ar[d]\\ 
X(\Nil^{\dagger}(\n{A})\otimes \s{O}(-1)[1]\oplus \n{A})  \ar[r] & Y(\Nil^{\dagger}(\n{A})\otimes \s{O}(-1)[1]\oplus \n{A})
\end{tikzcd}
\]
which is represented by 
\[
\Map_{\Mod_{\geq 0}(\n{A})}( \eta^* \bb{L}_{X/Y}, \Nil^{\dagger}(\n{A}) \otimes \s{O}(-1)[1] )= \eta^*\bb{L}_{X/Y}^{\vee} \otimes \Nil^{\dagger}(\n{A}) (-1)(*)[1].
\]
This shows that $X_{Hodge,Y}= (X\times */\bb{G}_m^{\an})/ \n{T}_{X/Y}^{\dagger}(-1)$ as wanted. 

\item  This follows immediately from the definition of the filtered analytic de Rham prestack, as limits and colimits of prestacks are computed at the level of points.
\end{enumerate}
\end{proof}

Similarly as for the algebraic de Rham stack, there is a notion of $\dagger$-formal completion or $\dagger$-neighbourhood for a Zariski closed immersion.

\begin{prop}
\label{PropDaggerFormalCompletion}
Let $X= \AnSpec \n{B} \to  Y=\AnSpec \n{A}$ be a morphism of bounded derived Tate adic spaces over $\bb{Q}_p$ which is surjective on $\pi_0$ and that  has the induced analytic structure (i.e. a Zariski closed immersion).  Let $\n{A}^{\dagger/\n{B}}$ be the idempotent $\n{A}$-algebra associated to the closed subspace $\Spa^{\dagger}\n{B} \subset \Spa^{\dagger} \n{A}$ of Proposition \ref{PropositionCLosedImageLocale}. Then there is a natural equivalence
\[
X_{dR,Y} \cong \AnSpec \n{A}^{\dagger/\n{B}}. 
\]
\end{prop}
\begin{proof}
By the proof of Proposition \ref{PropositionCLosedImageLocale} the map $\n{A}^{\dagger/\n{B}} \to \n{B}$ induces an equivalence in $\dagger$-reduced algebras. This gives rise an equivalence $\AnSpec(\n{A}^{\dagger/\n{B}})|_{\Aff^{b,\dagger-\red}_{\bb{Q}_p}} = \AnSpec (\n{B})|_{\Aff^{b,\dagger-\red}_{\bb{Q}_p}}$. But by definition, $X_{dR}$ is the right   Kan extension of the restriction of $X$ to $\Aff^{b,\dagger-\red}_{\bb{Q}_p}$. The proposition follows since $X_{dR} \to Y_{dR}$ is an immersion, and $X_{dR,Y} \subset Y$ is the full sub prestack mapping to $X_{dR}$. 
\end{proof}

As a consequence, we obtain Kashiwara equivalence for analytic $D$-modules and Zariski closed immersions. 

\begin{corollary}[Kashiwara equivalence]
\label{CoroKashiwaraEquivalence}
Let $X\to Y$ be a Zariski closed immersion of derived Tate adic spaces over $\bb{Q}_p$, and let $Y^{\dagger/X}\subset X$ be the overconvergent neighbouhood of $X$ in $Y$ obtained by gluing the rings $\n{B}^{\dagger/\n{A}}$ of Proposition \ref{PropDaggerFormalCompletion} in the analytic topology of $X$. Then there is a natural equivalence of analytic de Rham stacks $X_{dR}= Y^{\dagger/X}_{dR}$. In particular, the category of analytic $D$-modules of $Y$ supported on $X$ is equivalent to the category of analytic $D$-modules of $X$. 
\end{corollary}
\begin{proof}
By Lemma \ref{LemmaDescentAnalyticDmod} down below the formation of the analytic de Rham stacks satisfies descent for the analytic topology. Then, it suffices to prove the statement in the affinoid case of Proposition \ref{PropDaggerFormalCompletion}. But the analytic de Rham stack of $X$ if the $\s{D}$-sheafification of the right Kan extension of the restriction of $X$ to $\Aff^{b,\dagger-\red}_{\bb{Q}_p}$, and the $\dagger$-reductions of $\n{A}$ and  $\n{B}^{\dagger/\n{A}}$ are isomorphic by construction. This proves the corollary. 
\end{proof}

\subsubsection{Six functor formalism for analytic $D$-modules}

\begin{definition}
Let $S \in \PSh(\Aff^{b}_{\bb{Q}_p})$. We define the six functor formalisms $\s{D}_{dR^+,S}$, $\s{D}_{dR,S}$ and $\s{D}_{Hodge,S}$ for $\PSh(\Aff^{b}_{\bb{Q}_p})_{/S}$ to be the six functor formalism obtained by Lemma \ref{LemmaFunctoriality6Functors} applied to the functors $X \mapsto X_{dR^+,S}$, $X_{dR,S}$ and $X_{Hodge,S}$, landing in $\s{D}$-stacks over $S \times \bb{A}^{1,\an}/\bb{G}_m^{\an}$, $S $ and $S\times */\bb{G}_m^{\an}$ respectively. The six functor formalism $\s{D}_{dR^+,S}$ (resp. $\s{D}_{dR,S}$, resp. $\s{D}_{Hodge,S}$) is called the six functor formalism of filtered  analytic $D$-modules over $S$ (resp. of analytic $D$-modules over $S$, resp. of Hodge modules over $S$). 
\end{definition}

\begin{remark}
The analogue of Lemma \ref{LemmaTransmuteDcovers} holds for the analytic de Rham stack, namely, if $X \to Y$ is a canonical cover such that $X_{dR^+,S} \to Y_{dR^+,S}$ is a $\s{D}$-cover, then $X\to Y$ is a $\s{D}_{dR^+,S}$-cover. 
\end{remark}

Our next task is to show that the six functor formalism of filtered analytic $D$-modules admits $!$-functors for morphisms locally of solid finite presentation of derived Tate adic spaces. We will even prove the existence of $!$-functors for morphisms locally of finite presentation of Berkovich or $\dagger$-adic spaces, see Definition \ref{DefinitionBerkovichDaggerSpaces}. Finally, we will show that solid smooth and \'etale morphisms of derived Tate adic spaces give rise to cohomologically smooth and \'etale maps at the level of filtered analytic de Rham stacks.

\begin{lemma}
\label{LemmaKeyCasesAnalyticdeRham}
Let $\bb{G}_{a,\star}$ denote the analytic spectrum of one of the algebras $\bb{Q}_p \langle T \rangle_{\sol}$, $\bb{Q}_p\langle T \rangle$ or $\bb{Q}_p \langle T \rangle^{\dagger}$. Let $\bb{G}_{a}(-1) \to \bb{A}^{1,\an}/ \bb{G}_m^{\an}$ be the line bundle obtained by the analytic spectrum of $\Sym^{\bullet}_{\bb{A}^{1,\an}/\bb{G}_m^{\an}} (\s{O}(1))$, and let $\bb{G}_{a}(-1)^{\dagger}$ be its overconvergent neighbourhood at the zero section; we have a natural morphism of group objects $\bb{G}_a(-1)^{\dagger} \to \bb{G}_{a,\star}\times \bb{A}^{1,\an}/\bb{G}_{m}^{\an}$.  The following hold: 
\begin{enumerate}
\item $\bb{G}_{a,\star,dR^+}= (\bb{G}_{a,\star}\times \bb{A}^{1,\an}/\bb{G}_{m}^{\an})/ \bb{G}_{a}(-1)^{\dagger}$.  In particular, the maps 
\[
f:*_{dR^+}= \bb{A}^{1,\an}/ \bb{G}_m^{\an} \to \bb{G}_{a,\star,dR^+}
\]
and 
\[
g: \bb{G}_{a,\star,dR^+} \to *_{dR^+}
\]
admit $!$-functors. Furthermore, the following properties are satisfies: 
\begin{itemize}
\item[(a)]  $f$ is always weakly cohomologically proper.  Moreover, its pullback  to the de Rham stack is $(-1)$-truncated so cohomologically proper. 

\item[(b)] If $\bb{G}_{a,\star}= \bb{G}_{a,\sol}= \AnSpec \bb{Q}_p\langle T \rangle_{\sol}$, then $g$ is a cohomologically smooth map. 

\item[(c)] If $\bb{G}_{a,\star}= \AnSpec \bb{Q}_p\langle T \rangle$ or $\bb{G}_{a,\star}= \AnSpec \bb{Q}_p \langle T\rangle^{\dagger}$, then $g$ is weakly cohomologically proper. Furthermore, its pullback to the de Rham stack is  $0$-truncated so cohomologically proper, 

\end{itemize} 

\item  Let $n\geq 1$ be an integer and denote $X_n= \AnSpec \Sym_{\bb{Q}_p}^{\bullet} \bb{Q}_p[n]$.  Then 
\[
X_{n,dR^+}= (\bb{G}_{a}(-1)^{\dagger,\times_{\bb{G}_a} n+1})/ \bb{G}_{a}(-1)^{\dagger},
\]
where $\bb{G}_{a}(-1)^{\dagger}$ acts diagonally. In particular, $X_{dR^+} \to *_{dR^+}$ and $*_{dR^+} \to X_{dR^+} $ admit $!$-functors and are  weakly cohomologically proper.  Furthermore, their restriction to the de Rham stack is an equivalence. 

\end{enumerate} 
\end{lemma}
\begin{proof}
\begin{enumerate}

\item  Let $\bb{G}_a^{\an}= \bb{A}^{1,\an}$ be the analytic affine line seen as an additive group. By definition, $\bb{G}_{a,dR^+}^{\an}$ represents the functor on bounded affinoid algebras over $\bb{A}^{1,\an}/ \bb{G}_{m}$  given by 
\[
(\s{O}(-1)\to \n{A})\mapsto \cone(\Nil^{\dagger}\otimes \s{O}(-1) \to \n{A} )(*).
\]
For $\bb{G}_{a,\star}$, since the image of $\Nil^{\dagger}(\n{A}) \otimes \s{O}(-1)$ in $\pi_0(\n{A})$ is $\dagger$-nilpotent, by Proposition \ref{PropLiftingOverconvergentAlgebrasLift} the sub prestack $\bb{G}_{a,\star,dR^+} \subset \bb{G}_{a,dR}$ consists on the functor 
\[
(\s{O}(-1) \to \n{A}) \mapsto \cone(\Nil^{\dagger}(\n{A}) \otimes \s{O}(-1) (*) \to \bb{G}_{a,\star}(\n{A})). 
\]
This shows that $\bb{G}_{a,\star,dR^+}$ is represented by the stack $(\bb{G}_{a,\star}\times \bb{A}^{1,\an}/ \bb{G}_m^{\an})/ \bb{G}_a(-1)^{\dagger}$ as wanted. 

\begin{itemize}

\item[(a)] The map $*_{dR^+} \to \bb{G}_{a,\star,dR^+}$ is equivalent to the morphism
\[
\bb{G}_{a}(-1)^{\dagger}/ (\bb{G}_{a}(-1)^{\dagger} \to (\bb{G}_{a,\star}\times \bb{A}^{1,\an})/ \bb{G}_m^{\an})/ \bb{G}_a(-1)^{\dagger}.
\]This map admits $!$-functors and  is weakly cohomologically proper since  $ \bb{G}_{a}(-1)^{\dagger} \to (\bb{G}_{a,\star}\times \bb{A}^{1,\an}/ \bb{G}_m^{\an})$ has the induced analytic structure.  Furthermore, its pullback to the analytic de Rham stack is an immersion so $(-1)$-truncated. 

\item[(b) and (c)] We have factorizations for $\bb{G}_{a,\star,dR^+} \to *_{dR^+}$
\[
( \bb{G}_{a,\star} \times \bb{A}^{1,\an}/ \bb{G}_m^{\an})/ \bb{G}_a(-1)^{\dagger} \to (\bb{A}^{1,\an}/ \bb{G}_m^{\an})/ \bb{G}_{a}(-1)^{\dagger} \to \bb{A}^{1,\an}/ \bb{G}_m^{\an}.
\]
If $\bb{G}_{a,\star}=\bb{G}_{a,\sol}$, the first map is cohomologically smooth by Theorem \ref{TheoSerreDuality}, and the second is cohomologically smooth by Theorem \ref{TheoremAnalyticCartierII} and Proposition \ref{PropAlgebraicCartier2}, so the composite is also cohomologically smooth.  If $\bb{G}_{a,\star}= \AnSpec \bb{Q}_p \langle T \rangle $ or $\AnSpec \bb{Q}_p \langle T \rangle^{\dagger}$, then the first map is weakly cohomologically proper since it has the induced analytic structure, and the second is weakly cohomologically proper  by Theorem \ref{TheoremAnalyticCartierII} and Proposition \ref{PropAlgebraicCartier2}, thus, the composite is weakly cohomologically proper proving what we wanted.  It is clear that its pullback to the the Rham stack is $0$=truncated. 

\end{itemize}

\item By Proposition \ref{PropAnDeRhamStack} (3), the formation $X\mapsto X_{dR^+}$ commutes with all small limits and colimits. Then, since $X_{n+1}= * \times_{X_n} *$, an inductive argument gives the desired description of $X_{n,dR^+}$. The fact that $*_{dR^+} \to X_{n,dR^+}$ and $X_{n,dR^+} \to *_{dR^+}$ admit $!$-functors and that they are weakly cohomologically proper follows a similar  argument as part (1).  Finally, since the $\dagger$-reduction of $X_n$ is $*$, they give rise to the sane analytic de Rham stack. 
\end{enumerate}
\end{proof}

We now prove an analogue of Lemma \ref{LemmaDescentSmoothAlgdeRhamStack}. 

\begin{lemma}
\label{LemmaDescentAnalyticDmod}
Let $f:X \to \AnSpec \n{A} \to Y \AnSpec \n{B}$ be a morphism of bounded affinoid rings over $\bb{Q}_p$.

\begin{enumerate}
\item  If $f$ is   standard solid smooth (resp.   standard solid \'etale) then $f_{dR^+}: X_{dR^+} \to Y_{dR^+}$ is cohomologically smooth (resp. \'etale).  

\item Suppose that $X= \bigsqcup_{i=1}^d X_i$ with $X_i \to Y$  standard solid smooth. If $f$ is a smooth $\s{D}$-cover, then $f_{dR^+}:X_{dR^+} \to Y_{dR^+}$ is a smooth $\s{D}$-cover. 

\item Suppose that $f: X \to Y$ is a rational cover for the coordinate theories $\bb{Q}_p\langle T \rangle$ and $\bb{Q}_p \langle T \rangle^{\dagger}$. Then $X_{dR^+}\to Y_{dR^+}$ is a descendable $\s{D}$-cover.  

\end{enumerate}
\end{lemma}
\begin{proof}
The proof is virtually the same of Lemma \ref{LemmaDescentSmoothAlgdeRhamStack}; the only key step is to have generators $\n{B}$ of $\AffRing^{b}_{\bb{Q}_p}$ such that $Y \times \bb{A}^{1,\an}/ \bb{G}_{m}^{\an} \to Y_{dR^+}$ is surjective as $\s{D}$-stacks with $Y= \AnSpec \n{B}$. For this, we can take affinoid spaces of the form $Y= \AnSpec \n{B}$ with $\n{B}= \bb{Q}_p \langle \underline{X} \rangle_{\sol} \langle \bb{N}[K] \rangle$, where $\underline{X}$ is a finite set of variables and $K$ is a profinite set. Then, the surjection of $Y \to Y_{dR^+}$ is a consequence of Proposition \ref{PropLiftingOverconvergentAlgebrasLift}. 
\end{proof}

\begin{theorem}[Six functors for analytic $D$-modules]
\label{TheoSixFunctorsanDmodules}
Let us write $\bb{Q}_p(T)$ for $\bb{Q}_p\langle T\rangle_{\sol}$, $\bb{Q}_p\langle T \rangle$ and $\bb{Q}_p\langle T \rangle^{\dagger}$.  Let $f: X \to Y$ be a morphism of derived Tate adic spaces over $\bb{Q}_p$ locally of $\bb{Q}_p(T)$-finite presentation, and let $f_{dR^+}: X_{dR^+}\to Y_{dR^+}^{\alg}$ be the associated morphism of algebraic filtered de Rham stacks. Then $f_{dR^+}$ admits $!$-functors. Furthermore, if $f$ is solid smooth (resp. solid \'etale), then $f_{dR^+}$ is cohomologically smooth (resp. cohomologically \'etale).   Moreover, the formation of the filtered de Rham stack satisfies descent for solid smooth covers, namely, solid smooth maps $f$ such that $f^*$ is conservative (cf. Proposition \ref{PropositionSmoothDescent}).  Finally, if $f$ is a qcqs morphism of finite presentation of $\bb{Q}_p\langle T \rangle$ or $\bb{Q}_p\langle T \rangle^{\dagger}$-adic spaces, the map $f_{dR^+}: X_{dR^+} \to Y_{dR^+}$ is co-smooth. 
\end{theorem}
\begin{proof}
This follows the same argument of Theorem \ref{TheoSixFunctorsAlgebraicDMod}, by replacing Lemmas \ref{LemmaAlgdeRhamKeyCases} and \ref{LemmaDescentSmoothAlgdeRhamStack} with Lemmas \ref{LemmaKeyCasesAnalyticdeRham} and \ref{LemmaDescentAnalyticDmod} respectively. For the last statement about qcqs morphisms of finite presentation, by Lemma \ref{LemmaDescentAnalyticDmod} (3) it suffices to prove the claim when $f$ is a morphism of finite presentation of affinoid rings, this case follows from Lemma \ref{LemmaKeyCasesAnalyticdeRham} and an inductive argument. 
\end{proof}

\subsubsection{Comparison with algebraic $D$-modules}

Let $X \to Y$ be a solid smooth morphism of derived Tate adic spaces over $\bb{Q}_p$. Consider the relative algebraic de Rham stack $X_{dR^+,Y}^{\alg}$ and let $X_{dR^+,Y}^{\alg'}$ be its pullback to $\bb{A}^{1,\an}/ \bb{G}_{m}^{\an} \to \bb{A}^1/ \bb{G}_m$, or equivalently,  its restriction to $\Aff^{b}_{\bb{Q}_p}$.  The definition at the level of points yields a natural map of de Rham stacks $X_{dR^+,Y}^{\alg'} \to X_{dR^+,Y}$. In the following paragraph we will show that the category of analytic $D$-modules of $X$ over $Y$ embeds fully faithful in the category of algebraic $D$-modules of $X$ over $Y$. We will also deduce that the de Rham cohomology is the same when computed with the algebraic or analytic de Rham stacks. 

\begin{prop}
Let $f:X \to Y$ be a solid smooth morphism of derived Tate adic spaces over $\bb{Q}_p$. Then the natural map $g:X_{dR^+,Y}^{alg'} \to X_{dR^+,Y}$ is cohomologically co-smooth. Furthermore, the natural map 
\begin{equation}
\label{eqComparisonAlgAnDmod}
 1_{X_{dR^+,Y}} \to g_* 1_{X_{dR^+,Y}}
\end{equation}
is an equivalence. In particular, $g^*:\Mod_{\sol}(X_{dR^+,Y}) \to \Mod_{\sol}(X_{dR^+,Y}^{\alg,'})$ is a fully faithful embedding. 
\end{prop}
\begin{proof}
Both statements are local in the analytic topology of $X$ and $Y$, hence we can assume that $f$ is a standard solid smooth morphism of bounded affinoid spaces. Both claims are also preserved by base change on $Y$, so we can assume without loss of generality that $Y= \AnSpec \n{B}$ with $\n{B}= \bb{Q}_p \langle \underline{X} \rangle \langle \bb{N}[K] \rangle$, with $X$ a finite set of variables and $K$ a profinite set.  We can factor $X \to Z \to Y$ where $X\to Z$ is standard solid \'etale and $Z= Y \times \bb{G}_{a,\sol}^{d}$. We then have Cartesian diagrams 
\[
\begin{tikzcd} 
X\times \bb{A}^{1,\an}/ \bb{G}_m^{\an} \ar[r] \ar[d]& Z \times \bb{A}^{1,\an}/ \bb{G}_m^{\an} \ar[d] & &  X \times \bb{A}^{1,\an}/ \bb{G}_m^{\an} \ar[r] \ar[d]& Z  \times \bb{A}^{1,\an}/ \bb{G}_m^{\an} \ar[d]\\
X_{dR^+,Y}\ar[r] &   Z_{dR^+,Y} & & X_{dR^+,Y}^{\alg}\ar[r] &   Z_{dR^+,Y}^{\alg'}.
\end{tikzcd}
\]
Since $Z \times \bb{A}^{1,\an}/ \bb{G}_m^{\an} \to Z_{dR^+,Y}^{\alg'}$ and $X\times \bb{A}^{1,\an}/ \bb{G}_m^{\an} \to X_{dR^+,Y}^{\alg'}$ are surjective, we have that $X_{dR^+,Y}^{\alg'}= X_{dR^+,Y} \times_{Z_{dR^+,Y}} Z_{dR^+,Y}^{\alg'}$. Thus, by proper base change, we are reduced to consider the case of $X= Y\times \bb{G}_{a,\sol}^n$.  By proper base change, the statement is also stable under fiber products over $Y$, so it suffices to consider $X= Y \times \bb{G}_{a,\sol}$, and by base change assume that $Y=*$. By Lemmas \ref{LemmaAlgdeRhamKeyCases} and \ref{LemmaKeyCasesAnalyticdeRham} we have the explicit descriptions
\[
\bb{G}_{a,\sol,dR^+}^{\alg'}= (\bb{G}_{a,\sol}\times \bb{A}^{1,\an}/\bb{G}_m^{\an})/ \widehat{\bb{G}}_a(-1) 
\]
and 
\[\bb{G}_{a,\sol,dR^+}= (\bb{G}_{a,\sol}\times \bb{A}^{1,\an}/\bb{G}_m^{\an})/ \bb{G}_a(-1)^{\dagger}.  
\]
Consider the cartesian square
\[
\begin{tikzcd}
\left((\bb{G}_{a,\sol}\times \bb{A}^{1,\an}/ \bb{G}_m^{\an}) \times_{\bb{A}^{1,\an}/ \bb{G}_m^{\an}} \bb{G}_a(-1)^{\dagger} \right) /\widehat{\bb{G}}_{a}(-1) \ar[r, "h"] \ar[d] & (\bb{G}_{a,\sol}\times \bb{A}^{1,\an}/ \bb{G}_m^{\an}) \ar[d] \\ 
\bb{G}_{a,\sol,dR^+}^{\alg'}\ar[r] & \bb{G}_{a,\sol.dR^+} 
\end{tikzcd}
\] 
where $\widehat{\bb{G}}_{a}(-1)$ acts on the fiber product diagonally.  Then, to show that $g$ is co-smooth and that \eqref{eqComparisonAlgAnDmod} is an equivalence, it suffices to prove the analogue statements for the map $h$. We have an equivalence
\[
\left((\bb{G}_{a,\sol}\times \bb{A}^{1,\an}/ \bb{G}_m^{\an}) \times_{\bb{A}^{1,\an}/ \bb{G}_m^{\an}} \bb{G}(-1)^{\dagger} \right) /\widehat{\bb{G}}_{a}(-1) \cong  \bb{G}_{a,\sol} \times \bb{G}_a(-1)^{\dagger}/ \widehat{\bb{G}}_a(-1)
\]
induced from the action map by translations $\bb{G}_{a}(-1)^{\dagger}\times  (\bb{G}_{a,\sol}\times_{\bb{A}^{1,\an}/ \bb{G}_m^{\an}} \bb{A}^{1,\an}/\bb{G}_m) \to (\bb{G}_{a,\sol}\times \bb{A}^{1,\an}/\bb{G}_m)$. Under this equivalence, the action map becomes the projection 
\[
\bb{G}_{a,\sol} \times (\bb{G}_{a}(-1)^{\dagger}/ \widehat{\bb{G}}_{a}(-1)) \to \bb{G}_{a}\times \bb{A}^{1,\an}/ \bb{G}_m^{\an}. 
\]
Thus, by  base change, we are reduce to prove the claims for the map 
\[
(\bb{G}_{a}(-1)^{\dagger}/ \widehat{\bb{G}}_{a}(-1))  \to  \bb{A}^{1,\an}/ \bb{G}_m^{\an}.
\]
This follows from the following lemma 
\begin{lemma}
Let $X$ be an analytic  $\s{D}$-stack  over $\bb{Q}_p$ and $\s{F}$ a vector bundle over $X$ of rank $d$. Consider the quotient $\bb{V}(\s{F})^{\dagger}/ \widehat{\bb{V}(\s{F}) } $ where $\widehat{\bb{V}(\s{F})}$ acts by translations.   Then the  morphism 
\[
g:\bb{V}(\s{F})^{\dagger,\alg}_{dR,X}= \bb{V}(\s{F})^{\dagger}/ \widehat{\bb{V}(\s{F}) } \to X
\]
is co-smooth and the natural map 
\[
1_{X} \to g_* 1_{\bb{V}(\s{F})^{\dagger}/ \widehat{\bb{V}(\s{F}) }}
\]
is an equivalence. 
\end{lemma}
\begin{proof}
By proper base change we can reduce to the universal case $X=*/\GL_d$ and $\s{F}= \St$. Furthermore, since $* \to */\GL_d$ is surjective, it suffices to prove the claim after taking pullbacks to $*$. Then $\St \cong \bb{Q}_p^d$ and since the claim holds after finite fiber products, it suffices to consider the case of $h:\bb{G}_a^{\dagger}/ \widehat{\bb{G}}_{a} \to *$. But the map $h$ factors as  $\bb{G}_a^{\dagger}/ \widehat{\bb{G}}_{a}  \to */ \widehat{\bb{G}}_a \to *$, the first arrow is weakly cohomologically proper since it has the induced analytic structure, and the second is co-smooth by Proposition \ref{PropAlgebraicCartier2}; one deduces that $h$ is itself co-smooth. Finally, $g_* 1_{\bb{G}_{a}^{\dagger}/ \widehat{\bb{G}}_{a}}$ is nothing but the de Rham cohomology of $\bb{G}_{a}^{\dagger}$ which is equal to $\bb{Q}_p$ by the Poincar\'e lemma.  
\end{proof}
\end{proof}

\begin{corollary}
\label{CorollaryComparisonDeRhamCohomology}
Let $f:X \to Y$ be a solid smooth morphism of derived Tate adic spaces, and let $f_{dR^+,Y}: X_{dR^+,Y}\to Y \times \bb{A}^{1,\an}/ \bb{G}_m^{\an}$. Then $f_{dR^+,Y,*} 1_{X_{dR^+,Y}}$ is filtered complete and equal to the de Rham cohomology $DR(X/Y)$, namely, the dual of the de Rham cohomology with compact supports of Theorem \ref{TheoHodgeFiltration} (see Remark \ref{remarkDeRhamcohomology}).
\end{corollary}
\begin{proof}
Consider the commutative diagram 
\[
\begin{tikzcd}
X_{dR^+,Y}^{\alg'} \ar[r, "h"] \ar[rd, "f_{dR^+,Y}^{\alg}"'] & X_{dR^+,Y} \ar[d, "f_{dR^+,Y}"] \\ 
& Y
\end{tikzcd}
\]
Then 
\[
DR(X/Y)  = f_{dR^+,Y,*}^{\alg} 1   = f_{dR^+,Y,*} h_* 1= f_{dR^+,Y,*} 1. 
\]
\end{proof}

\subsection{Poincar\'e duality for  $D$-modules}
\label{SubsectionPoincareDmodules}

Next, we prove Poincar\'e duality for filtered algebraic and analytic $D$-modules. The strategy is similar as for coherent cohomology by taking the deformation to the normal cone. We shall adapt \cite[\S 4]{zavyalov2023poincare} to derived Tate adic spaces. 

\begin{definition}[{\cite[Definition 4.2.1]{zavyalov2023poincare}}]
Let $\s{C}$ be the category of derived Tate adic spaces over $\bb{Q}_p$. A six functor formalism $\s{D}$ on $\s{C}$ is \textit{premotivic} if the following hold: 
\begin{enumerate}
\item  It is $\bb{A}^{1,\an}$-acyclic, i.e., if  we denote $f: \bb{A}^{1,\an}\to *$, then the natural map $1 \to f_* 1_{\bb{A}^{1,\an}}$ is an equivalence in $\s{D}(*)$. 

\item Any any solid smooth morphism $f: X \to Y$ is cohomologically smooth with respect to $\s{D}$. 

\end{enumerate}
\end{definition}

\begin{remark}
Theorems \ref{TheoSixFunctorsAlgebraicDMod} and \ref{TheoSixFunctorsanDmodules} imply that  solid smooth maps are cohomologically smooth for the six functors $\s{D}_{dR^+}^{\alg}$ and $\s{D}_{dR^+}$. Furthermore, de Rham cohomology of the analytic affine line $\bb{A}^{1,\an}$ is trivial, by  Theorem \ref{TheoHodgeFiltration} and Corollary \ref{CorollaryComparisonDeRhamCohomology} we deduce that both $\s{D}_{dR^+}^{\alg}$ and $\s{D}_{dR^+}$ are motivic.  
\end{remark}

For a symmetric monoidal category $\s{E}$ let  $\Pic(\s{E})$ denote the full subcategory consisting on invertible objects. 

\begin{lemma}[{\cite[Lemma 2.1.11]{zavyalov2023poincare}}]
Let $\s{D}$ be a premotivic six functor formalism on $\s{C}$, $X\in \s{C}$ and $f:X\times \bb{A}^{1,\an}\to X$. Then the pullback functor  
\[
f^*:\Pic(\s{D}(X)) \to  \Pic(\s{D}(X\times \bb{A}^{1,\an})) 
\]
is fully faithful. 
\end{lemma}
\begin{proof}
The same proof of \textit{loc. cit.} applies. 
\end{proof}

\begin{definition}
\label{DefinitionDeformationNormalCone}
Let $\s{D}$ be a six functor formalism on $\s{C}$.  Let $f:X \to Z$ be a solid smooth morphism  and let $s:Z \to X$ be a Zariski closed immersion.

\begin{enumerate} 

\item We denote $C(f,s):= s^*f^! 1_{Z}\in \s{D}(Z)$. 

\item For a vector bundle $\s{F}$ over $X$ with projection $f:\bb{V}(\s{F})^{\an}\to X$ and zero section $s:X \to \bb{V}(\s{F})^{\an}$, we let $C_{X}(\s{F}):=C(f,s)$. 

\item Suppose that   $\s{O}_Z$ is a perfect $\s{O}_X$-module locally in the analytic topology. By Remark \ref{RemarkDeformationNormalCone} we can form the deformation to the normal cone 
\[
 Z\times \bb{P}^1 \to \widetilde{X}  \to Z \times \bb{P}^1
\]
living over $\bb{P}^1$. We let $D_{Z}(X)$ denote the pullback of $\widetilde{X}$ to $\bb{A}^{1,\an}=\bb{P}^1\backslash \{\infty\}$; we get maps 
\[
Z \times \bb{A}^{1,\an} \xrightarrow{\widetilde{s}} D_Z(X) \xrightarrow{\widetilde{f}} Z \times \bb{A}^{1,\an}. 
\]
\end{enumerate}
\end{definition}

\begin{prop}[{\cite[Proposition 4.2.6]{zavyalov2023poincare}}]
\label{PropDeformationNormalGeneralSix}
Suppose that the six functor formalism  $\s{D}$ over $\s{C}$ is premotivic. Let $f:X \to Z$ be a solid smooth morphism with section $s:Z \to X$ such that     $\s{O}_Z$ is a perfect $\s{O}_X$-module locally in the analytic topology. Then, in the notation of Definition \ref{DefinitionDeformationNormalCone}, the object  
\[
 \widetilde{s}^*\widetilde{f}^! 1_{Z \times \bb{A}^{1,\an}} \in \Pic(\s{D}(Z \times \bb{A}^{1,\an}))
\]
lies in the essential image of $\Pic(\s{D}(Z)$). 
\end{prop}
\begin{proof}
We perform the same series of reductions as in the proof of Theorem \ref{TheoSerreDuality}. In fact, we can assume that $Z $ is affinoid and replace $X$ by an open neighbourhood of the section of $X$. We can then assume that $X \to Z$ is standard solid smooth and that we have a standard solid \'etale map $X \to Z \times \bb{G}_{a,\sol}^d$. By further refining $Z$, we can even assume that the pullback of $Z \to X \to Z \times \bb{G}_{a,\sol}^d$ along $X \to Z\times \bb{G}_{a,\sol}$ is $Z$ itself, and reduce to the case where $X= Z\times \bb{G}_{a,\sol}^d$, see the proof of Theorem \ref{TheoSerreDuality}. By a change of coordinates,  we can suppose that $Z \to X$ is the zero section, and by base change that $Z=*$. This last case is covered in Step 3 of {\cite[Proposition 4.2.6]{zavyalov2023poincare}}.
\end{proof}

\begin{corollary}[{\cite[Corollary  4.2.7 and Theorem 4.2.8]{zavyalov2023poincare}}]
\label{CorollaryGeneralPoincare}
In the notation of the Proposition \ref{PropDeformationNormalGeneralSix}, let $\n{N}_{Z/X}^{\an}$ denote the  analytification of the normal cone of $Z$ in $X$. There is a natural equivalence  
\[
C(f,s) \cong C_Z(\n{N}_{X/Y}^{\an}). 
\]
Moreover, if $f: X \to Y $ is a solid smooth morphism, there is a natural equivalence 
\[
f^! 1_{Y} \cong C_{X}(\n{T}_{X/Y}^{\an}) \in \s{D}(X),
\]
where $\n{T}_{X/Y}^{\an}$ is the analytification of the tangent space of $X$ over $Y$. 
\end{corollary}
\begin{proof}
The same proof of \textit{loc. cit.} applies. 
\end{proof}

\begin{theorem}[Poincar\'e duality for $D$-modules]
\label{TheoPoincareDualityDmodules}
Let $f: X \to Y$ be a solid smooth morphism of derived Tate adic spaces of relative dimension $d$, and let $f_{dR^+}^{\alg}: X_{dR^+}^{\alg} \to Y_{dR^+}^{\alg}$ and  $f_{dR^+}: X_{dR^+} \to Y_{dR^+}$ be the associated maps of stacks. Then there are natural equivalences 
\[
f^{alg,!}_{dR^+} 1= \s{O}(-d).  
\]
and 
\[
f_{dR^+}^{!} 1 = \s{O}(-d)[2d]. 
\]
\end{theorem}
\begin{proof}
By Corollary \ref{CorollaryGeneralPoincare} it suffices to prove the theorem for a vector bundle $\s{F}$ over $X$. By further reducing to the universal case $X= */ \GL_d$ and $\s{F}=\St$, it suffices to prove it for the relative filtered de Rham stacks of $\bb{V}(\s{F})^{\an}$ over $X$.   Let $f: \bb{V}(\s{F})^{\an} \to X$ be the natural projection and $s: X\to  \bb{V}(\s{F})^{\an}$ the zero section.  We have a natural equivalence 
\[
\n{T}_{\bb{V}(\s{F})^{\an}/X} = \bb{V}(\s{F})^{\an}\times_X \bb{V}(\s{F})^{\an}
\] 
provided by the group structure of $\bb{V}(\s{F})^{\an}$. This implies that $f^{!}_{dR^+} 1 \cong f^* s^* f^!_{dR^+} 1$ (resp. for $f^{alg,!}_{dR^*} 1$). In particular, by $\bb{A}^{1,\an}$-invariance, we have that $f_{dR^+,*} f^{!}_{dR^+} 1 = s^* f^{!}_{dR^+} 1$ (resp. for $f_{dR^+}^{\alg}$).  

\textbf{Case of $f_{dR^+}^{\alg}$}.  We have a natural equivalence 
\[
\bb{V}(\s{F})_{dR^+,X}^{an,alg} =( \bb{V}(\s{F})^{\an}\times \bb{A}^{1,\an}/\bb{G}_m)/ \widehat{\bb{V}(\s{F})}(-1).
\]
Indeed, by Lemma \ref{LemmaAlgdeRhamKeyCases} there is a natural equivalence $\bb{G}_{a,dR^+}^{d,\alg}= (\bb{G}^d_a\times \bb{A}^1/\bb{G}_m)/ \widehat{\bb{G}}^d_a(-1)$, this isomorphism is clearly $\GL_d$-equivariant (eg. looking at the level of points) and it descends to an equivalence over the stack $*/\GL_d$; by base change one deduces the general case. The map $f_{dR^+,X}^{\alg}$ factors through 
\[
\bb{V}(\s{F})_{dR^+,X}^{an,alg} \xrightarrow{h} (X \times \bb{A}^{1}/ \bb{G}_m)/ \widehat{\bb{V}(\s{F})}(-1) \xrightarrow{g} X \times \bb{A}^{1}/ \bb{G}_m,
\]
 we find that 
 \[
 s^*f^{\alg, !}_{dR^+,X} 1 = s^*h^* g^! 1  \otimes  s^*h^! 1. 
 \]
 By Theorems \ref{TheoSerreDuality} and \ref{PropAlgebraicCartier2} we obtain that 
 \[
 s^*f^{\alg,!}_{dR^+,X} 1 = \bigwedge^d \s{F}^{\vee}[d]  \otimes  \bigwedge^d \s{F}  (-d) [-d] =  \s{O}(-d). 
 \]

\textbf{Case of $f_{dR^+}$}.   We have a natural equivalence 
\[
\bb{V}(\s{F})_{dR^+,X}^{\an} = \bb{V}(\s{F})^{\an}/ \bb{V}(\s{F})(-1)^{\dagger}. 
\]
Indeed, by Lemma \ref{LemmaKeyCasesAnalyticdeRham} there is a natural equivalence $\bb{G}_{a,dR^+}^{d,\an}= (\bb{G}^{d,\an}_a\times \bb{A}^{1,\an}/\bb{G}_m^{\an})/ \bb{G}^{d}_a(-1)^{\dagger}$, this isomorphism is clearly $\GL_d$-equivariant (eg. looking at the level of points) and it descends to an equivalence over the stack $*/\GL_d$; by base change one deduces the general case. The map $f_{dR^+,X}$ factors trough 
\[
\bb{V}(\s{F})_{dR^+,X}^{\an} \xrightarrow{h} (X \times \bb{A}^{1}/ \bb{G}_m)/ \bb{V}(\s{F})(-1)^{\dagger} \xrightarrow{g} X \times \bb{A}^{1}/ \bb{G}_m,
\]
we get
 \[
 s^*f^{ !}_{dR^+,X} 1 =s^* h^* g^! 1  \otimes s^* h^! 1. 
 \]
 By Theorems \ref{TheoSerreDuality}  and \ref{TheoremAnalyticCartierII} and Proposition \ref{PropAlgebraicCartier2}, we find that 
 \[
 s^*f^{!}_{dR^+,X} 1 = \bigwedge^d \s{F}^{\vee}[d] \otimes \bigwedge^d \s{F} (-d) [d]  = \s{O}(-d)[2d].
 \]
\end{proof}

\subsection{Analytic de Rham stack of rigid  spaces}
\label{SubsectionAnalyticdeRhamRigid}

Let $(K,K^+)$ be a non archimedean extension of $\bb{Q}_p$. We finish with the study of the de Rham stack for rigid spaces over $(K,K^+)$. We thank Alberto Vezzani for the questions that motivated this section. From now on all the analytic de Rham stacks are relative to $\AnSpec (K,K^+)_{\sol}$.

The main goal of the section is to prove the following theorem:

\begin{theorem}
\label{TheoremSurjectionXdRRigid}
Let $X$ be an adic space locally of finite type over $(K,K^+)$, then the morphism $X\to X_{dR}$ is a $\s{D}$-cover of Tate stacks. Futhermore, if $X$ is quasi-compact then $X\to X_{dR}$ is a descendable $\s{D}$-cover. In particular, we have that 
\[
\Mod(X_{dR})=\varprojlim_{[n]\in \Delta} \Mod(\Delta^{n+1}(X)^{\dagger} ),
\]
for both $*$ and $!$-pullbacks,  where $\Delta^{n+1}(X)^{\dagger}\subset X^{n+1}$ is the overconvergent neighbourhood of the locally closed diagonal map, obtained as the \v{C}ech nerve of $X\to X_{dR}$.  
\end{theorem}

\begin{lemma}
\label{LemmaSmoothLocus}
Let $X$ be a reduced and irreducible adic space locally of finite type over $(K,K^+)$. Then, locally in the analytic topology, there is an open Zariski subspace $U\subset X$ where $\bb{L}_{U/K}$ is a projective $\s{O}_U$-module. 
\end{lemma}
\begin{proof}
We can assume without loss of generality that $X=\AnSpec(A,A^
+)_{\sol}$.  By Noether's normalization lemma for rigid spaces (\cite[\S 2.2 Corollary 11]{BoschRigidFormalGeo}), there is a Tate algebra $B=K\langle T_1,\ldots, T_d\rangle$ and an injective and finite map $B\to A$. Let $\eta \in \Spec B(*)$ be the generic point of the underlying discrete ring of $B$ and $\kappa(\eta)$ its residue field, then $A\otimes_{B(*)}\kappa(\eta) $ is finite over   $B\otimes_{B(*)} \kappa(\eta)$, and  the underlying discrete ring of the last is a field. Then, the underlying discrete ring of $A\otimes_{B(*)}\kappa(\eta)$ is a finite field extension of $\kappa(\eta)$. By noetherian approximation we can find an element $b\in B(*)$ such that $A[\frac{1}{b}]$ is a finite \'etale extension of $B[\frac{1}{b}]$, in particular $\bb{L}_{A[\frac{1}{b}]/K}$ is projective. We can then take $U$ to be  the analytification of the space $\AnSpec A[\frac{1}{b}]$. 
\end{proof}

\begin{lemma}
\label{LemmaDaggerSmoothHigherRankPoints}
Let $(A,A^+)$ be an Huber pair with $A$ a Tate algebra of finite type over $K$. Suppose that $X':=\AnSpec (A,A^{0})_{\sol}$ is a solid smooth rigid space over $K$, then $X=\AnSpec (A,A^{+})_{\sol}$ is $\dagger$-smooth locally in the analytic topology of $X$. Furthermore, if $X'$ is standard smooth then $X$ is $\dagger$-smooth. 
\end{lemma}
\begin{proof}
By Theorem \ref{TheoFormalSmoothnesvsSmoothness} the space $X'$ is solid smooth if and only if locally in the analytic topology it is solid standard smooth. Then, we can assume without loss of generality that $(A,A^{\circ})$ is solid standard smooth. Let us write $A=K\langle T_1,\ldots, T_d\rangle/(f_1,\ldots, f_e)$ a standard smooth presentation of $A$, we want to show that $\AnSpec(A,A^{+})_{\sol}$ is $\dagger$-smooth.  Since $\Spa (A,A^{+})$ is an analytic open subspace of $\Spa (A,K^{+})$, it suffices to consider the case when $A^{+}$ is the open integral closure of $K^+$ in $A$. We can then  write 
\[
X\xrightarrow{f} \bb{D}^{d-e,/K}_{K}\xrightarrow{g} \AnSpec(K,K^{+})_{\sol},
\]
where $\bb{D}^{d-e,/K}_{K}=\AnSpec (K\langle T_{e+1},\ldots, T_d \rangle,K^+)_{\sol}$.  Then, it suffices to show that $f$ is $\dagger$-\'etale and that $g$ is $\dagger$-smooth. The fact that $g$ is $\dagger$-smooth follows from the fact that the polynomial algebra is a compact projective analytic ring, and by invariance of the bounded condition for analytic rings of Proposition \ref{PropLiftingOverconvergentAlgebrasLift}. The proof that $f$ is $\dagger$-formally \'etale follows exactly the same argument of Proposition \ref{PropFormallyInftSmoothEtale}, we left the details to the reader. 
\end{proof}

\begin{lemma}
\label{LemmaDescendabledevisage}
Let $X$ be a reduced affinoid adic space  of finite type over $(K,K^+)$, and suppose that there is a locally Zariski open subspace $U\subset X$ with reduced complement $Z$ such that $U\to U_{dR}$ and $Z\to Z_{dR}$ are descendable, then $X\to X_{dR}$ is descendable. 
\end{lemma}
\begin{proof}
Let us write $X=\AnSpec(A,A^+)$ and let $I\subset A$ be the ideal of definition of $Z$. Let $X^{\dagger/Z}= \AnSpec (A^{\dagger/Z},A^+)$ be the $\dagger$-formal completion of $X$ at $Z$. We have a morphism of Tate stacks
\[
Z\to X^{\dagger/Z} \to Z_{dR}. 
\]
Since $Z\to Z_{dR}$ is desendable, then $X^{\dagger/Z}\to Z_{dR}$ is also descendable. On the other hand, we have an excision sequence of de Rham stacks
\[
j_{dR}:U_{dR}\subset X_{dR} \supset Z_{dR}: \iota_{dR},
\]
we then have a fiber sequence
\begin{equation}
\label{eqDevisagedeRhamStacks}
j_{dR,!} 1_{U_{dR}} \to 1_{X_{dR}}\to \iota_{dR,*} 1_{Z_{dR}}. 
\end{equation}
We also have an excision 
\[
j:U\subset X \supset X^{\dagger/Z}:\iota
\]
giving rise to a fiber sequence
\[
j_{!} 1_{U}\to 1_X\to \iota_* \iota_{*} 1_{X^{\dagger/Z}}. 
\]
Note that we have cartesian squares
\[
\begin{tikzcd}
U \ar[r]\ar[d,"f_U"]& X \ar[d,"f_X"] & X^{\dagger/Z} \ar[l]\ar[d,"f_{X^{\dagger/Z}}"] \\ 
U_{dR} \ar[r] & X_{dR} & Z_{dR} \ar[l].
\end{tikzcd}
\]
This shows that $\iota_{dR}^{*}f_{X} 1_{X}=f_{X^{\dagger/Z},*} 1_{X^{\dagger}/Z}$. Then, by the projection formula of $\iota_{dR,*}$ and since $f_{X^{\dagger/Z}}$ is descendable, one deduces that $\iota_{dR,*} 1_{Z_{dR}}$ belongs to the thick tensor ideal generated by $f_{dR,*} 1_X$. Similarly, $j_{dR,!} (f_{U,*} 1_U)$ is in the thick tensor ideal of $f_{X,*} 1_X$ (being the fiber of $f_{X,*}1_X \to \iota_{dR,*} f_{X^{\dagger/Z},*} 1_{X^{\dagger}/Z}$), and by the projection formula and descendability of $f_U$, then so is $j_{dR,!} 1_{U,dR}$. Therefore, by \eqref{eqDevisagedeRhamStacks}, we get that $1_{X_{dR}}$ is in the thick tensor ideal of $f_{X,*} 1_X$, proving that $X\to X_{dR}$ is descendable by \cite[Definition 3.18]{MathewDescent}.
\end{proof}

\begin{proof}[Proof of Theorem \ref{TheoremSurjectionXdRRigid}] By taking affinoid covers, we can assume without loss of generality that $X$ is quasi-compact and separated.  Writing $X$ as union of  reduced irreducible spaces, we can assume  by Lemma \ref{LemmaDescendabledevisage} that $X$ is irreducible.  Moreover, let $\n{I}$ be the nilpotent radical of $X$, since $X$ is quasi-compact $\n{I}$ is nilpotent, and the map $X^{\red}\to X$ is descendable (see \cite[Proposition 3.35]{MathewDescent}). Therefore, we can assume that $X$ is reduced and irreducible. We proceed to prove the theorem by induction on the dimension of $X$, the zero dimensional case being trivial.

 By Lemma \ref{LemmaSmoothLocus} there is a locally open  Zariski subspace $U\subset X$ such that $\bb{L}_{X/K}$ is a projective $\s{O}_U$-module, we let $Z$ be its Zariski closed complement. By induction in the dimension, $Z\to Z_{dR}$ is descendable, then by Lemma \ref{LemmaDescendabledevisage} it suffices to show that $U\to U_{dR}$ is descendable.

  Let $U'\subset U$ be the maximal rigid space contained in $U$, by Theorem \ref{TheoFormalSmoothnesvsSmoothness} $U'$ is solid smooth over $K$, and Lemma \ref{LemmaDaggerSmoothHigherRankPoints} implies that $U$ is $\dagger$-smooth locally in the analytic topology. Then, Proposition \ref{PropAnDeRhamStack} and Theorem \ref{TheoHodgeFiltration} show that $U\to U_{dR}$ is descendable thanks to the  Hodge filtration of the de Rham complex.  This finish the proof of the theorem.
\end{proof}

\begin{corollary}
Let $X$ be an adic space  locally of finite type over $(K,K^+)$ and let $j:X'\subset X$ be maximal rigid space contained in $X$. Let $j_{dR}:X'_{dR}\to X_{dR}$ be the associated maps at the level of de Rham stacks. Then the natural map
\[
1_{X_{dR}}\xrightarrow{\sim} j_{dR,*}1_{X'_{dR}}
\]  
is an equivalence. 
\end{corollary}
\begin{proof}
We can assume without loss of generality that $X=\AnSpec(A,A^+)$ is affinoid, so that $X'=\AnSpec (A,A^0)$.  By Theorem  \ref{TheoremSurjectionXdRRigid} we have $\s{D}$-covers $f:X\to X_{dR}$ and $g:X'\to X'_{dR}$. Consider the \v{C}ech nerves $(\Delta^{n+1}(X)^{\dagger})_{[n]\in \Delta}$ and $(\Delta^{n+1}(X')^{\dagger})_{[n]\in \Delta}$ of $f$ and $g$ respectively, where $\Delta^{n}(Z)^{\dagger}$ is the overconvergent neighbourhood of the diagonal $\Delta^{n}(Z)\subset Z^n$. We have a map of simplicial affinoid spaces
\[
j^{\bullet}:\Delta^{\bullet+1}(X')^{\dagger}\to \Delta^{\bullet+1}(X)^{\dagger}. 
\]
Note that the underlying condensed rings of $\Delta^{n+1}(X')^{\dagger}$ and $\Delta^{n+1}(X)^{\dagger}$ are the same: this follows from the fact that the underlying condensed rings of $X^{'n+1}$ and $X^{n+1}$ are the same for all $n\in \bb{N}$, and that the overconvergent diagonal is defined using the same ideal of definition. This implies that $j^{\bullet}_* 1_{\Delta^{\bullet+1}(X')}=1_{\Delta^{\bullet+1}(X)}$ is a cocartesian section. Then, since
\[
\Mod(X_{dR})=\varprojlim_{[n]\in \Delta} \Mod(\Delta^{n+1}(X)^{\dagger})
\]
and
\[
\Mod(X'_{dR})=\varprojlim_{[n]\in \Delta} \Mod(\Delta^{n+1}(X')^{\dagger}),
\]
one deduces that $1_{X_{dR}}\xrightarrow{\sim} j_{dR,*} 1_{X'_{dR}}$ is an equivalence as wanted. 
\end{proof}

\section{Analytic de Rham stack and locally analytic representations}
\label{SectionAnalyticdRAndLocAn}

The last section of this paper concerns the relation between the analytic de Rham stack, the theory of locally analytic representations as in \cite{RRLocallyAnalytic,RJRCSolidLocAn2}, and the theory of equivariant $\wideparen{\n{D}}$-modules of \cite{ArdakovEquivariantD}. In \S \ref{SubsectionSmoothgroupoid} we introduce smooth $\dagger$-groupoids for derived Tate adic spaces. Geometric realizations of these kind of groupoids generalize the construction of the de Rham stack for solid smooth morphisms. Then, in \S \ref{SubsectionLocAnRep}, we use the notion of smooth $\dagger$-groupoid  together with actions of  $p$-adic Lie groups to give a very general notion of equivariant analytic $D$-module. 

\subsection{Smooth $\dagger$-groupoids}
\label{SubsectionSmoothgroupoid}

Different theories of $D$-modules over rigid spaces $X$ are built up from different epimorphisms of $\s{D}$-stacks $X\to X'$, equivalently, from different groupoid objects living over $X$. In the case of analytic $D$-modules, the kind of groupoid objects we encounter have a special shape, namely, they look like non-commutative deformations of the trivial group object $X\times \bb{G}_a^{\dagger,d}\to X$ for some $d\geq 1$, where $\bb{G}_a^{\dagger}=\AnSpec \bb{Q}_p\{T\}^{\dagger}$.  The previous observation leads us to the notion of a smooth $\dagger$-groupoid over a derived Tate adic space.

We start by briefly recalling the definition of a groupoid object in an $\infty$-category as well as some related notions. Then, we introduce smooth $\dagger$-groupoids on derived Tate adic spaces over $\bb{Q}_p$, and prove some cohomological properties of them. We end with some examples appearing in the theory of twisted $D$-modules of rigid spaces.

\subsubsection{Groupoids}
\label{Subsection-Groupoids}

\begin{definition}[{\cite[Definition 6.1.2.7]{HigherTopos}}]
Let $\s{C}$ be an  $\infty$-category with finite limits. A groupoid object on $\s{C}$ is a simplicial object $\n{G}:\Delta^{\op}\to \s{C}$ such that for all $[n]\in \Delta$ and all partition $[n]=S\bigcup S'$ with $S\cap S'=\{s\}$, the natural map 
\[
\n{G}([n])\to \n{G}(S)\times_{\n{G}(s)}\n{G}(S')
\]
is an equivalence. We let $\n{G}_{\bullet}$ denote the groupoid object in $\s{C}$. We call $\n{G}_0$ the objects of the groupoid, the map $d_1: \n{G}_1\to \n{G}_0$ the source map and $d_0:\n{G}_1\to \n{G}_0$ the target map. By an abuse of notation we say that $\n{G}$ is a groupoid over  $X:=\n{G}_0$, if $\s{C}$ admits geometric realizations we denote  
\[
X/\n{G}:=\varinjlim_{[n]\in \Delta^{\op}} \n{G}_{\bullet}.
\]
\end{definition}

Let $G$ be a group, a standard procedure to construct more groups from $G$ is to take quotients $G/H$ by normal subgroups. It turns out that being "normal" for a map of groups in higher category theory is not longer a property but additional datum:

\begin{definition}[Normal map of groupoids]
Let $\s{C}$ be an $\infty$-topos with effective epimorphisms and  let $\n{G}$ be groupoid over $X$ in $\n{C}$. Let $\n{H}$ be a group object over $X$ and $f:\n{H}\to \n{G}$ a morphism of groupoids over $X$ with geometric realizations $X/\n{H}\to X/\n{G}$. A \textit{normal quotient} of $f$ is the datum of a pullback  square in $\s{C}_{X/}$
\[
\begin{tikzcd}
X/\n{H} \ar[r] \ar[d] & X/\n{G} \ar[d] \\ 
X \ar[r] &  Y 
\end{tikzcd}
\]  
such that $X\to Y$ is an epimorphism. By an abuse of notation we let $\n{G}/\n{H}$ denote the groupoid associated to the epimorphism $X\to Y$. 
\end{definition}

\subsubsection{$\dagger$-groupoids}
\label{Subsection-daggerGroupoids}

We let $\s{C}=\AdicSp_{\bb{Q}_p}$ denote the category of derived Tate adic spaces over $\bb{Q}_p$.

\begin{definition}
\label{DefinitionSmoothDaggerGroupoids}
Let $X\in \s{C}$,  let $\n{G}$ be a groupoid in $\n{C}$ over $X$, and let $X/\n{G}$ be its geometric realization.

\begin{enumerate}
			\item We say that $\n{G}$ is a \textit{$\dagger$-groupoid} if the topological simplicial object $|\n{G}_{\bullet}|$ is the constant object $|X|$, i.e. if for all map $[n]\to [m]$ we have an homeomorphism of topological spaces $|\n{G}_m|\xrightarrow{\sim} |\n{G}_n|$. 
			 
			\item Let $\n{G}$ be a $\dagger$-groupoid over $X$. We say that $\n{G}$ is \textit{smooth of relative dimension $d$} if the target map $d_0:\n{G}_1\to X$ is, locally in the analytic topology of $X$, equivalent to the projection $X\times \bb{G}_a^{\dagger,d}\to X$.
					
\end{enumerate}

\end{definition}

\begin{remark}
Let $\n{G}$ be a $\dagger$-groupoid over $X$. The fact that any object $\n{G}_n$ has the same underlying topological space implies that we can localize the groupoid in the analytic topology of $X$. Namely, for any open subspace $U\subset X$, the preimages $U_n$ of the map $d_0: \n{G}_m \to X$ define a subsimplicial object $U_{\bullet}\subset \n{G}_{\bullet}$ that is clearly a groupoid over $U$.
\end{remark}

The following proposition implies that geometric realizations of smooth $\dagger$-groupoids have a well behaved theory of six functors.

\begin{prop}
\label{PropDescendableQuotientDaggerGRoup}
Let $\n{G}$  be a smooth $\dagger$-groupoid over $X$, then the natural map $f:X\to X/\n{G}$ is a descendable $\s{D}$-cover. 
\end{prop}
\begin{proof}
We have a cartesian diagram 
\[
\begin{tikzcd}
\n{G}_1 \ar[r,"d_0"] \ar[d,"d_1"] & X \ar[d] \\
X \ar[r] & X/\n{G}_1.
\end{tikzcd}
\]
Therefore, locally in the analytic topology of $X$, the arrow $X\to X/\n{G}$ has fibers given by $\bb{G}_a^{\dagger,d}$. Thus, we have a natural equivalence 
\[
f_{dR}:X_{dR, X/\n{G}} \xrightarrow{\sim} X/\n{G}
\]
and the de Rham cohomology $1_{X/\n{G}}= f_{dR,*} 1_{X_{dR, X/\n{G}}}$ is Hodge complete. By Proposition \ref{PropAnDeRhamStack},  $1_{X/\n{G}}$ has by Hodge graduation (after forgetting the weight)
\[
gr^i(1_{X/\n{G}})=f_* \Omega^i_{X/ (X/\n{G})}. 
\]
Since $\Omega^i_{X/ (X/\n{G})}$ is locally free in the analytic topology of $X$, one has descendability of $f$ as wanted.
\end{proof}

Let us now focus in the case of a smooth $\dagger$-group $G$ over a derived Tate adic space $X$, namely, group objects over $X$ that are in addition smooth $\dagger$-groupoids.  

\begin{lemma}
Let $X$ be a derived Tate adic space, $G$ a group object over $X$ in derived Tate adic spaces, and $e:X\to G$ the unit section. Then the co-lie complex $e^* \bb{L}_{G/X}$ has a natural structure of $G$-module defining an object $\ell_{G/X}\in \Mod_{\sol}(X/G)$. Moreover, if $f:X\to X/G$, then there is a natural equivalence $\bb{L}_{X/(X/G)}\cong f^* \ell_{X/G}$.
\end{lemma}
\begin{proof}
The diagonal map $G\to G\times G$ induces a morphism of classifying stacks 
\begin{equation}
\label{eqDiagonalMapClassStacks}
\Delta:X/G\to X/(G\times G).
\end{equation}
We have a natural equivalence $X/G\xrightarrow{\sim} G/(G\times G)$, where $X\to G$ is the unit map, $G\to G\times G$ is the diagonal map, and $G\times G$ acts on $G$ by $(g_1,g_2)\cdot g = g_1 g g_2^{-1}$. Then, the arrow \eqref{eqDiagonalMapClassStacks} is equivalent to the map 
\[
G/(G\times G)\to X/(G\times G).
\]
This implies that the underlying object of $\bb{L}_{\Delta}$ is precisely $\ell_{G/X}$, and that $G$ acts on $\ell_{G/X}$ by the adjoint action.  The last statement follows from the previous computation and the following cartesian diagram 
\[
\begin{tikzcd}
X \ar[r] \ar[d] & X/G  \ar[d,"h"]\\ 
X/G  \ar[r,"\Delta"]& X/(G\times G),
\end{tikzcd}
\]
where $h$ corresponds to the map of groups $(\id,e): G \to G\times G$.
\end{proof}

\begin{prop}
\label{PropDualizingSheafDaggerGRoupoid}
Let $G$ be a smooth $\dagger$-group over $X$ of relative dimension $d$. Then the map   $g:X/G \to X$ is cohomologically smooth and there is a natural equivalence $g^! 1_X \cong \bigwedge^d \ell_{G/X}^{\vee}[d]$. 
\end{prop}
\begin{proof}
Let $f:X\to X/G$, by Proposition \ref{PropDescendableQuotientDaggerGRoup} the map $g$ admits $!$-functors and is weakly cohomologically proper. In order to show that $g$ is cohomologically smooth we first compute the right adjoint $g^!$. The proof of Proposition \ref{PropDescendableQuotientDaggerGRoup} produced a Hodge filtration for the unit object $1_{X/G}$. By proper base change, we have Hodge filtrations for all $M\in\Mod_{\sol}(X/G)$ such that 
\[
\gr^i(M)= f_* f^*(\bigwedge^i \ell_{G/X} \otimes M).
\]
Let $N\in \Mod_{\sol}(X)$. The $\Hom$ space
\[
\iHom_X(g_* M, N)
\]
has a filtration with graded pieces 
\[
\begin{aligned}
gr^{-i}(R\Hom_X(g_* M, N)) & =R\Hom_X(g_*(f_*f^*(\bigwedge^i \ell_{G/X}\otimes M)),N) \\
& = R\iHom_X(f^*(\bigwedge^i \ell_{G/X}\otimes M), N) \\
& = R\Hom_{X/G}(M,  \bigwedge^i \ell^{\vee}_{G/X} \otimes f_* N)\\
& =  R\Hom_{X/G}(M, f_* f^*( \bigwedge^i \ell^{\vee}_{G/X} \otimes g^* N)).
\end{aligned}
\]
But this filtration is also induced by the dual of the Hodge-filtration of $g^*N$ which is nothing but a $\bigwedge^d \ell_{G/X}^{\vee}[d]$-twist of the Hodge filtration. This shows that there is a natural equivalence  $g^ !N \cong  \bigwedge^d \ell_{G/X}^{\vee}[d] \otimes g^* N$. We still need to prove that $g$ is cohomologically smooth, for this we employ Lemma \ref{KeyLemmaSmoothRetraction}. We let $\n{L}=g^! 1_{X}= \bigwedge^d \ell^{\ell}_{G/X}[d]$. The Hodge filtration gives rise a map $f_* 1_X \to \n{L}$, and the adjunction $g_*\n{L}=g_*g^! 1_X \to 1_X$ produces a splitting $ 1_X\to g_* \n{L} \to 1_X$. Thus, we have all the data and hypothesis needed in Lemma \ref{KeyLemmaSmoothRetraction}, proving that $g$ is cohomologically smooth. 
\end{proof}

\begin{example}
\label{ExampleDaggerGroupoids}
In the following we give some examples of smooth $\dagger$-groups and groupoids that appear in the theory of analytic $D$-modules.

\begin{enumerate}
\item Let $f:X\to Y$ be a solid smooth morphism of derived rigid spaces. Since $X\to Y$ is formally $\dagger$-smooth in the analytic topology of $Y$, (cf. Proposition \ref{PropFormallyInftSmoothEtale}), the map $X\to X_{dR^+,Y}$  is an epimorphism by Proposition \ref{PropFormallyInftSmoothEtale}. Then, by Proposition \ref{PropFormallyInftSmoothEtale} one deduces that the \v{C}ech nerve of $X\to X_{dR^+,Y}$ is equal to $(\Delta^{n+1,\dagger}_{Y}X )_{[n]\in \Delta^{\op}}$, where $\Delta^{n}_Y: X\to X^{\times_Y n}$ is the diagonal map, and where $\Delta^{n,\dagger}_{Y}X \subset X^{\times_Y n}$ is the immersion attached to the locally closed Zariski immersion $|\Delta^{n}_Y(X)| \subset |X^{\times_{Y} n}|$.  Then, locally in the analytic topology of $Y$ and $X$, the morphism $f$ is standard solid smooth and by taking a factorization $X\to Y\times \bb{G}_{a,\sol}^d\to Y$ with the first arrow being standard solid \'etale, one gets that $(\Delta^{n+1,\dagger}_{Y}X )_{[n]\in \Delta^{\op}}$ is a smooth $\dagger$-groupoid over $X$ by Lemma \ref{LemmaKeyCasesAnalyticdeRham} (1). Moreover, the previous description holds for any locally closed subspace  of $X$ in the sense of locale for the analytic topology.

\item Let $X$ be a derived Tate adic space and let $G$ be a group object over $X$ such that $G\to X$ is, locally in the analytic topology of $X$ and $G$, a locally closed subspace  of a solid smooth map. Let $\exp(\Lie G)^{\dagger}\subset G$
 be the locally closed subspace associated to the unit map $|X|\to |G|$. Then there is a natural equivalence  $G_{dR,X}= G/\exp(\Lie G)^{\dagger}$. Indeed, by (1) $G_{dR,X}$ is the geometric realization of the overconvergent diagonals of the \v{C}ech nerve of $G\to X$, but the  \v{C}ech nerve  of $G\to X$ is equivalent to the simplicial space $(G^{\times_X n+1})_{[n]\in \Delta^{\op}}$ that encodes the group structure of $G$, and the \v{C}ech nerve of the de Rham stack of $G$ corresponds to the subspace given by $(\exp(\Lie G)^{\dagger, \times_X n}\times G)_{[n]\in \Delta^{\op}}$, whose geometric realization is precisely $G/ \exp(\Lie G)^{\dagger}$.

\item  Let $(K,K^+)$ be a non-archimedean extension of $\bb{Q}_p$ and let $X$ be a rigid space over $(K,K^+)$, seen as a derived Tate adic space over $\AnSpec (K,K^+)_{\sol}$. Let $\f{L}$ be a $K$-linear Lie algebroid over $X$ (cf. \cite[\S 9.1]{ArdakovWadsleyDI})  which is locally finite free in the analytic topology of $X$. Let $U(\f{L})$ be its enveloping algebra over $\s{O}_X$ and $\n{D}(\f{L})$  its algebra of locally analytic distributions, i.e. the Fr\'echet completion of  \cite[\S 9.3]{ArdakovWadsleyDI}. The diagonal map $\f{L}\to \f{L}\oplus \f{L}$ defines a commutative co-algebra structure on $\n{D}(\f{L})$ compatible with its algebra structure, that endows $\n{D}(\f{L})$  with a Hopf-algebra structure over $K$. Let us fix the left $\s{O}_X$-action on $\n{D}(\f{L})$. Taking duals with respect to $\s{O}_X$ of the natural projection $\n{D}(\f{L})\to \n{D}(\f{L})/\n{D}(\f{L})(\f{L}) \cong \s{O}_X$ we get a morphism of commutative algebras $d^0: \s{O}_X\to  C^{\dagger}(\f{L})$. By fixing a basis of $\f{L}$ over $\s{O}_X$, the Poincar\'e-Birkhoff-Witt theorem implies that $C^{\dagger}(\f{L})$ is isomorphic to $\s{O}_X\{T_1,\ldots ,T_d\}^{\dagger}$, where $d$ is the rank of $\f{L}$ over $\s{O}_X$. On the other hand, the orbit map $d^1: \s{O}_X\to C^{\dagger}(\f{L})$ obtained by the action of $\f{L}$ on $\s{O}_X$ is also a morphism of commutative algebras. The natural map $\s{O}_X \to \n{D}(\f{L})$ induces an augmentation map $s:C^{\dagger}(\f{L})\to \s{O}_X$. Taking analytic spectrum over $X$ we end up with the data of a $(\leq 1)$-simplicial space
\begin{equation}
\label{eqGRoupoidExp}
\begin{tikzcd}
\exp(\f{L})^{\dagger} \ar[r,"d_0", shift left = 8pt] \ar[r,"d_1"', shift right = 8pt] & X \ar[l, "e"']
\end{tikzcd}
\end{equation}
with $\exp(\f{L})^{\dagger}= \AnSpec_X C^{\dagger}(\f{L})$.  It is not hard to see that the Lie algebra structure of $\f{L}$ defines a groupoid object structure on \eqref{eqGRoupoidExp}, we call this groupoid the \textit{exponential of $\f{L}$}. We also call $\Mod_{\sol}(X/\exp(\f{L})^{\dagger})$ the category of \textit{analytic $U(\f{L})$ or $\n{D}(\f{L})$-modules}.

\item In the notation of the previous point, let $X$ be a smooth rigid space over $(K,K^+)$. Then the tangent space $\n{T}_{X/K}$ has a natural structure of Lie algebroid over $X$. The exponential $\exp(\n{T}_{X/K})^{\dagger}$ is nothing but the \v{C}ech nerve of the de Rham stack of $X$. Indeed, it suffices to prove this locally in the analytic topology of $X$, and we can assume that we have an \'etale map towards a relative polydisc over $K$. By naturality under \'etale maps, it suffices to prove it for $\bb{G}_{a,\sol}^d$, which follows from the case of group objects of point (2).

\item Let $(K,K^+)$ be a non-archimedean extension of $\bb{Q}_p$ and let  $\bf{G}$ be a reductive group over $K$. Let $\bf{P}\subset \bf{G}$ be a parabolic subgroup, let $\bf{N}\to \bf{P}\to \bf{M}$ be the short exact sequence of its unipotent radical and the Levi quotient, and let $\Fl= \bf{P} \backslash \bf{G}$ be the flag variety.   For a group $\bf{H}$ we let $\f{h}$ denote its Lie algebra. There is a natural action of $\f{g}$ on $\Fl$ by derivations, this defines a Lie algebroid $\f{g}^{0}:=\s{O}_{\Fl}\otimes \f{g}$ over $\Fl$ whose exponential is the smooth $\dagger$ groupoid $\exp(\f{g})^{\dagger}\times \Fl \to \Fl$ induced by the natural multiplication. The Lie algebras $\f{n}$ and $\f{p}$ have a natural adjoint action by $\bf{P}$, and they define Lie algebroids $\f{n}^{0}\subset \f{p}^{0} \subset \f{g}^{0}$. In fact, these Lie algebroids act trivially on $\s{O}_{\Fl}$ and they are ideals of $\f{g}^{0}$, thus the associated smooth $\dagger$-groupoids $\exp(\f{n}^{0})^{\dagger}$ and $\exp(\f{p}^0)^{\dagger}$ are normal subgroups of $\exp(\f{g}^{0})^{\dagger}$ (they are actually normal subgroups of the bigger groupoid $\bbf{G}^{\an}\times \Fl\to\Fl$, where $\bbf{G}^{\an}$ is the analytitfication of $\bf{G}$ to a rigid space). The quotient $\f{g}^{0}/ \f{p}^{0}$ is the tangent space of $\Fl$, this implies that we have a fiber sequence 
\[
\Fl/\exp(\f{p}^{0})^{\dagger} \to \Fl/\exp(\f{g})^{\dagger}\to \Fl_{dR}. 
\]
On the other hand, we call $\Fl/(\exp(\f{g}^0/\f{n}^0))^{\dagger}$ the universal twisted analytic de Rham stack of $\Fl$, and call $\Mod_{\sol}(\Fl/(\exp(\f{g}^0/\f{n}^0))^{\dagger})$ the category of analytic universal twisted $D$-modules of $\Fl$.
\end{enumerate}
\end{example}

\subsection{$p$-adic Lie groups and  analytic $D$-modules}
\label{SubsectionLocAnRep}

We end this section with the relation between analytic $D$-modules, locally analytic representations of $p$-adic Lie groups, and the theory of equivariant twisted $\wideparen{\n{D}}$-modules. We need some notations.

\begin{definition}
Let $G$ be a $p$-adic Lie group. We let $G$, $G^{sm}$ and $G^{la}$  denote the analytic adic spaces obtained by sending a compact open subspace $U\subset G$ to the spaces $C(U,\bb{Q}_p)$, $C^{sm}(U,\bb{Q}_p)$ and $C^{la}(U,\bb{Q}_p)$ of continuous, locally constant, and locally analytic functions of $U$. 
\end{definition}

The following lemma provides a clean relation between the groups $G$, $G^{la}$ and $G^{sm}$.

\begin{lemma}
\label{LemmaRelationLieGroupsLocAn}
Let $G^{\dagger}\subset G$ be the closed immersion of locales corresponding to the unit section. Then $G^{\dagger}$ is a normal subgroup of $G$ and there is a natural equivalence $G/G^{\dagger}= G^{sm}$. Similarly, we have that $G^{la}_{dR}= G^{la}/\exp(\f{g})^{\dagger}=G^{sm}$.  
\end{lemma}
\begin{proof}
We can assume without loss of generality that $G$ is compact.  We can write $G^{\dagger}=\varprojlim_{1\in H \subset G} H$ where $H$ runs over all compact open subgroups of $G$. Then one finds that 
\[
G/G^{\dagger}=\varprojlim_H G/H = G^{sm}. 
\]

For the claim about $G^{la}$, by Example \ref{ExampleDaggerGroupoids} (2) we have that $G^{la}_{dR}=G^{la}/\exp(\f{g})^{\dagger}$. We can also write $\exp(\f{g})^{\dagger}= \varprojlim_H H^{la}$ where $H$ runs over all the open compact subgroups of $G$. One finds that 
\[
G^{la}_{dR}=\varprojlim_H G^{la}/H^{la} = \varprojlim_H G/H =G^{sm}. 
\]
\end{proof}

Next we show that the classifying stacks of $G$, $G^{la}$ and $G^{sm}$ have $!$-functors. 

\begin{prop}
Let $G$ be a $p$-adic Lie group. The maps $*\to */G$,  $*\to */G^{la}$  and $*\to */G^{sm}$ are $\s{D}$-covers. Furthermore, if $G$ is compact they are  descendable $\s{D}$-covers.
\end{prop}
\begin{proof}
Let $H\subset G$ be an open and compact subgroup, the natural map $*/H\to */G$ is fibered on $G/H$ which is discrete over $*$, so cohomologically \'etale. Thus, to show that $*\to */G$ is a $\s{D}$-cover it suffices to show that $*\to */H$ is a $\s{D}$-cover (resp. for $H^{la}$ and $H^{sm}$), so we can assume that $G$ is compact. Let us write $f$ for any of the projections of $*$ to the classifying stacks. In the case of $G^{sm}$, the object $f_* \bb{Q}_p$ is nothing but the algebra $C^{sm}(\bb{Q}_p)$ of smooth functions endowed with the left regular action. Since $C^{sm}(G,\bb{Q}_p)$  admits $\bb{Q}_p$ as an equivariant  direct summand, we get that $*\to */G^{sm}$ is descendable. Descendability for $G$ and $G^{la}$ follows from the Lazard-Serre resolution (\cite[Theorems 5.7 and 5.8]{RRLocallyAnalytic}), namely, the Lazard-Serre resolution is a long exact sequence of $\bb{Z}_{p,\sol}[G]$-modules
\[
0\to \bb{Z}_{p,\sol}[G]^{\dim G}\to \cdots \to \bb{Z}_{p,\sol}[G]\to \bb{Z}_p\to 0
\]
which by a theorem of Kolhaase  extends to a long exact sequence of the locally analytic distribution algebra of $G$:
\[
0\to \n{D}^{la}(G)^{\dim G} \to \cdots \to \n{D}^{la}(G) \to \bb{Q}_p\to 0.
\]
Taking duals with respect to $\bb{Q}_p$, we got long exact sequences of representations of $G$
\[
0\to \bb{Q}_p \to C(G,\bb{Q}_p)\to\cdots \to C(G,\bb{Q}_p)^{\dim G} \to 0
\]
and 
\[
0\to \bb{Q}_p \to C^{la}(G, \bb{Q}_p)\to \cdots \to C^{la}(G,\bb{Q}_p)^{\dim G}\to 0,
\]
which proves descendability of $*\to */G$ and $*\to */G^{la}$ respectively. 
\end{proof}

We now study cohomological properties of the classifying stacks of $G$ and $G^{la}$.

\begin{prop}
\label{PropCohoSmoothClassifyingG}
Let $G$ be a $p$-adic Lie group and consider the maps $f:*/G^{sm}\to *$,  $g:*/G \to *$ and $h: */G^{la}\to *$. Then $f$, $g$ and $h$ are cohomologically smooth,  both $g^! \bb{Q}_p$ and $h^! \bb{Q}_p$ are naturally isomorphic to $\bigwedge^{\dim G} \f{g} [d]$, and $f^! \bb{Q}_p= \delta_G$ is the unimodular character. 
\end{prop}
\begin{proof}
We can assume without loss of generality that $G$ is compact. Indeed, given $H$ a compact open subgroup of $G$, the \v{C}ech nerve $X_{\bullet}$ of the map $*/H\to */G$ is given by  $X_{n}=H\backslash G\times^H \cdots \times^{H} G/H$ ($n$-copies of $G$), and all the arrows $X_n\to X_m$ are cohomologically \'etale (resp. for $G^{la}$). Thus, all the  maps $g_n:X_n\to *$ would be cohomologically smooth and one has natural isomorphisms $g_n^! \bb{Q}_p= d_0^* g_H^! \bb{Q}_p$ where $g_H:*/H\to *$, proving that the object $g^! \bb{Q}_p$ is already determined by its restriction to $*/H$. An explicit but tedious bookkeeping of the maps in the \v{C}ech nerve will show that the action is the adjoint for $G$ and $G^{la}$, and the unimordular action for $G^{sm}$ (see \cite[Example 4.2.4]{HanKalWein}). 

Now let us suppose that $G$ is compact, we can even assume that $G$ is a uniform pro-$p$-group and fix a coordinate system $G\cong \bb{Z}_p^{d}$. We first deal with $f$. By \cite[Theorem 5.4.2]{RJRCSolidLocAn2} the category $\Mod(*/G^{sm})$ is equivalent to the category of solid smooth representations. In particular, $f_*$ is identified with the invariant functor which is exact, and has by right adjoint the formation of the trivial representations, namely, $f^*$. Then, Lemma \ref{KeyLemmaSmoothRetraction} can be applied with $\n{L}=f^* \bb{Q}$ being the trivial representation, proving that $f$ is cohomologically smooth.

 Finally, we deal with $g$ and $h$.  By the Lazard-Serre resolution, we know that both $g$ and $h$ are cohomologically smooth, namely, $g_*$ and $h_*$ are group cohomology, and their right adjoints are the trivial representation after twisting by a character, see \cite[Theorem 5.19]{RRLocallyAnalytic} (one can also apply Lemma \ref{KeyLemmaSmoothRetraction} with $\n{L}$ to the line bundle $\chi=\iHom_{G}(\bb{Q}_p, \bb{Q}_{p,\sol}[G])$). It is left to compute the dualizing sheaf $\chi$ of the classifying stacks.   Let $\bb{G}^{(h)}$ be the affinoid group consisting on finitely many disjoint affinoid polydiscs of radius $p^{-h}$ around the elements of $g\in G$, we also let $\overline{\bb{G}}^{(h)}=\varprojlim_{h'<h} \bb{G}^{(h')}$ be the overconvergent affinoid group of radius $p^{-h}$. Then, since the Lazard-Serre resolution is already extended for analytic distribution algebras (see proof of \cite[Theorem 4.4]{Kohlhaase}), letting $f: */\overline{\bb{G}}^{(h)}\to *$ and $k:*/G^{la}\to */\overline{\bb{G}}^{(h)}$, we have that $h^! \bb{Q}_p= k^* f^! \bb{Q}_p$ for some $h>>0$. Hence, it suffices to compute the dualizing sheaf of the map $f$. By taking the connected component of the identity, we are reduced to compute the dualizing sheaf of the classifying stack of an affinoid group $\overline{\bb{G}}$ whose underlying adic space is a closed polydisc of dimension $d$, this follows the same argument as Proposition \ref{PropDualizingSheafDaggerGRoupoid} obtaining $\bigwedge^{d} \Lie \overline{\bb{G}} [d]$ endowed with its adjoint action. This finishes the proof. 
\end{proof}

With the previous preparations we can finally define equivariant $D$-modules on derived Tate adic spaces.

\begin{definition}
\label{DefinitionEquivariantDmodules}
Let $X\to Y$ be a morphism of derived Tate adic spaces over $\bb{Q}_p$, let $G$ be a $p$-adic Lie group and suppose that we have an action of $G^{la}$ on $X$ over $Y$. Let $\exp(\f{g}^0)^{\dagger}$ denote the groupoid over $X$ obtained by the restriction of the action of $G^{la}$ to $\exp(\f{g})^{\dagger}$. Let $\bb{H}^{\dagger}\to \exp(\f{g}^{0})^{\dagger}$ be a normal morphism of groupoids such that the composite $ \bb{H}^{\dagger}\to G^{la}$ is also normal. We define the category of analytic equivariant $D(G^{la}/\bb{H}^{\dagger})$-modules to be $\Mod_{\sol}(X/(G^{la}/\bb{H}^{\dagger}))$. 
\end{definition}

Finally, the following theorem computes dualizing sheaves for equivariant analytic $D$-modules of solid smooth morphisms.

\begin{theorem}
\label{TheoDializingEquivariantDmod}
Let $X\to Y$ be a solid smooth morphism of derived Tate adic spaces over $\bb{Q}_p$ of relative dimension $d$, and let $G$ be a $p$-adic Lie group of dimension $g$ acting locally analytically   on $X$ over $Y$. Let us denote  $\f{g}=\Lie G$.     Let $\bb{H}^{\dagger}$ be a $\dagger$-smooth group  over $X$ of relative dimension $e$,   let $\bb{H}^{\dagger}\to G^{la}\times X$  be a map of groupoids with given normal quotient $G^{la}/\bb{H}^{\dagger}$. Then $g: X/(G^{la}/\bb{H}^{\dagger})\to Y$ is cohomologically smooth and its underlying $G^{la}$-equivariant dualizing sheaf is equivalent to 
\[
g^!1_{Y}=\Omega^{d}_{X/Y}[d]\otimes \bigwedge^{g} \f{g} [g]\otimes  \bigwedge^{e} \ell_{\bb{H}^{\dagger}/X} [-e]. 
\] 
\end{theorem}
\begin{proof}
 By hypothesis, the map $h:X/G^{la}\to X/(G^{la}/\bb{H}^{\dagger})$ is an epimorphism fibered on $X/\bb{H}^{\dagger}$. Then, the pullback along $h$ is conservative and it is cohomologically smooth by Proposition \ref{PropDualizingSheafDaggerGRoupoid}. Therefore, $h$ is a smooth $\s{D}$-cover and by Corollary \ref{CorollaryDescentSmoothProperCovers} $g$ is cohomologically smooth if $g\circ h$ is so. On the other hand, we can write 
 \[
 X/G^{la}\xrightarrow{f} Y/G^{la} \xrightarrow{k} Y,
 \]
 the map $f$ is representable by a solid smooth map so it is  cohomologically smooth, and the map $k$ is cohomologically smooth by Proposition \ref{PropCohoSmoothClassifyingG}. Finally, it is left to compute the pullback  of $g^{!}\bb{Q}_p$ along $h$ . Since $g\circ h= k \circ f$ we find that 
 \[
 f^{*} k^{!} 1_Y \otimes f^{!} 1_{Y/G^{la}}= h^{*} g^{!} 1_Y \otimes h^{!} 1_{ X/(G^{la}/\bb{H}^{\dagger})}.
 \] 
Therefore, 
\[
h^{*} g^{!} 1_Y =f^{*} k^{!} 1_Y \otimes f^{!} 1_{Y/G^{la}} \otimes (h^{!} 1_{ X/(G^{la}/\bb{H}^{\dagger})})^{-1}.
\]
The theorem follows since $k^{!} 1_Y = \bigwedge^{g} \f{g}[g]$ by Proposition \ref{PropCohoSmoothClassifyingG}, $f^{!}1_{Y/G^{la}} =\Omega^{d}_{X/Y}[d]$ by Theorem \ref{TheoSerreDuality}, and $h^{!} 1_{ X/(G^{la}/\bb{H}^{\dagger})}=\bigwedge^{e}\ell_{\bb{H}^{\dagger}/X}^{\vee}[e]$ by Proposition \ref{PropDualizingSheafDaggerGRoupoid} since $h$ is fibered on $X/\bb{H}^{\dagger}$. 
\end{proof}

\begin{example}
 Let $X\to Y$ be a smooth morphism of rigid spaces and let $G$ be a $p$-adic Lie group acting on $X$ over $Y$. The action of $G$ on $X$ is locally analytic and extends to an action of $G^{la}$. Let $\alpha: \s{O}_X\otimes \f{g}\to \n{T}_{X/Y}$ be the anchor map, and let us assume that it is surjective. Let $\f{k}^0=\ker(\alpha)$, then $\f{k}^0$ is a locally finite free Lie algebroid over $X$ acting trivially on $\s{O}_X$, and it defines a group object $\bb{H}^{\dagger}$ over $X$. Furthermore, since $\f{k}^{0}$ is a $G$-equivariant sheaf, the map $\bb{H}^{\dagger}\to G^{la}\times X$ is a normal map of $1$-groupoids and we can perform the groupoid quotient $G^{la}/\bb{H}^{\dagger}$. Then, the category of analytic $D(G^{la}/\bb{H}^{\dagger})$-modules will be an enhancement of the category of equivariant $\wideparen{\n{D}}(X,G)$-modules of \cite{ArdakovEquivariantD}. A concrete relation between these two categories  is left to a future work.  
\end{example}

\bibliographystyle{alpha}
\bibliography{BiblioDerivedAdic}

\begin{thebibliography}{RJRC23}

\bibitem[And21]{Andreychev}
Grigory Andreychev.
\newblock Pseudocoherent and perfect complexes and vector bundles on analytic
  adic spaces.
\newblock \url{https://arxiv.org/abs/2105.12591}, 2021.

\bibitem[Aok23]{aoki2023sheavesspectrum}
Ko~Aoki.
\newblock The sheaves-spectrum adjunction.
\newblock \url{https://arxiv.org/abs/2302.04069}, 2023.

\bibitem[Ard21]{ArdakovEquivariantD}
Konstantin Ardakov.
\newblock Equivariant {$\widehat{\mathcal{D}}$}-modules on rigid analytic
  spaces.
\newblock {\em Ast\'{e}risque}, (423):161, 2021.

\bibitem[AW18]{ArdakovWadsleyII}
Konstantin Ardakov and Simon Wadsley.
\newblock {$\widehat{\mathcal{D}}$}-modules on rigid analytic spaces {II}:
  {K}ashiwara's equivalence.
\newblock {\em J. Algebraic Geom.}, 27(4):647--701, 2018.

\bibitem[AW19]{ArdakovWadsleyDI}
Konstantin Ardakov and Simon~J. Wadsley.
\newblock {$\widehat{\mathcal{D}}$}-modules on rigid analytic spaces {I}.
\newblock {\em J. Reine Angew. Math.}, 747:221--275, 2019.

\bibitem[Bha22]{BhattGauges}
Bhargav Bhatt.
\newblock Prismatic {$F$}-gauges.
\newblock
  \url{https://www.math.ias.edu/~bhatt/teaching/mat549f22/lectures.pdf}, 2022.

\bibitem[Bod21]{bode2021operations}
Andreas Bode.
\newblock Six operations for {D}-cap-modules on rigid analytic spaces.
\newblock \url{https://arxiv.org/abs/2110.09398}, 2021.

\bibitem[Bos14]{BoschRigidFormalGeo}
Siegfried Bosch.
\newblock {\em Lectures on formal and rigid geometry}, volume 2105 of {\em
  Lect. Notes Math.}
\newblock Cham: Springer, 2014.

\bibitem[CS19]{ClausenScholzeCondensed2019}
Dustin Clausen and Peter Scholze.
\newblock Lectures on {C}ondensed {M}athematics.
\newblock \url{https://www.math.uni-bonn.de/people/scholze/Condensed.pdf},
  2019.

\bibitem[CS20]{ClauseScholzeAnalyticGeometry}
Dustin Clausen and Peter Scholze.
\newblock Lectures on {A}nalytic {G}eometry.
\newblock \url{https://www.math.uni-bonn.de/people/scholze/Analytic.pdf}, 2020.

\bibitem[CS22]{CondensedComplex}
Dustin Clausen and Peter Scholze.
\newblock Condensed {M}athematics and {C}omplex {G}eometry.
\newblock \url{https://people.mpim-bonn.mpg.de/scholze/Complex.pdf}, 2022.

\bibitem[FS21]{FarguesScholze}
Laurent Fargues and Peter Scholze.
\newblock Geometrization of the local {L}anglands correspondence.
\newblock \url{https://arxiv.org/abs/2102.13459}, 2021.

\bibitem[GR14]{GRdeRhamStack}
Dennis Gaitsgory and Nick Rozenblyum.
\newblock Crystals and {D}-modules.
\newblock \url{https://people.mpim-bonn.mpg.de/gaitsgde/GL/Crystalstext.pdf},
  Oct 2014.

\bibitem[GR23]{gaitsgory2023dg}
Dennis Gaitsgory and Nick Rozenblyum.
\newblock {DG} {I}ndschemes.
\newblock \url{https://arxiv.org/abs/1108.1738}, 2023.

\bibitem[HKW22]{HanKalWein}
David Hansen, Tasho Kaletha, and Jared Weinstein.
\newblock On the {Kottwitz} conjecture for local shtuka spaces.
\newblock {\em Forum Math. Pi}, 10:79, 2022.
\newblock Id/No e13.

\bibitem[Hub93]{HuberValuations}
Roland Huber.
\newblock Continuous valuations.
\newblock {\em Math. Z.}, 212(3):455--477, 1993.

\bibitem[Hub94]{HuberAdicSpaces}
Roland Huber.
\newblock A generalization of formal schemes and rigid analytic varieties.
\newblock {\em Math. Z.}, 217(4):513--551, 1994.

\bibitem[Hub96]{HuberEtaleCohomology}
Roland Huber.
\newblock {\em \'{E}tale cohomology of rigid analytic varieties and adic
  spaces}.
\newblock Aspects of Mathematics, E30. Friedr. Vieweg \& Sohn, Braunschweig,
  1996.

\bibitem[Koh11]{Kohlhaase}
Jan Kohlhaase.
\newblock The cohomology of locally analytic representations.
\newblock {\em J. Reine Angew. Math.}, 651:187--240, 2011.

\bibitem[Lau96]{laumon1996transformation}
Gerard Laumon.
\newblock Transformation de fourier generalisee.
\newblock \url{https://arxiv.org/abs/alg-geom/9603004}, 1996.

\bibitem[Lur04]{LurieDerivedAlgebraic}
Jacob Lurie.
\newblock {\em Derived algebraic geometry}.
\newblock ProQuest LLC, Ann Arbor, MI, 2004.
\newblock Thesis (Ph.D.)--Massachusetts Institute of Technology.

\bibitem[Lur09]{HigherTopos}
Jacob Lurie.
\newblock {\em Higher topos theory}, volume 170 of {\em Ann. Math. Stud.}
\newblock Princeton, NJ: Princeton University Press, 2009.

\bibitem[Lur17]{HigherAlgebra}
Jacob Lurie.
\newblock Higher algebra.
\newblock 2017.

\bibitem[LZ22]{lu_zheng_2022}
Qing Lu and Weizhe Zheng.
\newblock Categorical traces and a relative lefschetz--verdier formula.
\newblock {\em Forum of Mathematics, Sigma}, 10:e10, 2022.

\bibitem[Man22a]{MannSix2}
Lucas Mann.
\newblock The 6-functor formalism for $\mathbb{Z}_\ell$- and
  $\mathbb{Q}_\ell$-sheaves on diamonds.
\newblock \url{https://arxiv.org/abs/2209.08135}, 2022.

\bibitem[Man22b]{MannSix}
Lucas Mann.
\newblock A $p$-adic 6-{F}unctor {F}ormalism in {R}igid-{A}nalytic {G}eometry.
\newblock \url{https://arxiv.org/abs/2206.02022}, 2022.

\bibitem[Mao21]{MaoCrystalline}
Zhouhang Mao.
\newblock Revisiting derived crystalline cohomology.
\newblock \url{https://arxiv.org/abs/2107.02921}, 2021.

\bibitem[Mat16]{MathewDescent}
Akhil Mathew.
\newblock The {Galois} group of a stable homotopy theory.
\newblock {\em Adv. Math.}, 291:403--541, 2016.

\bibitem[RJRC22]{RRLocallyAnalytic}
Joaqu{\'{\i}}n Rodrigues~Jacinto and Juan~Esteban Rodr{\'{\i}}guez~Camargo.
\newblock Solid locally analytic representations of {{\(p\)}}-adic {Lie}
  groups.
\newblock {\em Represent. Theory}, 26:962--1024, 2022.

\bibitem[RJRC23]{RJRCSolidLocAn2}
Joaqu\'in Rodrigues~Jacinto and Juan~Esteban Rodr\'iguez~Camargo.
\newblock Solid locally analytic representations.
\newblock \url{https://arxiv.org/abs/2305.03162}, 2023.

\bibitem[Sch15]{ScholzeTorsion2015}
Peter Scholze.
\newblock On torsion in the cohomology of locally symmetric varieties.
\newblock {\em Ann. of Math. (2)}, 182(3):945--1066, 2015.

\bibitem[Sch23]{SixFunctorsScholze}
Peter Scholze.
\newblock Six-{F}unctor {F}ormalisms.
\newblock \url{https://people.mpim-bonn.mpg.de/scholze/SixFunctors.pdf}, 2023.

\bibitem[Sim96]{SimpsonDeRham}
Carlos Simpson.
\newblock Homotopy over the complex numbers and generalized de {R}ham
  cohomology.
\newblock In {\em Moduli of vector bundles ({S}anda, 1994; {K}yoto, 1994)},
  volume 179 of {\em Lecture Notes in Pure and Appl. Math.}, pages 229--263.
  Dekker, New York, 1996.

\bibitem[ST97]{SimpsonTelemandeRham}
Carlos Simpson and Constantin Teleman.
\newblock The de {R}ham theorem for $\infty$-stacks.
\newblock \url{https://math.berkeley.edu/~teleman/math/simpson.pdf}, 1997.

\bibitem[ST03]{SchTeitDist}
Peter Schneider and Jeremy Teitelbaum.
\newblock Algebras of {$p$}-adic distributions and admissible representations.
\newblock {\em Invent. Math.}, 153(1):145--196, 2003.

\bibitem[{Sta}22]{stacks-project}
The {Stacks project authors}.
\newblock The stacks project.
\newblock \url{https://stacks.math.columbia.edu}, 2022.

\bibitem[Tam15]{Tamme}
Georg Tamme.
\newblock On an analytic version of {L}azard's isomorphism.
\newblock {\em Algebra Number Theory}, 9(4):937--956, 2015.

\bibitem[Zav23]{zavyalov2023poincare}
Bogdan Zavyalov.
\newblock Poincar\'e {D}uality {R}evisited.
\newblock \url{https://arxiv.org/abs/2301.03821}, 2023.

\end{thebibliography}

\end{document}